\documentclass[12pt,a4paper,english]{book}
\pdfoutput=1
\usepackage[a4paper,left=30mm,right=30mm,top=32mm,bottom=32mm]{geometry}
\usepackage{authblk,amsmath,amsfonts,amssymb,amsthm}
\usepackage{graphicx}
\usepackage[doublespacing]{setspace}
\usepackage{epsfig}
\usepackage{nicefrac}
\usepackage{mathrsfs}
\usepackage{mdframed}
\usepackage{framed}
\usepackage{enumitem} 
\usepackage{fancyhdr}
\usepackage{dsfont}
\usepackage{placeins}
\usepackage{verbatim}
\usepackage{amsthm} 
\usepackage{pdfpages,graphicx}
\usepackage{mathtools}
\usepackage{longtable}
\usepackage{stmaryrd}
\usepackage{standalone}
\usepackage{tikz}
\usetikzlibrary{trees,calc,positioning,decorations.markings,arrows,backgrounds}
\usetikzlibrary{shapes, shapes.geometric, fit}
\usepackage{tikz-cd}
\usepackage{bm}
\usepackage{booktabs}
\usepackage{datetime}
\usepackage{tabularx}
\usepackage{subcaption}

\usepackage{pdfpages}

\usepackage[colorlinks, linkcolor = black, citecolor = black, filecolor = black, urlcolor = blue]{hyperref}
\usepackage{titlesec}
\titleformat{\section}[block]{\Large\bfseries\filcenter}{\thesection}{1em}{}
\newdateformat{monthyeardate}{%
  \monthname[\THEMONTH] \THEYEAR}
  
\newcommand{\red}[1]{\textcolor{red}{#1}}
\newcommand{\blue}[1]{\textcolor{blue}{#1}}

\newcommand{\Ha}{\mathrm{Ha}}

\newtheorem{theorem}{Theorem}
\newtheorem{remark}{Remark}

\newtheorem{lemma}{Lemma}
\newtheorem{problem}{Problem}

\newtheorem{algorithm}{Algorithm}

\newcommand\grad{\operatorname{grad}}
\renewcommand\div{\nabla \cdot }
\newcommand\curl{\operatorname{curl}}

\newcommand{\scurl}{\curl}
\newcommand{\del}{\partial}
\newcommand{\cross}{\times}
\DeclareMathOperator{\vcurl}{\mathbf{curl}}

\newcommand{\Hhc}{\mathbf{H}^{h}_{0}(\curl, \Omega)}
\newcommand{\Ltz}{L^2_0(\Omega)}
\newcommand{\Hhd}{\mathbf{H}^{h}_{0}(\operatorname{div}, \Omega)}
\newcommand{\Hd}{{H}(\mathrm{div}; \Omega)}
\newcommand{\Hc}{{H}(\mathrm{curl}; \Omega)}

\newcommand{\Ned}{N\'{e}d\'{e}lec\ }

\newcommand{\R}{\mathbb{R}}
\newcommand{\hcurl}{\mathbf{H}(\curl)}
\newcommand{\hdiv}{\mathbf{H}(\operatorname{div})}
\newcommand{\hdivz}{\mathbf{H}(\operatorname{div},0)}
\newcommand{\hzcurl}{\mathbf{H}_0(\curl, \Omega)}
\newcommand{\hzdiv}{\mathbf{H}_0(\operatorname{div}, \Omega)}
\newcommand{\Hozv}{\mathbf{H}^1_0(\Omega)}
\newcommand{\Hoz}{H^1_0(\Omega)}

\renewcommand\S{{S}}

\newcommand{\ou}{\overline{U}}

\newcommand{\eps}[1]{\ensuremath{\varepsilon(#1)}}
\newcommand\x{\mathbf{x}}
\newcommand\A{{\mathcal A}}

\renewcommand{\Re}{\ensuremath{\mathrm{Re}}}
\newcommand{\Rem}{\ensuremath{\mathrm{Re_m}}}
\newcommand{\Reminv}{\ensuremath{\mathrm{Re_m^{-1}}}}
\newcommand{\RH}{\ensuremath{\mathrm{R_H}}}
\renewcommand{\Pr}{\ensuremath{\mathrm{Pr}}}
\newcommand{\Pm}{\ensuremath{\mathrm{Pm}}}
\newcommand{\Ra}{\ensuremath{\mathrm{Ra}}}

\newcommand{\Qc}{\mathbb{Q}_{c}}
\newcommand{\Qd}{\mathbb{Q}_{d}}

\newcommand{\MM}{\mathcal{M}}
\newcommand{\ZZ}{\mathcal{Z}}
\newcommand{\UU}{\mathcal{U}}
\newcommand{\NN}{\mathcal{N}}
\newcommand{\II}{\mathcal{I}}
\newcommand{\JJ}{\mathcal{J}}
\newcommand{\CC}{\mathcal{C}}
\newcommand{\DD}{\mathcal{D}}
\renewcommand{\AA}{\mathcal{A}}
\newcommand{\BB}{\mathcal{B}}
\newcommand{\GG}{\mathcal{G}}
\newcommand{\FF}{\mathcal{F}}
\newcommand{\KK}{\mathcal{K}}
\newcommand{\LL}{\mathcal{L}}

\newcommand{\B}{\mathbf{B}}
\newcommand{\E}{\mathbf{E}}
\renewcommand{\u}{\mathbf{u}}
\renewcommand{\v}{\mathbf{v}}
\renewcommand{\j}{\mathbf{j}}
\renewcommand{\k}{\mathbf{k}}
\newcommand{\F}{\mathbf{F}}
\newcommand{\C}{\mathbf{C}}
\newcommand{\n}{\mathbf{n}}

\newcommand{\zerov}{\mathbf{0}}
\newcommand{\f}{\mathbf{f}}
\newcommand{\V}{\mathbf{V}}
\newcommand{\Rr}{\mathbf{R}}
\newcommand{\W}{\mathbf{W}}

\makeatletter
\renewcommand*\env@matrix[1][\arraystretch]{%
	\edef\arraystretch{#1}%
	\hskip -\arraycolsep
	\let\@ifnextchar\new@ifnextchar
	\array{*\c@MaxMatrixCols c}}
\makeatother

\newcommand\scalemath[2]{\scalebox{#1}{\mbox{\ensuremath{\displaystyle #2}}}}

\DeclareUnicodeCharacter{202F}{\,}

\renewcommand{\Re}{\mathrm{Re\,}}

\def\XXint#1#2#3{{\setbox0=\hbox{$#1{#2#3}{\int}$}
     \vcenter{\hbox{$#2#3$}}\kern-.5\wd0}}

\numberwithin{equation}{section}
\newcommand*{\ldblbrace}{\{\mskip-5mu\{}
\newcommand*{\rdblbrace}{\}\mskip-5mu\}}

\begin{document}
\raggedbottom 
\pagestyle{empty}

\begin{center}
	\LARGE{\textbf{
			Discretisations and Preconditioners for Magnetohydrodynamics Models
	}}
\end{center}
\vspace{1cm}
\begin{center}
 \includegraphics[scale=0.30]{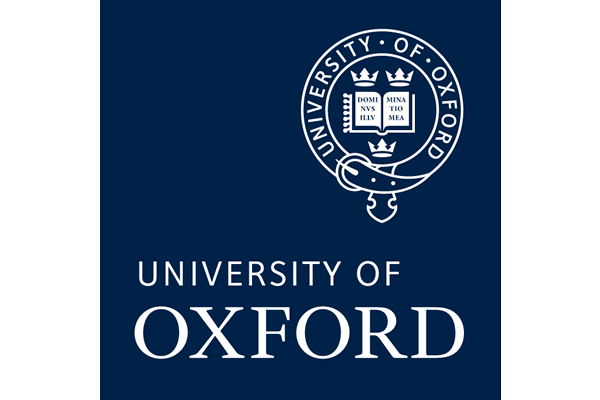}
\end{center}
\vspace{0.5cm}
\begin{center}
\large{Keble College \\University of Oxford}
\end{center}
\vspace{1cm}
\begin{center}
{\Large{A thesis submitted for the degree of\\
Doctor of Philosophy}}
\end{center}


\begin{center}
by
\end{center}

\begin{center}
\large{\textbf{Fabian Laakmann}} \\
\end{center}

\begin{center}
\large{Trinity 2022}
\end{center}

\newpage
\pagestyle{empty}
\section*{Abstract}
	Magnetohydrodynamics (MHD) models describe the behaviour of electrically conducting fluids such as astrophysical and laboratory plasmas or liquid metals in the presence of magnetic fields. They are generally known to be difficult to solve numerically, due to their highly nonlinear structure and the strong coupling between the electromagnetic and hydrodynamic variables, especially for high Reynolds and coupling numbers.
	
	In the first part of this work, we present a scalable augmented Lagrangian preconditioner for a finite element discretisation of the $\B$-$\E$ formulation of the incompressible viscoresistive MHD equations. For stationary problems, our solver achieves robust performance with respect to the Reynolds and coupling numbers in two dimensions and good results in three dimensions. We extend our method to fully implicit methods for time-dependent problems which we solve robustly in both two and three dimensions. Our approach relies on specialised parameter-robust multigrid methods for the hydrodynamic and electromagnetic blocks. The scheme ensures exactly divergence-free approximations of both the velocity and the magnetic field up to solver tolerances.
	We confirm the robustness of our solver by numerical experiments in which we consider fluid and magnetic Reynolds numbers and coupling numbers up to 10,000 for stationary problems and up to 100,000 for transient problems in two and three dimensions.
	
	In the second part, we focus on incompressible, resistive Hall MHD models and derive structure-preserving finite element methods for these equations. These equations incorporate the Hall current term in Ohm's law and provide a more appropriate description of fully ionized plasmas than the standard MHD equations on length scales close to or smaller than the ion skin depth. In particular, we present a variational formulation of Hall MHD that enforces the magnetic Gauss's law precisely (up to solver tolerances) and prove the well-posedness of a Picard linearisation. For the transient problem, we present time discretisations that preserve the energy and magnetic and hybrid helicity precisely in the ideal limit for two types of boundary conditions. Additionally, we investigate an augmented Lagrangian preconditioning technique for both the stationary and transient cases. Finally, we confirm our findings by several numerical experiments.
	
	In the third part, we investigate anisothermal MHD models. We start by performing a bifurcation analysis for a magnetic Rayleigh--B\'enard problem at a high coupling number $S=1{,}000$ by choosing the Rayleigh number in the range between 0 and $100{,}000$ as the bifurcation parameter. We study the effect of the coupling number on the bifurcation diagram and outline how we create initial guesses to obtain complex solution patterns and disconnected branches for high coupling numbers. Moreover, we extend the parameter-robust augmented Lagrangian preconditioner for the standard MHD equations to the anisothermal case. Again, we obtain excellent robustness with respect to the Rayleigh number, Prandtl number, magnetic Prandtl number and coupling number in two dimensions and good robustness in three dimensions. We verify our finding by reporting iteration numbers for a magnetic double glazing problem and a magnetic cooling channel problem.
	
    \newpage
 
\section*{Acknowledgement}

First, I want to thank Patrick Farrell, who has done an incredible job in proposing and supervising my project. 
 You have a remarkable talent in fascinating people for mathematics, explaining complicated things in an understandable way and  coming up with new ideas when one gets stuck with his project. You really support your students, constantly check on them and offer your support. You always respond quickly and know exactly what is going on in your students' projects. And on top of that you are also a very likeable person. 

Furthermore, I want to express my gratitude to Kaibo Hu for being such a helpful person over the years. Since most of my work is based on your fundamental contributions to the field of finite elements methods for MHD, I am convinced that this thesis would not have been possible in this form without your input. 

I also want to thank Lawrence Mitchell for his continuous help with implementations in Firedrake and, in general, for all the contributions to Firedrake in the past years. I'm very glad that I had access to this software and how (relatively) easy it made it to produce all the numerical results included in this thesis. 

Further thanks belongs to Prof.~Ben Schweizer for fascinating me about the world of partial differential equations and all the incredible effort he puts into his lectures and students. Similarly, I have to thank Prof.~Stefan Turek and Prof.~Sandra May for laying my foundation in all kinds of numerical methods for partial differential equations. 

Special thanks also belongs to all my friends in Oxford, who were mainly responsible for giving me such an awesome time in Oxford. I want to explicitly mention: Aili, Chris, Christoph, Jingmin, Pablo, Tian, Timo and Tommaso.

Finally, I want to thank all my friends at home, above all Florian, as well as my brother and my parents. Without your continuous support and unconditional love none of this would have been possible. 

\newpage
\includepdf[pages=-]{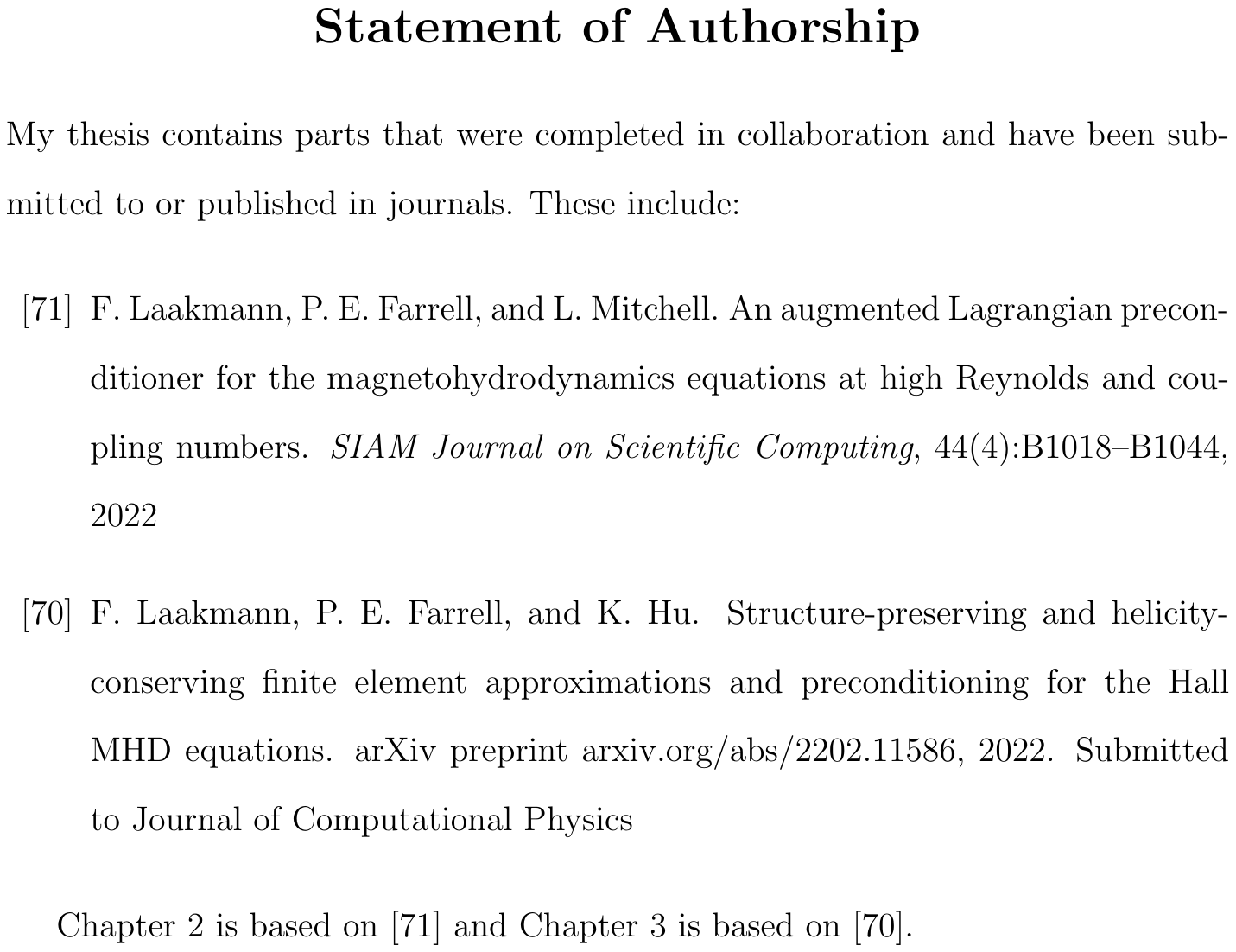}

\tableofcontents
\listoftables
\listoffigures
\pagestyle{fancy}
	
\chapter{Introduction}
\section{The magnetohydrodynamics equations}
Magnetohydrodynamics (MHD) models describe the flow of electrically conducting fluids in the presence of electromagnetic fields. They have numerous important applications in astrophysics, geophysics, the liquid metal industry and thermonuclear fusion. Mathematically, they are described by the Navier--Stokes and Maxwell's equations which are coupled through the Lorentz force and Ohm's law. The development of numerical methods is an area of active research and known to be very challenging due to the highly nonlinear character of the system and the strong coupling between the hydrodynamic and electromagnetic variables.

The flow of conducting fluids can be mainly characterised by three dimensionless parameters, which are called the fluid Reynolds number $\Re$, the magnetic Reynolds number $\Rem$ and the coupling number $S$. The size of these numbers has important consequences for the physical behaviour of the fluid and electromagnetic field, but also for the development of numerical methods. Low fluid Reynolds numbers describe laminar flow, which is characterised by a smooth and constant fluid motion due to the domination of viscous forces. On the other hand, high fluid Reynolds numbers model turbulent flow driven by inertial forces that exhibit eddies, vortices and other irregular behaviour.
Small magnetic Reynolds number imply that the magnetic field is in a purely diffusive state and inhomogeneities in the field will be smoothed out. For high magnetic Reynolds numbers the magnetic field lines are mainly advected by the fluid flow. The size of the coupling number $S$ quantifies the strength of the coupling between the electromagnetic and hydrodynamic unknowns.

Most applications are in the regime of high Reynolds and coupling numbers and hence it is of great interest to build robust solvers that work in these parameter regimes.
For liquid metals, the fluid Reynolds number $\Re$ tends to be much larger than $\Rem$. For example, the flow of liquid mercury is characterized by a 
ratio of $10^6$ between these two constants; typical values in aluminium electrolysis are $\Rem = 10^{-1}$ and $\Re = 10^5$ \cite{Gerbeau2006}. High 
magnetic Reynolds numbers occur on large length scales, as in geo- and astrophysics. The magnetic Reynolds number of the outer Earth's core is in the range of $10^3$ and of the sun is in the range of $10^6$ \cite{Davies2015}. Magnetic Reynolds numbers between $10^1-10^3$ have been used in several dynamo experiments that investigate planetary magnetic fields \cite{Molokov2007}. The coupling number $S$ is around $10^0$ for aluminium electrolysis \cite{Gerbeau2006} and Armero \& Simo \cite{Armero1996} define strong coupling for $S$ in the range of $10^2-10^9$.

From a mathematical point of view, high Reynolds numbers imply a strong nonlinear coupling between the fluid and magnetic field, while singular terms start to dominate the equations. This setting demands the use of appropriate discretisations, linearisation schemes and globalisation techniques and the construction of highly specialised linear solvers that can capture the kernel of the singular terms.

\section{MHD models}
MHD models can be derived under specific assumptions from the Vlasov-Maxwell equations which provide a general kinetic description of collisionless plasmas.
These assumptions include that the considered fluid is quasi-neutral, the considered length scales are larger than the kinetic length scales such as the gyro radius or the ion skin depth and the considered time scales are shorter than the cyclotron frequencies of the ions and electrons. Moreover, displacement currents, electric forces, electron inertia and relativistic effects are neglected \cite{Gerbeau2006}. 

These assumptions are common to all MHD models. Depending on the physical setting or application additional assumptions or simplifications are often considered. Ideal MHD models \cite[Section 2.4.3]{Galtier2015} further neglect the resistivity of the fluid and hence assume that the fluid is perfectly conducting (which formally corresponds to an infinite magnetic Reynolds number). This results in the so-called frozen-in condition, where the magnetic field lines are tied to the fluid flow. Electron MHD models \cite[Section 2.4.2]{Galtier2015} describe the limit of small time and space scales where the velocity of the electron is much higher than that of the ions. In this case, the ions build a neutralising background and the electric current is only carried by the electrons. The Hall MHD equations \cite[Section 2.4]{Galtier2015} which include an additional current term in Ohm's law are described in more detail in the next section. A further extension is given by multi-fluid models that track the separate movement of ions, electrons and neutral species. 
 
Depending on the application, the standard MHD equations can describe compressible or incompressible fluids and isothermal or anisothermal systems with homogeneous or spatially-varying physical parameters. In this work, we focus on incompressible MHD models that include visco-resistive effects for homogeneous physical parameters. The most common applications for these models are liquid metals, but, e.g., also the modelling of solar winds is well approximated in the incompressible regime since the occurring density fluctuations are usually small \cite[Section 2.4.1]{Galtier2015}. On the other hand, compressible MHD models provide a more suitable description of plasmas. The first two main chapters deal with isothermal MHD models, while the last main chapter is devoted to the Boussinesq approximation for anisothermal magnetic convection problems.

\section{Hall MHD equations}
The Hall MHD equations extend standard MHD models by including the so-called Hall effect.
This provides a more appropriate description of fully ionized plasmas on length scales close to or smaller than the ion skin depth \cite{Galtier2015}. On these length scales the Hall MHD equations take into account the different motions of ions and electrons in a two-fluid approach. While the electron motion is frozen to the magnetic field in this regime, it remains to solve for the velocity of the centre of mass $\u$ \cite{Huba2003}. 
The Hall MHD equations can be used to describe many important plasma phenomena, such as magnetic reconnection processes \cite{Forbes1991, Morales2005}, the expansion of sub-Alfv\'{e}nic plasma \cite{Ripin1993} and the dynamics of Hall drift waves and Whistler waves \cite{Huba2003}.  

The essence of the Hall effect is described by adding the Hall-term $\j \times \B$, the cross product of the current density and magnetic field, in the generalised Ohm's law \cite[Section 2.2.2]{Galtier2015}.
Hence, the Hall-MHD equations only differ by the Hall term $\RH\,\j \times \B$ with the Hall parameter $\RH$ from the standard MHD equations. Nevertheless, the extension of existing theory and algorithms for the standard MHD equations is highly non-trivial, since the Hall term represents the highest order term in the final system of equations given by $\RH \nabla \times ((\nabla \times \B)\times \B))$. Furthermore, the current density $\j$ cannot be eliminated with the help of Ohm's law and has to be kept in the formulation.

Several analytical results for the continuous Hall MHD problem \cite{Chae2014, Danchin2020} and computational results of physical simulations \cite{Gmez2008, Chacn2003, Tth2008} are available in the literature.

\section{Anisothermal MHD models}\label{sec:IntroAnisothermalMHDmodels}

There are several ways to incorporate temperature dependence in MHD models. Here, we mainly focus on the Boussinesq approximation which assumes that the flow is bouyancy driven and that density differences only appear in the buoyancy term, while the remaining unknowns are assumed to neither depend on the density nor temperature. This results in an additional term in the velocity equation of the form $\Pr\, \Ra\, \sigma\, \mathbf{g}$ where $\Pr$ denotes the Prandtl number, $\Ra$ the Rayleigh number, $\sigma$ the temperature unknown and $\mathbf{g}$ a unit vector in the direction of the gravity. Anisothermal formulations usually include $\Ra$, $\Pr$  and the magnetic Prandtl number $\Pm$ as parameters instead of the fluid and magnetic Reynolds numbers $\Re$ and $\Rem$. 

The Prandtl number describes the ratio of kinematic viscosity and thermal diffusivity. It is solely determined by the fluid and its state and does not depend on the problem configuration or length scale. Typical values for liquid metals are 0.065 for lithium at 975 Kelvin, 0.015 for mercury at room temperature and 0.003 for potassium at 975 Kelvin \cite{Coulson1999}. 

The magnetic Prandtl number corresponds to the quotient $\Rem/\Re$. For liquid-metals this number tends to be quite small, for example, a typical value for liquid mercury is $10^{-7}$, for aluminium electrolysis is $10^{-6}$, for liquid sodium is $10^{-4}$, while $\Pm\approx 1$ occurs in tokamaks \cite{Davies2015, Galtier2015}.  

The Rayleigh number encodes the importance of the bouyancy-driven natural convection in comparison to viscous dissipation and heat conduction. Depending on the application, the Rayleigh number usually varies between $10^3-10^9$, but can also reach values up to $10^{20}$, e.g., when the Earth's core is considered.

\section{Structure and contribution of thesis}
This thesis is split into three main chapters. Chapter \ref{chap:2} deals with robust solvers for the $\B$-$\E$ formulation of the stationary and transient standard incompressible resistive MHD equations.
The main contribution of this chapter is to provide block preconditioners with good convergence even at high Reynolds and coupling numbers. The performance relies on the following three (novel) approaches:
\begin{itemize}
	\item[1.)] We consider a fluid-Reynolds-robust augmented Lagrangian preconditioner for an $\hdiv\times L^2$-discretisation of the Navier--Stokes equations that relies on a specialised multigrid method.
	\item[2.)]  We introduce a new monolithic multigrid method for the electromagnetic block.
	\item[3.)]  We discover that using the outer Schur complement which eliminates the $(\u,p)$ block instead of the $(\B, \E)$ block has crucial advantages for ensuring robustness for high parameters.
\end{itemize}
Furthermore, we show that our preconditioners extend in a straightforward manner to time-dependent versions. This has the substantial advantage that the choice of the time-stepping scheme is no longer restricted by the ability to solve the linear systems. In particular, it allows the use of fully implicit methods for high Reynolds numbers and coupling parameters.

In Section~\ref{sec:MHDModel}, we introduce the MHD model that we mainly focus on in this work.
In Section~\ref{sec:discretization}, we derive an augmented Lagrangian formulation and describe the finite element discretisation and linearisation schemes. In Section~\ref{sec:derivationofblockpreconditioner}, we introduce block preconditioners for these schemes, present a calculation of the corresponding (approximate) Schur complements and describe robust linear multigrid solvers for the different blocks. Numerical examples and a detailed description of the algorithm are presented in Section~\ref{sec:numericalresults}. 
\newline

In Chapter \ref{chap:3}, we introduce structure-preserving finite element discretisations for the Hall MHD equations and derive conservative algorithms that preserve energy and helicity precisely on the discrete level in the ideal limit.
Here the main contribution includes the following results:
\begin{enumerate}
	\item[1.)] We provide a variational formulation and structure-preserving discretisation for the stationary and time-dependent Hall-MHD equations and prove a well-posedness results for a Picard type linearisation.
	\item[2.)] We construct numerical schemes that preserve the energy, magnetic helicity and hybrid helicity precisely in the ideal limit of $\Re=\Rem=\infty$.
	\item[3.)] We investigate parameter-robust preconditioners and report corresponding iteration numbers.
\end{enumerate}

In Section \ref{sec:vatiational-formulation}, we derive a variational formulation of the stationary Hall MHD system and prove the well-posedness of a Picard type iteration. In Section \ref{sec:timedepproblems}, we derive algorithms that preserve the energy, magnetic and hybrid helicity precisely in the ideal limit. We investigate an augmented Lagrangian preconditioner for the Hall MHD system in Section \ref{sec:ALP}. Finally, we present numerical results in Section \ref{sec:numericalresultsHall}, which include iterations numbers for a lid-driven cavity problem, the simulation of magnetic reconnection for an island coalescence problem and a numerical verification of the conservation properties for our algorithms in the ideal limit.
\newline

In Chapter 4, we investigate anisothermal MHD models by performing a bifurcation analysis for a magnetic Rayleigh-B\'enard problem and deriving a parameter-robust preconditioner for these models. The main results of this chapter include:

\begin{itemize}
	\item[1.)] We show how to create bifurcation diagrams over the bifurcation parameter $\Ra$ at a high coupling number of $S=1{,}000$. We investigate the dependence of the coupling number on the bifurcation analysis, observe how increasing the coupling number can stabilise unstable branches and describe how we discover disconnected branches.
	\item[2.)] We extend the parameter-robust preconditioner that was developed for the standard MHD equations to the anisothermal case and report iteration numbers for two test problems.
	
\end{itemize}

In Section \ref{sec:AnIsoFormulationAndDiscretisation}, we introduce the strong and weak formulation of our anisothermal model based on the Boussinesq approximation and outline our finite element approximation. In Section \ref{sec:BifurcationAnalysis}, we perform the bifurcation analysis for the magnetic Rayleigh-B\'enard problem over $0 \leq \Ra \leq 100{,}000$ at $S=1{,}000$ and describe in detail the evolution and stability of our discovered branches and how we obtained suitable initial guesses that allowed us to compute these solutions. Finally, we describe in Section \ref{sec:BlockPreconAniso} how we construct the augmented Lagrangian preconditioner for the anisothermal model and underline our findings with numerical tests.

\section{Notation}
Throughout this work, we assume that $\Omega\subset \R^d$, $d \in \{2,3\}$ is a bounded Lipschitz polygon or polyhedron that is simply connected in order to ensure that the de Rham complexes, considered later in this work, are exact. We use the convention that vector-valued functions and function spaces are denoted by bold letters. We use $(\cdot,\cdot)$ and $\|\cdot\|$ (sometimes $\|\cdot\|_0$) to denote the $L^2(\Omega)$ inner product and norm. The dual pairing between an $H^{-1}$ (with norm $\|\cdot\|_{-1}$) and $H^1$ (with norm $\|\cdot\|_{1}$) function is denoted as $\langle \cdot , \cdot \rangle$.
We define the function spaces
\begin{align}
    \Hoz &= \{v \in H^1(\Omega) \, | \, v = 0 \text{ on } \partial \Omega \}, \\
    \hzdiv &= \{\C \in \mathbf{L}^2(\Omega) \, | \, \div \C \in \mathrm{L}^2(\Omega), \, \C \cdot \n = 0 \text{ on } \partial \Omega\}, \\
    \hzcurl &= \{\F \in \mathbf{L}^2(\Omega) \, | \, \vcurl \F \in \mathbf{L}^2(\Omega), \, \F \times \n = \mathbf{0} \text{ on } \partial \Omega\}, \\
    \mathrm{L}^2_0(\Omega) &= \{q \in L^2(\Omega) \, | \, \int_{\Omega} q\,  \mathrm{d}x = 0 \},
\end{align}
where $\n$ denotes the unit outward normal vector on the boundary of $\Omega$. In some cases, we also denote $\Hoz$ as $H_{0}(\grad, \Omega)$. Further, we drop the domain $\Omega$ in the notation of the function spaces if it is obvious which domain we consider.

In our formulations, $\u:\Omega\to\R^3$ denotes the velocity, $p:\Omega\to\R$ the fluid pressure, $\j:\Omega \to \R^3$ the current density, $\B:\Omega\to\R^3$ the magnetic field, $\E:\Omega\to\R^3$ the electric field and $\theta: \Omega \to \R$ the temperature. Furthermore, $\Re$ denotes the fluid Reynolds number, $\Rem$ the magnetic Reynolds number, $S$ the coupling number, $\RH$ the Hall parameter, $\Ra$ the Rayleigh number, $\Pr$ the Prandtl number, $\Pm$ the magnetic Prandtl number, $\f: \Omega \to \R^3$ a source term and $\eps \u = \frac{1}{2} (\grad \u + \grad \u^\top )$ the symmetric gradient. We regularly use the notation $\bm \omega = \nabla \times \u$ for the vorticity.

We often use properties of the following continuous de Rham complexes in two and three dimensions 
\begin{equation}\label{eq:contDC2d}
	\R \xrightarrow[]{\text{id}} H_0(\mathrm{curl}, \Omega) \xrightarrow[]{\vcurl} \hzdiv \xrightarrow[]{\mathrm{div}} L^2_0(\Omega) \xrightarrow[]{\text{null}} 0,
\end{equation}
\begin{equation}\label{eq:contDC3d}
	\R \xrightarrow[]{\text{id}} H_{0}(\grad, \Omega) \xrightarrow[]{\mathrm{grad}} \hzcurl\xrightarrow[]{\vcurl} \hzdiv \xrightarrow[]{\mathrm{div}} L^2_0(\Omega) \xrightarrow[]{\text{null}} 0.
\end{equation}
De~Rham complexes are called exact if the kernel of an operator in the sequence is given by the range of the preceding operator, e.g., $\mathrm{range}\ \mathrm{grad} = \mathrm{ker} \vcurl$ or $\mathrm{range} \vcurl = \mathrm{ker}\ \mathrm{div}$. Both sequences are exact for the simply connected domains we consider \cite{Arnold2006}.

For  numerical approximations, we use the finite element de~Rham sequences
\begin{equation}
	\label{eqn:derhamfe2d}
	\R\xrightarrow[]{\text{id}}  H^{h}_{0}(\scurl, \Omega) \xrightarrow[]{\vcurl} \Hhd \xrightarrow[]{\mathrm{div}}  L^{2}_{h}(\Omega)\xrightarrow[]{\text{null}} 0, 
\end{equation}
\begin{equation}
	\label{eqn:derhamfe}
		\R \xrightarrow[]{\text{id}}  H^{h}_{0}(\grad, \Omega) \xrightarrow[]{\mathrm{grad}}  \Hhc \xrightarrow[]{\vcurl} \Hhd \xrightarrow[]{\mathrm{div}}  L^{2}_{h}(\Omega) \xrightarrow[]{\text{null}}0, 
\end{equation}
to discretise the variables from \eqref{eq:contDC2d} and \eqref{eq:contDC3d}, where $H^{h}_{0}(D, \Omega)\subset H_{0}(D, \Omega), ~D = \grad, \curl, \operatorname{div}$ are conforming finite element spaces, see e.g.~Arnold, Falk, Winther
\cite{Arnold2006,Arnold2010},
Hiptmair \cite{Hiptmair.R.2002a}, Bossavit \cite{Bossavit.A.1998a} for
more detailed discussions on discrete differential forms. A concrete example for such  finite element de~Rham sequences are given by 
\begin{equation}\label{eq:deRhamdiscret2d}
	\mathbb{CG}_k \xrightarrow[]{\vcurl} \mathbb{RT}_k \xrightarrow[]{\mathrm{div}} \mathbb{DG}_{k-1} \xrightarrow[]{\text{null}} 0,
\end{equation} 
\begin{equation}\label{eq:deRhamdiscret3d}
	\mathbb{CG}_k  \xrightarrow[]{\mathrm{grad}} \mathbb{NED}1_k \xrightarrow[]{\vcurl} \mathbb{RT}_k \xrightarrow[]{\mathrm{div}} \mathbb{DG}_{k-1} \xrightarrow[]{\text{null}} 0.
\end{equation}
Here, $\mathbb{RT}_k$ denotes the Raviart--Thomas elements \cite{Raviart1977} of degree $k$, $\mathbb{NED}1_k$ the  \Ned elements of first kind \cite{Nedelec1980},  $\mathbb{CG}_k$ continuous Lagrange elements and $\mathbb{DG}_k$ discontinuous Lagrange elements.

 In the schemes presented in this work, we require that $\E_h, \j_h \in \Hhc$ and $\B_h \in \Hhd$, i.e.~that they are drawn from the same sequence. We denote the finite element spaces used for the velocity $\bm{u}_h$ and pressure $p_h$
by $\bm{V}_{h}$ and $Q_{h}$ respectively,
and assume that the choice is inf-sup stable
\cite{Girault.V;Raviart.P.1986a}.

In two dimensions, there exist two different curl operators \mbox{given by}
\begin{equation}\label{eq:curl2d}
	\scurl \B \coloneqq \del_x B_2 - \del_y B_1, \qquad  \vcurl E \coloneqq \begin{pmatrix}
		\del_y E\\
		-\del_x E
	\end{pmatrix},
\end{equation}
that correspond to the cross-products
\begin{equation}\label{eq:cross2d}
	\u \times \B \coloneqq u_1 B_2 - u_2 B_1, \qquad
	\B \times E \coloneqq \begin{pmatrix}
		B_2 E \\
		-B_1 E
	\end{pmatrix}.
\end{equation}
We also regularly use the notation $\nabla \cross$ to denote the curl operator in three dimensions. 

We define the weak $\curl$-operator $\tilde\nabla_{h}\times  : [L^{2}(\Omega)]^{3}\to \Hhc$ by
\begin{equation}\label{eq:weakcurl}
(\tilde \nabla_{h}\times \B_h, \k_{h})=(\B_h, \nabla\times \k_{h}) \quad\forall\, \k_{h}\in \Hhc.
\end{equation}
We regularly use the generalised Gaffney inequality
\begin{equation}
	\|\B_h\|_{L^{3+\delta}} \leq \|\tilde{\nabla} \times \B_h\| + \|\nabla \cdot \B_h\| \quad \forall \ \B_h \in \Hhd
\end{equation}
for  $0\leq \delta \leq3$, where $\delta$ depends on the regularity of $\Omega$. For a proof, we refer to \cite[Theorem 1]{he2019generalized} and references therein.

The interpolant of a function $u$ into a finite element space $V_h$ with a set of degrees of freedom $\{{\ell_{h,i}(\cdot)}\}$ and basis functions $\{\varphi_i\}$ is represented by
\begin{equation}
	\II^h_{V_h}(u) = \sum_i \ell_{h,i}(u) \varphi_i.
\end{equation}
		
\renewcommand{\grad}{\nabla}
\chapter{Robust solvers for standard magnetohydrodynamics models}\label{chap:2}
The content of this chapter was developed in collaboration with Patrick Farrell and Lawrence Mitchell. A manuscript \cite{laakmann2021} has been published in the SIAM Journal of Scientific Computing.
\vspace{-0.3cm}
\section{MHD model}\label{sec:MHDModel}
In this chapter, we consider the incompressible viscoresistive magnetohydrodynamics equations on a bounded polytopal Lipschitz domain $\Omega\subset\R^d$, $d \in \{2,3\}$. In the stationary three-dimensional setting, we investigate the formulation
\begin{subequations}
	\label{eq:MHD3d}
	\begin{align}
		- \frac{2}{\Re} \div \eps \u + \u \cdot \nabla \u + \nabla p + S\, \B \times (\E + \u \times \B) &= \f, \label{eq:MHD3d1}\\
		\div \u&=0, \label{eq:MHD3d2}\\
		\E + \u \times \B - \frac{1}{\Rem} \vcurl \B &= \mathbf{0}, \label{eq:MHD3d3}\\
		\vcurl \E &= \mathbf{0}, \label{eq:MHD3d4}\\
		\div \B &= 0. \label{eq:MHD3d5}
	\end{align}
\end{subequations}
We mainly consider the perfectly conducting boundary conditions
\begin{equation}\label{eq:boundarycond3d}
	\u=\mathbf{0}, \quad  \E \times \n = \mathbf{0}, \quad \B \cdot \n=0 \quad \text{ on } \del \Omega,
\end{equation}
although the treatment of the alternative boundary conditions
\begin{equation}\label{eq:boundarycond3dalt}
	\u=\mathbf{0}, \quad  \E \cdot \n = 0, \quad \B \times \n=\mathbf{0} \quad \text{ on } \del \Omega,
\end{equation}
is also possible, see \cite{hu2020convergence}. The term $\mathbf{f}$ on the right-hand side of \eqref{eq:MHD3d1} can, e.g., represent a force like gravity. The above formulation based on the electric and magnetic fields was first rigorously analysed by Hu et al.\ \cite{hu2020convergence}.

 It is straightforward to derive the two-dimensional formulation from this work, which is given by
\begin{subequations}
	\label{eq:MHD2d}
	\begin{align}
		- \frac{2}{\Re} \div \eps \u + \u \cdot \nabla \u + \nabla p + S\, \B \times  (E + \u \times \B) &= \f, \label{eq:MHD2d1}\\
		\div \u&=0, \label{eq:MHD2d2}\\
		E + \u \times \B - \frac{1}{\Rem} \scurl \B &= 0, \label{eq:MHD2d3}\\
		\vcurl E &= \mathbf{0}, \label{eq:MHD2d4}\\
		\div \B &= 0, \label{eq:MHD2d5}
	\end{align}
\end{subequations}
subject to the boundary conditions
\begin{equation}\label{eq:boundarycond2d}
	\u=\mathbf{0}, \quad  E = 0, \quad \B \cdot \n=0 \quad \text{ on } \del \Omega.
\end{equation}

Note that the electric field $E$ is a scalar field in 2D and we have used two different curl operators and cross products corresponding to \eqref{eq:curl2d} and \eqref{eq:cross2d}.
Moreover, the boundary conditions for the electric field change to $E=0$ on $\del \Omega$ in two dimensions.

Other formulations include the current density $\mathbf{j}= \E + \u \times \B$ \cite{Hu2018} as an unknown or eliminate the electric field using equation \eqref{eq:MHD3d3}. In addition to the stationary case, we also consider the time-dependent version of \eqref{eq:MHD3d} where  the time-derivatives $\frac{\del\u}{\del t}$ and $\frac{\del\B}{\del t}$ are added to \eqref{eq:MHD3d1} and \eqref{eq:MHD3d4} respectively with suitable initial conditions 
\begin{equation}
    \u(\x,0)=\u_0(\x) \text{ and } \B(\x,0)=\B_0(\x) \, \forall \x \in \Omega.
\end{equation} Note that MHD models neglect displacement currents $\frac{\del\E}{\del t}$ \cite[Sec. 1.5]{Gerbeau2006}.

The above equations are derived from the Navier--Stokes and Maxwell's equations for a single, incompressible, homogeneous fluid in steady state, which are given by 
\begin{subequations}
\begin{align}
   - \div(2 \nu \mathbf{\varepsilon} (\u)) +  \u \cdot \nabla \u + \nabla p + \frac{1}{\rho_0\mu_0}\B \times \vcurl \B &= \f,  \label{eq:MHDparam1}\\
  \div \u&=0, \label{eq:MHDparam2}\\
  \vcurl \E &= \mathbf{0},  \label{eq:MHDparam3} \\ 
 \mu_0 \j - \vcurl \B &= \mathbf{0}, \label{eq:MHDparam4}  \\
  \div \B &=0, \label{eq:MHDparam5}\\
  \eta \vcurl \B - \E - \u \cross \B &= 0, \label{eq:MHDparam6}
\end{align}
\end{subequations}
with the kinematic viscosity $\nu>0$, the magnetic permeability of free space $\mu_0>0$, a reference density $\rho_0$ and the magnetic resistivity $\eta>0$. We treat each of these parameters as constant throughout the domain. Equation \eqref{eq:MHDparam1} and \eqref{eq:MHDparam2} describe the incompressible Navier--Stokes equations where the Lorentz force $\B \times \j$ acts on the fluid. The stationary forms of the Maxwell-Faraday law and Amp\`{e}re's circuital law are given by \eqref{eq:MHDparam3} and \eqref{eq:MHDparam4}. The system is completed by the magnetic Gauss's law \eqref{eq:MHDparam5} and Ohm's law \eqref{eq:MHDparam6}

To obtain the MHD system \eqref{eq:MHD3d}, we non-dimensionalise the resulting system 
by introducing the new unknowns
\begin{align}
    \u^\star(\xi)&=\frac{\u(L\xi)}{\ou},\\
    p^\star(\xi)&= \frac{p(L\xi)}{ \rho_0 \ou^2}, \\
    \B^\star(\xi)&= \frac{\B(L\xi)}{\overline{B}},\\
    \E^\star(\xi) &= \frac{\E(L\xi)}{\ou\overline{B}},\\
    \f^\star(\xi)&=\frac{f(L\xi)L}{ \ou^2},
\end{align}
with a characteristic value for the velocity $\ou$, magnetic field $\overline{B}$ and the length scale $L$. Finally, we obtain \eqref{eq:MHD3d} by defining the fluid Reynolds number, magnetic Reynolds number and coupling number
\begin{equation}
    \Re= \frac{\ou L}{\nu}, \quad \quad \Rem= \frac{ \ou L}{\eta}, \quad \quad S = \frac{\overline{B}^2 L}{\rho_0\mu_0\eta \overline{U}}.
\end{equation}
Note that some formulations ignore the coupling number $S$ in front of the Lorentz force $\B \cross \j$. With the alternative scaling of 
\begin{align}
    \B^\star(\xi)&= \frac{\B(L\xi)}{\overline{B}},\qquad \overline{B} = \ou \sqrt{\mu_0},\\
    \E^\star(\xi) &= \frac{\E(L\xi)}{\overline{E}}, \qquad \overline{E}=\ou \overline{B},
\end{align}
one can achieve that $S=1$. However, in this case the characteristic value $\overline{B}$ cannot be chosen freely, but has to be expressed in terms of $\overline{U}$. In order to include the case where one wants to choose $\overline{B}$ freely depending on the problem we include the parameter $S$ and in particular consider high values of $S$ in the numerical examples.

An important point for discretisations is the enforcement of the magnetic Gauss' law $\div \B=0$ in the weak formulation, achieved in most cases by a non-physical Lagrange multiplier $r$ \cite{Schoetzau2004}. However, in general, a Lagrange multiplier only enforces the divergence constraint in a weak sense, which can cause severe problems for the discretisation and numerical simulations \cite{Brackbill1980, Dai1998}. Therefore, in recent years increased attention has been paid to derive discretisations that enforce $\div \B=0$ pointwise. These approaches include the use of a magnetic vector-potential \cite{Adler2016, Adler2019, Cyr2013, Hiptmair2018, Shadid2010}, exact penalty methods on convex domains \cite{Phillips2014},  compatible discretisations \cite{Hu20162, Hu2018}, the use of divergence-free basis functions \cite{Cai2013} and divergence-cleaning methods \cite{Brackbill1980, dedner2002hyperbolic}. For the $\B$-$\E$ formulation \eqref{eq:MHD3d} Hu et al.~\cite{hu2020convergence} show that both a Lagrange multiplier and an augmented Lagrangian term lead to a pointwise preservation of Gauss' law with appropriate choices of spaces. In this work, we consider the latter approach by replacing \eqref{eq:MHD3d4} with
\begin{equation}\label{eq:augdivB0}
\frac{1}{\Rem}\grad \div\B + \vcurl \E = \mathbf{0},
\end{equation}
which we show below enforces $\div \B = 0$ exactly.

The literature proposes numerous numerical schemes and preconditioning strategies for the numerical solution of the different formulations. The most common approach is based on block preconditioners in both the stationary \cite{Li2017, PhillipsPHD, Phillips2014, Wathen2017, Wathen2020} and time-dependent \cite{Chacn2008,Cyr2013, Phillips2016} cases. Here, the main challenges are to find suitable approximations of one or more Schur complements and robust linear solvers for the inner auxiliary problems. Phillips et al.~\cite{Phillips2016} simplify the Schur complement by the use of vector identities and approximate the remaining parts based on a spectral analysis. They report iteration counts for a stationary three-dimensional lid-driven cavity problem up to $\Re=\Rem=100$. A similar approach is used by Wathen and Greif in \cite{Wathen2020} where they construct an approximate inverse block preconditioner by sparsifying a derived formula for the exact inverse and drop low order terms. Here, results for Hartmann numbers $\Ha=\sqrt{S\Rem\Re}$ up to 1,000 are reported for stationary problems. Other approaches include fully-coupled geometric \cite{adler2020monolithic, Adler2016} and algebraic \cite{Shadid2010, Shadid2016} monolithic multigrid methods. In \cite{adler2020monolithic}, Adler et al.\ present results for a two-dimensional Hartmann problem for parameters up to $\Re = \Rem = 64$.

However, the performance of most of these preconditioners deteriorates significantly for high Reynolds and coupling numbers. To the best of our knowledge, a practical robust preconditioner for the stationary MHD equations has not yet been proposed. The common problem for high magnetic Reynolds numbers and coupling numbers for the stationary case is that all available Schur complement approximations become less accurate for Newton-type linearisations, causing the linear solver to fail to converge. Conversely, Picard-type linearisations can allow an exact computation of the Schur complement but fail to converge in the nonlinear iteration at high magnetic Reynolds number. This stands, e.g., in contrast to the Navier--Stokes equations where under certain assumptions the Picard iteration converges globally even for high fluid Reynolds numbers \cite[page 346]{Elman2014}.

Ma et al.~\cite{Ma2016} have developed Reynolds-robust preconditioners for the time-depen\-dent MHD equations that are based on norm-equivalent and field-of-values equivalent approaches. To the best of our knowledge, their strategy does not extend to the stationary case; in general, the time-dependent case offers crucial advantages for the development of robust solvers. For example, Ma et al.\ treat complicated terms like the hydrodynamic convection term $\u \cdot \grad \u$ explicitly in the time-stepping scheme, which can cause problems for convection-dominated problems and does not apply in the stationary case. The discretisation of the time derivative causes mass matrices with a scaling of $1/\Delta t$, where $\Delta t$ denotes the time step size, to appear in the block matrix on the diagonal blocks. As we will see also in our numerical results for the time-dependent problems, these extra terms dominate the scheme for small $\Delta t$ and hence simplify the development of robust solvers.

In this work, we consider two different linearisations. The first is the Picard iteration proposed by Hu et al.~\cite{hu2020convergence}. We compute an approximation to the outer Schur complement of the arising block system and introduce a robust linear solver for the different blocks. This scheme works well for small magnetic Reynolds numbers but the nonlinear iteration fails to converge for higher $\Rem$, as anticipated in the analysis of \cite{hu2020convergence}. The second is a full Newton linearisation, which converges well for high Reynolds numbers and coupling numbers for suitable initial guesses that we obtain with parameter continuation. However, our approximation of the Schur complement deteriorates slightly for high parameters.

\section{Formulation, linearisation, and discretisation} \label{sec:discretization}

\subsection{An augmented Lagrangian formulation}\label{sec:AnaugmentedLagrangianFormulation}
We modify \eqref{eq:MHD3d} by introducing two augmented Lagrangian terms: $-\gamma \grad \div \u$ for $\gamma>0$ is added to \eqref{eq:MHD3d1}, and (as previously discussed)  $-1/\Rem\ \grad \div \B$ is added to \eqref{eq:MHD3d4}. Note that both terms leave the continuous solution of the problem unchanged. We use the first term to control the Schur complement of the fluid subsystem \cite{BenziOlshanskii,Farrell2020} and the second term to enforce the divergence constraint $\div \B=0$.

Following these modifications, we consider the system
\begin{subequations}
	\label{eq:MHDFinal}
	\begin{align}
		- \frac{2}{\mathrm{Re}} \div \mathbf{\varepsilon}( \u)  + \u \cdot \nabla \u - \gamma \grad \div \u  + \nabla p + S\, \B \times  (\E + \u \times \B) &= \f, \label{eq:MHDFinal1}\\
		\div \u&=0, \label{eq:MHDFinal2}\\
		\E + \u \times \B - \frac{1}{\Rem} \vcurl \B &= \mathbf{0}, \label{eq:MHDFinal3}\\
		-\frac{1}{\Rem} \grad \div \B + \vcurl \E &= \mathbf{0} \label{eq:MHDFinal4},
	\end{align}
\end{subequations}
subject to the boundary conditions \eqref{eq:boundarycond3d}.
For convenience, we consider homogeneous boundary conditions in this section but all the results extend in a mathematically straightforward manner to inhomogeneous boundary conditions. However, there are subtle technicalities for the implementation of the degrees of freedom in the finite element method in the inhomogeneous case, which are explained in detail in Section~\ref{sec:interp-boundary-data}.

The weak formulation of \eqref{eq:MHDFinal} seeks $\UU \coloneqq(\u,p,\E,\B)\in \ZZ \coloneqq \V\times Q \times \Rr \times \W$ with
\begin{equation}
	\V \coloneqq \Hozv, \quad Q\coloneqq L^2_0(\Omega),\quad \Rr \coloneqq \hzcurl, \quad  \W \coloneqq \hzdiv.
\end{equation}
In two dimensions,  the space for the electric field is scalar-valued and can be identified with $R\coloneqq \Hoz$.
The weak formulation is to find $\mathcal{U} \in \ZZ$ such that for all $\mathcal{V}\coloneqq(\v,q,\F,\C) \in \ZZ$ and $\mathcal{F} =  (\f, 0, \mathbf{0}, \mathbf{0})$ there holds
\begin{equation} \label{eq:weakform}
	\mathcal{R}(\mathcal{U}, \mathcal{V})	\coloneqq \NN(\mathcal{U}, \mathcal{V}) - (\mathcal{F},\mathcal{V}) = 0
\end{equation}
with
\begin{align}
	\begin{split}
		\NN(\mathcal{U},\mathcal{V}) & =\frac{2}{\Re}(\mathbf{\varepsilon}(\u), \mathbf{\varepsilon}(\v)) + (\u\cdot \grad \u, \v) + \gamma (\div \u, \div \v) \\
		& - (p,\div \v) +  S (\B \times \E, \v) + S (\B \times (\u \times \B), \v)  \label{eq:NN3} \\
		&- (\div \u, q) \\
		&+ (\E, \F) + (\u\times \B, \F) - \frac{1}{\Rem}(\B, \vcurl \F) \\
		&+ \frac{1}{\Rem} (\div \B, \div \C) + ( \vcurl \E, \C).
	\end{split}
\end{align}
All boundary integrals that result from integration by parts vanish because of the choice of the boundary conditions \eqref{eq:boundarycond3d}.

Note that $\W$ and $\Rr$ are chosen from the same exact de Rham complex \eqref{eq:contDC3d} or \eqref{eq:contDC2d}.
This ensures that formulation \eqref{eq:MHDFinal} enforces the divergence constraint $\div \B =0$ and $\vcurl \E = \mathbf{0}$ \cite[Theorem 9]{hu2020convergence}. To see this, we test \eqref{eq:weakform} with $\mathcal{V} = (\mathbf{0}, 0, \mathbf{0}, \vcurl \E)$ and conclude that $\vcurl \E =\mathbf{0}$. Here, $\mathcal{V}$ is a valid test function because the above exact sequence implies that $\vcurl(\Rr) = \W$. Similarly, testing with $\mathcal{V} = (\mathbf{0}, 0, \mathbf{0}, \B)$ results in $\div \B = 0$.

\subsection{Linearisation: Newton and Picard}

The Newton linearisation of \eqref{eq:weakform} for the initial guess $\mathcal{U}^n = (\u^n,p^n,\E^n,\B^n)$ is to find an update $\delta \mathcal{U}$ such that
\begin{align}
	\NN_\text{N}(\delta\mathcal{U},\mathcal{U}^n,\mathcal{V})&=\mathcal{R}(\mathcal{U}^n,\mathcal{V}) \quad \forall \ \mathcal{V} \in \ZZ,\\
	\mathcal{U}^{n+1} &= \mathcal{U}^n + \delta\mathcal{U},
\end{align}
with the weak form of the nonlinear residual $\mathcal{R}(\mathcal{U}^n, \mathcal{V}) := (\mathcal{F}, \mathcal{V}) - \mathcal{N}(\mathcal{U}^n, \mathcal{V})$ evaluated at $\mathcal{U}^n$ and
\begin{align}
	\label{eq:Newton}
	\begin{split}
		\NN_\text{N}(\delta\mathcal{U},\mathcal{U}^n,\mathcal{V}) &= \frac{2}{\Re}(\mathbf{\varepsilon}( \delta\u), \mathbf{\varepsilon}(\v)) + (\u^n\cdot \grad \delta\u, \v) + (\delta\u\cdot \grad \u^n, \v)  \\
		&+ \gamma (\div \delta\u, \div \v) - (\delta p,\div \v)\\
		& +  S (\B^n\times \delta \E, \v)  + S (\delta\B \times \E^n, \v)  \\
		& + S  (\B^n \times (\delta \u \times \B^n), \v) +
		S (\delta\B \times (\u^n \times \B^n), \v) \\
		&+ S (\B^n \times  (\u^n \times \delta\B), \v)\\
		&- (\div \delta\u, q) \\
		& + (\delta \E, \F) + (\u^n\times\delta\B, \F) + (\delta\u \times \B^n, \F) \\
		&- \frac{1}{\Rem} (\delta \B, \vcurl \F) \\
		&+ \frac{1}{\Rem} (\div \delta \B, \div \C) + (\vcurl \delta \E, \C).
	\end{split}
\end{align}
The bilinear form for the Picard iteration that we consider is given by
\begin{align}
	\label{eq:Picard}
	\begin{split}
		\NN_\text{P}(\delta\mathcal{U},\mathcal{U}^n,\mathcal{V}) =\, & \NN_{N} (\delta\mathcal{U},\mathcal{U}^n,\mathcal{V}) - S (\delta\B \times \E^n, \v) - S (\B^n \times  (\u^n \times \delta\B), \v)\\
		& - S (\delta\B \times (\u^n \times \B^n), \v) - (\u^n\times\delta\B, \F).
	\end{split}
\end{align}
Note that in contrast to \cite{hu2020convergence}, we do not scale the term $(\vcurl \delta \E, \C)$ with $S/\Rem$ and consider the full Newton linearisation of the advection term $(\u\cdot \grad) \u$.
The advantage of the Picard linearisation \eqref{eq:Picard} in comparison to the Newton linearisation \eqref{eq:Newton} is that it allows an exact Schur complement computation in two dimensions and converges well for high $\Re$. However, its major disadvantage is the failure of nonlinear convergence for high $\Rem$.

\subsection{Discretisation}\label{sec:Discretisationchap2}

For a finite element discretisation, we seek $\mathcal{U}_h:=(\u_h, p_h,$ $\E_h, \B_h) \in \ZZ_h \coloneqq \mathbf{V}_h \times Q_h \times \mathbf{R}_h\times \W_h$ such that
\begin{equation} \label{eq:weakformdiscr}
	\NN(\mathcal{U}_h, \mathcal{V}_h)=  (\mathcal{F},\mathcal{V}_h) \quad \forall\, \mathcal{V}_h \in \ZZ_h.
\end{equation}
We choose Raviart--Thomas elements of degree $k$ $\mathbb{RT}_k$ for $\W_h$, \Ned elements of first kind $\mathbb{NED}1_k$ for $\Rr_h$ in 3D and continuous Lagrange elements $\mathbb{CG}_k$ for $R_h$ in 2D. 
Note that these elements belong to the discrete subcomplexes \eqref{eq:deRhamdiscret3d} and \eqref{eq:deRhamdiscret2d}. 
This implies that we enforce $\div \B_h=0$ and $\vcurl \E_h=\mathbf{0}$ pointwise with the same proof as for the continuous case. These identities also hold for inhomogeneous boundary conditions, since the interpolation operator $\II^h_{\W_h}$ into the Raviart--Thomas space satisfies for all divergence-free $\B \in \mathbf{H}_0(\mathrm{div}, \Omega)$~\cite[Prop.~2.5.2]{Boffi2013}
\begin{equation}\label{eq:InterpolationPreserveRT}
	\div(\II^h_{\W_h} \B) = 0.
\end{equation}

To be more precise, we consider the non-homogeneous boundary condition $\B \cdot \n = g_1$ for $g_1 \in H^{\frac{1}{2}}(\partial \Omega)$ and assume for the solvability of the problem that there exists a $\B_{g_1}  \in \mathbf{H}(0, \mathrm{div})$ such that $\B_{g_1}\cdot \n = g_1$. For a finite element approximation, one then computes $\II^h_{\W_h}\B_{g_1} = \sum_{i=1}^{N} B_{i, g_1} \mathbf{\Phi}_i  $ in terms of the $N$ basis function $\mathbf{\Phi}_i$ of $\W_h$ and sets, cf. \cite{LectureNotesBoundary},
\begin{equation}
	\B_{h,g_1} = \sum_{i=1}^{N_D} B_{i, g_1} \mathbf{\Phi}_i.
\end{equation} Here, we have chosen the ordering that the basis functions that have non-vanishing normal components on the boundary are the first $N_D$ ones. Then, we look for a solution of the form 
\begin{equation}
	\B_h = \B_{h,{g_1}}+ \B_{h,0} \quad \text{ with }\quad \B_{h,0} = \sum_{i=N_D+1}
	^N B_i \mathbf{\Phi}_i.
\end{equation} Hence, the approximation of the weak form of \eqref{eq:MHDFinal4} is given by
\begin{equation}\label{eq:weaknonhom}
	\frac{1}{\Rem}(\div \B_{h,0}, \div \mathbf{\Phi}_i) + (\vcurl \E_h, \mathbf{\Phi}_i) = - \frac{1}{\Rem} (\div \B_{h, g_1}, \div \mathbf{\Phi}_i), \quad i=N_D + 1,...,N.
\end{equation}
As before, we can conclude that $\vcurl \E_h=\mathbf{0}$ since $\vcurl \E_h \in \mathrm{span}\{\mathbf{\Phi}_i \, | \, i=N_D+1,...,N\}$ by the sequence \eqref{eq:deRhamdiscret3d}.
Equation \eqref{eq:weaknonhom} implies for the homogeneous test functions $\mathbf{\Phi}=\B_{h,0}$ and $\mathbf{\Phi}=\B_{h,g_1}-\II^h_{\W_h}\B_{g_1}$ 
\begin{equation}\label{eq:divres}
	(\div \B_{h}, \div \B_{h,0}) = 0 \quad\text{ and }\quad (\div \B_{h}, \div (\B_{h,g_1}-\II^h_{\W_h}\B_{g_1})) = 0.
\end{equation}
But there holds $\div (\II^h_{\W_h}\B_{g_1})=0$  by the interpolation property \eqref{eq:InterpolationPreserveRT} and thus adding the two expressions in \eqref{eq:divres} shows that $\div \B_h=0$. 

The same approach applied for non-homogeneous boundary conditions $\E=\mathbf{g}_2$ on $\partial \Omega$ would add a term $ (\vcurl \E_{h,g_2}, \mathbf{\Phi}_i)$ to the right-hand side of \eqref{eq:weaknonhom}. But since $\vcurl \E=\mathbf{0}$ implies that $\mathbf{g}_2$ is constant, this extra term vanishes and thus $\vcurl \E_h=\mathbf{0}$ holds by the same proof as before.

Moreover, following \cite{douglas1976}, we add the following interior penalty stabilisation term to address the problem that the Galerkin discretisation of advection-dominated problems can be oscillatory \cite{Elman2014}
\begin{equation}\label{eq:stabBurman}
	\sum_{K\in \MM_h} \frac{1}{2} \int_{\del K} \mu \, h_{\del K}^2 \llbracket \grad \u_h \rrbracket : \llbracket \grad \v_h \rrbracket \ \mathrm{d} s.
\end{equation}
Here, $\llbracket \grad \u_h \rrbracket$ denotes the jump of the gradient, $h_{\del K}$ is a function giving the facet size, and $\mu$ is a free parameter that is chosen according to \cite{burman_edge_2006}.

Note that a fully robust discretisation should also include a stabilisation term for the magnetic field $\B$ in the case of dominating magnetic advection. The literature does not propose many stabilisation types for this problem. The most promising work by Wu and Xu \cite{Wu2020} uses the so-called SAFE-scheme for stabilisation which is based on an exponential fitting approach. While the original SAFE-scheme is only a first order method, it can be extended to higher order as shown in \cite{wu2020unisolvence}. We aim to include this stabilisation in future work.

For the hydrodynamic part, we consider the $\hdiv\times L^2$-conforming element pair $\mathbb{BDM}_k\times \mathbb{DG}_{k-1}$ with the Brezzi-Douglas-Marini element $\mathbb{BDM}_k$ of order $k$ \cite{Brezzi1985, Ndlec1986}. This discretisation ensures that $\div \u_h =0$ holds pointwise since $\div \mathbf{V}_h \subset Q_h$. Additionally, it exhibits pressure robustness, i.e., the error estimates do not degrade for high Reynolds numbers \cite{john2017}.

Since the discretisation is nonconforming, we must consider a discontinuous Galer\-kin formulation of the hydrodynamic advection and diffusion terms \cite[Section 7]{GauerLinke2019}. We denote by $\FF_h = \FF_h^i \cup \FF_h^\partial$ all facets of the triangulation, which consists of the interior facets $\FF_h^i $ and the Dirichlet boundary facets $\FF_h^\partial$. We assign to each $F \in \mathcal{F}_h$ its diameter $h_F$ and unit normal vector $\mathbf{n}_F$. The jump and average operators across a facet are denoted by $\llbracket\cdot \rrbracket$ and $\ldblbrace\cdot \rdblbrace$, respectively, and are defined as $\llbracket \Phi \rrbracket=\Phi^+ - \Phi^-$ and $\ldblbrace\Phi \rdblbrace=\frac{1}{2}(\Phi^+ + \Phi^-)$. The penalisation parameter is chosen as $\sigma = 10 k^2$, $k$ being the degree of the velocity space. Inhomogeneous boundary data are described by $\mathbf{g}_D$. We then add the following bilinear forms to \eqref{eq:weakformdiscr}:
\begingroup
\begin{align}\label{eq:hidvl2form}
	\begin{split}
		a_h^{DG}(\u_h, \v_h)= &-\frac{2}{\Re}\sum_{F \in \FF_h} \int_F \ldblbrace\varepsilon( \u_h) \rdblbrace \mathbf{n}_F \cdot \llbracket \v_h \rrbracket \,\mathrm{d}s\\
		&-\frac{2}{\Re}\sum_{F \in \FF_h} \int_F  \llbracket \u_h \rrbracket \cdot \ldblbrace\varepsilon(\v_h) \rdblbrace \mathbf{n}_F  \,\mathrm{d}s\\
		&+\frac{1}{\Re}\sum_{F \in \FF_h} \frac{\sigma}{h_F} \int_F \llbracket\u_h \rrbracket \cdot \llbracket \v_h \rrbracket \,\mathrm{d}s  \\
		& - \frac{1}{\Re}\sum_{F \in \FF_h^\partial} \frac{\sigma}{h_F} \int_F \mathbf{g}_D \cdot \v_h \,\mathrm{d}s\ +  \frac{2}{\Re}\sum_{F \in \FF_h^\partial} \int_F \mathbf{g}_D \cdot \varepsilon( \v_h) \mathbf{n}_F \,\mathrm{d}s,
	\end{split}
	\\
	\begin{split}
		c_h^{DG}(\u_h, \v_h)= &\ \ \ \ \frac{1}{2}\sum_{F \in \FF^i_h}\int_F \llbracket (\u_h \cdot \mathbf{n}_F + |\u_h\cdot \n_F|)\u_h \rrbracket \cdot
		\llbracket \v_h \rrbracket \,\mathrm{d}s  \\
		& +\frac{1}{2} \sum_{F \in \FF_h^\partial} \int_F (\u_h \cdot \mathbf{n}_F + |\u_h\cdot \n_F|) \u_h \cdot \v_h \,\mathrm{d}s \\
		& +\frac{1}{2} \sum_{F \in \FF_h^\partial} \int_F (\u_h \cdot \mathbf{n}_F - |\u_h\cdot \n_F|) \mathbf{g}_D \cdot \v_h \,\mathrm{d}s.
	\end{split}
\end{align}
\endgroup

Alternatively, a discretisation with Scott--Vogelius elements \cite{scott_conforming_1985}, i.e.\ $(\mathbb{CG}_k)^d\times \mathbb{DG}_{k-1}$ elements, also enforces $\div \u_h = 0 $. While this conforming discretisation does not require stabilisation terms to weakly enforce continuity, it is only stable on certain types of meshes. For this reason, the mesh hierarchy we consider here is barycentrically refined and ensures stability for polynomial order $k=d$ \cite{zhang2004}. That means in three dimensions one has to use at least polynomial order three for a stable discretisation, which can be relatively expensive. For this reason and the fact that the $\hdiv \times L^2$-discretisation does not have restrictions on the mesh types, we mainly focus on the $\hdiv\times L^2$-conforming discretisation in the following. However, we also explain in the next section how a parameter robust multigrid method can be constructed for this discretisation and include a comparison for the iteration numbers in the numerical results in Section \ref{sec:numericalresults}.


Hu et al.\ prove in \cite[Theorem 4]{hu2020convergence} that \eqref{eq:weakformdiscr} is well-posed and has at least one solution. The solution is unique for suitable source and boundary data.
While the well-posedness and convergence of the Newton iteration remains an open problem, Hu et al.\ prove that the Picard iteration converges to the unique solution of \eqref{eq:weakformdiscr} if both $\Re$ and $\Rem$ are small enough \cite[Theorem 6]{hu2020convergence}.

For the Newton linearisation \eqref{eq:Newton}, we must solve the following linear system at each step:
\begin{equation}
	\label{eq:matrix_upBE}
	\begin{bmatrix}
		\mathcal{F} +\DD & \BB^\top &  \JJ & \tilde{\JJ} + \tilde{\DD}_1 + \tilde{\DD}_2 \\
		\BB & \mathbf{0} & \mathbf{0} & \mathbf{0} \\
		\GG& \mathbf{0}& \MM_\E & \tilde{\GG} -\frac{1}{\Rem} \AA \\
		\mathbf{0} & \mathbf{0} & \AA^\top  &\CC
	\end{bmatrix}
	\begin{bmatrix}
		x_\u \\ x_p \\ x_\E \\ x_\B
	\end{bmatrix} =
	\begin{bmatrix}
		\mathcal{R}_\u\\ 	\mathcal{R}_p\\ 	\mathcal{R}_\E\\ 	\mathcal{R}_\B
	\end{bmatrix},
\end{equation}
where $x_\u$, $x_p$, $x_\E$ and $x_\B$ are the coefficients of the discretised Newton corrections and $\mathcal{R}_\u$, $\mathcal{R}_p$, $\mathcal{R}_\E$ and $\mathcal{R}_\B$ the corresponding nonlinear residuals.
The correspondence between the discrete and continuous operators is illustrated in Table \ref{tab:Operators}. We have chosen the notation that operators that include a tilde are omitted in the Picard linearisation $\eqref{eq:Picard}$. Moreover, we introduce $\eta\in\{0,1\}$ to distinguish between the stationary ($\eta=0$) and transient ($\eta=1$) cases.

For the time-dependent equations, we concentrate here on the implicit Euler method, but the following computations are straightforward to adapt to other implicit multi-step methods. We use the same finite element discretisation as in the stationary case. Note that in the transient case, the equation
\begin{equation}
	\del_t \B + \vcurl \E = \mathbf{0}
\end{equation}
immediately implies $\div \B=0$ if the initial condition satisfies $\div \B_0=0$, and this remains true on the discrete level up to solver tolerances; see \cite[Theorem 1]{Hu20162} for a proof for implicit Euler which can be extended to other multi-step methods in a straightforward manner, provided all starting values are divergence-free.
Hence, the augmented Lagrangian term $-	\frac{1}{\Rem}\grad \div \B$ is no longer necessary to enforce the divergence constraint and could therefore be omitted. Nevertheless, we retain it in our scheme since we employ the identity
\begin{equation}\label{eq:laplaceidentity}
	\frac{1}{\Rem} \vcurl \vcurl \u - 	\frac{1}{\Rem} \grad \div \u = - 	\frac{1}{\Rem} \Delta \u
\end{equation}
in our derivation of Schur complement approximations below.

\begin{table}[htb!]
	\centering
	\resizebox{\textwidth}{!}{
		\begin{tabular}{c|c|c}
			\toprule
			\textbf{Discrete} & \textbf{Continuous} & \textbf{Weak form}\\
			\midrule
			$\mathcal{F} \u$ & $\frac{\eta}{\Delta t} \u-\frac{2}{\Re} \div \varepsilon(\u) + \u^n\cdot \grad \u $ & $\frac{\eta}{\Delta t} (\u, \v)+\frac{2}{\Re}(\varepsilon( \u), \varepsilon( \v)) + (\u^n\cdot \grad \u , \v)$  \\
			& $+ \u \cdot \grad \u^n-\gamma\grad \div \u$ &  $+(\u\cdot \grad \u^n, \v) + \gamma(\div \u, \div \v)  $  \\
			$\DD \u$& $S \B^n\times(\u\times\B^n)$ & $ S (\B^n\times (\u \times \B^n), \v)$ \\
			$\JJ \E$ & $ S \B^n\times\E$ & $ S (\B^n\times\E, \v)$   \\
			$\tilde{\JJ} \B$ & $ S \B\times\E^n$ & $ S (\B\times\E^n, \v)$   \\
			$\tilde{\DD}_1\B$ & $ S \B \times (\u^n\times\B^n)$ & $ S (\B \times (\u^n\times\B^n),\v)$  \\
			$\tilde{\DD}_2\B$ & $ S \B^n \times (\u^n\times\B)$ & $ S (\B^n \times (\u^n\times\B),\v)$  \\
			$\MM_\E \E$ & $\E$ & $(\E,\F)$  \\
			$\GG \u$ & $\u \times \B^n$ & $(\u\times\B^n, \F)$  \\
			$\tilde{\GG} \B$ & $\u^n \times \B$ & $(\u^n\times\B, \F)$  \\
			$\AA \B$ & $ \vcurl \B$ & $ (\B, \vcurl \F)$ \\
			$\CC \B$ & $\frac{\eta}{\Delta t} \B- \frac{1}{\Rem}\grad \div \B$ & $\frac{\eta}{\Delta t}(\B, \C) +  \frac{1}{\Rem}(\div\B,\div\C)$  \\
			$\AA^\top \E$ & $\vcurl \E$ & $(\vcurl \E, \C)$ \\
			$\BB^\top p$ & $\grad p$ & $-(p, \div \v)$ \\
			$\BB \u$ & $-\div \u$ & $-(\div \u, q)$  \\
			\bottomrule
	\end{tabular}}
	\caption{Overview of operators. Terms that include a tilde are dropped in the Picard iteration. The stationary and transient cases are distinguished by $\eta\in\{0,1\}$. }
	\label{tab:Operators}
\end{table}

\section{Derivation of block preconditioners}\label{sec:derivationofblockpreconditioner}
We now consider block preconditioners for \eqref{eq:matrix_upBE}. The inverse of a $2 \times 2$  block matrix can factorised as \cite{Benzi2005,Elman2014} 
\begin{equation}\label{eq:blockfactor}
	\begin{pmatrix} \MM & \mathcal{K} \\ \mathcal{L} & \mathcal{N} \end{pmatrix}^{-1}
   =
	\begin{pmatrix}
		\II &-\MM^{-1}\mathcal{K} \\
		\mathbf{0} &\II
	\end{pmatrix}
	\begin{pmatrix}
		\MM^{-1} &\mathbf{0} \\
		\mathbf{0} &\mathcal{S}^{-1}
	\end{pmatrix}
	\begin{pmatrix}
		\II &\mathbf{0} \\
		-\mathcal{L}\MM^{-1} &\II
	\end{pmatrix}
\end{equation}
provided the top-left block $\MM$ and the Schur complement $\mathcal{S}=\mathcal{N} - \LL\MM^{-1} \KK$ are invertible. Since the Schur complement is usually a dense matrix, the main task is to find a suitable approximation $\tilde{\mathcal{S}}$ for the Schur complement $\mathcal{S}$ as well as efficient solvers for $\MM$ and $\tilde{\mathcal{S}}$.

 In Sections \ref{sec:OuterSchurEB} and \ref{sec:OuterSchurup} we derive approximations of the Schur complements for two different block elimination strategies.  We briefly introduce the theory of parameter-robust multigrid relaxation in Section \ref{sec:robustssc}, and then describe the multigrid methods that we use to solve systems with the top-left block $\MM$ and with the Schur complement approximations $\tilde{\mathcal{S}}$ in Sections \ref{sec:solverforschurcomp} and \ref{sec:solverformagneticblock}.

Both block preconditioners we consider gather the variables as $(\E, \B)$ and $(\u, p)$. They differ in the order of block elimination: the first takes the Schur complement that eliminates (inverts) the $(\E, \B)$ block, while the second takes the Schur complement that eliminates the $(\u, p)$ block. References for the first choice are given by \cite{Li2017,Phillips2016} and for the second choice by \cite{Cyr2013, Cyr2016}. As we will see, for small $\Rem$ and $S$ both preconditioners perform similarly, while for more difficult parameter regimes the second choice notably outperforms the first. We therefore recommend the second strategy and mainly report numerical results for this choice. Nevertheless, we also investigate the first option, both for comparison and because it allows a much more detailed description of the Schur complement. In two dimensions it even allows an exact computation of the Schur complement. The two strategies are compared in Section \ref{sec:stationaryliddrivencavityproblemin3d} below.

Another overview for Schur complement approximations and physics-based preconditioners is given in \cite{Chacn2003} in the context of the Hall MHD equations. This theory is also applicable to our standard MHD system when the case of vanishing ion skin depth $d_i=0$ is considered in \cite{Chacn2003}.

\subsection{Outer Schur complement eliminating the $(\E, \B)$ block}\label{sec:OuterSchurEB}
Reordering \eqref{eq:matrix_upBE} for convenience, we consider
\begin{equation}
	\label{eq:matrix_EBup}
	\left[
	\begin{array}{cc|cc}
		\MM_\E & \tilde{\GG} - \frac{1}{\Rem} \AA & \GG & \mathbf{0}\\
		\AA^\top & \CC & \mathbf{0} & \mathbf{0} \\
		\hline \rule{0pt}{1.0\normalbaselineskip}
		\JJ & \tilde{\JJ} + \tilde{\DD}_1 + \tilde{\DD}_2 & \mathcal{F} + \DD & \BB^\top\\
		\mathbf{0} & \mathbf{0} & \BB & \mathbf{0}
	\end{array}
	\right]
	\begin{bmatrix}
		x_\E \\ x_\B \\ x_\u \\ x_p
	\end{bmatrix} =
	\begin{bmatrix}
		\mathcal{R}_\E\\ \mathcal{R}_\B\\ \mathcal{R}_\u\\ \mathcal{R}_p
	\end{bmatrix}.
\end{equation}
In the following, we refer to the Schur complement of the $4 \times 4$ matrix as the outer Schur complement, while we call the Schur complements of the resulting $2 \times 2$ blocks inner Schur complements.
The outer Schur complement eliminating the $(\E, \B)$ block is given by
\begin{equation}
	\label{eq:outerSchurCompNewton}
	\mathcal{S}^{(\E, \B)} =
	\begin{bmatrix}
		\FF+\DD & \BB^\top \\
		\BB & \mathbf{0}
	\end{bmatrix}
	-
	\begin{bmatrix}
		\JJ & \tilde{\JJ} + \tilde{\DD}_1 + \tilde{\DD}_2\\
		\zerov & \zerov
	\end{bmatrix}
	\begin{bmatrix}
		\MM_\E& \tilde{\GG} - \frac{1}{\Rem} \AA \\
		\AA^\top & \CC
	\end{bmatrix}^{-1}
	\begin{bmatrix}
		\GG & \zerov\\
		\zerov & \zerov
	\end{bmatrix}.
\end{equation}

We simplify $\mathcal{S}^{(\E, \B)}$ by applying the identity \eqref{eq:blockfactor}
to the top-left electromagnetic block
\begin{equation}
	\label{eq:M}
	\MM =
	\begin{bmatrix}
		\MM_\E& \tilde{\GG} - \frac{1}{\Rem} \AA \\
		\AA^\top & \CC
	\end{bmatrix}.
\end{equation}
This results in
\begin{equation}
	\mathcal{S}^{(\E, \B)}=
	\begin{bmatrix}
		\FF + \DD - \JJ\MM^{-1}_{1,1} \GG - (\tilde{\JJ} + \tilde{\DD}_1 + \tilde{\DD}_2)\MM^{-1}_{2,1}\GG& \BB^\top \\
		\BB & \zerov
	\end{bmatrix}
\end{equation}
with
\begin{equation}
	\MM^{-1}_{1,1} = \MM_\E^{-1} + \MM_\E^{-1}\left(\tilde{\GG}-\frac{1}{\Rem} \AA\right)\left(\CC - \AA^\top \MM_\E^{-1}\left(\tilde{\GG}-\frac{1}{\Rem} \AA\right)\right)^{-1}\AA^\top\MM_\E^{-1}
\end{equation}
and
\begin{equation}
	\MM^{-1}_{2,1} = -\left(\CC - \AA^\top \MM_\E^{-1}\left(\tilde{\GG}-\frac{1}{\Rem} \AA\right)\right)^{-1}\AA^\top\MM_\E^{-1}.
\end{equation}

We precondition $\mathcal{S}^{(\E, \B)}$ for both linearisations in the stationary case by 
\begin{equation}\label{eq:SchurStat}
	\tilde{\mathcal{S}}^{(\E, \B)} =
	\begin{bmatrix}
		\FF+\DD & \BB^\top \\
		\BB & \mathbf{0}
	\end{bmatrix},
\end{equation} 
and in the transient case by
\begin{equation}\label{eq:SchurAlpha}
	\tilde{\mathcal{S}}^{(\E, \B)}_{\alpha}:=
	\begin{bmatrix}
		\FF+\alpha\DD & \BB^\top \\
		\BB & \mathbf{0}
	\end{bmatrix},
	\quad \quad \alpha = \frac{\Delta t}{\Delta t + \Rem h^2 + \delta \Rem h \|\u^n\|_{L^2}\Delta t }.
\end{equation}

In the following, we motivate this choice of preconditioners and emphasise the cases in which these Schur complement approximations are exact. Therefore, we mainly follow \cite{Phillips2016}, but adapt the computations for our formulation which includes the electric field $\E$ instead of a Lagrange multiplier $r$.

For the simplification of the outer Schur complement $\mathcal{S}^{(\E, \B)}$ we must find approximations for
\begin{equation}\label{eq:SchurcompExtraterms}
	\mathcal{K}_1 \coloneqq	\DD - \JJ\MM^{-1}_{1,1} \GG \quad \text{  and  } \quad  \mathcal{K}_2 \coloneqq -(\tilde{\JJ} + \tilde{\DD}_1 + \tilde{\DD}_2)\MM^{-1}_{2,1}\GG.
\end{equation}
Note that the first summand of $\JJ\MM^{-1}_{1,1} \GG $ is $\JJ \MM_\E^{-1} \GG$ which equals $\DD$. Hence, $\mathcal{K}_1$ simplifies to the second summand of  $\JJ \MM_\E^{-1} \GG$, i.e., 
\begin{equation}\label{eq:Schurcomppart1}
	\mathcal{K}_1 = -\JJ \MM_\E^{-1}\left(\tilde{\GG}-\frac{1}{\Rem} \AA\right)\left(\CC - \AA^\top \MM_\E^{-1}\left(\tilde{\GG}-\frac{1}{\Rem} \AA\right)\right)^{-1}\AA^\top\MM_\E^{-1}\GG
\end{equation}
which corresponds on a continuous level to
\begin{equation}\label{eq:Schurcomppart1cont}
	\scalemath{0.945}{
		-S\, \B^n\times \left( \left(\delta\,\u^n \times \cdot - \frac{1}{\Rem} \vcurl\right)\left(\frac{\eta}{\Delta t} I-\frac{1}{\Rem} \Delta - \delta\,\vcurl(\u^n \times \cdot)\right)^{-1}  \vcurl(\u \times \B^n)\right),
	}
\end{equation}
where  $\cdot$ denotes a placeholder for the input of the corresponding operators.
Moreover, we have used $\delta \in \{0,1\}$ to distinguish between the Picard ($\delta=0$) and Newton ($\delta=1$) linearisations. In the discrete counterpart \eqref{eq:Schurcomppart1}, the matrix arising in the Picard iteration is made by dropping all terms with a tilde.
The continuous expression for $\mathcal{K}_2$ is given by
\begin{align}\label{eq:Schurcomppart2cont}
	\begin{split}
		\delta\,S\,(\cdot \times \E^n + \cdot & \times(\u^n \times \B^n) + \B^n\times(\u^n\times \cdot)) \\
		& \left(\frac{\eta}{\Delta t}{I}-\frac{1}{\Rem}
		\Delta - \delta \vcurl(\u^n \times \cdot)\right)^{-1} \vcurl(\u \times \B^n).
	\end{split}
\end{align}

\subsubsection{The two-dimensional case}
For the Picard linearisation, expression \eqref{eq:Schurcomppart1} simplifies to $\DD$ in the stationary case. This follows immediately from the two-dimen\-sional analogue of \eqref{eq:laplaceidentity} and the identity \cite{Phillips2014}
\begin{equation}\label{eq:curlidentity2d}
	\scurl (-\Delta)^{-1} \vcurl \varphi = \varphi
\end{equation}
which implies for our structure-preserving discretisation that
\begin{equation}
	\AA (\CC + \AA^\top \MM_\E^{-1}\AA)^{-1}\AA^\top = \MM_\E.
\end{equation}
That means in the two-dimensional stationary case the outer Schur complement for the Picard iteration is exactly given by
\eqref{eq:SchurStat},
i.e.,~the Navier--Stokes block with the linearised Lorentz force.

In the transient case, the Schur complement for the Picard linearisation can no longer be calculated exactly. The behaviour of the Schur complement now depends on which of the terms $\frac{1}{\Delta t} I$ and $\frac{1}{\Rem} \Delta$ dominates in \eqref{eq:Schurcomppart1cont}. If $\frac{1}{\Delta t}$ is small in comparison to $\frac{1}{\Rem h^2}$, a good approximation of \eqref{eq:Schurcomppart1cont} is given, as in the stationary case, by $\tilde{\mathcal{S}}^{(\E, \B)}$. If $\frac{1}{\Delta t}$ dominates over $\frac{1}{\Rem h^2}$, \eqref{eq:Schurcomppart1cont}  is approximately given by 
\begin{equation}
	S\, \B^n\times \left( \frac{1}{\Rem} \vcurl\, \left(\frac{1}{\Delta t} I \right)^{-1}  \vcurl(\u \times \B^n)\right).
\end{equation}
Hence, its magnitude can be approximated by $\frac{S \|\B^n\|^2\Delta t}{\Rem\, h^2}\ll1$ for moderate coupling numbers and therefore we neglect this term by using the approximation
$\begin{bmatrix}
	\FF & \BB^T \\
	\BB & \mathbf{0}
\end{bmatrix} $ for the Schur complement in this case.

 To also include the intermediate regime, we use the approximation of Phillips et al.\ \cite{Phillips2016} who suggest to use
\eqref{eq:SchurAlpha}.
The expression for $\alpha$ interpolates between the above mentioned dominating cases, since $\alpha \approx 0$ if $\frac{1}{\Delta t}\gg \frac{1}{\Rem h^2}$  and $\alpha \approx 1$ if  $\frac{1}{\Delta t}\ll \frac{1}{\Rem h^2}$.


A simplification for the full Newton linearisation of $\mathcal{S}^{(\E, \B)}$ is not straightforward, but our numerical tests suggest that $\tilde{\mathcal{S}}^{(\E, \B)}$ and $\tilde{\mathcal{S}}^{(\E, \B)}_{\alpha}$ are acceptable preconditioners for $\mathcal{S}^{(\E, \B)}$ in the stationary and transient cases, deteriorating only for high $S$ and $\Rem$.
This can be explained by the fact that for small $\Rem$ or $\Delta t$ the terms $\frac{1}{\Rem}\vcurl$ and $\frac{\eta}{\Delta t}{I}-\frac{1}{\Rem} \Delta$ dominate over $\delta \u^n\times \cdot$ and $\delta \vcurl(\u^n\times \cdot)$ in \eqref{eq:Schurcomppart1cont}. Remember that the terms that include a $\delta$ do not appear in the Picard iteration and were hence neglected in the previous derivation for the Picard iteration. Moreover, the term $\mathcal{K}_2$ is not included in our preconditioner for the Newton scheme which should deteriorate the performance for large $S$.


\subsubsection{The three-dimensional case}\label{sec:derivationofblockpreconditioner3d}
The main difficulty in three dimensions is that the identity \eqref{eq:curlidentity2d} no longer holds. Therefore, $\tilde{\mathcal{S}}^{(\E, \B)}$ is not the exact outer Schur complement for the Picard linearisation in the stationary case. In \cite{Phillips2016} the same approximation from the two-dimensional case is used in three dimensions. Based on the argument for the two-dimensional case in the previous subsection, we expect this approximation to work well when the term $\Delta t$ dominates and to deteriorate in the other cases, especially in the stationary case.
The three-dimensional performance of this preconditioner could be substantially improved with a better approximation of $\vcurl \Delta^{-1} \vcurl$ than a scaled identity.

We briefly comment on the main part of the outer Schur complement in the stationary case, given by
\begin{equation}\label{eq:SchurComp3dMain}
	S \B^n \times \left[\vcurl \Delta^{-1} \vcurl(\u \times \B^n) \right].
\end{equation}
As shown in \cite[Chapter 4]{PhillipsPHD} one can rewrite $\vcurl \Delta^{-1} \vcurl$ as $I -\nabla \Delta_r^{-1} \nabla \cdot$, where $\Delta_r$ denotes a scalar Laplacian. These two representations show that the operator is the identity on divergence-free functions and maps curl-free functions to zero. Hence, this operator corresponds to the orthogonal $L^2$-projection of a vector field onto its divergence-free part, which we denote by $\mathbb{P}$. Thus, the weak form of \eqref{eq:SchurComp3dMain} is given by
\begin{equation}\label{eq:SchurComp3dMainWeakForm}
	S(\mathbb{P}(\u \times \B^n), \mathbb{P}(\v \times \B^n)).
\end{equation}
The key challenge is then to find a sparse approximation of \eqref{eq:SchurComp3dMainWeakForm}. We do not further address this challenge here and focus instead on the outer Schur complement that eliminates the $(\u,p)$ block.


%

\subsection{Outer Schur complement eliminating the $(\u, p)$ block}\label{sec:OuterSchurup}
The outer Schur complement eliminating the $(\u, p)$ block is given by
\begin{equation}
	\label{eq:outerSchurCompNewton_BE}
	\mathcal{S}^{(\u, p)} =
	\begin{bmatrix}
		\MM_\E& \tilde{\GG} - \frac{1}{\Rem} \AA \\
		\AA^\top & \CC
	\end{bmatrix}
	-
	\begin{bmatrix}
		\GG & \zerov\\
		\zerov & \zerov
	\end{bmatrix}
	\begin{bmatrix}
		\FF+\DD & \BB^\top \\
		\BB & \mathbf{0}
	\end{bmatrix}^{-1}
	\begin{bmatrix}
		\JJ & \tilde{\JJ} + \tilde{\DD}_1 + \tilde{\DD}_2\\
		\zerov & \zerov
	\end{bmatrix}.
\end{equation}
The outer Schur complement for the Newton iteration is given by
\begin{equation}
	\mathcal{S}^{(\u, p)}=
	\begin{bmatrix}
		\MM_\E - \GG\NN^{-1}_{1,1} \JJ & \tilde{\GG} - \frac{1}{\Rem}\AA -\GG\NN^{-1}_{1,1}(\tilde{\JJ} + \tilde{\DD}_1 + \tilde{\DD}_2) \\
		\AA^\top & \CC
	\end{bmatrix},
\end{equation}
where
\begin{equation}
	\NN^{-1}_{1,1} = (\FF+\DD)^{-1} - (\FF+\DD)^{-1}\BB^\top(-\BB(\FF+\DD)^{-1}\BB^\top)^{-1}\BB(\FF+\DD)^{-1}.
\end{equation}
For this strategy, further simplifications of the Picard or Newton linearisations are not straightforward. Our numerical results in the next section show that 
\begin{equation}
	\tilde{\mathcal{S}}^{(\u, p)} =
	\begin{bmatrix}
		\MM_\E& \tilde{\GG} - \frac{1}{\Rem} \AA \\
		\AA^\top & \CC
	\end{bmatrix}
\end{equation}
works very well as a preconditioner for both schemes. Indeed, in contrast to the previous order of elimination, this approximation works qualitatively the same in two and three dimensions. We expect the approximation to deteriorate in the stationary case for very high $\Rem$, since the missing term $-\GG\NN^{-1}_{1,1}(\tilde{\JJ}+\tilde{\DD}_1 + \tilde{\DD}_2)$ in the Schur complement approximation gains more influence in comparison to $- \frac{1}{\Rem}\AA$. We also observe this numerically in the next section. 

However, we make the crucial observation that a good approximation of the outer Schur complement is maintained for high coupling numbers $S$, which will clearly be seen in our numerical results in Section \ref{sec:stationaryliddrivencavityproblemin3d} below. This behaviour is perhaps explained by the fact that $\NN^{-1}_{1,1}$ also includes a factor $S$ in the inverse of $(\FF+\DD)$, which balances the factor of $S$ in the matrices $\JJ, \tilde{\JJ}, \tilde{\DD_1}$ and $\tilde{\DD_2}$. We believe that this inclusion of $S$ in the Schur complement approximation is crucial for the outperformance of the previous order of elimination for which it was obvious that the Schur complement approximation deteriorates with higher $S$.

To use these block preconditioners in practice, we must develop robust preconditioners for the electromagnetic and hydrodynamic subsystems.

\subsection{Parameter-robust relaxation}\label{sec:robustssc}

The equations in the hydrodynamic and electromagnetic blocks become difficult to solve in the parameter regimes of interest at high Reynolds and coupling numbers both due to the non-symmetric linearised advection and Lorentz force terms, and the addition of the symmetric positive semi-definite (SPSD) augmented Lagrangian terms. 
Standard multigrid methods are known to perform poorly for these kinds of problems. 

The key components for a robust multigrid method  for the SPSD augmented Lagrangian terms are a parameter-robust relaxation method, that efficiently damps error modes in the kernel of the singular operators, and a kernel-preserving prolongation operator, as revealed in the seminal work of Sch\"oberl~\cite{schoberl1999b}. The non-symmetric terms are more troublesome, but numerical results have shown \cite{Farrell2020,Farrell2018} that subspace correction methods can still perform well for the Navier--Stokes equations at high Reynolds numbers.

A recent summary of the theory of robust relaxation methods can be found in \cite{farrell2019pcpatch}. Briefly, we consider the multigrid relaxation methods in the framework of subspace correction methods \cite{xu1992}. These decompose a (finite-dimensional) trial space $V$ as
\begin{equation}\label{eq:decomp}
	V = \sum_{i} V_i,
\end{equation}
where the sum is not necessarily direct.
The parallel subspace correction method applied to a linear variational problem $a(u,v)=(f,v) \, \forall v \in V$ computes for an initial guess $u^k$ a correction $\delta u_i$ to the error $e=u-u^k$ in each subspace $V_i$ by solving
\begin{equation}
	a(\delta u_i, v_i) = (f, v_i) - a(u^k, v_i) \text{ for all } v_i \in V_i,
\end{equation}
and sets $u^{k+1} = u^k + \sum_i w_i \delta u_i$ for damping parameters $w_i$.
A rigorous statement regarding the properties the decomposition \eqref{eq:decomp} and the considered bilinear form $a$ must fulfil to yield a robust relaxation method can be in found in \cite[Theorem 4.1]{schoberl1999b}.
A key property is that the kernel $\mathcal{N}$ of the SPSD terms is decomposed over the subspaces, i.e.,
\begin{equation}\label{eq:kerneldecomp}
	\mathcal{N} = \sum_i (V_i \cap \mathcal{N}).
\end{equation}
This property means that it must be possible to write any kernel function as the sum of kernel functions in the subspaces $V_i$. This implies that the subspaces $V_i$ must
be at least rich enough to support nonzero kernel functions. The choice of the space decomposition \eqref{eq:decomp} is often made by consideration of the
discrete Hilbert complexes underpinning the discretisation. We outline specific examples for such decompositions and discrete complexes in the next two subsections. 

\subsection{Solver for the hydrodynamic block}\label{sec:solverforschurcomp}


In order to implement the block factorisation \eqref{eq:blockfactor} as the outer preconditioner, we need a solver for the Navier--Stokes subsystem. To do this,
we will apply ideas of parameter-robust multigrid relaxation described in Section \ref{sec:robustssc}, albeit without a theoretical guarantee of success. The
variational statement of the PDE we wish to solve is
\begin{alignat}{1}
	\frac{2}{\Re}(\mathbf{\varepsilon}(\u), \mathbf{\varepsilon}(\v)) + (\u^n\cdot \grad \u, \v) +  (\u\cdot \grad \u^n, \v) &+ \gamma (\div \u, \div \v) \nonumber\\
	+ S (\B^n \times (\u \times \B^n), \v) - (p,\div \v) &= (\mathbf{f},\v) \quad \, \forall\, \v \in \Hozv,\\
	-(\div \u, q) &= 0 \qquad \quad \, \forall\, q \in L^2(\Omega). \nonumber
\end{alignat}
This corresponds to the standard Newton linearisation of the Navier--Stokes equations with an augmented Lagrangian term, plus the linearisation of the Lorentz force $\mathcal{D}$. We follow the approach of~\cite{Hong2015,Farrell2018,Farrell2020} to solve this system. The first idea is to use the augmented Lagrangian term $-\gamma \grad \div \u$ to approximate the inner Schur complement of the hydrodynamic block by choosing a large $\gamma$, e.g., $\gamma \approx 10^4$. One can show \cite[Theorem 3.2]{Bacuta2006} that the inner Schur complement of the augmented system $\tilde{\mathcal{S}}_\text{NS}$ satisfies
\begin{equation}\label{eq:SchurCompNS}
	\tilde{\mathcal{S}}_\text{NS}^{-1} = \mathcal{S}_\text{NS}^{-1}-\gamma \MM_p^{-1},
\end{equation}
where $\mathcal{S}_\text{NS}^{-1}$ denotes the Schur complement of the system without the augmented Lagrangian term and $\MM_p$ denotes the pressure mass matrix. Therefore, for large $\gamma$ the pressure mass matrix scaled by $-1/(1/\Re + \gamma)$ is a good approximation for $\tilde{\mathcal{S}}_\text{NS}$. As the discretisation considered in this work uses discontinuous pressures, the pressure mass matrix is block-diagonal and hence directly invertible. In the transient case $\mathcal{S}_\text{NS}^{-1}$ can be further approximated by the inverse of the stationary Schur complement plus an extra term $- \Delta t L_p^{-1}$  \cite{Heister2013}, where $L_p$ corresponds to the Poisson problem for $p$ with Neumann boundary conditions. In our numerical examples this extra term makes little difference as we only consider time steps $\frac{1}{\Delta t} \ll \gamma$, and we therefore neglect it.

Since the augmented Lagrangian term has a large kernel that consists of all solenoidal vector fields, a robust multigrid scheme as described in Section \ref{sec:robustssc} must be used to solve the augmented momentum block.
For the $\hdiv\times L^2$-conforming discretisation the \emph{star iteration} \cite[Section 4]{Farrell2018} can be used as a robust relaxation method. The subspace decomposition is defined as
\begin{equation}\label{eq:stardecomp}
	\mathbf{V}_i = \{\v \in \mathbf{V}_h: \mathrm{supp}(\v) \subset K_i \},
\end{equation}
where $K_i$ is the patch of elements sharing the vertex $i$ in the mesh. Example patches are shown in Figure~\ref{fig:star}. Since we use a structure-preserving discretisation, the properties of the de Rham complexes \eqref{eq:deRhamdiscret3d} and \eqref{eq:deRhamdiscret2d} imply that \eqref{eq:stardecomp} fulfils the kernel decomposition property \eqref{eq:kerneldecomp}. This property was also used in \cite{Arnold2000} to construct a robust smoother for the $\hdiv$ and $\hcurl$ Riesz maps and in \cite{Hong2015} for the Stokes equations.
In this case we may employ the standard prolongation operator induced by the finite element discretisation, because the uniformly refined mesh hierarchy we consider is nested.
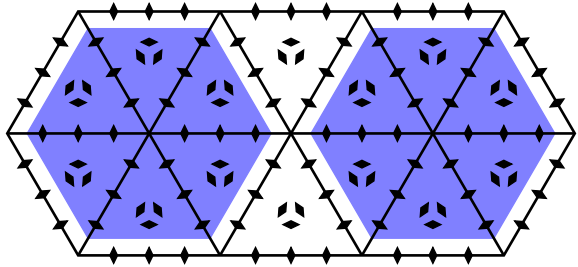
\begin{figure}[htbp]
	\centering
	\begin{tikzpicture}
  \foreach \i in {0, 1, 2, 3, 4} {
    \coordinate (0-\i) at ($(\i*1.5, 0)$);
    \coordinate (1-\i) at ($(0-\i) + (60:1.5)$);
    \coordinate (2-\i) at ($(1-\i) + (120:1.5)$);
    \coordinate (3-\i) at ($(\i*1.5, 2.6-2*1.3/3)$);
    \coordinate (4-\i) at ($(\i*1.5, 1.3-1*1.3/3)$);
    \coordinate (5-\i) at ($(0.75+ \i*1.5, 2.6-1*1.3/3)$);
    \coordinate (6-\i) at ($(0.75+ \i*1.5, 1.3-2*1.3/3)$);
  }
  
  \foreach \i in {1, 2, 3, 4}  {
    \node[transform shape,
    rotate=90,
    shape=diamond,
    fill=black,
    black,
    inner sep=0pt,
    minimum height=5pt,
    minimum width=2.5pt, draw] at ([shift={(90:3pt)}]4-\i) {};

    \node[transform shape,
    rotate=30,
    shape=diamond,
    fill=black,
    black,
    inner sep=0pt,
    minimum height=5pt,
    minimum width=2.5pt, draw] at ([shift={(30:-3pt)}]4-\i) {};

    \node[transform shape,
    rotate=-30,
    shape=diamond,
    fill=black,
    black,
    inner sep=0pt,
    minimum height=5pt,
    minimum width=2.5pt, draw] at ([shift={(-30:3pt)}]4-\i) {};
  }

  \foreach \i in {1, 2, 3, 4}  {
    \node[transform shape,
    rotate=90,
    shape=diamond,
    fill=black,
    black,
    inner sep=0pt,
    minimum height=5pt,
    minimum width=2.5pt, draw] at ([shift={(270:3pt)}]3-\i) {};

    \node[transform shape,
    rotate=210,
    shape=diamond,
    fill=black,
    black,
    inner sep=0pt,
    minimum height=5pt,
    minimum width=2.5pt, draw] at ([shift={(210:-3pt)}]3-\i) {};

    \node[transform shape,
    rotate=150,
    shape=diamond,
    fill=black,
    black,
    inner sep=0pt,
    minimum height=5pt,
    minimum width=2.5pt, draw] at ([shift={(150:3pt)}]3-\i) {};
  }
  
  \foreach \i in {1, 2, 3}  {
    \node[transform shape,
    rotate=90,
    shape=diamond,
    fill=black,
    black,
    inner sep=0pt,
    minimum height=5pt,
    minimum width=2.5pt, draw] at ([shift={(90:3pt)}]5-\i) {};

    \node[transform shape,
    rotate=30,
    shape=diamond,
    fill=black,
    black,
    inner sep=0pt,
    minimum height=5pt,
    minimum width=2.5pt, draw] at ([shift={(30:-3pt)}]5-\i) {};

    \node[transform shape,
    rotate=-30,
    shape=diamond,
    fill=black,
    black,
    inner sep=0pt,
    minimum height=5pt,
    minimum width=2.5pt, draw] at ([shift={(-30:3pt)}]5-\i) {};
  }

  \foreach \i in {1, 2, 3}  {
    \node[transform shape,
    rotate=90,
    shape=diamond,
    fill=black,
    black,
    inner sep=0pt,
    minimum height=5pt,
    minimum width=2.5pt, draw] at ([shift={(270:3pt)}]6-\i) {};

    \node[transform shape,
    rotate=210,
    shape=diamond,
    fill=black,
    black,
    inner sep=0pt,
    minimum height=5pt,
    minimum width=2.5pt, draw] at ([shift={(210:-3pt)}]6-\i) {};

    \node[transform shape,
    rotate=150,
    shape=diamond,
    fill=black,
    black,
    inner sep=0pt,
    minimum height=5pt,
    minimum width=2.5pt, draw] at ([shift={(150:3pt)}]6-\i) {};
  }

  \draw[line join=miter, thick]
  (0-1) -- (0-2) -- (0-3) -- (0-4)
  -- (1-4)
  -- (2-4)
  -- (2-3) -- (2-2) -- (2-1)
  -- (1-0)
  -- cycle;
  \draw[line join=miter, thick] (1-0) -- (1-4);
  \foreach \i/\j in {1/2, 2/3, 3/4} {
    \foreach \k in {0, 2} {
      \foreach \n in {0.25, 0.5, 0.75} {
        \node[transform shape,
        shape=diamond,
        fill=black,
        black,
        inner sep=0pt,
        minimum height=5pt,
        minimum width=2.5pt, draw] at ($(\k-\i)!\n!(\k-\j)$) {};
      }
    }
    \foreach \n in {0.25, 0.5, 0.75} {
      \node[transform shape,
      rotate=60,
      shape=diamond,
      fill=black,
      black,
      inner sep=0pt,
      minimum height=5pt,
      minimum width=2.5pt, draw] at ($(0-\i)!\n!(1-\i)$) {};

      \node[transform shape,
      rotate=-60,
      shape=diamond,
      fill=black,
      black,
      inner sep=0pt,
      minimum height=5pt,
      minimum width=2.5pt, draw] at ($(1-\i)!\n!(2-\i)$) {};

      \node[transform shape,
      rotate=-60,
      shape=diamond,
      fill=black,
      black,
      inner sep=0pt,
      minimum height=5pt,
      minimum width=2.5pt, draw] at ($(0-\j)!\n!(1-\i)$) {};

      \node[transform shape,
      rotate=60,
      shape=diamond,
      fill=black,
      black,
      inner sep=0pt,
      minimum height=5pt,
      minimum width=2.5pt, draw] at ($(1-\i)!\n!(2-\j)$) {};
    }
    \draw[line join=miter, thick] (0-\i) -- (1-\i);
    \draw[line join=miter, thick] (1-\i) -- (2-\i);
    \draw[line join=miter, thick] (0-\j) -- (1-\i);
    \draw[line join=miter, thick] (1-\i) -- (2-\j);
  }

  \foreach \n in {0.25, 0.5, 0.75} {
      \node[transform shape,
      rotate=-60,
      shape=diamond,
      fill=black,
      black,
      inner sep=0pt,
      minimum height=5pt,
      minimum width=2.5pt, draw] at ($(0-1)!\n!(1-0)$) {};

      \node[transform shape,
      rotate=60,
      shape=diamond,
      fill=black,
      black,
      inner sep=0pt,
      minimum height=5pt,
      minimum width=2.5pt, draw] at ($(1-0)!\n!(2-1)$) {};

      \node[transform shape,
      rotate=60,
      shape=diamond,
      fill=black,
      black,
      inner sep=0pt,
      minimum height=5pt,
      minimum width=2.5pt, draw] at ($(0-4)!\n!(1-4)$) {};

      \node[transform shape,
      rotate=-60,
      shape=diamond,
      fill=black,
      black,
      inner sep=0pt,
      minimum height=5pt,
      minimum width=2.5pt, draw] at ($(1-4)!\n!(2-4)$) {};

      \foreach \i/\j in {0/1, 1/2, 2/3, 3/4} {
        \node[transform shape,
        shape=diamond,
        fill=black,
        black,
        inner sep=0pt,
        minimum height=5pt,
        minimum width=2.5pt, draw] at ($(1-\i)!\n!(1-\j)$) {};
        }
  }

  \begin{scope}[on background layer]
    \draw[blue!50, fill=blue!50] 
    ([shift={(0:6pt)}]1-0) --
    ([shift={(60:6pt)}]0-1) --
    ([shift={(120:6pt)}]0-2) --
    ([shift={(180:6pt)}]1-2) --
    ([shift={(240:6pt)}]2-2) --
    ([shift={(300:6pt)}]2-1) --
    cycle;
    \draw[blue!50, fill=blue!50]
    ([shift={(0:6pt)}]1-2) --
    ([shift={(60:6pt)}]0-3) --
    ([shift={(120:6pt)}]0-4) --
    ([shift={(180:6pt)}]1-4) --
    ([shift={(240:6pt)}]2-4) --
    ([shift={(300:6pt)}]2-3) --
    cycle;
  \end{scope}
\node[fit=(current bounding box),inner sep=2mm]{};
\end{tikzpicture}
	\caption{Star patch for $\mathbb{BDM}_2$-elements.}
	\label{fig:star}
\end{figure}

A fluid-Reynolds-robust multigrid method for the conforming Scott--Vogelius discretisation was recently developed in \cite{Farrell2020}. Remember that this discretisation is only stable on certain types of meshes and we consider a mesh hierarchy here which is barycentrically refined. This required the design of another specialised multigrid method and we refer to  \cite{Farrell2020} for the details of this method. The two main parts include a so-called macrostar iteration which is used for the relaxation and special prolongation operator which is necessary since the barycentrically refined mesh hierarchy is not nested. 

The velocity block further includes terms given by the convection-diffusion term $(\u\cdot \grad) \u$, the linearisation of the Lorentz force $S\, \B^n \times (\u \times \B^n)$ and the stabilisation term \eqref{eq:stabBurman}.
Numerical experiments in \cite{Farrell2020} and in the next Section~\ref{sec:numericalresults} show that these terms only degrade the performance of the preconditioner at high Reynolds and coupling numbers. As we have mentioned before, these somewhat surprising numerical observations are not backed up by theory since these terms do not fit in the framework of Section \ref{sec:robustssc}, and applying geometric multigrid methods to problems with strong advection typically requires special care, since these methods are primarily developed for elliptic PDEs. The kernel of the stabilisation \eqref{eq:stabBurman} consists of all $C^1$ vector fields. Therefore, the stabilisation term slightly degrades the performance of the solver, but the impact is not very significant as the factor $\mu h_{\partial K}^2$ is small.

\subsection{Solver for the electromagnetic block}\label{sec:solverformagneticblock}
The weak formulation of the electromagnetic block is given by
\begin{equation}
	\begin{aligned}
		(\E, \F) - \frac{1}{\Rem}(\B, \vcurl \F) + \delta\, (\u^n \times \B, \F) &= 0 &&\forall\, \F \in \hzcurl,\\
		\frac{\eta}{\Delta t} (\B, \C) + (\vcurl \E, \C) + \frac{1}{\Rem}(\div \B, \div \C) &= (\f, \C) &&\forall\, \C \in \hzdiv.
	\end{aligned}
\end{equation}
Recall that $\eta,\delta \in \{0,1\}$ distinguish between the stationary ($\eta=0$) and transient ($\eta = 1$) cases and the Picard ($\delta=0$) and Newton ($\delta = 1$) linearisations.
Eliminating $\E$, this corresponds to a mixed formulation of
\begin{equation}
	\begin{aligned}
		\frac{\eta}{\Delta t} \B +  \frac{1}{\Rem}\left(\vcurl \vcurl \B - \grad \div \B\right)   + \delta\, \vcurl(\u^n\times\B) &= \f \text{ in } \Omega,\label{eq:LaplaceGtilde} \\
		\B \cdot \mathbf{n} &= 0 \text{ on } \partial \Omega, \\
		\frac{1}{\Rem} \vcurl \B - \delta\, \u^n\times \B &= \mathbf{0} \text{ on } \partial \Omega.
	\end{aligned}
\end{equation}
For the Picard linearisation, this problem simplifies to the mixed formulation for the standard vector Laplace problem \cite{Arnold2006} with boundary conditions $\B\cdot \n = \vcurl \B =\mathbf{0}$ on $\del\Omega$. Chen et al.\ \cite{Chen2018} propose a Schur complement solver and Arnold et al.\ \cite{Arnold2006} propose a norm-equivalent block diagonal preconditioner for the mixed formulation.
We also found that the star multigrid solver applied monolithically to the electromagnetic block \eqref{eq:LaplaceGtilde} results in an efficient solver and employ this solver in our numerical examples.
All of the solvers described show $\Rem$-robust performance.

In contrast, the presence of the additional term $\vcurl(\u^n \times \B)$ in the Newton linearisation, which has a non-trivial kernel, makes the problem almost singular for high $\Rem$ in the stationary case and hence requires a special multigrid method.
Unfortunately the troublesome term $\vcurl(\u^n \times \B)$ is not symmetric and thus does not fit the available analytical framework of Sch\"oberl. Our considerations on this point are therefore necessarily heuristic. Some insight may be gained by employing the vector identity
\begin{equation}
	\vcurl (\mathbf{A} \times \mathbf{B}) =  (\mathbf{B} \cdot \nabla) \mathbf{A} -  (\mathbf{A} \cdot \nabla) \mathbf{B} +
	\mathbf{A}(\nabla \cdot \mathbf{B}) - \mathbf{B}(\nabla \cdot \mathbf{A})
\end{equation}
to rewrite \eqref{eq:LaplaceGtilde} to
\begin{equation}
	\frac{\eta}{\Delta t} \B -\frac{1}{\Rem}\Delta \B - (\B \cdot \grad) \u^n + (\u^n \cdot \nabla) \B  - \u^n(\nabla \cdot \mathbf{B}) - \mathbf{B}(\nabla \cdot \u^n).
\end{equation}
The last term $ - \mathbf{B}(\nabla \cdot  \u^n)$ vanishes since we exactly enforce $\nabla \cdot \u^n=0$ in each step.
The terms  $-(\B \cdot \grad) \u^n + (\u^n \cdot \nabla) \B $ are reminiscent of the Newton linearisation of the advection term $(\u \cdot \grad) \u$ of the Navier--Stokes equation, for which it has been demonstrated that a star multigrid method is effective \cite{Farrell2018}. On the other hand, this similarity also suggests that we cannot hope for a simpler solver than a specialised kernel-capturing multigrid method to solve for this block. Numerical experiments with our approach applied monolithically do indeed yield a robust solver for the stationary and transient cases in two dimensions, and in the transient case in three dimensions for sufficiently small $\Delta t$. We observe in our numerical tests that in three dimensions the solver breaks down for $\Rem \approx 700$ for a stationary lid-driven cavity problem.

\section{Numerical results} \label{sec:numericalresults}

In this section, we present numerical results for the Picard and Newton linearisation described in the previous sections. We investigate three test problems: the stationary Hartmann problem, the stationary and transient version of a lid-driven cavity problem and a transient island-coalescence problem. The numerical results were produced on ARCHER2, the UK national supercomputer, which consists of 5,860 compute nodes each built of two AMD Zen2 7742 processors with 64 2.25 GHz cores  and 256 GB of memory.

\subsection{Algorithm details}\label{sec:algorithmdetails}
The algorithm is implemented in Firedrake \cite{rathgeber2016} and uses the solver package PETSc \cite{balay2019}.
It is well-known that the convergence of the nonlinear scheme depends heavily on the initial guess and might fail to converge for high Reynolds numbers with poor initial guesses. To circumvent this problem we perform continuation in the Reynolds numbers and coupling number, for the stationary problems. In the presented tables we always apply continuation to the variable in the column first and use each solution as the starting point for the continuation over the rows. We use the steps $1,100,1{,}000,5{,}000,10{,}000$ for $S$ and $1,500,1{,}000,3{,}000,5{,}000,7{,}000,10{,}000$ for $\Re$ and $\Rem$.
The reported nonlinear and linear iteration numbers correspond to the final solve in the continuation; however, the extra cost for the continuation should be kept in mind for stationary problems. For time-dependent problems, continuation is not necessary.

We use flexible GMRES \cite{saad1993} as the outermost Krylov solver since we apply GMRES in the multigrid relaxation. Moreover, we apply a block upper triangular preconditioner~\cite{Benzi2005}
\begin{equation}
	\mathcal{P}=
	\begin{pmatrix}
		\II &-\tilde{\MM}^{-1}\mathcal{K} \\
		\mathbf{0} &\II
	\end{pmatrix}
	\begin{pmatrix}
		\tilde{\MM}^{-1} &\mathbf{0} \\
		\mathbf{0} &\tilde{\mathcal{S}}^{-1}
	\end{pmatrix}
\end{equation}
to \eqref{eq:matrix_upBE}, where we denoted \eqref{eq:matrix_upBE} here as $\begin{pmatrix} \MM & \mathcal{K} \\ \mathcal{L} & \mathcal{N} \end{pmatrix}$. We also investigated a full block-LDU preconditioner without notable improvements in terms of iteration counts, which fits with the recent theoretical results in \cite{Southworth2020} and is also suggested by the eigenvalue analysis in \cite{Wathen2000}.

Both the block matrix $\MM$ and the outer Schur complement approximation $\mathcal{S}^{(\u, p)}$ are inverted approximately with two iterations of preconditioned FGMRES (denoted $\tilde{\MM}^{-1}$ and $\tilde{S}^{-1}$, respectively). The former uses the block preconditioner for the hydrodynamic block described in Section \ref{sec:solverforschurcomp}, the latter the monolithic multigrid method described in Section \ref{sec:solverformagneticblock}. In the numerical results we focus on taking the outer Schur complement that eliminates the hydrodynamic block, except for one case in Section \ref{sec:stationaryliddrivencavityproblemin3d}. In both multigrid methods we use six preconditioned GMRES iterations as the smoother on each level and the direct solver MUMPS \cite{MUMPS} to solve the problem on the coarsest grid. Since this relaxation is quite expensive, convergence in a very small number of outer iterations is required for efficiency. See Figure \ref{fig:solver_diagramm} for a graphical representation of the solver.
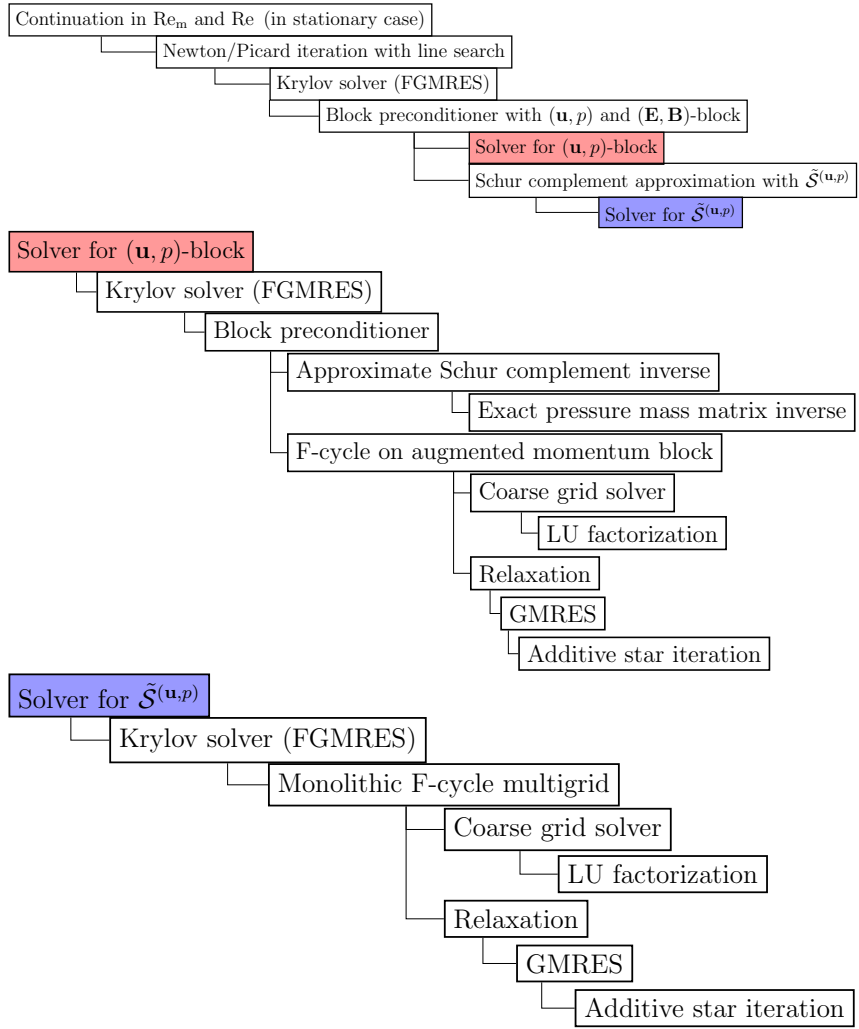
\begin{figure}[htb!]
	\centering
	\resizebox{0.75\textwidth}{!}{\begin{tikzpicture}[%
  every node/.style={draw=black, thick, anchor=west},
  grow via three points={one child at (-1.4,-0.7) and
  two children at (-1.4,-0.7) and (-1.4,-1.4)},
  edge from parent path={(\tikzparentnode.188) |- (\tikzchildnode.west)}]
  \node {Continuation in $\Rem$ and $\Re$ (in stationary case)}
  child { node{Newton/Picard iteration with line search}
  child { node{Krylov solver (FGMRES)}
    child { node {Block preconditioner with $(\u, p)$ and $(\E,\B)$-block}
      child { node[fill=red!40] {Solver for $(\u, p)$-block}
      }
      child { node {Schur complement approximation with $\tilde{\mathcal{S}}^{(\u, p)}$}
      child { node[fill=blue!40] {Solver for $\tilde{\mathcal{S}}^{(\u, p)}$}
      }
      }
     }
     }
    };
    
\end{tikzpicture}}
	\resizebox{0.75\textwidth}{!}{\begin{tikzpicture}[%
  every node/.style={draw=black, thick, anchor=west},
  grow via three points={one child at (-0.6,-0.7) and
  two children at (-0.6,-0.7) and (-0.6,-1.4)},
  edge from parent path={(\tikzparentnode.201) |- (\tikzchildnode.west)}]
   \node[fill=red!40]{Solver for $(\u, p)$-block}
      child { node {Krylov solver (FGMRES)}
        child { node {Block preconditioner}
          child { node {Approximate Schur complement inverse}
              child{ node {Exact pressure mass matrix inverse}}
          }
          child [missing] {}
          child { node {F-cycle on augmented momentum block}
              child { node {Coarse grid solver}
                child { node {LU factorization}}
              }
              child [missing] {}
              child { node {Relaxation}
                child { node {GMRES}
                  child { node {Additive star iteration}}
                }
              }
          }
        }
      };
\end{tikzpicture}}
	\resizebox{0.75\textwidth}{!}{\begin{tikzpicture}[%
  every node/.style={draw=black, thick, anchor=west},
  grow via three points={one child at (0.0,-0.7) and
  two children at (0.0,-0.7) and (0.0,-1.4)},
  edge from parent path={(\tikzparentnode.210) |- (\tikzchildnode.west)}]
   \node[fill=blue!40] {Solver for $\tilde{\mathcal{S}}^{(\u, p)}$}
      child { node {Krylov solver (FGMRES)}
          child { node {Monolithic F-cycle multigrid}
              child { node {Coarse grid solver}
                child { node {LU factorization}}
              }
              child [missing] {}
              child { node {Relaxation}
                child { node {GMRES}
                  child { node {Additive star iteration}}
                }
              }
          }
      };
\end{tikzpicture}}
	\caption{Graphical outline of the solver.}
	\label{fig:solver_diagramm}
\end{figure}

We have chosen relative and absolute tolerances of $10^{-10}$ and $10^{-6}$ for the residuals of the nonlinear solver and $10^{-7}$ and $10^{-7}$ for the residuals of the outermost linear solver, measured in the Euclidean norm. We use the  $\hdiv\times L^2$-conforming elements $\mathbb{BDM}_2 \times \mathbb{DG}_{1}$ for $(\u_h,p_h)$.   Moreover, we apply $\mathbb{CG}_2 \times \mathbb{RT}_2$ elements for $(E_h,\B_h)$ in 2D and $\mathbb{NED}1_2 \times \mathbb{RT}_2$ elements for $(\E_h,\B_h)$ in 3D. All problems are posed over the domain $\Omega=(-1/2,1/2)^d$, unless stated otherwise. For the multigrid hierarchy we use a regular coarse mesh of $16 \times 16$ cells and five levels of uniform refinement in 2D resulting in a $512\times512$ grid with 18.5 million degrees of freedom (DoFs). In three dimensions a coarse grid of $6\times6\times6$ cells with 3 levels of refinement is used which results in a $48\times48\times48$ grid with 25 million DoFs. When we consider a manufactured solution we always subtract $\int_\Omega p \,\mathrm{d}x$ from the pressure to fix the average of $p$ to be zero. We present iteration numbers in tables where two of the three parameters $\Re$, $\Rem$ and $S$ are varied and the third one is fixed to be 1 unless stated otherwise.

For time-dependent problems, we apply the second-order, L-stable BDF2 method with a fixed time-step. We compute the first time-step with Crank-Nicolson to provide the second starting value for BDF2. For the transient lid-driven cavity problem, we use a time-step of $\Delta t = 0.01$ and a final time of $T=0.1$. We did not choose a higher final time T because of budget limitations. However, we confirmed that the reported numbers are representative for higher T by computing the solution for a few parameters until $T=1$ without noticeable changes in the iteration counts. Moreover, we confirm the efficiency for more time-steps in the island coalescence problem where we iterate in the finest run until $T=15$ in 2400 time steps.


\subsection{Interpolating boundary data}
\label{sec:interp-boundary-data}
The theory from the previous sections has been formulated for homogeneous boundary conditions, but the generalisation is straightforward for non-homogeneous boundary conditions as shown in Section \ref{sec:discretization}. However, there is a subtle technicality in the implementation if one wants to enforce the divergence constraint $\div \B_h=0$ pointwise. Strong boundary conditions are enforced in a finite element code by interpolating the given boundary data onto the corresponding finite element space. If the interpolation of the boundary values $\mathbf{g}$ were exact, identity \eqref{eq:InterpolationPreserveRT} would imply that $\div \B_h=0$ holds. However, the degrees of freedom for the interpolation are moments and are usually implemented by a quadrature rule whose quadrature degree is based on the polynomial degree of the finite element space. If $\mathbf{g}$ is a non-polynomial expression, this quadrature rule might not interpolate the boundary condition exactly and therefore one loses the property that $\div \mathbf{g}_h=0$ on $\del \Omega$ holds exactly.

To circumvent this problem we use high-order quadrature rules for the evaluation of the degrees of freedom to ensure that the interpolation is exact up to machine precision.
In Figure \ref{tab:divBQuadDeg} we have illustrated the effect of the quadrature degree on the enforcement of the divergence constraint. We have used the method of manufactured solutions for a smooth problem to compute $\|\div \B_h\|_0$ for different quadrature degrees. Moreover, we have plotted the $L^2$-norm over $\partial \Omega$ of the interpolation of the divergence-free function $\B$ into the $\mathbb{RT}_2$ space. One can clearly observe that a quadrature degree of 2 for $\mathbb{RT}_2$ elements is not sufficient to enforce $\div \B_h=0$ up to machine precision. A higher quadrature degree preserves the divergence of the boundary data more accurately and leads to the pointwise enforcement of $\div \B_h=0$.
\begin{figure}[htb!]
	\centering
	\includegraphics[width=0.5\textwidth]{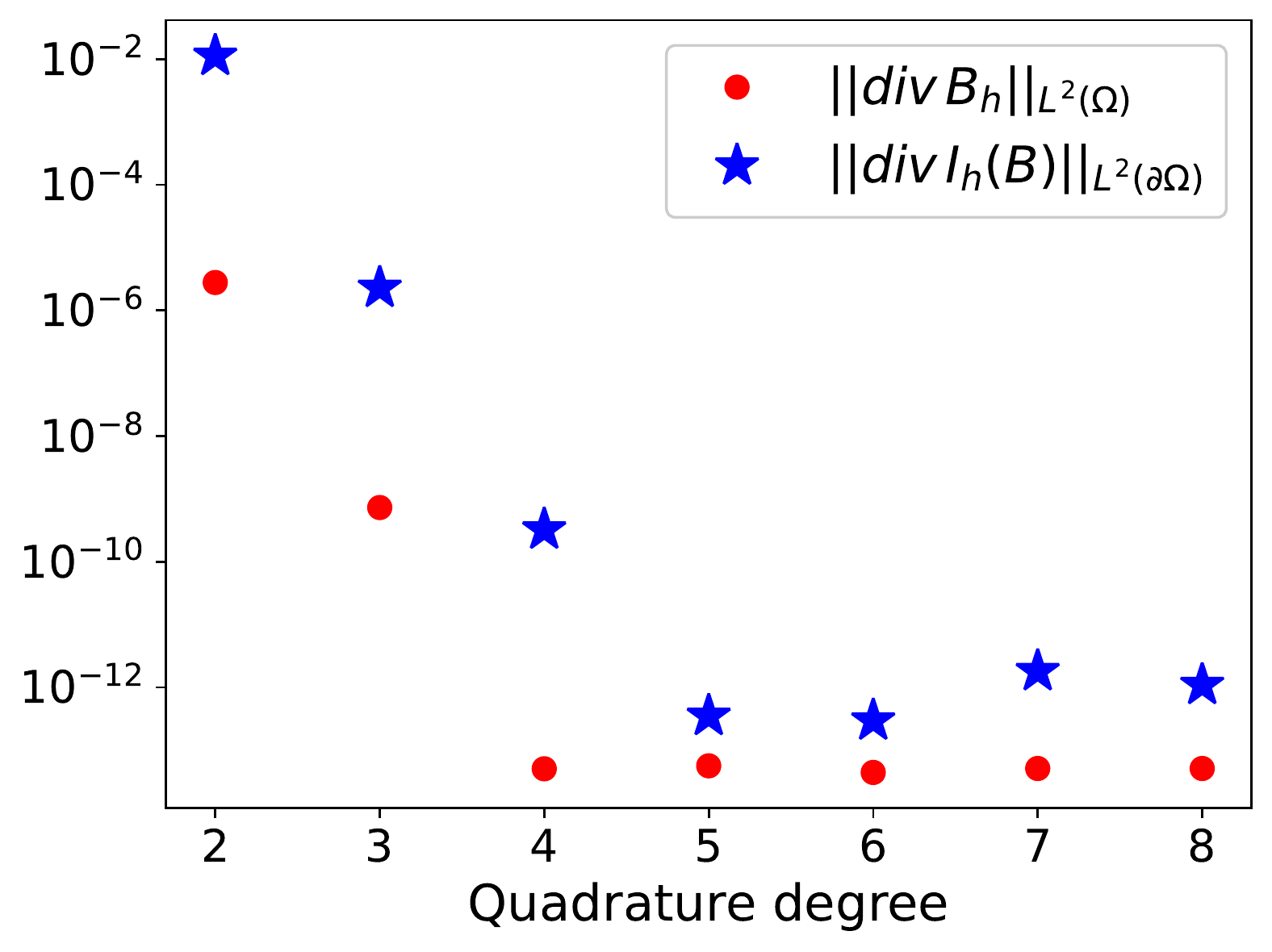}
	\caption{$L^2$-norm of the divergence of the solution $\B_h$ and the interpolant of the boundary condition for different quadrature degrees in the evaluation of the degrees of freedom for the Raviart--Thomas space.}
	\label{tab:divBQuadDeg}
\end{figure}

\subsection{Two-dimensional results}

\subsubsection{Hartmann flow}\label{sec:HartmannFlow}
First, we consider the Hartmann flow problem posed over $\Omega = (-1/2, 1/2)^2$ that describes the flow of a conducting fluid through a section of a channel to which a transverse magnetic field $\B_0=(0,1)^\top$ is applied. This problem was considered in \cite{adler2020monolithic, Wathen2017}.
The analytical solution is given by $\u=(u_1(y), 0)^\top$ and $\B=(B_1(y), 1)^\top$ with
\begin{align*}
	u_1(y) &= \frac{G \Re}{2 \Ha \tanh(\Ha/2)}\left(1 - \frac{\cosh(y\Ha)}{\cosh(\Ha/2)} \right),\\
	B_1(y) &= \frac{G}{2}\left( \frac{\sinh(y/\Ha)}{\sinh(\Ha/2)} -2y \right), \\
	p(x,y) &= - Gx - \frac{B_1^2(y)}{2}.
\end{align*}
Here, we used the Hartmann number $\Ha=\sqrt{S\Re \Rem}$ and $\mathrm{G} = \frac{2 \Ha \sinh(\Ha/2)}{\Re(\cosh(\Ha/2)-1)}$. The analytical solution for $E$ is computed via \eqref{eq:MHD2d3}. Note that for high $\mathrm{Ha}$ the computation of, e.g., $\sinh(\Ha/2)$ exceeds the range that double precision floating point numbers can represent. Therefore, we have chosen the following approximation for Hartmann numbers $\Ha \geq 100$ with $G = 2 \frac{\Ha}{\Re}$
\begin{align*}
	u_1(y) &= \frac{G \Re}{2 \Ha} \left(1 + \exp(\Ha(-y-\frac{1}{2})) - \exp(\Ha(y-\frac{1}{2}))  \right),\\
	B_1(y) &= \frac{G}{2} \left(\exp(\Ha(-y-\frac{1}{2})) - \exp(\Ha(y-\frac{1}{2}))  - 2y\right).
\end{align*}

The iteration counts for the three different linearisation methods are presented in Table \ref{tab:hartmann2d}.
For both schemes, we observe fairly constant Krylov iteration counts for $\Re$ and $S$ in the range of 1 to 10,000. In terms of the nonlinear convergence, the Picard linearisation takes more iterations for higher values of $S$ than the Newton linearisation. 

\begin{table}[htbp!]
	\centering
	\begin{tabular}{r|ccc|ccc}
		\toprule
		& \multicolumn{3}{c|}{Picard} &  \multicolumn{3}{c}{Newton} \\
		\midrule
		$S\backslash\Re$ & 1        & 1,000    & 10,000     & 1        & 1,000    & 10,000   \\
		\midrule
					1           & ( 3) 6.3 & ( 2) 6.0 & ( 2) 6.0 & ( 2) 7.0 & ( 2) 5.5 & ( 2) 6.0  \\
			1,000     & ( 6) 6.0 & ( 4) 5.8 & ( 4) 4.2 & ( 2) 7.5 & ( 2) 5.5 & ( 2) 4.5     \\
			10,000    & ( 6) 6.3 & ( 5) 5.0 & ( 4) 4.5 & ( 2) 7.0 & ( 2) 6.0 & ( 2) 6.0    \\
		\bottomrule
	\end{tabular}
	\caption{(Nonlinear iterations) Average outer Krylov iterations per nonlinear step for the Hartmann problem.\label{tab:hartmann2d}}
\end{table}

\subsubsection{Stationary lid-driven cavity in two dimensions}

Next, we consider a lid-driven cavity problem posed over $\Omega=(-1/2,1/2)^2$ for a background magnetic field $\B_0=(0,1)^\top$ which determines the boundary conditions $\B\cdot \n = \B_0\cdot \n$ on $\del \Omega$ and set $\f = \mathbf{0}$ \cite{Ma2016}. We impose the boundary condition $\u=(1,0)^\top$ at the boundary $y=0.5$ and homogeneous boundary conditions elsewhere. The problem models the flow of a conducting fluid driven by the movement of the lid at the top of the cavity. The magnetic field imposed orthogonal to the lid creates a Lorentz force that perturbs the flow of the fluid. 

For both linearisations we observe fairly constant Krylov iteration counts for $\Re$ and $S$ in the range of 1 to 10,000 in Table \ref{tab:ldc2dSRE}. In terms of the nonlinear convergence, the Picard linearisation sometimes takes slightly more iterations than the Newton linearisations, with slightly better linear iteration numbers.

\begin{table}[htbp!]
	\centering
	\begin{tabular}{r|ccc|ccc}
		\toprule
		& \multicolumn{3}{c|}{Picard} &  \multicolumn{3}{c}{Newton} \\
		\midrule
		$S\backslash\Re$ & 1        & 1,000    & 10,000     & 1        & 1,000    & 10,000   \\
		\midrule
		1           & ( 3) 5.3 & ( 4) 3.5 & ( 3) 4.3 & ( 2) 6.5 & ( 4) 3.5 & ( 3) 4.3   \\
		1,000     & ( 4) 3.5 & ( 3) 4.7 & ( 2) 8.5 & ( 2) 5.5 & ( 3) 4.7 & ( 2) 6.5     \\
		10,000    & ( 3) 5.0 & ( 3) 4.3 & ( 2) 7.0 & ( 2) 6.5 & ( 2) 6.0 & ( 2) 7.0     \\
		\bottomrule
	\end{tabular}
	\caption{(Nonlinear iterations) Average outer Krylov iterations per nonlinear step for the stationary lid-driven cavity problem in 2D.\label{tab:ldc2dSRE}}
\end{table}

As mentioned earlier, our scheme does not include a stabilisation for high magnetic Reynolds numbers. However, we have verified that our solutions do not exhibit oscillations in this regime. A plot of the streamlines for different $\Re$ and $\Rem$ can be found in Figure \ref{fig:Streamlinesldc}. One can clearly observe the phenomenon that for high magnetic Reynolds numbers the magnetic field lines are advected with the fluid flow. Iteration counts are displayed in Table \ref{tab:ldc2dREREMHdiv}. 

For the Picard linearisation we observe that the nonlinear scheme already fails to converge for a magnetic Reynolds number of 100. The poor nonlinear convergence of the Picard iteration for high $\Rem$ even with continuation was previously observed for other formulations \cite{PhillipsPHD,Phillips2014}.


For the Newton linearisation the linear iterations increase slightly since the approximation of the Schur complement $\tilde{\mathcal{S}}^{(\u, p)}$ becomes less accurate for high Reynolds numbers. On the other hand, the number of nonlinear iterations remains fairly constant which seems to indicates that the linear solver for the $(\E, \B)$ block described in Section \ref{sec:solverformagneticblock} works very well for high $\Rem$ in two dimensions.


\begin{table}[htbp!]
	\centering
	\begin{tabular}{r|ccc|ccc}
		\toprule
		& \multicolumn{3}{c|}{Picard} & \multicolumn{3}{c}{Newton} \\
		\midrule
		$\Rem\backslash\Re$  &1 &     1,000 &    10,000 &1 &     1,000 &    10,000 \\
		\midrule
        1              &( 3) 5.3&(4) 3.5 & (3) 4.3&( 2) 6.0 & ( 3) 4.3 & ( 3) 4.3   \\
		1,000      &NF&NF&NF& ( 2) 4.5 & ( 3) 3.0 & ( 3) 3.0  \\
		10,000    &NF&NF&NF& ( 2) 4.5 & ( 4) 5.5 & ( 3) 5.7      \\
		\bottomrule
	\end{tabular}
	\caption{Iteration counts for the stationary lid-driven cavity problem in 2D with $\hdiv\times L^2$-discretisation for different $\Rem$ and $\Re$. NF indicates that this entry was not computable due to the failure of nonlinear convergence. \label{tab:ldc2dREREMHdiv}}
\end{table}

Thus far we have considered the $\hdiv\times L^2$ discretisation for the hydrodynamic variables. For comparison, in this subsection we include results for Scott--Vogelius elements which we previously referred to in Section \ref{sec:Discretisationchap2} and Section \ref{sec:solverforschurcomp}.
The results are shown in Table \ref{tab:ldc2dREREMSV}. We observe that the Krylov iteration counts are in general similar for the Scott--Vogelius element, making this an attractive alternative for those wishing to employ conforming schemes. However, one must keep in mind that the work per Krylov iteration is substantially higher for this element, due to the use of larger patches in the macrostar relaxation method.

\begin{table}
	\centering
    \begin{tabular}{r|ccc|ccc}
	\toprule
	& \multicolumn{3}{c|}{Picard} & \multicolumn{3}{c}{Newton} \\
	\midrule
	$\Rem\backslash\Re$ & 1        & 1,000    & 10,000  & 1        & 1,000    & 10,000   \\
	\midrule
	1                   & ( 2) 4.0 & ( 2) 2.5 & ( 2) 8.5 & ( 2) 4.0 & ( 2) 2.5 & ( 3) 9.7\\
	1,000               & NF       & NF        & NF        & ( 5) 1.8 & ( 3) 3.0 & ( 2) 4.0 \\
	10,000              & NF       & NF        & NF        & ( 8) 5.2 & ( 4) 6.2 & ( 2) 5.5 \\
	\bottomrule
\end{tabular}
\caption{Iteration counts for the stationary lid-driven cavity problem in 2D with Scott--Vogelius elements for different $\Rem$ and $\Re$.\label{tab:ldc2dREREMSV}}
\end{table}

\begin{figure}[htbp!]
	\centering
	\begin{tabular}{ccc}
		\includegraphics[width=3.9cm]{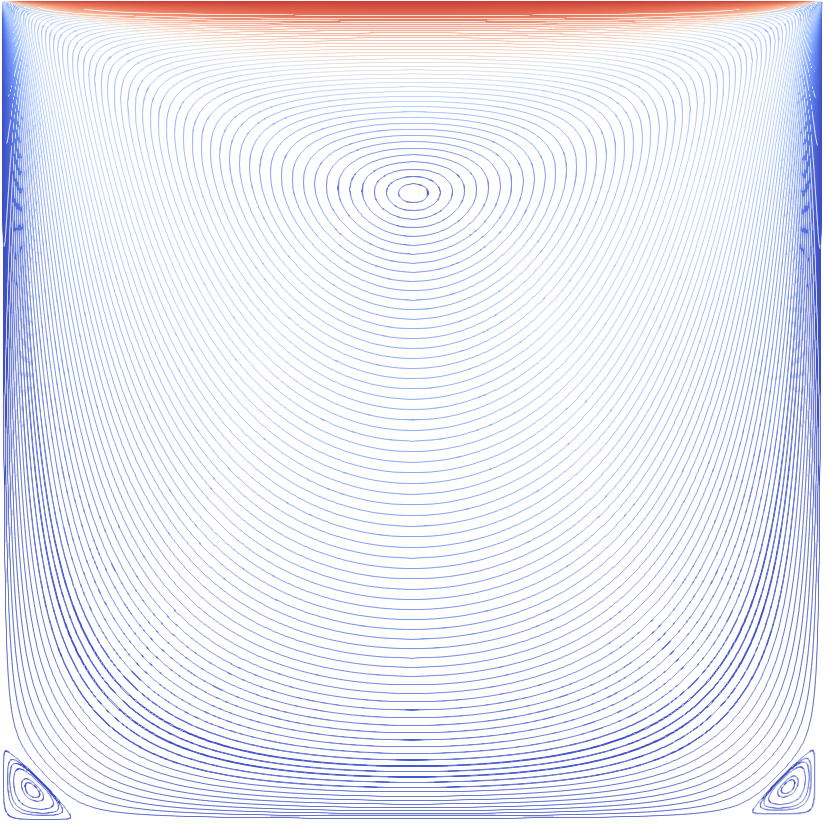} &
		\includegraphics[width=3.9cm]{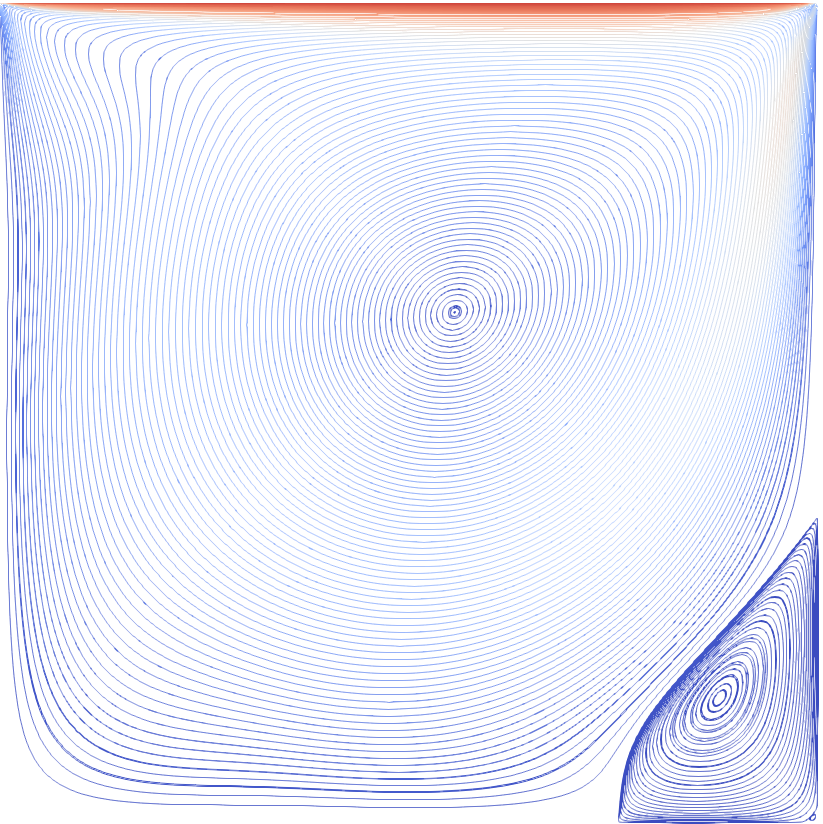} &
		\includegraphics[width=3.9cm]{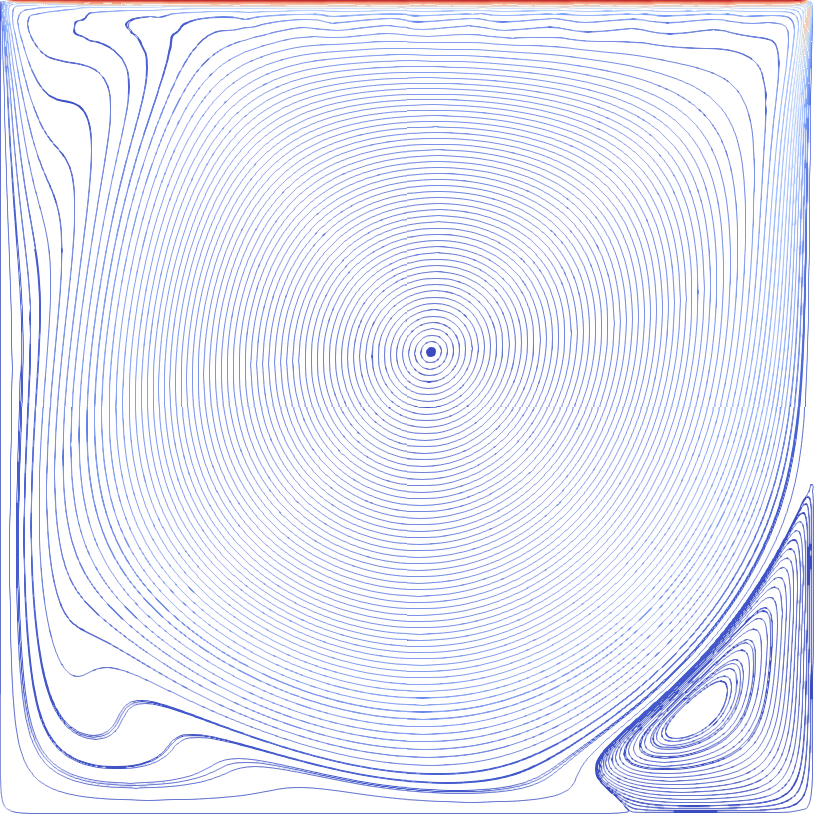} \\
		\includegraphics[width=3.9cm]{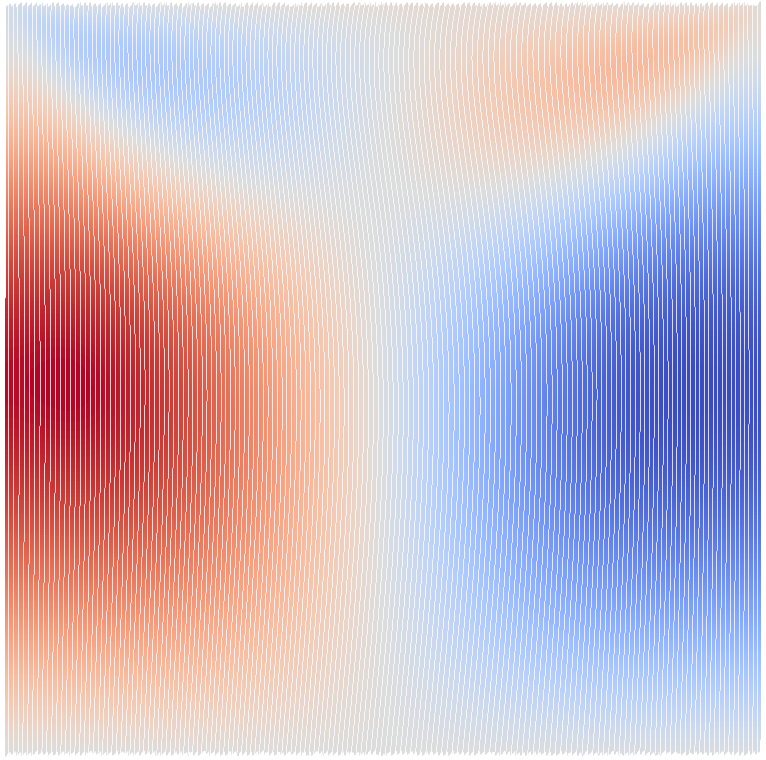} &
		\includegraphics[width=3.82cm]{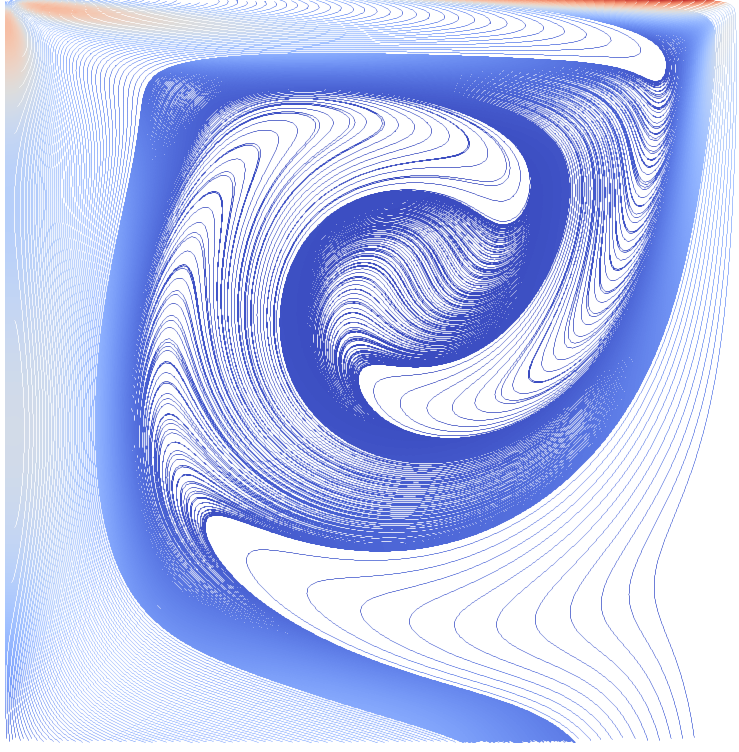} &
		\includegraphics[width=3.82cm]{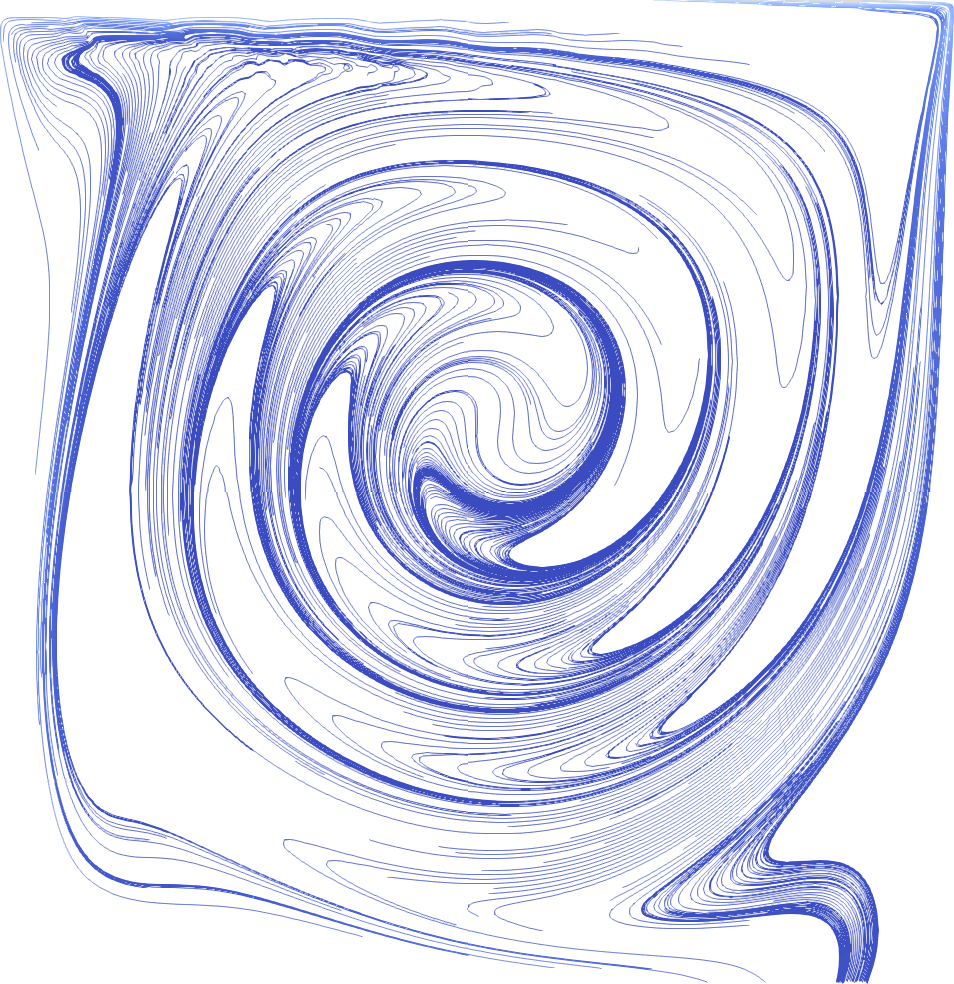} \\
		$\Re=\Rem=1$ & $\Re=\Rem=500$ & $\Re=\Rem=5,$000\\
	\end{tabular}
	\caption{Streamlines for the two-dimensional stationary lid-driven cavity problem for $\u$ (upper row) and $\B$ (lower row).\label{fig:Streamlinesldc}}
\end{figure}

\subsubsection{Time-dependent lid-driven cavity problem in two dimensions}
We next consider the time-dependent lid-driven cavity problem. We choose the same boundary conditions and right-hand side as in the stationary case. The numerical results for varying $S$ and $\Re$ are shown in Table \ref{tab:ldc2dtimedep}. As in the stationary case, the Krylov iteration counts remain almost constant for the two linearisations. We notice that the Picard iteration fails to converge for high $S$ and $\Rem$ for the chosen $\Delta t = 0.01$. However, we tested that one can get the Picard iteration to converge in most cases by choosing a smaller $\Delta t$ in the first time steps. We do not report these results here to keep the tables consistent.

Table \ref{tab:ldc2dtimedep} also shows iteration counts for varying $\Re$ and $\Rem$. The linear solver is robust for most parameter values, with iteration counts only increasing for $\Re=1$ and $\Rem=100,000$.

For completeness, we also study the case of high $\Rem$ and $S$ at the same time in Table \ref{tab:ldc2dtimedep}, which we expect to be the most challenging case. Again the slight increase of the Krylov iterations in the Newton iteration is due to inaccurate outer Schur complement approximation. However, the solvers perform very well, considering the difficulty of the problem.

\begin{table}[htbp!]
	\centering
	\begin{tabular}{r|ccc|ccc}
		\toprule
		& \multicolumn{3}{c|}{Picard} & \multicolumn{3}{c}{Newton} \\
		\midrule
		$S \backslash\Re$  &1 &     10,000 &    100,000 & 1 &     10,000 &    100,000 \\
		\midrule
        1 &  (2.0) 3.0 & (2.2) 3.6 & (3.1) 3.3 & (1.6) 3.6 & (2.2) 3.6 & (3.1) 3.3\\
		1,000 &   (3.0) 4.0 & (3.0) 3.3 & (2.7) 3.0 & (2.1) 4.7 & (2.2) 3.9 & (2.3) 3.3\\
		10,000 & NF & NF & NF &  (2.2) 6.5 & (2.5) 5.0 & (2.3) 5.6\\
		\bottomrule
		\multicolumn{7}{c}{} \\
		\toprule
		& \multicolumn{3}{c|}{Picard} & \multicolumn{3}{c}{Newton} \\
		\midrule
		$\Rem\backslash\Re$  &1 &     10,000 &    100,000 &1 &     10,000 &    100,000 \\
		\midrule
        1 & (2.0) 3.0 & (2.2) 3.6 & (3.1) 3.3 & (1.6) 3.6 & (2.2) 3.6 & (3.1) 3.3  \\
		10,000 & NF & NF & NF & (2.0) 3.1 & (2.3) 3.6 & (3.1) 3.3\\
		100,000 &  NF & NF & NF &  (2.2)10.9 & (3.0) 3.2 & (3.3) 3.2\\
		\bottomrule
		\multicolumn{7}{c}{} \\
		\toprule
		& \multicolumn{3}{c|}{Picard} & \multicolumn{3}{c}{Newton} \\
		\midrule
		$\Rem\backslash S$  &1 &     1,000 &    10,000 & 1 &     1,000 &    10,000 \\
		\midrule
        1 &  (2.0) 3.0 & (3.0) 4.0 & NF & (1.6) 3.6 & (2.1) 4.7 & (2.2)  6.5  \\
		1,000 &  (3.0) 2.5 & NF & NF & (2.0) 3.1 &  (2.2) 5.6 & (2.8)11.0 \\
		10,000 & NF & NF & NF &  (2.0) 3.1 & (2.2) 6.3 & (3.1)11.8\\
		\bottomrule
	\end{tabular}
	\caption{Iteration counts for the transient lid-driven cavity problem in 2D.\label{tab:ldc2dtimedep}}
\end{table}

\subsubsection{Time-dependent island coalescence problem in two dimensions}\label{sec:islandcoal}
Next, we consider a two-dimensional island coalescence problem to demonstrate the effectiveness of our method for a physically relevant model that shows behaviour which is unique to MHD problems. Furthermore, we report results for a weak parallel scalability test ranging from 4 processors and 160K DoFs to 256 processors and 41M DoFs to examine the performance of our algorithm. \\
The island coalescence problem is used to model magnetic reconnection processes in large aspect ratio tokamaks. For a strong magnetic field in the toroidal direction, the flow can be described in a two-dimensional setting by considering a cross-section of the tokamak. We consider the same problem as in \cite[Section 4.2]{adler2020monolithic}. The domain $\Omega=(-1,1)^2$ results from the unfolding of an annulus in the cross-sectional direction where the left and right edges are mapped periodically. The equilibrium solution for $k=0.2$ is given by
\begin{align*}
	\u_{eq}&=\mathbf{0},\qquad p_{eq}(x,y) = \frac{1-k^2}{2}\left (1+\frac{1}{(\cosh(2\pi y) + k \cos(2\pi x))^2}\right), \\
	\B_{eq}(x,y) &= \frac{1}{\cosh (2\pi y) + k \cos(2\pi x)} 
	\begin{pmatrix}
		\sinh (2\pi y) \\ k \sin(2\pi x)
	\end{pmatrix}, E_{eq}=\frac{1}{\Rem} \scurl \B_{eq} - u_{eq} \times B_{eq},
\end{align*}
which results in right-hand sides $\f=\mathbf{0}$ and $\mathbf{g}$ for \eqref{eq:MHD3d3} given by
\begin{equation}
	\mathbf{g} = \frac{-8 \pi^2 (k^2-1)}{\Rem(\cosh(2\pi y) + k\cos(2\pi x))^3}
	\begin{pmatrix}
		\sinh(2\pi y) \\ k \sin(2\pi x)
	\end{pmatrix}.
\end{equation}
The initial condition for $\B_{eq}$ is given by perturbing it for $\varepsilon=0.01$ with 
\begin{equation}
	\Delta \B = \frac{\varepsilon}{\pi} \begin{pmatrix}
		-\cos(\pi x) \sin(\pi y/2) \\
		2\cos(\pi y/2) \sin(\pi x)
	\end{pmatrix}.
\end{equation}
The authors believe that the reported $\Delta \B$ in \cite{adler2020monolithic} includes a typographical error, as it is not divergence-free, and amended the second component appropriately. Therefore, the problem setup is not exactly identical and hence we might see slightly different solutions.
The reconnection rate can be computed as the difference between $\curl \B$ evaluated at the origin $(0,0)$ at the current time and the initial time, divided by $\sqrt{\Rem}$. Note that in our formulation $\B \in \hdivz$ and therefore we apply the $\scurl$ weakly by solving a problem for $j_0 \in \mathrm{H}_0(\mathrm{curl}, \Omega)$ such that
\begin{equation}
	(j_0, k) = (\B, \vcurl k) \quad \forall\, k \in \mathrm{H}_0(\mathrm{curl}, \Omega).
\end{equation}
In order to make the point evaluation of $j_0$ at (0,0) well-posed we project $j_0$ to the space $\mathbb{CG}1$ as in \cite{adler2020monolithic}. 

Figure \ref{fig:reconnectionrate} shows the reconnection rates for $\Re=\Rem=1{,}000$, $\Re=\Rem=5{,}000$ and $\Re=\Rem=10{,}000$ for three different spatial and temporal resolutions. We have fixed a coarse grid of $16\times16$ cells and compute results for three (1.1M DoFs), four (4.6M DoFs) and five (18.4M DoFs) levels of refinement. For the three levels of refinement, we chose a fixed step size of $\Delta t =0.025$ and halved it with each refinement. We iterated until $T=15$ which results in 2400 time steps for the finest resolution. In \cite{adler2020monolithic}, a coarse mesh of $20 \times 20$ cells has been considered with four, five, six and seven levels of refinement and $\Delta t =0.025$ on the coarsest level, i.e., they considered a finer mesh for the same time-step size. 

One can observe a decreasing height of the peak for increasing Reynolds numbers and the so-called ``sloshing" \cite{Knoll2006} effect that results in further peaks after the main peak with higher Reynolds numbers. Convergence for our considered meshes can be observed for $\Re=\Rem=1{,}000$ and  $\Re=\Rem=5{,}000$ while a further refinement is needed for $\Re=\Rem=10{,}000$. Nevertheless, our finest grid results match the results of \cite[Fig. 4]{adler2020monolithic} where finer meshes of up to $2{,}560\times2{,}560$ cells and $\Delta t=0.0016$ have been considered. For example, our finest result for $\Re=\Rem=10{,}000$ clearly reproduces the second peak in the reconnection rate. The time evolution of the current density $j$ at $\Re=\Rem=5{,}000$ can be found in Figure \ref{fig:islandcoalplots}.

\begin{figure}[htbp!]
	\centering
	\begin{tabular}{ccc}
		\includegraphics[width=3.9cm]{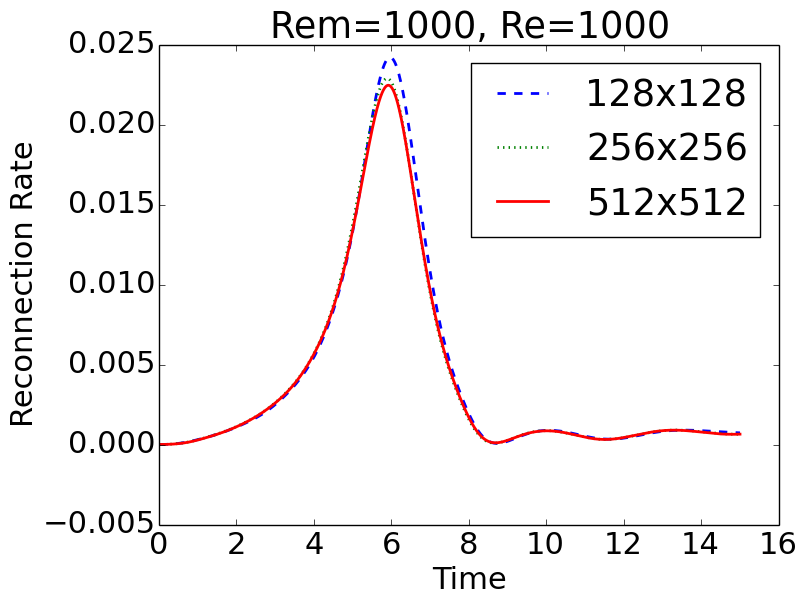} &
		\includegraphics[width=3.9cm]{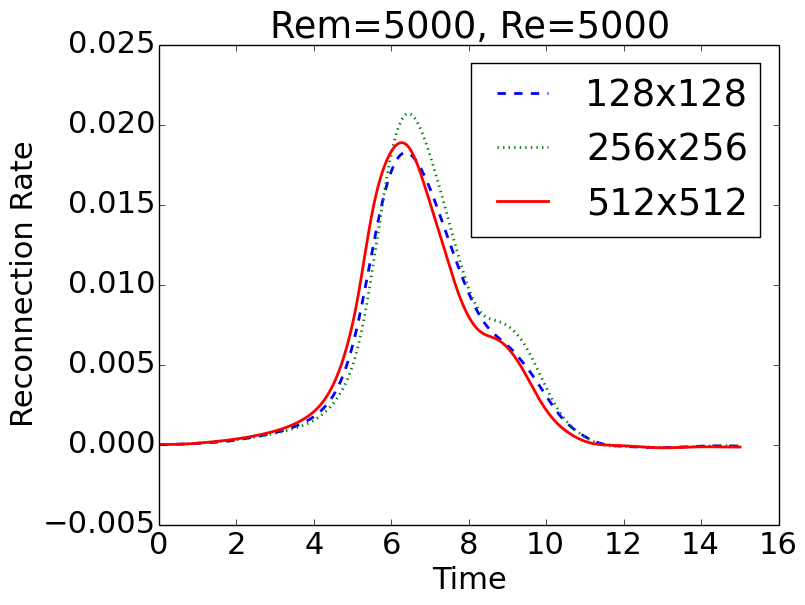} &
		\includegraphics[width=3.9cm]{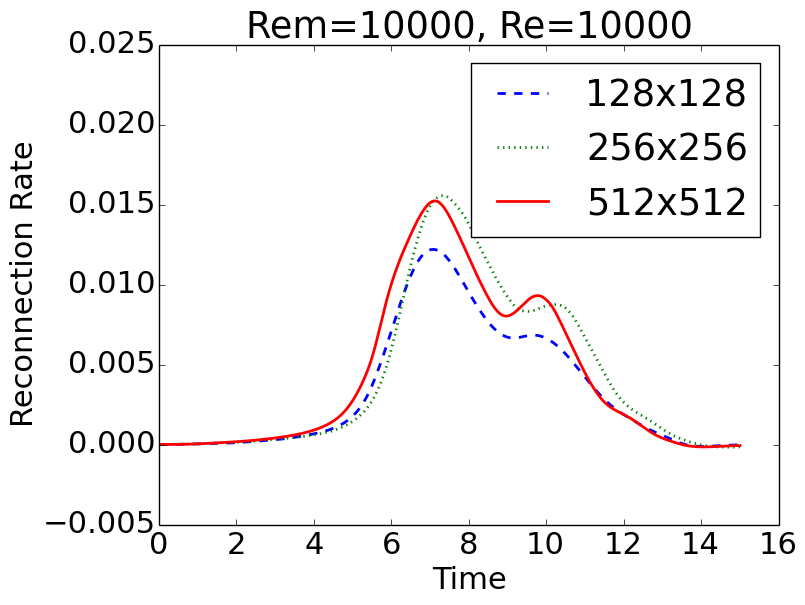} \\
	\end{tabular}
	\caption{Reconnection rates for the island coalescence problem.\label{fig:reconnectionrate}}
	\vspace{-0.3cm}
\end{figure}

 Furthermore, we performed a weak parallel scaling test for nine different combinations of the Reynolds numbers. We chose a coarse grid of $16\times16$ cells with three (1.1M DoFs), four (4.6M DoFs) and five (18.4M DoFs) levels of refinement. All tests were performed with 16 cores per node on 1, 4 and 16 nodes resulting in 16, 64 and 256 cores for the different refinements. We observed (not reported here) that scaling over the nodes with a fixed number of cores per node provides better results than increasing the number of cores per node. This seems to indicate that our code is mainly limited by the memory bandwidth. Furthermore, we ensured that the numbers of cores used in our simulations evenly divide the number of cells in the $16\times16$ coarse grid to minimize load imbalances. For an optimal scaling of the patch smoother in the multigrid relaxation the number of patches (and hence vertices) per processor should also be evenly balanced, but this was not implemented.

In Table \ref{tab:islandcoalscaling}, we report the average runtimes per linear iteration rather than the total runtime to take into account that the numbers of linear and nonlinear iterations change slightly between the different refinements. The runtimes only show a slight increase the more cores being used and hence underline good weak scaling of our method.

 As for the lid-driven cavity problem, we observe excellent robustness of the linear and nonlinear iteration counts with respect to the Reynolds numbers. Both linear and nonlinear solvers converge in either 1 or 2 iterations in the investigated ranges of $\Re$ and $\Rem$. We therefore do not report a table here that shows each iteration count.

\begin{table}[htbp!]
	\centering
	\resizebox{\textwidth}{!}{%
		\begin{tabular}{r|ccc|ccc|ccc}
			\toprule
			& \multicolumn{3}{c|}{$128\times 128$ on 16 cores} & \multicolumn{3}{c|}{$256\times 256$ on 64 cores} & \multicolumn{3}{c}{$512\times512$ on 256 cores} \\
			\midrule
			$\Rem \backslash\Re$  &1 &     1,000 &    10,000 &1 &     1,000 &    10,000 &1 &     1,000 &    10,000 \\
			\midrule
			1 & 0.13 & 0.12 & 0.13 & 0.14 & 0.14 & 0.14 & 0.17 & 0.17 & 0.17\\
			1,000 & 0.13 & 0.12 & 0.12 & 0.14 & 0.13 & 0.13 & 0.15 & 0.14 & 0.15\\
			10,000 &  0.12 & 0.11 & 0.12 & 0.14 & 0.13 & 0.13 & 0.15 & 0.15 & 0.15\\
			\bottomrule
	\end{tabular}}
	\caption{Average time per linear iteration in minutes for the two-dimensional island coalescence problem. \label{tab:islandcoalscaling}}
\end{table}

\begin{figure}[htbp!]
	\centering
	\begin{tabular}{c c c c}
		\includegraphics[width=3.3cm]{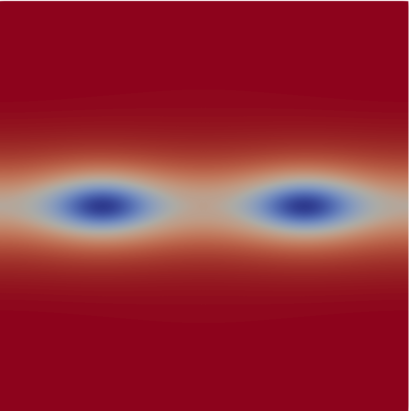} &
		\includegraphics[width=3.3cm]{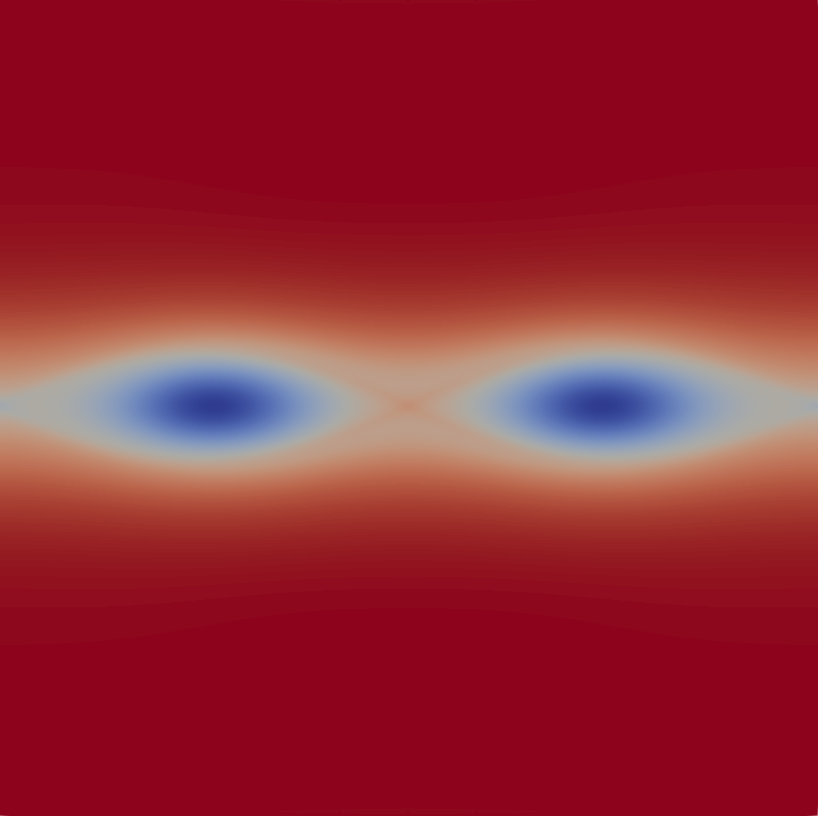} &
		\includegraphics[width=3.3cm]{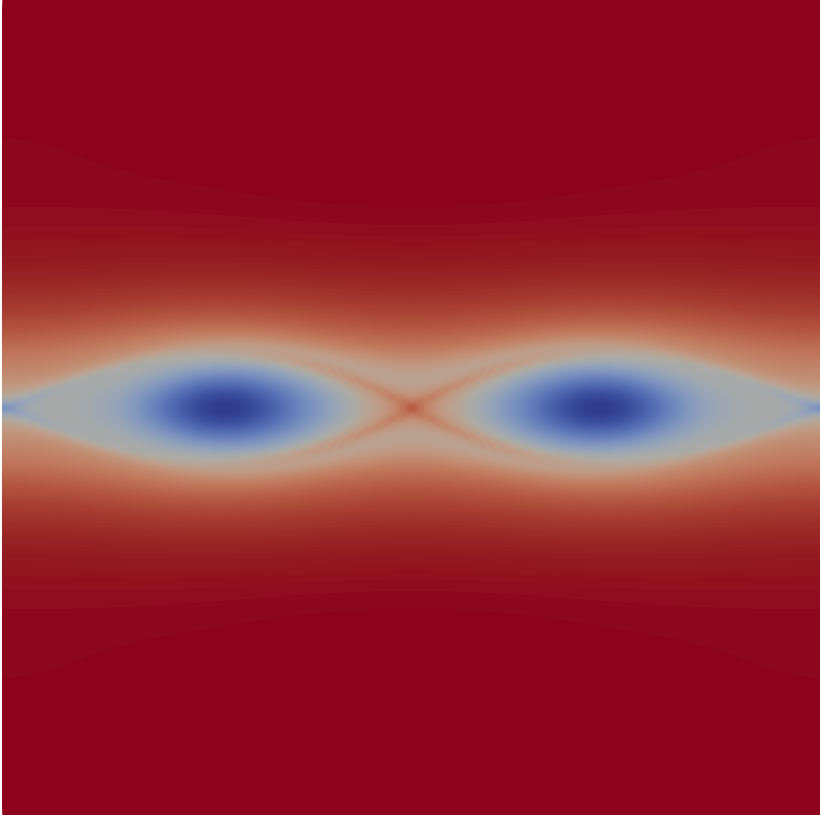} &
		\includegraphics[width=3.3cm]{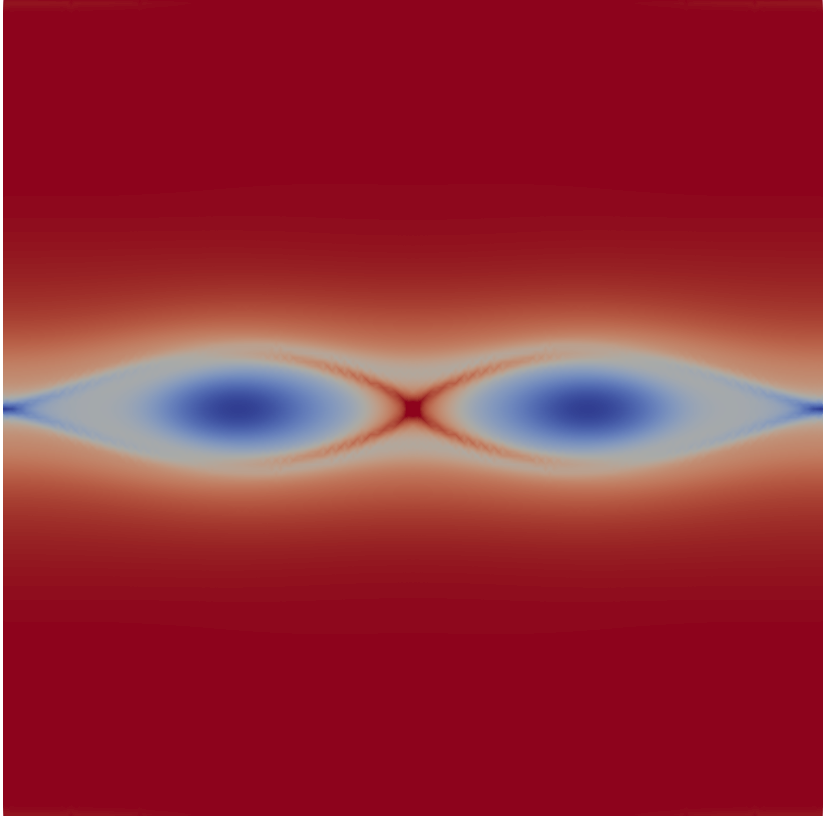} \\[0.3cm]
		\includegraphics[width=3.3cm]{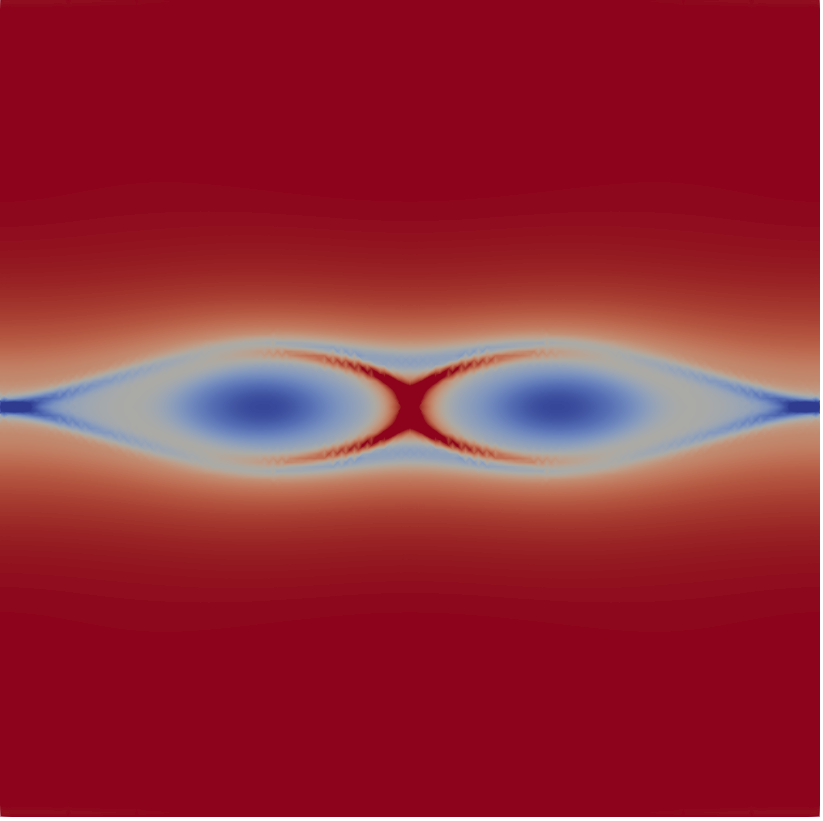} &
		\includegraphics[width=3.3cm]{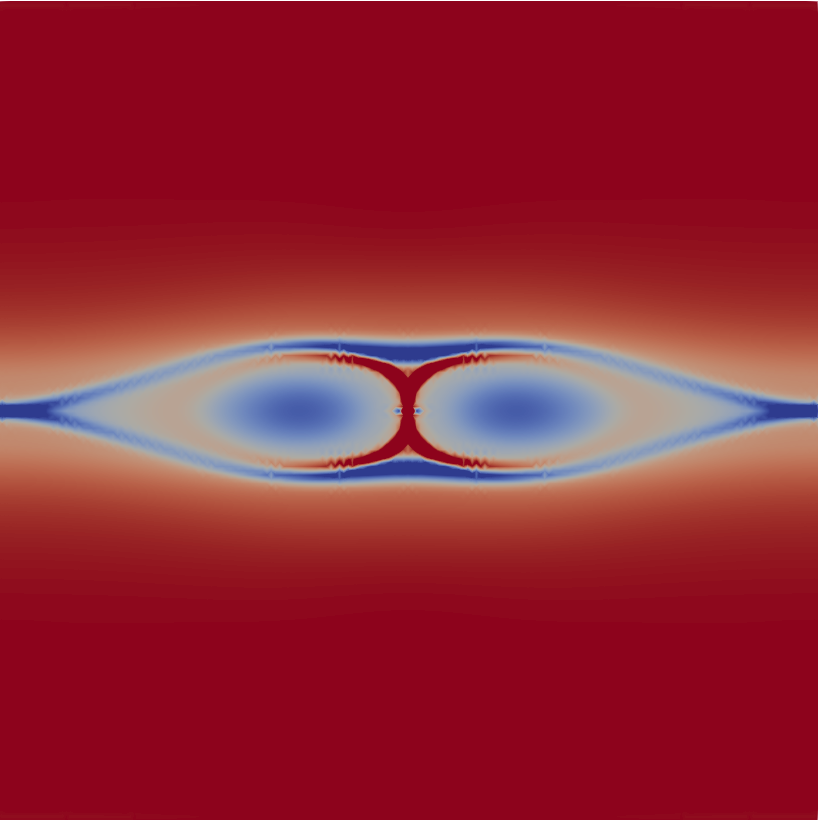} &
		\includegraphics[width=3.3cm]{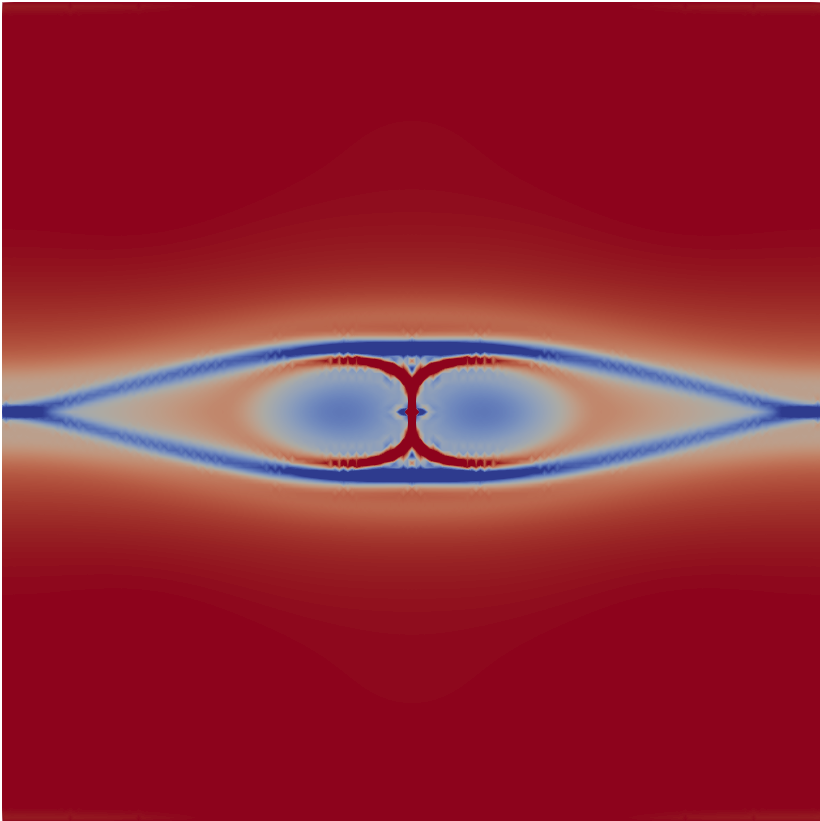} &
		\includegraphics[width=3.3cm]{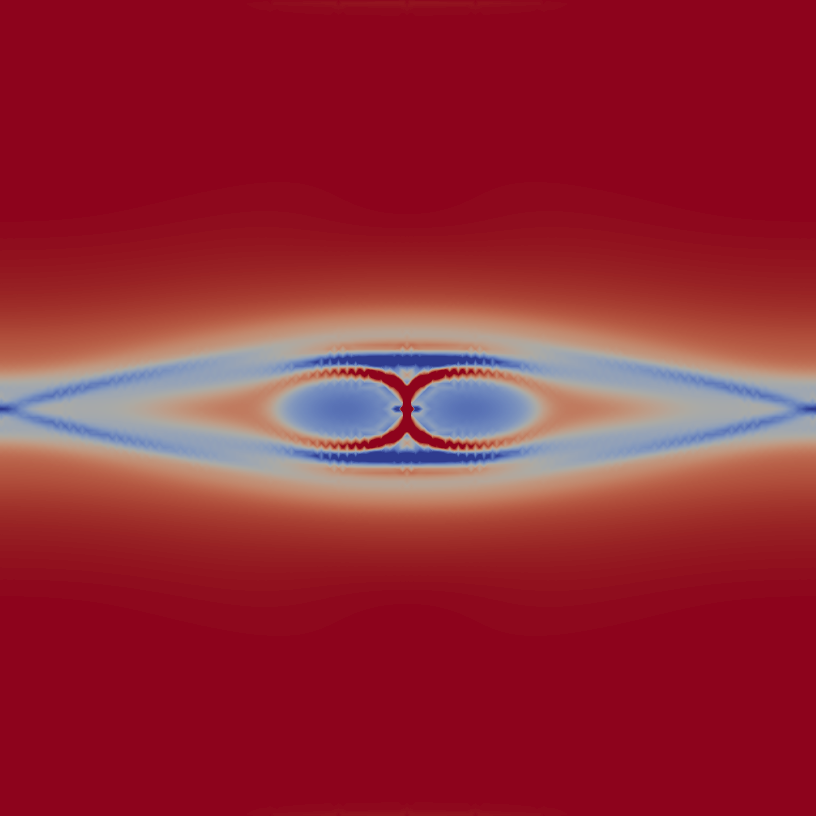} \\[0.3cm]
		\includegraphics[width=3.3cm]{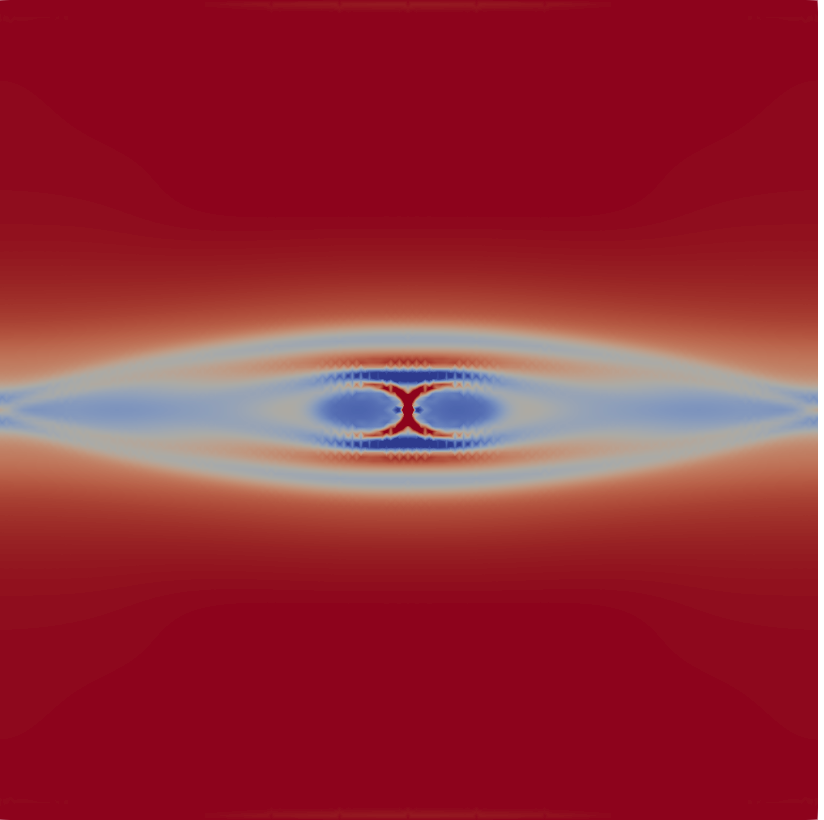} &
		\includegraphics[width=3.3cm]{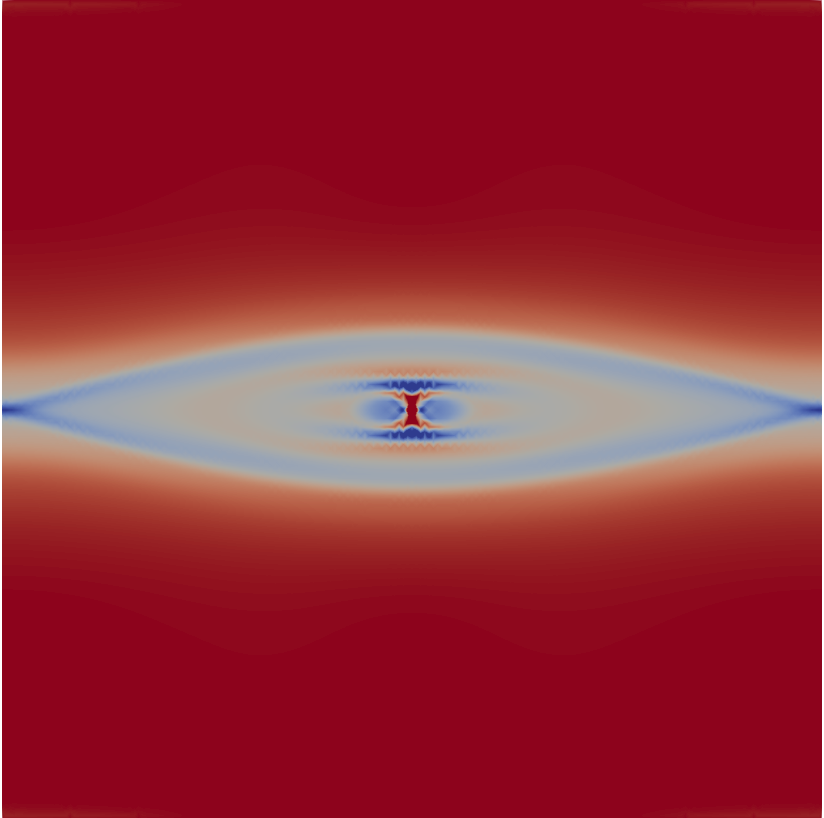} &
		\includegraphics[width=3.3cm]{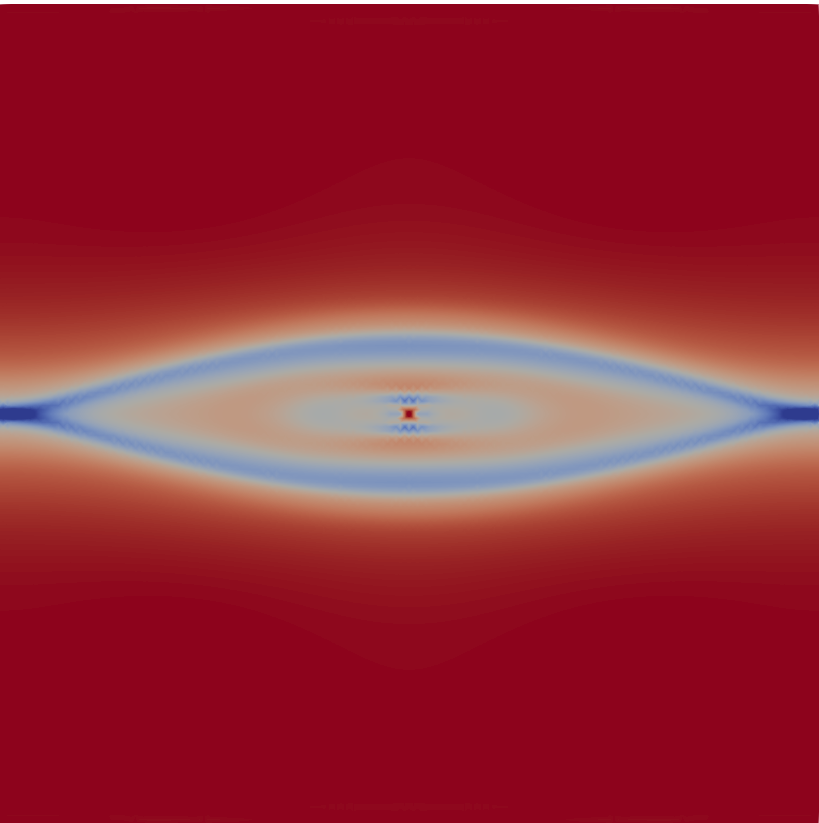} &
		\includegraphics[width=3.3cm]{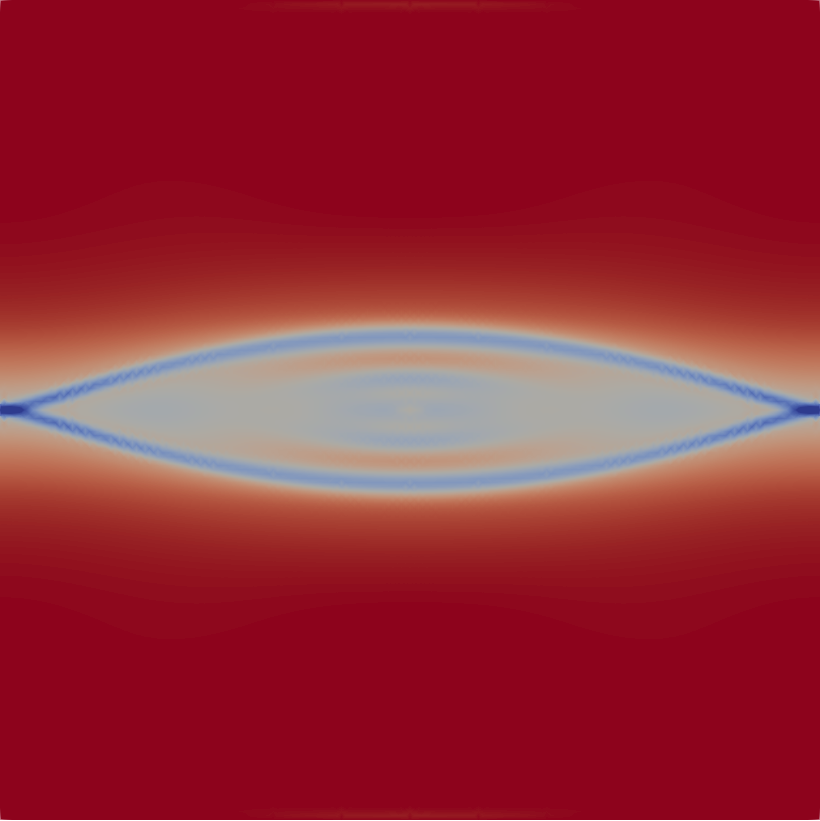} \\

	\end{tabular}
	\caption{Plots of the current density $j$ for the island coalescence problem at $\Re=\Rem=5{,}000$ for t = 1, 2, ..., 12 counted from top left to bottom right.\label{fig:islandcoalplots}}
\end{figure}

\subsection{Three-dimensional results}\label{sec:3dresults}

In three dimensions, we observe in general that the stationary problems are harder to solve for high parameters than in two dimensions. We believe that the following three points are mainly responsible for this behaviour. First of all, the discretisation of the electric field changes from a scalar-valued $\mathbb{CG}$-function to a vector-valued $\mathbb{NED}1$-function with tangential boundary conditions. Moreover, the kernel of the term $\vcurl(\u^n \times \B)$ is larger in three dimensions which degrades the performance of the monolithic solver for the $(\E, \B)$ block for high $\Rem$. Furthermore, the grids we consider are much coarser than in two dimensions because of computational costs.

\subsubsection{Stationary lid-driven cavity problem in three dimensions}\label{sec:stationaryliddrivencavityproblemin3d}

We adapt the two-dimensional lid-driven cavity problem to three dimensions by considering the domain $\Omega=(1/2,1/2)^3$ and the boundary conditions $\u=(1,0,0)^\top$ on the boundary $y=0.5$ and $\u=(0,0,0)^\top$ on the other faces. The background magnetic field $\B_0 = (0,1,0)^\top$ determines the boundary conditions for $\B$. For the three-dimensional problem, we only investigate the Newton linearisation as we have seen in two dimensions that the Newton iteration outperforms the Picard iteration in nearly all cases. The results on the left in Table \ref{tab:ldc3dstationarySRE} show a good control over the linear iteration numbers for the lid-driven cavity problem, where the case of $S=1$ and $\Re=10{,}000$ seems to be the most challenging case. A streamline plot of the solution at $\Re=\Rem=100$ can be found in Figure \ref{fig:3dldc}.

\begin{figure}[htbp!]
	\centering
	\includegraphics[width=14cm]{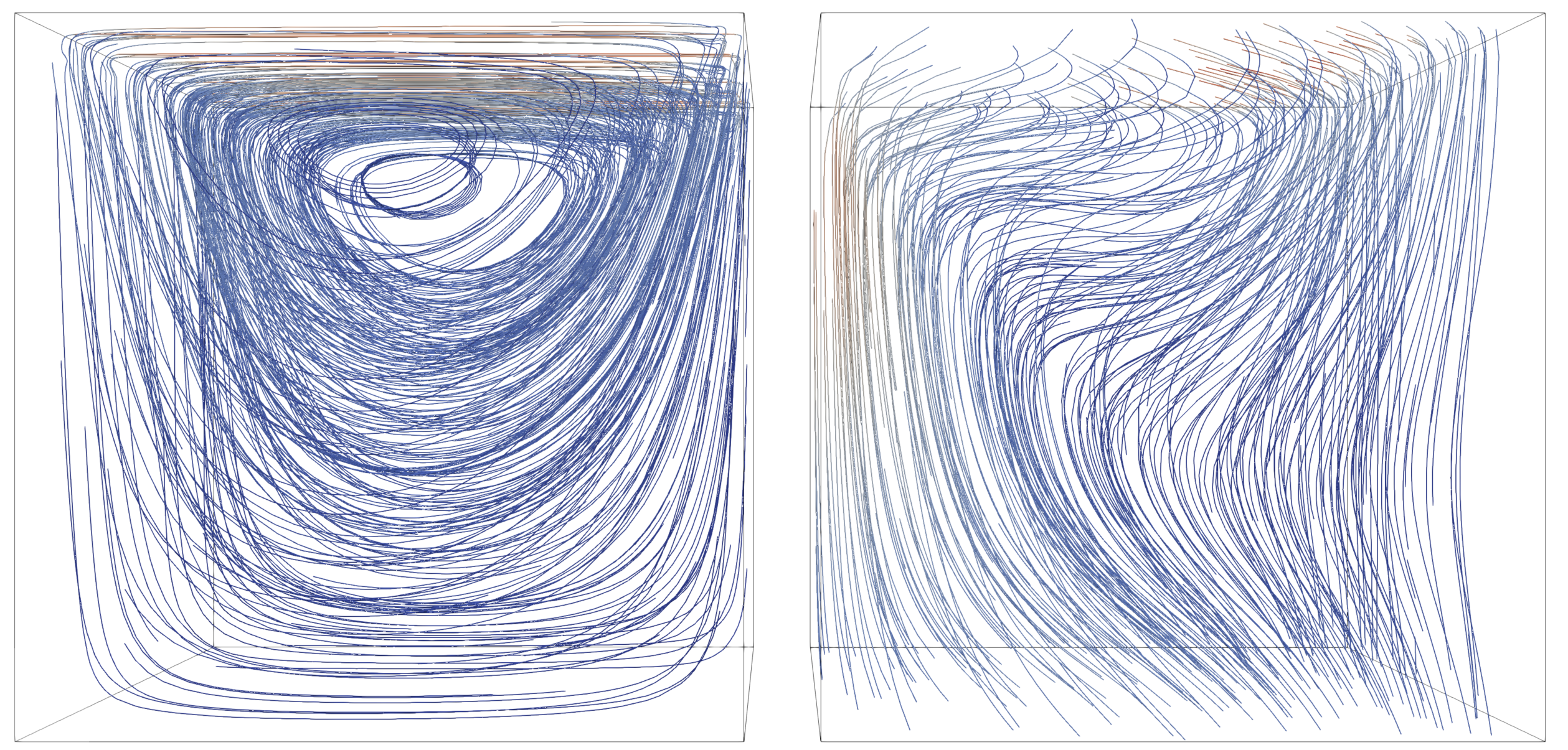} 
	\caption{Streamlines for the three-dimensional stationary lid-driven cavity problem for $\u$ (left) and $\B$ (right) at $\Re=\Rem=100$.}
	\label{fig:3dldc}
\end{figure}

On the right of Table \ref{tab:ldc3dstationarySRE} we report a comparison to taking the outer Schur complement that eliminates the $(\E, \B)$ block.  As mentioned in Section \ref{sec:derivationofblockpreconditioner}, we can clearly see that this choice performs worse for high values of $\Re$ and $S$ where no convergence in 50 linear iterations was reached. We observed similar behaviour for unreported experiments on transient and two-dimensional problems.

We do not include a full table for high $\Rem$, as in this case the monolithic multigrid solver cannot deal with the term $\vcurl(\u^n \times \B^n)$ that occurs in the Newton linearisation. As in two dimensions, this term is crucial for the convergence of the nonlinear iteration. For Newton, the iteration counts increase very slightly from $\Re=\Rem=1$ by 8.0 Krylov iterations per nonlinear step to 10.0 iterations for $\Rem=500$ and $\Re=1$ and fails to converge for higher $\Rem$. We want to emphasise that in this case the failure of convergence is indeed caused by the inner multigrid method and not by an inaccurate outer Schur complement approximation. To the best of our knowledge, preconditioning methods that robustly treat the vector Laplace operator with an additional $\vcurl(\u^n \times \B)$ term in three dimensions are not known, and we intend to investigate this problem further in future work.

\begin{table}[htbp!]
	\centering
	\begin{minipage}[c]{0.45\textwidth}
		\resizebox{\textwidth}{!}{
			\begin{tabular}{r|ccc}
				\toprule 
				\multicolumn{4}{c}{Using $\tilde{\mathcal{S}}^{(\u, p)}$ for order $(\u, p, \E, \B)$} \\
				\midrule
				$S \backslash\Re$  &1 &     1,000 &    10,000 \\
				\midrule
				1  & ( 3) 6.0 & ( 3) 7.0 & ( 4)20.0 \\
				1,000 & ( 3) 7.3 & ( 2) 9.5 & ( 2) 6.5\\
				10,000 & ( 3) 9.0 & ( 2)13.0   & ( 2)12.5  \\
				\bottomrule
		\end{tabular}}
	\end{minipage}
	\begin{minipage}[c]{0.43\textwidth}
		\resizebox{\textwidth}{!}{
			\begin{tabular}{r|ccc}
				\toprule
				\multicolumn{4}{c}{Using $\tilde{\mathcal{S}}^{(\E, \B)}$ for order $(\E, \B, \u, p)$} \\
				\midrule
				$S \backslash\Re$  &1 &     1,000 &    10,000 \\
				\midrule
				1  &( 3) 6.0 & (4)14.7 & (-)$>$50\\
				1,000 &  ( 3)12.7 & (-)$>$50 &(-)$>$50 \\
				10,000 &  ( 3)20.0 & (-)$>$50 &(-)$>$50\\
				\bottomrule
		\end{tabular}}
	\end{minipage}
	\caption{(left) Iteration counts for the stationary lid-driven cavity problem in 3D for the Newton linearisation. (right) Iteration counts for the stationary lid-driven cavity problem in 3D for taking the outer Schur complement that eliminates the $(\E, \B)$ block.}
	\label{tab:ldc3dstationarySRE}
\end{table}

\subsubsection{Time-dependent lid-driven cavity problem in three dimensions}
Finally, we consider the time-dependent version of the three-dimensional lid-driven cavity problem, which was also investigated in detail in \cite{Ma2016}.
The numerical results in Table \ref{tab:ldc3dtimedep} show good control of the iteration counts and the linear iteration numbers only notably increases for very high values of $S$. Moreover, we observe robust convergence of the monolithic multigrid solver for the $(\E, \B)$ block for high $\Rem$. As mentioned before in Section \ref{sec:solverformagneticblock}, this can be explained by the fact that the problem does not become nearly singular for high $\Rem$ due to the extra mass matrix. Therefore, the fact that the kernel of $\vcurl(\u^n\times \B)$ is not fully captured by the multigrid method has less influence. 
\newline 

\begin{table}[htbp!]
	\centering
\begin{minipage}[c]{0.45\textwidth}
	\resizebox{\textwidth}{!}{
		\begin{tabular}{r|ccc}
			\toprule
			$S \backslash\Re$  &1 &     10,000 &    100,000 \\
			\midrule
      		1  &   (2.1) 7.3 & (3.2) 2.1 & (3.3) 2.0\\
			1,000 & (3.0) 8.6 & (3.3) 2.8 & (3.5) 2.6 \\
			10,000 & (4.0)11.3 & (4.0) 7.0 & (4.0) 6.2 \\
			\bottomrule
	\end{tabular}}
\end{minipage}
\begin{minipage}[c]{0.463\textwidth}
	\resizebox{\textwidth}{!}{
		\begin{tabular}{r|ccc}
			\toprule
			$\Rem \backslash\Re$  &1 &     10,000 &    100,000 \\
			\midrule
			1  &     (2.1) 7.3 & (3.2) 2.1 & (3.3) 2.0\\
			10,000 & (2.5) 7.1   & (3.2) 2.0 & (3.3) 2.0   \\
			100,000 & (3.0)15.1 & (3.2) 2.6 & (3.3) 2.0 \\
			\bottomrule
	\end{tabular}}
\end{minipage}
\end{table}
\vspace{-0.8cm}
\begin{table}[htbp!]
	\centering
	\resizebox{0.45\textwidth}{!}{
		\begin{tabular}{r|ccc}
			\toprule
			$\Rem \backslash S$  &1 &     100 &    1,000 \\
			\midrule
			1  & (2.1) 7.3 & (2.1) 7.3 & (3.0) 8.6\\
			1,000 & (3.0) 7.6 & (3.0) 6.8& (3.1) 9.5\\
			10,000 & (2.5) 7.1  & (3.1) 7.1&(3.2) 9.7\\
			\bottomrule
	\end{tabular}}
	\caption{Iteration counts for the transient lid-driven cavity problem in 3D for the Newton linearisation.\label{tab:ldc3dtimedep}}
\end{table}

\renewcommand{\grad}{\operatorname{grad}}
\renewcommand{\div}{\operatorname{div}}
\chapter{Structure-preserving finite element methods and solvers for the Hall MHD equations}\label{chap:3}
The content of this chapter was developed in collaboration with Patrick Farrell and Kaibo Hu. A manuscript \cite{laakmann2022} has been submitted to the Journal of Computational Physics.
\section{Hall MHD model}   
In this chapter, we consider finite element methods for the solution of the incompressible, resistive Hall magnetohydrodynamics (MHD) equations. The stationary formulation on a bounded polyhedral Lipschitz domain $\Omega \subset \mathbb{R}^3$ is given by
\begin{subequations}
	\label{eq:HallMHD}
	\begin{align} \label{eq:HallMHDu}
		- \Re^{-1} \Delta \u + 
		( \u \cdot \nabla) \u
		- \S\, \j \times \B
		+ \nabla p
		&= \f , \\\label{eq:HallMHDj}
		\j - 
		\nabla \times \B
		&= \mathbf{0} , \\\label{eq:HallMHDB}
		\nabla \times \E &= \mathbf{0} , \\
		\nabla \cdot \B &= 0, \\
		\nabla \cdot \u &= 0, \\
		\label{eq:HallMHDE}
		\Reminv \j - (\E + \u \times \B-\RH\,\j\times \B) & = \mathbf{0} .
	\end{align}
\end{subequations}
When the Hall current term $\RH\,\j\times \B$ vanishes, one obtains the standard resistive MHD system that we considered in the previous chapter.

For time-dependent problems, the time derivatives $\frac{\partial \u}{\partial t}$ and $\frac{\partial \B}{\partial t}$ are added to the left-hand sides of \eqref{eq:HallMHDu} and \eqref{eq:HallMHDB} respectively. 
We mainly consider the boundary conditions
\begin{equation}
	\label{boundary_cond}
	\u = \mathbf{0}, \quad \B \cdot \n = 0, \quad \E \times \n = \mathbf{0}, \quad \j \times \n = \mathbf{0}, \quad \text{on $\partial \Omega$}.
\end{equation}
However, a treatment of the alternative boundary conditions (c.f., \cite{gunzburger1991existence})\begin{align}
	\label{boundary_cond2}
	\u = \mathbf{0}, \quad \B \times \n = \mathbf{0}, \quad \E \cdot \n = 0, \quad \j \times \n = \mathbf{0}, \quad \text{on $\partial \Omega$}
\end{align}
is also possible.

The essence of the Hall effect is described by adding the Hall-term $\j \times \B$ in the generalised Ohm's law \cite[Section 2.2.2]{Galtier2015}
\begin{equation}
	\mu_0 \eta \j = \E + \u \times \B - \frac{1}{n e} \j \times \B,
\end{equation}
where $\eta$ denotes the magnetic resistivity, $\mu_0$ the permeability of free space, $n$ the charge density and $e$ the electron charge. 
The non-dimensionalised form of the generalised Ohm's law corresponds to \eqref{eq:HallMHDE} where the Hall parameter $\RH$ is defined as 
\begin{equation}
	\RH = \frac{1}{\mu_0 n e} \frac{\overline{B}}{L \overline{U}}
\end{equation}
for a characteristic length $L$, magnetic field strength $\overline{B}$ and speed $\overline{U}$ of the fluid. We refer to the case $\RH = 0$ as the standard MHD equations.


Although \eqref{eq:HallMHD} only differs by one term from the standard MHD equations, the extension of existing theory and algorithms to the Hall case is non-trivial. Most formulations of the standard MHD equations use Ohm's law \eqref{eq:HallMHDE} directly to eliminate $\j$ as an unknown such as the one in Chapter \ref{chap:2}; with $\RH \neq 0$ this is no longer possible. Therefore our proposed variational formulation for the stationary problem includes as unknowns both the current density $\j$ and the electric field $\E$.
In the time-dependent case, the various conservation properties of the MHD system in the ideal limit are based upon the symmetries of the system; the introduction of the Hall term changes these symmetries, thus making it substantially more difficult to construct numerical methods that preserve several quantities simultaneously.
Finally, the development of preconditioning techniques becomes more difficult as an additional non-symmetric term with a non-trivial kernel enters the system.

Most variational formulations of the standard MHD system either look for the magnetic field $\B$ in an $\Hc$- or $\Hd$-conforming space as in Chapter~\ref{chap:2}. $\Hc$-conforming formulations have the advantage that they usually include the fewest unknown variables, typically $\u$, $p$, $\B$ and a Lagrange multiplier for the enforcement of the magnetic Gauss's law. However, such formulations only enforce the magnetic Gauss's law weakly, which can cause problems for numerical approximations \cite{Brackbill1980}. Therefore, in recent years much interest has been paid to structure-preserving $\Hd$-conforming approximations that enforce $\nabla \cdot \B = 0$ precisely on the discrete level \cite{Hu20162,Hu2018}. These formulations either eliminate $\E$ or $\j$ with help of \eqref{eq:HallMHDE} or \eqref{eq:HallMHDj}. Here, the augmented Lagrangian formulation \eqref{eq:MHDFinal} from Chapter \ref{chap:2} seems a natural approach, as it only includes $\u,p,\B$ and $\E$ as unknowns and enforces $\nabla \cdot \B = 0$ without the need for a Lagrange multiplier. Our proposed formulation with both $\j$ and $\E$ as unknowns tries to use the fewest number of unknown variables for the Hall system while still enforcing the magnetic Gauss's law precisely.

Another way of enforcing $\nabla \cdot \B = 0$ for the incompressible MHD system is to use formulations based on the vector potential $\mathbf{A}$ where $\B=\nabla \times \mathbf{A}$ (see, e.g., \cite{Adler2019, Hiptmair2018, pagliantini2016computational}). The Hall term is $\j\times \B=(\nabla\times \B)\times \B=(\nabla\times \nabla\times \mathbf{A})\times (\nabla\times \mathbf{A})$, which is a high order term in $\mathbf{A}$. It seems difficult to deal with this term with the magnetic potential and we will not pursue potential-based formulations in this work. 


The ideal limit in Hall MHD describes the case of vanishing magnetic resistivity $\eta$. We also include the case of vanishing fluid viscosity $\nu$ in this notion and hence the ideal limit formally corresponds to $\Re=\Rem=\infty$. It is well-known that in this case the energy, magnetic helicity and cross helicity are conserved properties of the standard MHD system \cite{Galtier2015}. For the ideal Hall MHD system, the cross helicity is not conserved any more; instead the so-called hybrid helicity \cite{Mininni_2003} is conserved, which is a suitable combination of magnetic, cross and fluid helicity. In \cite{gawlik2020} and \cite{hu2021helicity}, the authors propose numerical algorithms that preserve the conservative properties of the standard MHD equations precisely on the discrete level. We extend their work for the additional Hall term and propose algorithms that also preserve the hybrid helicity precisely.

Helicity characterises the linkage of field lines (the vortex lines for the fluid helicity, the magnetic lines for the magnetic helicity etc.), and is thus fundamentally important for the flow kinematics \cite{moffatt1992helicity}. 
The importance of the magnetic and cross helicity can be found in, e.g., \cite{taylor1974relaxation,pariat2005photospheric,perez2009role} and the references therein. 
Even in the non-ideal case, i.e., for non-vanishing resistivity, the total helicity is approximately preserved if the magnetic field undergoes small-scale turbulence \cite[Remark 7.19]{arnold1999topological}. Hence, algorithms that preserve the helicity and other quantities precisely (or nearly in the non-ideal case) at the discrete level are important and can lead to more physical solutions for the same resolution, because the pollution of the solutions through numerical errors is minimised. 
\section{Stationary variational formulation, linearisation and discretisation}\label{sec:vatiational-formulation}

{}

\subsection{Nonlinear scheme}
\label{sec_scheme}

We propose the following variational form for the stationary problem  \eqref{eq:HallMHD} with boundary conditions (\ref{boundary_cond}).
Define $\bm{X}_{h}:=\bm{V}_{h}\times Q_h \times \mathbf{H}_{0}^{h}(\curl)\times \mathbf{H}_{0}^{h}(\div)\times \mathbf{H}_{0}^{h}(\curl)$.
\begin{problem}\label{prob:continuous}
	Find $(\u_h, p_h, \E_h, \B_h, \j_h)\in \bm{X}_{h}$, such that for any $(\v_h, q_h, \F_h, \C_h, \k_h)\in \bm{X}_{h} $,
	\begin{subequations}\label{fem-discretisation}
		\begin{align}\label{fem-1}
			\Re^{-1} (\nabla \u_h, \nabla \v_h)  
			+(( \u_h \cdot \nabla) \u_h, \v_h) \qquad \qquad \qquad \qquad & \nonumber \\ 
			- \S (\j_h\times \B_h,\v_h ) - (p_h,\nabla\cdot \v_h) 
			&= \langle \f,\v_h \rangle,\\
			\label{fem-2}
			(\j_h,\F_h) -( \B_h, \nabla\times \F_h) &=0,\\
			\label{fem-3}
			(\nabla\times \E_h, \C_h) +(\nabla\cdot \B_h, \nabla\cdot \C_h)&= 0,\\ 
			\label{fem-4}
			\Reminv(\j_{h}, \k_{h})-(\E_{h}+\u_{h}\times \B_{h}-\RH\,\j_{h}\times \B_{h}, \k_{h})&=0,\\
			\label{fem-5}
			-(\nabla\cdot \u_h, q_h) &=0.
		\end{align}
	\end{subequations}
\end{problem} 

The above formulation includes the weak form of the augmented Lagrangian term $-\nabla \nabla \cdot \B_h$ in \eqref{fem-3}, which is used to enforce the magnetic Gauss's law $\nabla \cdot \B_h=0$ precisely with the same proof as in Section \ref{sec:discretization}.
We summarise some properties of the variational formulation in the next theorem.
\begin{theorem}\label{stability_fem}
	Any solution for  Problem \ref{prob:continuous} satisfies the 
	\begin{enumerate}
		\item magnetic Gauss's law:
		$$
		\nabla\cdot \B_h=0,
		$$
		\item  stationary Faraday's law:
		$$
		\nabla\times \E_h =\mathbf{0},
		$$
		\item energy estimates:
		\begin{align}\label{energy-1}
			{\Re^{-1}}\|\nabla\u_h\|^{2}+\Reminv \S\|\j_h\|^{2} &= \langle \f, \u_h\rangle, \\
			\label{energy-2}
			\frac{1}{2}\Re^{-1}\|\nabla\u_h\|^{2}+\Reminv \S\|\j_h\|^{2} &\leq \frac{\Re}{2}\|\f\|_{-1}^{2}.
		\end{align}
	\end{enumerate}
\end{theorem}

\begin{proof} 
	
	As for the standard MHD formulation, the stationary Faraday's law \mbox{$\nabla\times \E_{h}=\mathbf{0}$} follows from testing \eqref{fem-3} with $\C_h = \nabla\times \E_h$, and the magnetic Gauss' law $\nabla\cdot \B_h = 0$ then follows from testing \eqref{fem-3} with $\C_h = \B_h$.
	The proof of the energy law follows from testing \eqref{fem-4} with $\k_h = \j_h$. Since the additional Hall term $\RH\,(\j_{h}\times \B_{h}, \k_{h})$ vanishes for $\k_h=\j_h$, the proof coincides with the one in \cite{hu2020convergence} for the standard MHD system.
\end{proof}

\subsection{Picard iteration}

In the following, we propose a Picard-type iteration for Problem \ref{prob:continuous}. Even though the Picard iteration was clearly outperformed by the Newton iteration in the previous chapter in terms of nonlinear convergence for high $\Rem$, Picard-type iterations are still interesting to investigate since they allow for rigorous well-posedness proofs. In this section, we extend these proofs for the additional Hall term. The well-posedness of the full Newton linearisation is much more difficult to achieve or even unknown for certain MHD formulations, such as the one in Chapter \ref{chap:2}.

\begin{algorithm}[Picard step]
	\label{alg:picard-s}
	Given $(\u_{h}^{n-1},\B_{h}^{n-1})$,   find $(\u_{h}^{n}, p_h^n, \E_{h}^{n}, \B_{h}^{n}, \j_{h}^{n})\in \bm{X}_{h}$, such that for any $(\v_{h}, q_h, \F_{h}, \C_{h}, \k_{h})\in \bm{X}_{h} $,
	\begin{subequations}
		\begin{align}
			\label{picard1}
			{\Re^{-1}} (\nabla \u^{n}_{h}, \nabla \v_{h}) 
			+ (( \u^{n-1}_h \cdot \nabla) \u^{n}_h, \v_h) \qquad \qquad \qquad & \nonumber \\
			- \S (\j_{h}^{n}\times \B^{n-1}_{h},\v_{h} ) - (p_{h}^{n},\nabla\cdot \v_{h}) 
			& = \langle \f,\v_{h} \rangle,\\
			\label{picard2} 
			(\j_{h}^{n},\F_{h}) - ( \B_{h}^{n}, \nabla\times \F_{h}) &=0, \\
			\label{picard3}
			(\nabla\times \E_{h}^{n}, \C_{h}) +(\nabla\cdot\B_{h}^n, \nabla\cdot \C_{h})&= 0, \\ 
			\label{picard4} 
			\Reminv  (\j_{h}^{n}, \k_{h}
			)-(\E^{n}_{h}+\u^{n}_{h}\times \B^{n-1}_{h}-\RH\,\j_{h}^{n}\times \B^{n-1}_{h}, \k_{h})&=0,
			\\\label{picard5}
			-(\nabla\cdot \u^{n}_{h}, q_{h}) &=0.
		\end{align}
	\end{subequations}
\end{algorithm} 

\begin{algorithm}[Newton iteration]\label{alg:newton-s}
	The Newton iteration includes the additional terms $(( \u^{n}_h \cdot \nabla) \u^{n-1}_h, \v_h) - S (\j_h^{n-1} \times \B_h^n, \v_h)$ on the left-hand side of \eqref{picard1}, and $-(\u_h^{n-1}\times \B_h^n, \k_h) + \RH\,(\j_h^{n-1}\times \B_h^n, \k_h)$ on the left hand side of \eqref{picard4}.
\end{algorithm}

\begin{remark}
	By construction, any solution $(\u^n_h, p^n_h, \E^n_h, \B^n_h, \j^n_h)$ of Algorithm \ref{alg:picard-s} also fulfils $(1)$, $(2)$, and $(3)$ from Theorem \ref{stability_fem} precisely.
\end{remark}

We will use the Brezzi theory \cite{Brezzi.F.1974a} to prove the well-posedness of the Picard iteration. 
We recast Algorithm \ref{alg:picard-s} as follows.
We first formally eliminate the variables $\j_{h}^{n}$ and $\E_{h}^{n}$ from the system by 
\begin{equation}\label{jE}
	\j_{h}^{n}=\tilde{\nabla}_{h}\times \B_{h}^{n}, \quad \E_{h}^{n}=\Reminv\tilde{\nabla}_{h}\times \B_{h}^{n}-\mathbb{Q}_{c}(\u_{h}^{n}\times \B_{h}^{n-1})+\RH\, \mathbb{Q}_{c}((\tilde{\nabla}_{h}\times \B_{h}^{n})\times  \B_{h}^{n-1}),
\end{equation}
where $\mathbb{Q}_{c}$ is the $L^{2}$-projection to $\mathbf{H}^{h}_{0}(\curl, \Omega)$ and $\tilde{\nabla}_h\times$ is the weak $\curl$-operator as defined in \eqref{eq:weakcurl}. Then \eqref{picard1}-\eqref{picard5} becomes 
\begin{subequations}
	\begin{align}
		\label{picard-reduced1}
		{\Re^{-1}} (\nabla \u^{n}_{h}, \nabla \v_{h}) 
		+ (( \u^{n-1}_h \cdot \nabla) \u^{n}_h, \v_h) \qquad  \qquad \qquad \qquad \qquad &\nonumber\\
		- \S ((\tilde{\nabla}_{h}\times \B_{h}^{n})\times \B^{n-1}_{h},\v_{h} ) - (p_{h}^{n},\nabla\cdot \v_{h}) 
		&= \langle \f,\v_{h} \rangle,\\
		\label{picard-reduced2}
		\Reminv(\tilde{\nabla}_{h}\times \B_{h}^{n}, \tilde{\nabla}_{h}\times \C_{h} )-(\u_{h}^{n}\times \B_{h}^{n-1}, \tilde{\nabla}_{h}\times \C_{h} )\qquad  \qquad \quad &\nonumber\\ +\RH\,(\tilde{\nabla}_{h}\times \B_{h}^{n})\times  \B_{h}^{n-1}, \tilde{\nabla}_{h}\times \C_{h} ) +(\nabla\cdot\B_{h}^n, \nabla\cdot \C_{h})&= 0, \\ 
		\label{picard-reduced3}
		-(\nabla\cdot \u^{n}_{h}, q_{h}) &=0.
	\end{align}
\end{subequations}

Define $\bm{W}_{h}:=\bm{V}_{h}\times \mathbf{H}_{0}^{h}(\div, \Omega)$. Given $(\u^{-}, \B^{-})\in \bm{W}_{h}$, for ${\bm{x}}=(\u, \B)$, ${\bm{y}}=(\v, \C)\in \bm{W}_{h}$ and $p, 
q\in Q_{h}$, we define the bilinear forms
\begin{align*}
	a({\bm{x}}, {\bm{y}})
	:=\ &\Re^{-1}(\nabla \u, \nabla \v) + (( \u^{-} \cdot \nabla) \u, \v) - \S ((\tilde{\nabla}_{h}\times \B)\times \B^{-},\v) \\
	& +(\nabla\cdot\B, \nabla\cdot \C)
	+\Reminv( \tilde{\nabla}_{h}\times \B, \tilde{\nabla}_{h}\times\C)-(\u \times \B^{-}, \tilde{\nabla}_{h}\times \C)\\
	&+\RH\,((\tilde{\nabla}_{h}\times \B)\times  \B^{-}, \tilde{\nabla}_{h}\times \C), \\
	b(\bm{x}, q):=\ & (\nabla\cdot \u, q).
\end{align*}

The mixed form of the Picard step in Algorithm \ref{alg:picard-s}
can be written as: for $\bm{h}\in \bm{W}_{h}^{\ast} $ and $g\in Q_{h}^*$, 
find $({\bm{x}}, p)\in  \bm{W}_{h} \times Q_{h}$, such that for all
$({\bm{y}}, q) \in \bm{W}_{h} \times Q_{h}$,  
\begin{subequations}
	\begin{alignat}{3}
		& a({\bm{x}}, \bm{y})+ b(\bm{y}, p)
		&&=\langle \bm{h}, \bm{y} \rangle, \\
		& b(\bm{x}, q)
		&&=\langle g, q \rangle .
		\label{brezzip}
	\end{alignat}
\end{subequations}
Define the norms
\begin{subequations}
	\begin{align}
			\|(\u,  \B)\|^{2}_{X}&:=\|\nabla \u\|^{2}+\|\nabla\cdot \B\|^{2}+\|\tilde{\nabla}_{h}\times \B\|^{2}, \label{stationary-BE-norm-X} \\
		\|p\|_{Q}&:=\|p\| \label{norm-Y}.
	\end{align}	
\end{subequations}
We verify that $\|\cdot\|_{X}$ is a norm. Indeed, $\|(\u,   \B)\|_{X}^{2}$ is quadratic for $\bm x:=(\u,  \B)$. Moreover, when $\|(\u,   \B)\|_{X}=0$, we find $\u=\mathbf{0}$ (Poincar\'e inequality) and $\B=\mathbf{0}$ (generalised Poincar\'e inequality or the discrete Gaffney inequality).

\begin{theorem}\label{thm:wellposed-picard}
	Assume that $\B^{-}\in \mathbf{L}^{\infty}(\Omega)$. Then, problem \eqref{brezzip} is well-posed with the norms defined by \eqref{stationary-BE-norm-X} and \eqref{norm-Y}.
\end{theorem}

\begin{proof}
	To prove the well-posedness of \eqref{brezzip} based on the Brezzi theory, we need to verify the boundedness of each term, the inf-sup condition of $b(\cdot, \cdot)$ and the coercivity of $a(\cdot,\cdot)$ on the discrete kernel defined by 
	$$
	\bm{W}_{h}^{0}:=\{\bm x\in \bm{W}_{h}: ~(\nabla \cdot \u, q)=0 \quad \forall q\in Q_{h}\}.
	$$
	The boundedness of both bilinear forms is obvious from the definition of the norms. In particular, the Hall term fulfils
	$$
	|((\tilde{\nabla}_{h}\times \B)\times  \B^{-}, \tilde{\nabla}_{h}\times \C)|\leq \|\tilde{\nabla}_{h}\times \B\|\|\B^{-}\|_{L^{\infty}}\|\tilde{\nabla}_{h}\times \C\|.
	$$
	The inf-sup condition of $b(\cdot, \cdot)$ follows by assumption.
	To prove coercivity on the kernel, we take
	$\v=\u$ and $\C=S\B$, yielding
	\begin{align*}
		\bm a((\u,   \B), (\v,   \C))&= \Re^{-1}\|\nabla \u\|^{2}+ S\|\nabla\cdot \B\|^{2}+ S\Reminv \|\tilde{\nabla}_{h}\times \B\|^{2},
	\end{align*}
	and thus the coercivity of $a(\cdot, \cdot)$.
	Combining the boundedness of the variational forms, the inf-sup condition of $b(\cdot, \cdot)$ and the coercivity of $a(\cdot,\cdot)$ on $\bm{W}_{h}^{0}$, we complete the proof.
\end{proof}
\begin{remark}
	The assumption $\B^{-}\in \mathbf{L}^{\infty}(\Omega)$ is due to the Hall term, since we do not have higher regularity for $\tilde{\nabla}_{h}\times \B$ and $\tilde{\nabla}_{h}\times \C$ than $\mathbf{L}^2(\Omega)$. The other nonlinear terms can be controlled by $\|\cdot\|_{X}$ as, e.g., 
	\begin{align*}
	|(\u \times \B^{-}, \tilde{\nabla}_{h}\times \C)| &\leq \|\u\|_{L^{6}}\|\B^{-}\|_{L^{3}}\|\tilde{\nabla}_{h}\times \C\|\\
	 &\leq C\|\nabla \u\|(\|\tilde{\nabla}_{h}\times \B^{-}\|^{2}+\|\nabla\cdot \B^{-}\|^{2})^{\frac{1}{2}}\|\tilde{\nabla}_{h}\times \C\|,
	\end{align*}
	where we used the Poincar\'e inequality, the Sobolev embedding, and the discrete Gaffney inequality for the last step. On the discrete level, we always have that the finite element function $\B^{-}\in \mathbf{L}^{\infty}(\Omega)$ and hence we have proved the well-posedness of the discrete problem on a fixed mesh.
\end{remark}

\begin{remark}
	In the above proof, we have used that $(( \u^{-} \cdot \nabla) \u, \u)=0$ which holds if $\nabla\cdot\u=0$ is enforced exactly on the discrete level. If one wishes to use a Stokes pair that is not exactly divergence-free, one can replace this term by $(( \u^{-} \cdot \nabla) \u, \v)-(( \u^{-} \cdot \nabla) \v, \u)$. This approximation is equal to  $(( \u^{-} \cdot \nabla) \u, \u)=0$ if $\nabla\cdot\u$ and a consistent approximation otherwise, cf.~\cite{hu2020convergence}.
\end{remark}

\begin{remark}[Boundary conditions]
	For the standard MHD equations with $\u=\mathbf{0}$ and $\B \cdot \n = \mathbf{0}$ on $\partial \Omega$, the boundary conditions $\E \times \n = \mathbf{0}$ and $\j \times \n = \mathbf{0}$ are equivalent  due to Ohm's law $\j = \E + \u \times \B$. However, for the Hall MHD equations $\E \times \n = \mathbf{0}$ and $\j \times \n = \mathbf{0}$ are independent. The generalised Ohm's law then implies
	\begin{align*}
		&\Reminv  \j \times \n = \E\times \n + (\u \times \B) \times \n - \RH\,(\j \times \B) \times \n\\
		\Leftrightarrow\  & \RH\,(\j \times \B) \times \n= \mathbf{0}\\
		\Leftrightarrow\  & \RH \left[(\j \cdot \n) \B - \j\, \B \cdot \n \right] = \mathbf{0}\\
		\Rightarrow\  & \j \cdot \n = 0.
	\end{align*}
	Hence, there exists an additional compatibility condition that $\j \cdot \n = 0$. 
	
\end{remark}	

In the following, we consider the convergence of the Picard iteration. 
\begin{theorem}\label{thm:picard_convergence}
	For a fixed mesh drawn from a quasi-uniform sequence (so that the inverse estimates hold) and $f\in [H^{-1}]^{3}$, $\u_{h}^{n}$, $p_h^n$, $\E_h^n$, $\j_{h}^{n}$ and $\B_{h}^{n}$ from Algorithm \ref{alg:picard-s} converge if $\Rem$ and $\Re$ are small enough.  
\end{theorem}

The proof is similar to \cite[Theorem 7]{Hu2018}, and we only give a sketch of the proof focusing on the additional Hall term. 
The essence of the proof is to show that one gets a contraction in the errors $\bm e_{u}^{n}:=\u_{h}^{n}-\u_{h}^{n-1}$ and $\bm e_{j}^{n}:=\tilde{\nabla}_{h}\times \B_{h}^{n}-\tilde{\nabla}_{h}\times \B_{h}^{n-1}$, i.e.,
\begin{equation}\label{contraction0}
	\frac{1}{2}(	\Re^{-1}\|\nabla \bm e_{u}^{n}\|^{2}+S \Reminv\|\bm e_{j}^{n}\|^{2})\leq \frac{1}{4}(\Re^{-1}\|\nabla \bm e_{u}^{n-1}\|^{2}+S \Reminv\| \bm e_{j}^{n-1}\|^{2}),
\end{equation}
if $\Re$ and $\Rem$ are small enough. One gets an expression for these errors by subtracting the $(n-1)$-th step of 
\eqref{picard-reduced1}-\eqref{picard-reduced3} from the $n$-th step and using the test functions 
$\v_{h}=\bm e_{u}^{n}$  and $\C_{h}=\B_{h}$. 
This gives
\begin{align}\label{contraction}
	\begin{split}
	\Re^{-1}\|\nabla \bm e_{u}^{n}\|^{2}+S\Reminv \|\bm e_{j}^{n}\|^{2}=&(\u_{h}^{n}\times \B_{h}^{n-1}-\u_{h}^{n-1}\times \B_{h}^{n-2}, \bm e_{j}^{n})+\cdots\\
	&-\RH(\j_{h}^{n}\times \B_{h}^{n-1}-\j_{h}^{n-1}\times \B_{h}^{n-2}, \bm e_{j}^{n}).
\end{split}
\end{align}
Here we have omitted other terms of the standard MHD system which are treated in detail in  \cite[Theorem 7]{Hu2018}.
The last term is the Hall term. The first term can be estimated by 
\begin{align*}
	|(\u_{h}^{n}\times \B_{h}^{n-1}-\u_{h}^{n-1}\times \B_{h}^{n-2}, \bm e_{j}^{n})|&=|(\bm e_{u}^{n}\times \B_{h}^{n-1},  \bm e_{j}^{n})+(\u_{h}^{n-1}\times  \bm e_{B}^{n-1},  \bm e_{j}^{n})|\\&
	\leq C( \| \bm e_{u}^{n}\|_{L^6}\|\B_{h}^{n-1}\|_{L^3}\| \bm e_{j}^{n}\|+ \|\u_{h}^{n-1}\|_{L^6}\| \bm e_{B}^{n-1}\|_{L^3}\|\bm e_{j}^{n}\|)\\&
	\leq C(\|\nabla \bm e_{u}^{n}\|^{2}+\|\bm e_{j}^{n}\|^{2}+\|\bm e_{j}^{n-1}\|^{2}),
\end{align*}
where in the last step we have used the Sobolev embedding $\| \bm e_{u}^{n}\|_{L^6}\leq C\|\nabla \bm e_{u}^{n}\|$, the generalised Gaffney inequality $\| \bm e_{B}^{n-1}\|_{L^3}\leq C\|\tilde{\nabla}\times \bm e_{B}^{n-1}\|=C\|\bm e_{j}^{n-1}\|$, and the energy bounds $\|\B_{h}^{n-1}\|_{L^3}\leq C\|\f\|_{-1}$, $\|\u_{h}^{n-1}\|\leq C\|\f\|_{-1}$ ($\|\f\|_{-1}$ is assumed to be a given finite number). For $\Re^{-1}$ and $\Reminv$ large enough, we can move $\|\nabla \bm e_{u}^{n}\|^{2}$ and $\|\nabla \bm e_{j}^{n}\|^{2}$ to the left hand side of \eqref{contraction}. 

The boundedness of the Hall term is more complicated. In fact, for some $0\leq\delta\leq 3$ depending on the domain, 
\begin{align*}
	|(\j_{h}^{n}\times \B_{h}^{n-1}-\j_{h}^{n-1}\times \B_{h}^{n-2}, \bm e_{j}^{n})|&=|(\bm e_{j}^{n}\times \B_{h}^{n-1}, \bm e_{j}^{n})+(\j_{h}^{n-1}\times \bm e_{B}^{n-1}, \bm e_{j}^{n})|\\
	&=|(\j_{h}^{n-1}\times \bm e_{B}^{n-1}, \bm e_{j}^{n})|\leq C\|\j_{h}^{n-1}\|_{L^{\frac{6+2\delta}{1+\delta}}}\|\bm e_{B}^{n-1}\|_{L^{3+\delta}}\|\bm e_{j}^{n}\|\\
	& \leq  Ch^{-\frac{3}{3+\delta}}\|\j_{h}^{n-1}\|\|\bm e_{j}^{n-1}\| \|\bm e_{j}^{n}\|\\&
	\leq  Ch^{-\frac{3}{3+\delta}}(\|\bm e_{j}^{n-1}\| ^{2}+\|\bm e_{j}^{n}\|^{2}),
\end{align*}
where we used the inverse estimate, the generalised Gaffney inequality and the energy bound $\|\j_{h}^{n-1}\|\leq C\|\f\|_{-1}$. Again, we move $\|\bm e_{j}^{n}\|^{2}$ to the left hand side of \eqref{contraction} if $\Re^{-1}$ and $\Reminv$ are large enough.
The contraction \eqref{contraction0} proves the convergence of $\u^n_h$ and $\j^n_h$. Note that the convergence of $\j^n_h$ also implies the convergence of $\B^n_h$ since $\| \bm e_{B}^{n}\|\leq C\|\tilde{\nabla}\times \bm e_{B}^{n}\|=C\|\bm e_{j}^{n}\|$.

To show the convergence of $p_{h}^{n}$, we note that from \eqref{picard-reduced1}, 
\begin{align*}
	(p_{h}^{n}-p_{h}^{n-1}, \nabla\cdot \v_{h})=\Re^{-1}(\nabla \bm e_{u}^{n}, &\nabla \v_{h})+((\bm e_{u}^{n-1}\cdot \nabla)\u_{h}^{n}, \v_{h})+((\u_{h}^{n-2}\cdot \nabla)\bm e_{u}^{n}, \v_{h})\\&-\S (\bm e_{j}^{n}\times \B_{h}^{n-1}, \v_{h})-\S (\j_{h}^{n-1}\times \bm e_{B}^{n-1}, \v_{h}).
\end{align*}
From the inf-sup condition of the velocity-pressure pair, there exists $\v_{h}$ such that
$$
(p_{h}^{n}-p_{h}^{n-1}, \nabla\cdot\v_{h})\geq C\|p_{h}^{n}-p_{h}^{n-1}\|^{2}, \quad\mbox{and}\quad \|\v_{h}\|_{1}\leq \|p_{h}^{n}-p_{h}^{n-1}\|. 
$$
Taking this $\v_{h}$ as the test function, we get
\begin{align*}
	C\|p_{h}^{n}-p_{h}^{n-1}\|^{2}\leq &\Re^{-1}\| \bm e_{u}^{n}\|_{1}\| \v_{h}\|_{1}+\|\bm e_{u}^{n-1}\|_{1}\|\u_{h}^{n}\|_{1} \|\v_{h}\|_{1}+\|\u_{h}^{n-2}\|_{1}\|\bm e_{u}^{n}\|_{1}\|\v_{h}\|_{1}\\&+\S \|\bm e_{j}^{n}\|\|\B_{h}^{n-1}\|_{L^{3}}\| \v_{h}\|_{1}+\S \|\j_{h}^{n-1}\|\| \bm e_{B}^{n-1}\|_{L^{3}}\| \v_{h}\|_{1}.
\end{align*}
Since $\u_{h}^{n}$ converges in $H^{1}(\Omega)$ and $\B_{h}^{n}$ converges in $L^{3}(\Omega)$ (alternatively, $\j_{h}^{n}=\tilde{\nabla}_{h}\times \B_{h}^{n}$ converges in $L^{2}(\Omega)$), we obtain the $L^{2}$-convergence of $p_{h}^{n}$ by the Cauchy-Schwarz inequality.

	For the standard MHD equations, the convergence of the electric field $$\E_{h}^{n}=\Reminv\tilde{\nabla}_{h}\times \B_{h}^{n}-\mathbb{Q}_{c}(\u_{h}^{n}\times \B_{h}^{n-1})+\RH\,\mathbb{Q}_{c}((\tilde{\nabla}_{h}\times \B_{h}^{n})\times  \B_{h}^{n-1})$$ follows from the strong convergence of $\B_{h}^{n}$ in $\mathbf{H}_0^h(\div)\cap \mathbf{H}_{0}^{h}(\curl)\hookrightarrow L^{3+\delta}$  and $\u_{h}^{n}$ in $\mathbf{H}^{1}\hookrightarrow \mathbf{L}^{6}$. For the convergence of the Hall-term, we can apply the inverse estimate as before.

\begin{remark}
	For the standard MHD system, the condition on the size of $\Re^{-1}$ and $\Reminv$ only depends on $\|\f\|_{-1}$. Due to the Hall term, this condition also involves a factor $h^{-\frac{3}{3+\delta}}$ which might suggest that the convergence of the Picard iteration deteriorates on finer meshes. Theorem \ref{thm:picard_convergence} proves the convergence of the Picard iteration on a fixed mesh.
\end{remark}

\subsection{2.5D Hall MHD formulation}\label{sec:2.5Dform}
In this section, we introduce the 2.5-dimensional formulation of \eqref{eq:HallMHD}, which refers to the assumption that vector fields still have three components but derivatives in the $z$-direction vanish. That means we assume that a three-dimensional vector-field can be decomposed into a two-dimensional vector field and scalar field with the notation
\begin{equation}
	\B(x,y,z) = \begin{pmatrix}
		\tilde{\B} (x,y) \\
		B_3(x,y)
	\end{pmatrix}.
\end{equation}
Recall that there exist two different curl operators in two dimensions, \mbox{given by}
\begin{equation}
	\scurl \tilde{\B} = \partial_x B_2 - \partial_y B_1, \qquad  \vcurl B_3 = \begin{pmatrix}
		\partial_y B_3\\
		-\partial_x B_3
	\end{pmatrix},
\end{equation}
that correspond to the cross-products
\begin{equation}
	\tilde{\u} \times \tilde{\B} =u_1 B_2 - u_2 B_1, \qquad
	\tilde{\B}  \times E_3 =\begin{pmatrix}
		B_2 E_3 \\
		-B_1 E_3
	\end{pmatrix}.
\end{equation}
Hence, we can rewrite the three-dimensional cross-product and curl operator as
\begin{equation}
	\j \times \B = \begin{pmatrix}
		\tilde{\j} \times B_3 - \tilde{\B} \times j_3 \\
		\tilde{\j} \times \tilde{\B}
	\end{pmatrix} 
	\quad \text{ and } \quad
	\nabla \times \B = 
	\begin{pmatrix}
		\vcurl B_3 \\
		\scurl \tilde{\B}
	\end{pmatrix}.
\end{equation}

With this notation we are able to rewrite \eqref{eq:HallMHD} on a bounded polygonal Lipschitz domain $\Omega \subset \mathbb{R}^2$ as 
\begin{subequations}
	\label{eq:2.5DHallMHD}
	\begin{align} \label{eq:2.5DHallMHDu}
		- \Re^{-1} \Delta \tilde{\u} + 
		( \tilde{\u} \cdot \tilde{\nabla}) \tilde{\u}
		- \S\, ( \tilde{\j} \times B_3 -  \tilde{\B} \times j_3)
		+ \tilde{\nabla} p 
		&= \tilde{\f} , \\
		- \Re^{-1} \Delta u_3 + 
		( \tilde{\u} \cdot \tilde{\nabla}) u_3
		- \S\, \tilde{\j} \times \tilde{\B}
		&= f_3 , \\
		\tilde{\j} - 
		\vcurl B_3
		&= \mathbf{0} , \\
		j_3 - 
		\scurl \tilde{\B}
		&= 0 , \\
		\vcurl E_3 &= \mathbf{0} , \\
		\scurl \tilde{\E} &= 0 , \\
		\tilde{\nabla} \cdot \tilde{\B} &= 0, \\
		\tilde{\nabla} \cdot \tilde{\u} &= 0, \\
		\Reminv \tilde{\j} - (\tilde{\E} + \tilde{\u} \times B_3 - \tilde{\B} \times u_3 -\RH\,(\tilde{\j}\times B_3 - \tilde{\B} \times j_3)) & = \mathbf{0}, \label{eq:islandcoal-j}\\
		\Reminv j_3 - (E_3 + \tilde{\u} \times \tilde{\B}-\RH\,\tilde{\j}\times \tilde{\B}) & = 0, \label{eq:islandcoal-j3}
	\end{align}
\end{subequations}

subject to the boundary conditions
\begin{equation}
	\tilde{\u} = \mathbf{0}, \,\, u_3 = 0, \,\, \tilde{\B} \cdot \tilde{\n} = 0, \,\, B_3 = 0, \,\, \tilde{\j} \times \tilde{\n} =  \mathbf{0}, \,\, j_3 = 0, \,\, \tilde{\E} \times \tilde{\n} =  \mathbf{0}, \,\, E_3 = 0. 
\end{equation}
For a finite element discretisation, as before we can look for $\tilde{\B}_h$ in an $\mathbf{H}_{0}^{h}(\div)$-conforming space and for $\tilde{\j}_h$ and $\tilde{\E}_h$ in an $\mathbf{H}_{0}^{h}(\curl)$-confirming space. The other components $u_3$, $B_3$, $j_3$ and $E_3$ are approximated in an $H_{0}^{h}(\grad)$-conforming space.

\section{Conservative discretisations for time-dependent problems}\label{sec:timedepproblems}

For time-dependent problems, we include the time derivatives in the formulation for the stationary problem, i.e., we add $\frac{\partial \u_{h}}{\partial t}$ to \eqref{fem-1} and $\frac{\partial \B_{h}}{\partial t}$ to \eqref{fem-3}. 

\subsection{Conserved quantities}
In the ideal limit of $\Re=\Rem=\infty$ it is well-known that the energy, magnetic helicity and cross helicity are conserved properties of the standard incompressible MHD system \cite{Galtier2015}. The energy is defined as
\begin{equation}
	E := \int_\Omega |\u|^2 + S|\B|^2 \ \mathrm{d} x,
\end{equation}
the magnetic helicity is defined as
\begin{equation}
	H_M := \int_\Omega \mathbf{A} \cdot \B \ \mathrm{d} x,
\end{equation}
for a vector potential $\mathbf{A}$ such that $\nabla \times \mathbf{A} = \B$, and the cross helicity is defined as
\begin{equation}
	H_C := \int_\Omega \mathbf{\u} \cdot \B \ \mathrm{d} x.
\end{equation}
For the ideal Hall MHD equations, the energy and magnetic helicity are still conserved, while the cross helicity is not. Here, hybrid helicity replaces the cross helicity as a conserved property and is defined as 
\begin{equation}
	H_H := \int_\Omega (\mathbf{A} + \alpha \u)\cdot (\B + \beta \nabla \times \u) \, \mathrm{d}x,
\end{equation}
for $\alpha$ and $\beta$ satisfying the relation 
\begin{equation}\label{eqn:alphabeta}
	2S\alpha\beta-{\RH}(\alpha+\beta)=0.
\end{equation}
We prove the conservation of hybrid helicity in the next theorem. Note, that the hybrid helicity is a combination of the magnetic, cross and fluid helicity, which is defined as 
\begin{equation}
	H_F := \int_\Omega \u \cdot \nabla \times \u \ \mathrm{d}x.
\end{equation} If ${\RH}=0$, i.e., when the Hall term vanishes, the above equality \eqref{eqn:alphabeta} holds if $\alpha=0$ or $\beta=0$. For $\alpha=\beta=0$,  the hybrid helicity is just the magnetic helicity. If $\alpha=0$ and  $\beta\neq 0$ (alternatively, $\alpha\neq 0$ and $\beta=0$), the hybrid helicity becomes a combination of magnetic and cross helicity. Thus the conservation of hybrid helicity implies the conservation of both magnetic and cross helicity in standard MHD. In Hall MHD, $\alpha=\beta=0$ still corresponds to the magnetic helicity. But in this case \eqref{eqn:alphabeta} does not allow the case $\alpha=0$, $\beta\neq 0$, or $\alpha\neq 0$, $\beta=0$. This means that the cross helicity is not conserved. There exist many non-trivial choices of $\alpha$ and $\beta$, for example, $\alpha=\beta = S^{-1}\RH$.

\begin{theorem}
	The generalised hybrid helicity $H_{H}$ is conserved in the time-dependent Hall MHD system with $\f=\mathbf{0}$ and formally $\Re^{-1}=\Reminv =0$ for any $\alpha$, $\beta$ such that \eqref{eqn:alphabeta} holds.
\end{theorem}
\begin{proof}
	We have
	\begin{align*}
		\frac{d}{dt}H_{H}&=\frac{d}{dt}(\mathbf{A}, \B)+\frac{d}{dt}[\alpha(\u, \B)+\beta(\mathbf{A}, \bm \omega)]+\frac{d}{dt}\alpha\beta(\u, \bm \omega)\\
		&=\frac{d}{dt}(\mathbf{A}, \B)+\frac{d}{dt}(\alpha+\beta)(\u, \B)+\frac{d}{dt}\alpha\beta(\u,  \bm \omega).
	\end{align*}
	First, the magnetic helicity is conserved, i.e., 
	\begin{align*}
		\frac{d}{dt}(\mathbf{A}, \B) &= 2 (\B_t, \mathbf{A}) = 2 (\nabla \times [\u \times \B] , \mathbf{A}) - 2 \RH (\j \times \B , \mathbf{A})\\
		& = 2 (\u \times \B, \B) - 2\RH (\j \times \B, \B) = 0.
	\end{align*}
	It remains to check the other two terms. In fact,
	from \eqref{eq:HallMHDu},
	\begin{equation*}
		(\u_{t}, \B)=(\B, \u\times \bm \omega+S\j\times \B-\nabla p)=(\B, \u\times \bm \omega).
	\end{equation*}
	From \eqref{eq:HallMHDB}, 
	\begin{align*}
		(\B_{t}, \u)&=-(\nabla\times \E, \u)=-(\E, \nabla\times\u)=(\u\times\B-{\RH}\j\times \B, \nabla\times\u)\\
		&=(\u\times\B,\bm \omega)-{\RH}(\j\times \B, \bm \omega).
	\end{align*}
	Consequently,
	\begin{align*}
		\frac{d}{dt}(\alpha+\beta)(\u, \B)=(\alpha+\beta)[(\u_{t}, \B)+(\u, \B_{t})]=-{\RH}(\alpha+\beta)(\j\times \B, \bm \omega).
	\end{align*}
	Moreover, 
	\begin{align*}
		\frac{d}{dt}\alpha\beta(\u, \bm \omega)=2\alpha\beta(\u_{t}, \bm \omega)=2\alpha\beta(\u\times \bm \omega+S\j\times \B-\nabla p, \bm \omega)=2S\alpha\beta(\j\times \B, \bm \omega).
	\end{align*}
	This implies that
	$$
	\frac{d}{dt}H_{H}=[2S\alpha\beta-{\RH}(\alpha+\beta)](\j\times \B, \bm \omega)
	$$
	and proves the desired result. 
\end{proof}

Similar to the discussions in \cite{arnold1999topological}, we show that the hybrid helicity provides a lower bound for the energy when $\alpha=\beta=S^{-1}\RH$. This bound, which was referred to as the Arnold inequality in the case of the magnetic helicity \cite[Section 8]{moffatt2021some}, shows that non-zero hybrid helicity, as a measure of the knottedness, provides a topological barrier which prevents a hybrid energy defined by $\|\B+ S^{-1}\RH \bm \omega\|^{2}$ from decaying below a certain value. The conclusion also holds for dissipative flows where the helicity is not conserved.
\begin{theorem}\label{thm:arnold}
	$$
	\|\B+ S^{-1}\RH \bm \omega\|^{2}\geq C^{-1}|H_{H}|,
	$$
	where $C$ is the positive constant in the Poincar\'e inequality.
\end{theorem}
\begin{proof}
	\begin{align*}
		|H_{H}|= \left | \int (\mathbf{A} + S^{-1}\RH \u)\cdot (\B + S^{-1}\RH \bm \omega)\, dx \right |&\leq \|\mathbf{A} + S^{-1}\RH \u\|\|\B + S^{-1}\RH \bm \omega\|\\
		&\leq C\|\B + S^{-1}\RH \bm \omega\|^{2}.
	\end{align*}
\vspace{-2cm}

\end{proof}

Next, we present time discretisations that preserve the above quantities precisely on the discrete level.
The MHD system has delicate differential structures reflected in its various conserved quantities, e.g., the energy, the magnetic Gauss law, and the magnetic and cross/hybrid helicity. In fact, in the proof of the energy conservation, the Lorentz force and the magnetic convection cancel each other, and the fluid convection cancels itself. For the cross helicity, the fluid and magnetic convection cancel each other, and the Lorentz force cancels itself.
To construct conservative numerical methods, it is important to respect these symmetries on the discrete level. This in turn requires certain algebraic structures among the discrete spaces; for example, to preserve the magnetic Gauss law, we discretise unknowns on discrete de~Rham sequences, as in \eqref{eqn:derhamfe}. The magnetic helicity involves the magnetic field and its potential. Therefore it is largely independent of the fluid discretisation. However, the energy law and the conservation of cross/hybrid helicity essentially derive from the symmetric coupling between fluids and electromagnetic fields.  Thus it is not surprising that to preserve them on the discrete level, the finite element spaces for the velocity and pressure (Stokes pairs) have to interplay with the spaces for the electromagnetic fields (de~Rham sequences). 

Therefore, the imposition of the boundary condition $\u=\mathbf{0}$ on $\partial \Omega$ can cause difficulties in designing conservative methods, because the description of all components of $\u$ on the boundary does not fit to the electromagnetic boundary conditions. Hence, the literature distinguishes for the standard MHD system between the boundary conditions $\u \times \n$ \cite{hu2021helicity} and $\u \cdot \n$ \cite{gawlik2020}, where the velocity field $\u$ is discretised with $\Hhc$- and $\Hhd$-conforming finite element spaces respectively. Both schemes conserve the energy, magnetic and cross helicity precisely on the discrete level. In the following, we also focus on these two cases and extend the proposed algorithms for the additional Hall-term and the hybrid helicity.

\subsection{Helicity and energy preserving scheme for $\u \times \n =  \mathbf{0}$}\label{sec:schemeutimesn}
In this section, we present a time discretisation that preserves the energy and magnetic and hybrid helicity precisely for the boundary condition $\u \times \n = \mathbf{0}$ on $\partial \Omega$. Since these quantities are only preserved for $\f = \mathbf{0}$ and formally $\Re^{-1} = \Reminv =\infty$, we focus only on this case from now on for this section.

The following approach is mainly taken from \cite{hu2021helicity}, but adapted for the additional Hall-term. Let $\Qc$ denote the projection to $\Hhc$, $\Qd$ the projection to $\Hhd$ and $P_h:= p_h + 1/2 |\u_h|^2$ the total pressure. 

We first consider a semi-discrete formulation, discretised in space. We formally eliminate the electric field $\E_h$ by the generalised Ohm's law \eqref{eq:HallMHDE}.
The problem is: find $(\u_h(t), P_h(t), \B_h(t), \j_h(t)) \in \Hhc \times  H^1_0(\Omega) \times \Hhd \times \Hhc$ such that (we drop the argument $t$ in the following)
\begin{subequations}\label{alg:helicity-operator2}
	\begin{align}
		((\u_h)_t, \v_h) + (\Qc[\nabla \times \u_h] \times \u_h, \v_h) \qquad \qquad \quad &
		\nonumber \\- S (\j_h \times \Qc \B_h ,\v_h) + (\nabla P_h, \v_h) = 0 &\quad \forall\, \v_h \in \Hhc,  \label{alg:helicity-operator2-u}\\
		( \u_h, \nabla Q_h) = 0 &\quad \forall\, Q_h \in H^1_0(\Omega), \\
		((\B_h)_t, \C_h) - (\nabla \times \Qc[\u_h \times \Qc \B_h], \C_h ) \qquad \qquad \quad & \nonumber \\ 
		+ \RH (\nabla \times \Qc[\j_h \times \Qc \B_h], \C_h ) = 0    &\quad \forall\, \C_h \in \Hhd, \label{alg:helicity-operator2-b}\\
		(\j_h, \k_h) - (\B_h, \nabla \times \k_h) = 0&\quad \forall\, \k_h \in \Hhc. \label{alg:helicity-operator2-j}
	\end{align}
\end{subequations}
This formulation is useful for analysis but not yet amenable to computation, due to the presence of the projection operators.

\begin{theorem}
	Any solution  $(\u_{h}, p_h, \B_{h}, \j_{h})$ of \eqref{alg:helicity-operator2} fulfils the magnetic Gauss's law $\nabla \cdot \B_h = 0$ precisely if $\nabla \cdot \B_h^0=\mathbf{0}$.
\end{theorem}
\begin{proof}
	Choosing
	\begin{equation*}
		\C_h = (\B_h)_t - \nabla \times \Qc[\u_h \times \Qc \B_h + \RH \j_h \times \Qc \B_h]
	\end{equation*}
	in \eqref{alg:helicity-operator2-b} gives $(\B_h)_t = \nabla \times \Qc[\u_h \times \Qc \B_h + \RH \j_h \times \Qc \B_h] $ and hence $\nabla \cdot \B_h = 0$ if $\nabla \cdot \B_h^0=\mathbf{0}$.
\end{proof}
\begin{theorem}
	Any solution $(\u_{h}, p_h, \B_{h}, \j_{h})$ of \eqref{alg:helicity-operator2} satisfies the energy identity
	$$
	\frac{1}{2}\frac{d}{dt}(\|\u_{h}\|^{2}+S\|\B_{h}\|^{2})= 0.
	$$
\end{theorem}
\begin{proof}
	Testing \eqref{alg:helicity-operator2-u} with $\u_{h}$,
	$$
	\frac{1}{2}\frac{d}{dt}\|\u_{h}\|^{2}= S (\j_h \times \Qc \B_h, \u_h).
	$$
	Testing \eqref{alg:helicity-operator2-b} with $\B_{h}$, 
	\begin{align*}
		\frac{1}{2}\frac{d}{dt}\|\B_{h}\|^{2}
		& = (\nabla \times \Qc[\u_h \times \Qc \B_h], \B_h ) - \RH (\nabla \times \Qc[\j_h \times \Qc \B_h], \B_h ) \\
		& = ( \Qc[\u_h \times \Qc \B_h], \j_h ) - \RH (\Qc[\j_h \times \Qc \B_h], \j_h ) \\
		& = - ( \j_h \times \Qc \B_h, \u_h ).
	\end{align*}
	Here we have used the definition of $\j_h$ in \eqref{alg:helicity-operator2-j} and that $\j_h \in \Hhc$.
	Consequently, the desired result holds by adding the above equalities. 	
\end{proof}

On the discrete level we define the hybrid helicity as 
\begin{equation}
	H_{H}:=\int_\Omega (\mathbf{A}_h + \alpha \u_h)\cdot (\B_h + \beta \bm \omega_h) \, \mathrm{d}x,
\end{equation}
where $\bm \omega_h := \Qc \nabla \times \u_h$ and ($\alpha, \beta$) satisfies \eqref{eqn:alphabeta}. 

\begin{theorem}
	The hybrid helicity of \eqref{alg:helicity-operator2} is conserved if $\f=\mathbf{0}$ and formally $\Re^{-1}=\Reminv =0$ for any $\alpha, \beta$ such that \eqref{eqn:alphabeta} holds.
\end{theorem}

\begin{proof}
	Similar to the continuous level, we have
	\begin{align*}
		\frac{d}{dt}H_{H}=\frac{d}{dt}(\mathbf{A}_{h}, \B_{h})+\frac{d}{dt}(\alpha+\beta)(\u_{h}, \B_{h})+\frac{d}{dt}\alpha\beta(\u_{h}, \bm \omega_h).
	\end{align*}
	
	Testing \eqref{alg:helicity-operator2-b} with $\Qd \mathbf{A}_{h}$, using that $\nabla \times \Hhc \subseteq \Hhd$ and integrating by parts, we have
	\begin{align*}
		\frac{d}{dt}(\mathbf{A}_{h}, \B_{h})=2((\B_{h})_{t}, \mathbf{A}_{h})
		& = 2 (\nabla \times \Qc[ \u_h \times \Qc \B_h], \mathbf{A}_h ) - 2 \RH (\nabla \times \Qc[\j_h \times \Qc \B_h],  \mathbf{A}_h )\\
		& = 2 ( \Qc[\u_h \times \Qc \B_h], \B_h) - 2 \RH (\Qc[\j_h \times \Qc \B_h],  \B_h )\\
		& = 2 ( \u_h \times \Qc \B_h, \Qc \B_h) - 2 \RH (\j_h \times \Qc \B_h,  \Qc\B_h )\\
		&=0.
	\end{align*}
	Testing \eqref{alg:helicity-operator2-u} with $\Qc \B_{h}$, we have
	\begin{align*}
		((\u_{h})_{t}, \B_{h})&=  -(\Qc[\nabla \times \u_h] \times \u_h, \Qc \B_h) + S (\j_h \times \Qc \B_h ,\Qc \B_h) \\
		& = -(\Qc[\nabla \times \u_h] \times \u_h, \Qc \B_h).
	\end{align*}
	Testing \eqref{alg:helicity-operator2-b} with $\Qd \u_{h}$, using that $((\B_{h})_{t}, \Qd \u_{h}) = ((\B_{h})_{t}, \u_{h})$, we have
	\begin{align*}
		((\B_{h})_{t}, \u_{h})
		& = (\nabla \times \Qc[\u_h \times \Qc \B_h], \u_h ) -  \RH (\nabla \times \Qc[\j_h \times \Qc \B_h],  \u_h )\\
		& = (\u_h \times \Qc \B_h, \Qc\nabla \times \u_h ) -  \RH (\j_h \times \Qc \B_h, \Qc[ \nabla \times \u_h]) \\
		& = (\Qc[\nabla \times \u_h] \times \u_h, \Qc \B_h) -  \RH (\j_h \times \Qc \B_h,   \Qc[ \nabla \times \u_h]). 
	\end{align*}
	Consequently,
	\begin{align*}
		\frac{d}{dt}(\alpha+\beta)(\u_{h}, \B_{h})&=(\alpha+\beta)[((\u_{h})_{t}, \B_{h})+(\u_{h}, (\B_{h})_{t})]\\
		&=-{\RH}(\alpha+\beta)(\j_h \times \Qc \B_h, \Qc[ \nabla \times \u_h]) .
	\end{align*}
	Moreover, testing  \eqref{alg:helicity-operator2-u} with $\Qc [\nabla \times \u_{h}]$, we get
	\begin{align*}
		\frac{d}{dt}\alpha\beta(\u_{h}, \nabla \times \u_h)
		&=2\alpha\beta((\u_{h})_{t},  \nabla \times \u_h)\\
		& = -2\alpha\beta(\Qc[\nabla \times \u_h] \times \u_h, \Qc [\nabla \times \u_h]) \\
		& \quad + 2\alpha\beta S (\j_h \times \Qc \B_h ,\Qc [\nabla \times \u_h])\\
		& = 2\alpha\beta S (\j_h \times \Qc \B_h ,\Qc [\nabla \times \u_h]).
	\end{align*}	
	This implies that
	\begin{equation*}
	\frac{d}{dt}H_{H}=[2S\alpha\beta-{\RH}(\alpha+\beta)] (\j_h \times \Qc \B_h, \Qc[ \nabla \times \u_h]).
	\end{equation*}
\vspace{-1.5cm}

\end{proof}

Similar to Theorem \ref{thm:arnold} on the continuous level, we have the following.
The proof is analogous, only using the discrete Poincar\'e inequality \cite[Theorem 5.11]{Arnold2006}.
\begin{theorem}[discrete Arnold inequality]
	\begin{equation*}
	\|\B_{h}+ S^{-1}\RH \bm \omega_{h}\|^{2}\geq C^{-1}|H_{H}|,
	\end{equation*}
	where $C$ is a positive constant.
\end{theorem}

To render the semi-discrete problem \eqref{alg:helicity-operator2} amenable to computation, we introduce auxiliary variables for the projection operators. The resulting problem is: find $(\u_{h}(t), P_h(t), \B_{h}(t), \E_{h}(t), \j_{h}(t), \mathbf{H}_{h}(t),\bm \omega_{h}(t)) \in$ $\Hhc \times H^1_0(\Omega) \times \Hhd \times [\Hhc]^4 ) $, such that for any $(\v_h, q_h, \C_h, \F_h, \k_h, \bm G_h, \bm \mu_h)$ in the same space,
\begin{subequations}\label{alg:helicityutimesn}
	\begin{align}\label{alg:helicity-cross-u}
		((\u_{h})_{t}, \v_h)- (\u_{h}\times\bm \omega_{h}, \v_h)-S(\j_{h}\times  \bm H_{h}, \v_h)\qquad \qquad \qquad \qquad \nonumber\\
		+\Re^{-1}(\nabla\times \u_{h}, \nabla \times \v_h)+( \v_h, \nabla P_{h})&= 0,
		\\
		( \u_{h}, \nabla Q_{h})&=0,
		\\\label{alg:helicity-cross-b}
		((\B_{h})_{t}, \C_h)+(\nabla\times  \E_{h}, \C_h)&=0,\\\label{alg:helicity-cross-j2}
		(\j_{h}, \F_h)-(\B_{h},  \nabla\times \F_h)&=0,\\\label{alg:helicity-cross-h1}
		(\bm H_{h}, \bm G_h)-(\B_{h},  \bm G_h)&=0,\\\label{alg:helicity-cross-h2}
		(\bm \omega_{h}, \bm \mu_h)-(\nabla\times \u_{h},  \bm \mu_h)&=0,\\\label{alg:helicity-cross-w2}
		-\Reminv (\j_{h}, \k_h) + (\E_{h}, \k_h)-((\RH\,\j_{h}-\u_{h})\times \bm H_{h},  \k_h)&=0.
	\end{align}
\end{subequations}
Now \eqref{alg:helicity-cross-j2} gives $\j_{h}=\tilde\nabla_{h}\times \B_{h}$,  \eqref{alg:helicity-cross-h1} gives $\bm H_{h}=\Qc \B_{h}$; and \eqref{alg:helicity-cross-h2} gives $\bm \omega_{h}=\Qc\nabla\times \u_{h}$.

For the time-discretisation, we replace the time-derivatives of $(\u_h)_t$ and $(\B_h)_t$ by the difference quotients
\begin{equation}
	D_t \u_h = \frac{\u^{k+1}_h - \u^{k}_h}{\Delta t} \quad \text{ and } \quad D_t \B_h = \frac{\B^{k+1}_h - \B^{k}_h}{\Delta t}.
\end{equation}
We replace $\u_h$ and $\B_h$ with the average of two neighbouring time steps defined as 
$\u^{k+\frac{1}{2}}:=\frac{1}{2}(\u^{k+1}+\u^{k})$ and $\B^{k+\frac{1}{2}}:=\frac{1}{2}(\B^{k+1}+\B^{k})$. All the other auxiliary variables are only defined on the midpoints of two time steps $k+\frac{1}{2}$ (not an average) and denoted as $P^{k+\frac{1}{2}}_h, \E^{k+\frac{1}{2}}_{h}, \j^{k+\frac{1}{2}}_{h}, \mathbf{H}^{k+\frac{1}{2}}_{h}$ and $\bm \omega^{k+\frac{1}{2}}_{h}$. This way we only have to provide initial data $\u^0_h$ and $\B^0_h$ and then solve the time-discretised version of \eqref{alg:helicityutimesn} for each $k \geq 1$; compare with \cite[Algorithm 1]{hu2021helicity}.

\begin{theorem}\label{thm:timeconsutimesn}
	The time-discretised version of \eqref{alg:helicityutimesn} preserves the energy, magnetic and hybrid helicity precisely and enforces $\nabla \cdot \B_h=0$ for all time steps; i.e., for all $ k \geq 0$ there holds
	\begin{align}
		\int_\Omega \mathbf{u}^{k+1}_h  \cdot \u_h^{k+1} + S \mathbf{B}^{k+1}_h  \cdot \B_h^{k+1} \mathrm{d} x &= \int_\Omega  \mathbf{u}^{k}_h  \cdot \u_h^{k} + S \mathbf{B}^{k}_h  \cdot \B_h^{k} \mathrm{d} x,\\	
		\int_\Omega \mathbf{A}^{k+1}_h  \cdot \B_h^{k+1} \mathrm{d} x &= \int_\Omega \mathbf{A}^{k}_h  \cdot \B_h^{k} \mathrm{d} x, \label{eq:magheltime}\\
		\int_\Omega \left(\mathbf{A}^{k+1}_h + \alpha \u^{k+1}_h\right)\cdot \left(\B_h^{k+1} + \beta \bm \omega^{k+1/2}_h\right) \mathrm{d} x &= \int_\Omega \left(\mathbf{A}^{k}_h + \alpha \u^{k}_h\right)\cdot \left(\B_h^{k} + \beta \bm \omega^{k-1/2}_h\right) \mathrm{d} x,\\
		\div \B^k_h = 0.
	\end{align}
\end{theorem}
\begin{proof}
	These results follow immediately from the proofs of the continuous results by replacing the continuous time-derivative $\partial_t$ by $D_t$. As an example, we prove the conservation of the magnetic helicity. It holds that
	\begin{equation*}
		\frac{1}{\Delta t}	\int_\Omega \mathbf{A}^{k+1}_h  \cdot \B_h^{k+1} - \mathbf{A}^{k}_h  \cdot \B_h^{k} \ \mathrm{d} x = (D_t \B_h, \mathbf{A}^{k+1/2}_h) + (D_t \mathbf{A}_h, \mathbf{B}^{k+1/2}_h). 
	\end{equation*}
	From the definition of the scheme, it follows that 
	\begin{align*}
		(D_t \B_h, \mathbf{A}^{k+1/2}_h) &= -\left(\nabla \times \E_h^{k+1/2}, \frac{\mathbf{A}_h^{k+1}+\mathbf{A}_h^k}{2}\right) \\
		& = -\left(\E_h^{k+1/2}, \frac{\mathbf{B}_h^{k+1}+\mathbf{B}_h^k}{2}\right) = -\left(\E_h^{k+1/2}, \mathbf{H}^{k+1/2}_h\right) \\
		& = - \left([\RH \j_h^{k+1/2} -  \u_h^{k+1/2}] \times \mathbf{H}^{k+1/2}_h, \mathbf{H}^{k+1/2}_h\right)=0.
	\end{align*}
	The term $(D_t \mathbf{A}_h, \mathbf{B}^{k+1/2}_h) $ vanishes with an analogous proof.
\end{proof}

\subsection{Helicity and energy preserving scheme for $\u \cdot \n =0$}\label{sec:schemeucdotn}
We now consider the boundary conditions $\u \cdot \n = 0$ on $\partial \Omega$. The presented scheme preserves the energy and magnetic helicity precisely, and in contrast to the previous algorithm also enforces $\nabla \cdot \u_h = 0$ precisely, but it does not preserve the hybrid helicity. Again, we only focus on $\f = \mathbf{0}$ and formally $\Re^{-1} = \Reminv =\infty$. 

The following algorithm is mainly taken from \cite{gawlik2020}, but adapted for the additional Hall-term.  
The semi-discrete form of our algorithm is given by: find $(\u_h(t), p_h(t), \B_h(t), \j_h(t)) \in \Hhd \times  \Ltz \times \Hhd \times \Hhc$ such that
\begin{subequations}\label{alg:helicity-operator3}
	\begin{align}
		((\u_h)_t, \v_h) + (\Qc[(\tilde \nabla_h \times \u_h) \times \Qc \u_h ], \v_h) \qquad \quad \nonumber &\\
		- S (\Qc[\j_h \times \Qc \B_h] ,\v_h) - (p_h, \nabla \cdot \v_h) = 0 &\quad \forall\, \v_h \in \Hhd,  \label{alg:helicity-operator3-u}\\
		(\nabla \cdot \u_h, q_h) = 0 &\quad \forall\, q_h \in \Ltz, \\
		((\B_h)_t, \C_h) - (\nabla \times \Qc[\Qc \u_h \times \Qc \B_h], \C_h ) + \qquad \nonumber &\\ \RH (\nabla \times \Qc[\j_h \times \Qc \B_h], \C_h ) = 0    &\quad \forall\, \C_h \in \Hhd, \label{alg:helicity-operator3-b}\\
		(\j_h, \k_h) - (\B_h, \nabla \times \k_h) = 0&\quad \forall\, \k_h \in \Hhc. \label{alg:helicity-operator3-j}
	\end{align}
\end{subequations}

For the following theorems, we only show the part of the proof that involves the additional Hall-term. The remainders of the proofs then coincide with the ones in \cite{gawlik2020}.

\begin{remark}
	Similar to before, every solution satisfies $\nabla \cdot \B_h = 0$ if $\nabla \cdot \B_h^0=0$. Furthermore, the $\Hd$-$L^2(\Omega)$ discretisation allows the exact enforcement of $\nabla \cdot \u_h=0$, e.g., for $\V_h = \mathbb{BDM}_k$ or $\V_h = \mathbb{RT}_k$ and $Q_h = \mathbb{DG}_{k-1}$ since then $\nabla \cdot \V_h \subset Q_h$.
\end{remark}

\begin{theorem}\label{thm:energyudotn}
	Any solution $(\u_{h}, p_h, \B_{h}, \j_{h})$ of \eqref{alg:helicity-operator3} satisfies the energy identity
	$$
	\frac{1}{2}\frac{d}{dt}(\|\u_{h}\|^{2}+S\|\B_{h}\|^{2})= 0.
	$$
\end{theorem}
\begin{proof}
	For the energy identity, it is crucial that the additional Hall term vanishes when \eqref{alg:helicity-operator3-b} is tested with $\B_h$. Indeed, we have that 
	\begin{equation*}
		\RH\,(\nabla \times \Qc [\j_h \times \Qc \B_h], \B_h) = \RH\,(\Qc[\j_h \times \Qc \B_h], \j_h) = 0,
	\end{equation*}
	since $\Qc \j_h=\j_h$ for $\j_h \in \Hhc$.
\end{proof}

\begin{theorem}\label{thm:magheludotn}
	The magnetic helicity of \eqref{alg:helicity-operator3} is conserved if $\f=\mathbf{0}$ and formally $\Re^{-1}=\Reminv =0$.
\end{theorem}

\begin{proof}
	We have to show that the Hall-term vanishes when \eqref{alg:helicity-operator3-b} is tested with a vector-potential $\mathbf{A}_h$. Calculating,
	\vspace{-0.5cm}
	\begin{align*}
		\RH (\nabla \times \Qc[\j_h \times \Qc \B_h],  \mathbf{A}_h ) & =  \RH (\Qc[\j_h \times \Qc \B_h],  \mathbf{B}_h ) \\
		& =\RH (\j_h \times \Qc \B_h,  \Qc \mathbf{B}_h ) = 0.
	\end{align*}
\vspace{-2.0cm} 

\end{proof}

\begin{remark}
	We discuss why a scheme that conserves hybrid helicity is difficult to construct for the boundary conditions $\u \cdot \n = 0$. First, these boundary conditions naturally fit with $\u_h \in \Hhd$. Therefore, the definition of the discrete hybrid helicity is not straight-forward due to the term $\nabla \times \u$. Two possible choices could be
	\begin{equation}
		H_{H}:=\int_\Omega (\mathbf{A}_h + \alpha \u_h)\cdot (\B_h + \beta \bm \omega_h) \, \mathrm{d}x,
	\end{equation}
	with either $\bm \omega_h = \nabla \times \Qc \u_h$ or $\bm \omega_h = \tilde{\nabla}_h \times \u_h$. The evolution of the fluid helicity would coincide for both definitions since
	\begin{align*}
		\frac{d}{dt} (\u_h,  \nabla \times \Qc \u_h) &= ((\u_h)_t,  \nabla \times \Qc \u_h) + (\u_h,  \nabla \times \Qc (\u_h)_t) \\
		&= ((\u_h)_t, \tilde{\nabla}_h \times \u_h  + \nabla \times \Qc \u_h  )
	\end{align*}
	and 
	\begin{align*}
		\frac{d}{dt} (\u_h,  \tilde{\nabla}_h \times \u_h ) & = ((\u_h)_t,  \tilde{\nabla}_h \times \u_h ) + (\u_h,  \tilde{\nabla}_h \times (\u_h)_t ) \\
		& = ((\u_h)_t, \nabla \times \Qc \u_h + \tilde{\nabla}_h \times \u_h ).
	\end{align*}
	The right-hand side can be modified to $ \nabla \times \Qc \u_h + \Qd^0\tilde{\nabla}_h \times \u_h$, where $\Qd^0$ denotes the projection to the divergence-free functions in $\Hhd$. This ensures that this term is a suitable test function in the velocity equation and that the term $(p_h, \nabla \cdot \v_h)$ vanishes.
	
	An essential step in a proof for the hybrid helicity conservation on the continuous level is that the advection term from the Navier--Stokes equations vanishes when tested against $\bm \omega$, i.e., 	$(\u \times \bm \omega, \bm \omega) = 0$. This already requires a complicated discretisation of the advection term. A possible choice could be to approximate $u\times \bm \omega$ by
	\begin{equation}
		\frac{1}{2}\Qd[\u \times [ \nabla \times \Qc \u_h + \Qd^0\tilde{\nabla}_h \times \u_h]].
	\end{equation}
	However, the essence of the conservation proofs is the cancellation of corresponding terms that result from the symmetry in the discretisation. That means also the Lorentz force, the Hall-term and magnetic advection terms have to be discretised in a similar complicated way. We were not able to find an elegant discretisation that does not require the introduction of many additional terms and auxiliary variables.
\end{remark}

Again, to render \eqref{alg:helicity-operator3} computable we introduce auxiliary variables for the projections, yielding: find $(\u_{h}(t), p_h(t), \B_{h}(t), \E_{h}(t), \j_{h}(t), \mathbf{H}_{h}(t),\bm \omega_{h}(t), \mathbf{U}_h(t), \bm \alpha_h(t)) \in$ $\Hhd \times L^2_0(\Omega) \times \Hhd \times [\Hhc]^6 ) $, such that for any $(\v_h, q_h, \C_h, \F_h, \k_h, \bm G_h, \bm \mu_h, \mathbf{V}_h, \bm \beta_h)$ in the same space,
\begin{subequations}\label{alg:helicityudotn}
	\begin{align}\label{alg:helicity-u}
		((\u_{h})_{t}, \v_h)+ (\bm \alpha_h, \v_h), \v_h)+( \nabla \cdot \v_h,  p_{h})&= 0,
		\\
		( \nabla \cdot \u_{h}, q_{h})&=0,
		\\\label{alg:helicity-b}
		((\B_{h})_{t}, \C_h)+(\nabla\times  \E_{h}, \C_h)&=0,\\\label{alg:helicity-j}
		(\j_{h}, \F_h)-(\B_{h},  \nabla\times \F_h)&=0,\\\label{alg:helicity-h1}
		(\bm H_{h}, \bm G_h)-(\B_{h},  \bm G_h)&=0,\\\label{alg:helicity-h2}
		(\bm \omega_{h}, \bm \mu_h)-(\u_{h},  \nabla\times \bm \mu_h)&=0,\\\label{alg:helicity-U2}
		(\mathbf{U}_{h}, \mathbf{V}_h)-(\u_{h},  \mathbf{V}_h)&=0,\\\label{alg:helicity-alpha2}
		(\bm \alpha_{h}, \bm \beta_h)+(\bm \omega_h \times \mathbf{U}_h,  \bm \beta_h)- S (\j_h \times \mathbf{H}_h,  \bm \beta_h)&=0,\\\label{alg:helicity-omega2}
		(\E_{h}, \k_h)-((\RH\,\j_{h}-\mathbf{U}_{h})\times \bm H_{h},  \k_h)&=0.
	\end{align}
\end{subequations}
Now \eqref{alg:helicity-j} gives $\j_{h}=\tilde\nabla_{h}\times \B_{h}$,  \eqref{alg:helicity-h1} gives $\bm H_{h}=\Qc \B_{h}$; \eqref{alg:helicity-h2} gives $\bm \omega_{h}=\tilde{\nabla}_h\times \u_{h}$, \eqref{alg:helicity-U2} gives $\mathbf{U}_{h}=\Qc \u_{h}$ and \eqref{alg:helicity-alpha2} gives $\bm \alpha_{h}= \Qc[(\tilde{\nabla}_h \times \u_h) \times \Qc \u_h] - S \Qc[\j_h \times \Qc \B_h]$.

We use the same time discretisation as in Section \ref{sec:schemeutimesn}; compare also to \cite[Section 6]{gawlik2020} for a detailed proof of the next theorem. The proofs for the Hall-term follow immediately from the continuous proofs of Theorem \ref{thm:energyudotn} and Theorem \ref{thm:magheludotn}.

\begin{theorem}
	The time-discretised version of \eqref{alg:helicityudotn} preserves the energy and magnetic helicity precisely and enforces $\div \B_h=\div \u_h=0$  for all time steps; i.e. for all $ k \geq 0$ there holds
	\begin{align}
	\int_\Omega \mathbf{u}^{k+1}_h  \cdot \u_h^{k+1} + S \mathbf{B}^{k+1}_h  \cdot \B_h^{k+1} \mathrm{d} x &= \int_\Omega  \mathbf{u}^{k}_h  \cdot \u_h^{k} + S \mathbf{B}^{k}_h  \cdot \B_h^{k} \mathrm{d} x,\\	
	\int_\Omega \mathbf{A}^{k+1}_h  \cdot \B_h^{k+1} \mathrm{d} x &= \int_\Omega \mathbf{A}^{k}_h  \cdot \B_h^{k} \mathrm{d} x, \\
	\div \u^k_h = 0,\\
	\div \B^k_h = 0.
\end{align}	

\end{theorem}

\section{An augmented Lagrangian preconditioner for the Hall MHD equations}\label{sec:ALP}
In this section, we try to extend the block preconditioner approach from Section \ref{sec:derivationofblockpreconditioner} to the stationary and time-dependent versions of the Picard and Newton linearisations from  Algorithm \ref{alg:picard-s} and  Algorithm \ref{alg:newton-s}. In each nonlinear step, we have to solve a linear system of the form

\begin{equation}
	\label{eq:matrix_upBEj}
	\begin{bmatrix}
		\mathcal{F} & \BB^\top &  \mathbf{0} & \tilde{\KK} & \KK\\
		\BB & \mathbf{0} & \mathbf{0} & \mathbf{0} & \mathbf{0} \\
		\mathbf{0}& \mathbf{0}& \mathbf{0} & -\AA & \MM \\
		\mathbf{0} & \mathbf{0} & \DD  & \CC & \mathbf{0}\\
		-\GG& \mathbf{0}& -\mathcal{P} & -\tilde{\GG} + \tilde{\NN}  & \LL + \NN
	\end{bmatrix}
	\begin{bmatrix}
		x_{\u_h} \\ x_{p_h} \\ x_{\E_h} \\ x_{\B_h} \\ x_{\j_h}
	\end{bmatrix} =
	\begin{bmatrix}
		R_{\u_h} \\ R_{p_h} \\ R_{\E_h} \\ R_{\B_h} \\ R_{\j_h}
	\end{bmatrix},
\end{equation}
where $x_{\u_h}$, $x_{p_h}$, $x_{\E_h}$, $x_{\B_h}$ and $x_{\j_h}$ are the coefficients of the discretised corrections and $R_{\u_h}$, $R_{p_h}$, $R_{\E_h}$, $R_{\B_h}$ and $R_{\j_h}$ the corresponding nonlinear residuals.
The correspondence between the discrete and continuous operators is illustrated in Table \ref{tab:OperatorsHall}. As before, we have chosen the notation that operators that include a tilde are omitted in the Picard linearisation from Algorithm $\ref{alg:picard-s}$. 

\begin{table}[htb!]
	\centering
	\begin{tabular}{c|c|c}
		\toprule
		\textbf{Discrete} & \textbf{Continuous} & \textbf{Weak form}\\
		\midrule
		$\mathcal{F} \u^n_h$ & $-\frac{1}{\Re} \Delta \u_h^n + \u_h^{n-1}\cdot \nabla \u_h^n + \u_h^n \cdot \nabla \u_h^{n-1}$ & $\frac{1}{\Re}(\nabla\u_h^n, \nabla \v_h) + (\u_h^{n-1}\cdot \nabla \u_h^{n} , \v_h)$  \\
& $-\gamma\nabla \nabla \cdot \u_h^n$ &  $+(\u_h^n\cdot \nabla \u_h^{n-1}, \v_h) + \gamma(\nabla \cdot \u_h^n, \nabla \cdot \v_h)  $  \\
		$\KK \j^n_h$& $-S\, \j^n_h \times \B^{n-1}_h$ & $ -S\, (\j^n_h \times \B^{n-1}_h, \v_h)$ \\
		$\tilde{\KK} \B^n_h$& $-S\, \j^{n-1}_h \times \B^{n}_h$ & $ -S\, (\j^{n-1}_h \times \B^{n}_h, \v_h)$ \\
		$\BB^\top p^n_h$ & $\nabla p^n_h$ & $-(p^n_h, \div \v_h)$ \\
		$\BB \u^n_h$ & $-\div \u^n_h$ & $-(\div \u^n_h, q)$  \\
		$\LL \j^n_h$ & $ \frac{1}{\Rem} \j^n_h$ & $\frac{1}{\Rem} (\j^n_h, \k_h)$ \\
		$\mathcal{P} \E^n_h$ & $\E^n_h$ & $ (\E^n_h, \k_h)$ \\
		$\GG \u^n_h$ & $\u^n_h \times \B^{n-1}_h$ & $(\u^n_h\times\B^{n-1}_h, \k_h)$  \\
		$\tilde{\GG} \B^{n}_h$ & $\u^{n-1}_h \times \B^{n}_h$ & $(\u^{n-1}_h\times\B^{n}_h, \k_h)$  \\
		$\NN \j^n_h$ & $\RH\, \j^n_h \times \B^{n-1}_h$ & $\RH\,(\j^n_h\times\B^{n-1}_h, \k_h)$  \\
		$\tilde{\NN} \B^{n}_h$ & $\RH\, \j^{n-1}_h \times \B^{n}_h$ & $\RH\,(\j^{n-1}_h\times\B^{n}_h, \k_h)$  \\
		$\DD \E^n_h$ & $ \nabla \times \E^n_h$ & $ (\nabla \times \E^n_h, \C_h)$ \\
		$\CC \B^n_h$ & $-\nabla \nabla \cdot \B^n_h$ & $(\nabla \cdot \B^n_h,\nabla \cdot \C_h)$  \\
		$\MM \j^n_h$ & $\j^n_h$ & $ (\j^n_h, \F_h)$ \\
		$\A \B^n_h$ & $ \nabla \times \B^n_h$ & $ (\B^n_h, \nabla \times  \F_h)$ \\
		\bottomrule
	\end{tabular}
	\caption{Overview of operators. As before, the stationary formulation corresponds to $\eta=0$ and the transient formulation for implicit Euler to $\eta=1$.}
	\label{tab:OperatorsHall}
\end{table}
The following preconditioning approach is similar to one developed in Section \ref{sec:derivationofblockpreconditioner} for the standard incompressible resistive MHD equations. The main idea is to do a Schur complement approximation which separates the hydrodynamic and electromagnetic unknowns and then to apply parameter-robust multigrid methods to the different subproblems.

We start by simplifying the outer Schur complement that eliminates the $(\u_h, p_h)$ block given by
\begin{equation}\label{eq:outerSchurCompup}
	\mathcal{S}^{(\u_h, p_h)} =
	\begin{bmatrix}
		\mathbf{0} & -\AA & \MM\\
		\DD & \CC & \mathbf{0}\\
		-\mathcal{P}& -\tilde{\GG} + \tilde{\NN} & \LL + \NN
	\end{bmatrix}
	-
	\begin{bmatrix}
		\zerov & \zerov\\
		\zerov & \zerov\\
		-\GG & \zerov
	\end{bmatrix}
	\begin{bmatrix}
		\FF & \BB^\top \\
		\BB & \mathbf{0}
	\end{bmatrix}^{-1}
	\begin{bmatrix}
		\zerov & \tilde{\KK} & \KK \\
		\zerov & \zerov & \zerov
	\end{bmatrix}.
\end{equation}
This order of elimination worked the best for the standard MHD model as we described in Chapter \ref{chap:2} and hence we choose the same order here.
Applying the identity
\begin{equation}
	\label{eq:matrixinvers}
	\begin{bmatrix}
		A & B \\
		C & D
	\end{bmatrix}^{-1} = \begin{bmatrix}
		A^{-1} + A^{-1} B (D - C A^{-1} B)^{-1} C A^{-1} &
		-A^{-1}B(D-C A^{-1} B)^{-1} \\
		- (D - C A^{-1} B)^{-1} C A^{-1}&
		(D-C A^{-1}B)^{-1}
	\end{bmatrix}
\end{equation}
for non-singular matrices $A$ and $D-C A^{-1} B$ to the $(\u_h, p_h)$ block results in
\begin{equation}\label{eq:outerSchurCompup2}
	\mathcal{S}^{(\u_h, p_h)} =
	\begin{bmatrix}
		\mathbf{0} & -\AA & \MM \\
		\DD & \CC & \mathbf{0} \\
		-\mathcal{P}& -\tilde{\GG} + \tilde{\NN} + \GG \mathcal{S}^{-1}_{1,1} \tilde{K} & \LL + \NN + \GG \mathcal{S}^{-1}_{1,1} \KK 
	\end{bmatrix}
\end{equation}
with 
\begin{equation}
	\mathcal{S}^{-1}_{1,1} = \FF^{-1} - \FF^{-1}  \BB^\top (\BB \FF^{-1} \BB^T)^{-1} \BB \FF^{-1}.
\end{equation}
Note that the magnitude of the matrices $\GG \mathcal{S}^{-1}_{1,1} \tilde{K}$ and $\GG \mathcal{S}^{-1}_{1,1} \KK $ is approximately a factor of $\mathcal{O}(h^2)$ smaller than of the other matrices at the corresponding entries. Therefore, a good approximation for a reasonably refined mesh and moderate coupling numbers $S$ is given by
\begin{equation}\label{eq:outerSchurCompupapprox}
	\tilde{\mathcal{S}}^{(\u_h, p_h)} =
	\begin{bmatrix}
		\mathbf{0} & -\AA & \MM \\
        \DD & \CC & \mathbf{0} \\
        -\mathcal{P}& -\tilde{\GG} + \tilde{\NN} & \LL + \NN 
	\end{bmatrix}.
\end{equation}

We treat the  hydrodynamic block 
\begin{equation}\label{eq:Schurcomphycdro}
	\mathcal{M}_{NS}=	
	\begin{bmatrix}
		\FF & \BB^\top \\
		\BB & \mathbf{0}
	\end{bmatrix}
\end{equation}
as before in Section \ref{sec:solverforschurcomp}. Therefore, we also add the augmented Lagrangian term $\gamma (\nabla \cdot \u^n_h, \nabla \cdot \v^n_h)$ to the velocity equation with a large $\gamma$ to gain control over the Schur complement of \eqref{eq:Schurcomphycdro}. Moreover, we use the $\Hd$-conforming discretisation for $\u_h$ from \eqref{eq:hidvl2form}  to allow the use of parameter-robust multigrid methods that can deal with the non-trivial kernels of the occurring semi-definite terms.

We found that applying the same parameter-robust multigrid methods monolithically to the Schur complement approximation $\tilde{\mathcal{S}}^{(\u_h, p_h)}$ shows good results for the three dimensional lid-driven cavity problem as long as $\Rem$, $\S$ and $\RH$ are not chosen too high at the same time. However, as we explain below, this solver does not work well for the 2.5D formulation tested on island coalescence  problem if $\RH$ is chosen higher than 0.01. Therefore, this solver needs to be investigated further and we applied a direct solver to this block for the 2.5D formulation.

Alternatively, one could do a further Schur complement approximation of $\tilde{\mathcal{S}}^{(\u_h, p_h)}$ with one of the blockings $(\E, \B)-\j$, $(\E,\j) - \B$ or $(\B, \j)- \E$. However, the further Schur complement approximation is not straight-forward and would include the use of non-conforming discretisations by suitable discontinuous Galerkin methods. We report iteration numbers for the monolithic approach in the next section.

For completeness, we also outline the block structure of the 2.5D formulation introduced in Section \ref{sec:2.5Dform}. We use $\eta \in \{0,1\}$ to distinguish between the stationary $(\eta=0)$ and transient $(\eta=1)$ cases. The hydrodynamic block 	$\begin{bmatrix}
	\FF & \BB^\top \\
	\BB & \mathbf{0}
\end{bmatrix}$ arises now as the discretisation of the forms 
\begin{equation} 
	\begin{bmatrix}
		A_1 & \mathbf{0} & (p_h, \nabla \cdot \tilde{\v}_h)\\
		\mathbf{0} & A_2 &  \mathbf{0} \\
		(\nabla \cdot \tilde{\u}_h, q_h) &  \mathbf{0} &  \mathbf{0}
	\end{bmatrix}
\end{equation}
with
\begin{align*}
	A_1 &= 	\frac{\eta}{\Delta t} (\tilde{\u}^n_h, \tilde{\v}_h) -\frac{1}{\Re} (\nabla \tilde{\u}^n_h, \nabla \tilde{\v}_h) + 	( (\tilde{\u}^n_h \cdot \tilde{\nabla}) \tilde{\u}^{n-1}_h, \tilde{\v}_h) + ((\tilde{\u}^{n-1}_h \cdot \tilde{\nabla}) \tilde{\u}^{n}_h, \tilde{\v}_h)\\
	&  + \gamma (\nabla \cdot \tilde{\u}^n_h, \nabla \cdot \tilde{\v}^n_h), \\
	A_2 &= \frac{\eta}{\Delta t} (u^n_{3h}, v_{3h}) -\frac{1}{\Re} (\nabla u^n_{3h}, \nabla v_{3h}) + 	( (\tilde{\u}^n_h \cdot \tilde{\nabla}) u_{3h}^{n-1}, v_{3h}) + ((\tilde{\u}^{n-1}_h \cdot \tilde{\nabla}) u^n_{3h}, v_{3h}).
\end{align*}

Furthermore,
\begin{equation}
	\begin{bmatrix}
		\mathbf{0} & \tilde{\KK} & \KK \\
		\mathbf{0} & \mathbf{0} & \mathbf{0} \\
	\end{bmatrix}, 
	\begin{bmatrix}
		\mathbf{0} & \mathbf{0} \\
		\mathbf{0} & \mathbf{0} \\
		-\GG & \mathbf{0}
	\end{bmatrix}
	\text{ and }
	\begin{bmatrix}
		\mathbf{0} & -\AA & \MM \\
		\DD & \CC & \mathbf{0} \\
		-\mathcal{P}& -\tilde{\GG} + \tilde{\NN} & \LL + \NN 
	\end{bmatrix}
\end{equation}
correspond to 
\begin{equation} 
	\resizebox{\textwidth}{!}{$
		\begin{bmatrix}
			\mathbf{0}  &  \mathbf{0}  & S(\tilde{\B}^n_h\times j^{n-1}_{3h}, \tilde{\v}_h) &  -S(\tilde{\j}^{n-1}_h \times \B^{n}_{3h}, \tilde{\v}_h) &  -S (\tilde{\j}^n_h \times \B^{n-1}_{3h}, \tilde{\v}_h)& S(\tilde{\B}^{n-1}_h\times j^{n}_{3h}, \tilde{\v}_h) \\
			\mathbf{0} & \mathbf{0} & -S (\tilde{\j}^{n-1}_{h}\times \tilde{\B}^{n}_h, v_{3h}) & \mathbf{0} &  -S (\tilde{\j}^{n}_{h}\times \tilde{\B}^{n-1}_h, v_{3h}) & \mathbf{0} \\
			\mathbf{0} & \mathbf{0} & \mathbf{0} & \mathbf{0} &  \mathbf{0} & \mathbf{0} \\
		\end{bmatrix}$},
\end{equation}

\begin{equation} 
	\begin{bmatrix}
		\mathbf{0} & \mathbf{0} & \mathbf{0}\\
		\mathbf{0} & \mathbf{0} & \mathbf{0}\\
		\mathbf{0} & \mathbf{0} & \mathbf{0}\\
		-(\tilde{\u}^n_h \times B^{n-1}_{3h}, \k_h) & (\tilde{\B}^{n}_h \times u^{n-1}_{3h}, \k_h) & \mathbf{0}\\
		-(\tilde{\u}^n_h \times \tilde{\B}^{n-1}_h,k_{3h})& \mathbf{0} & \mathbf{0}
	\end{bmatrix}
\end{equation}

and 

\begin{equation}
	\resizebox{\textwidth}{!}{$
		\begin{bmatrix}
			\mathbf{0} & \mathbf{0} &\mathbf{0} & -(B^n_{3h}, \scurl \tilde{\F}_h) &  (\tilde{\j}^n_h, \tilde{\F}_h)&\mathbf{0} \\
			\mathbf{0} & 	\mathbf{0} & -(\tilde{\B}^n_h, \vcurl F_{3h}) &	\mathbf{0} & 	\mathbf{0} &	(j^n_{3h}, F_{3h}) \\
			\mathbf{0} & (\vcurl E^n_{3h} , \tilde{\C}_h) & \substack{\frac{\eta}{\Delta t} (\tilde{\B}^n_h, \tilde{\C}_h)\\+ \frac{1}{\Rem} (\nabla \cdot \tilde{\B}^n_h, \nabla \cdot \tilde{\C}_h)}& \mathbf{0} & \mathbf{0} & \mathbf{0} \\ 
			(\scurl \tilde{\E}^n_h, C_{3h}) & \mathbf{0}& \mathbf{0}&  \frac{\eta}{\Delta t} (B^n_{3h}, C_{3h}) &  \mathbf{0}& \mathbf{0}\\
			-(\tilde{\E}^n_h, \k_h) & \mathbf{0} &\substack{(\tilde{\B}^n_h \times u^{n-1}_{3h}, \k_h) \\- \RH (\tilde{\B}^n_h \times j^{n-1}_{3h}, \k_h)} & 
			\substack{(\tilde{\u}^{n-1}_h \times B^n_{3h}, \k_h) \\+ \RH (\tilde{\j}^{n-1}_h \times B^n_{3h},\k_h)} &  \substack{\frac{1}{\Rem} (\tilde{\j}^n_h, \k_h) \\ + \RH (\tilde{\j}^n_h \times B^{n-1}_{3h}, \k_h)} & -\RH (\tilde{\B}^{n-1}_h \times j^n_{3h}, \k_h)  \\ 
			\mathbf{0} & -(E^n_{3h}, k_{3h}) & \substack{- (\tilde{\u}^{n-1}_h \times \tilde{\B}^n_h, k_{3h}) \\+ \RH (\tilde{\j}^n_h  \times \tilde{\B}^{n-1}_h, k_{3h})} &\mathbf{0}  & \RH(\tilde{\j}^n_h \times \tilde{\B}^{n-1}_h, k_{3h}) & \frac{1}{\Rem} (j^n_{3h}, k_{3h})
		\end{bmatrix}$}.
\end{equation}
Our numerical experiments suggest that the same outer Schur complement approximation (now applied to the blocking $(\tilde{\u}_h, u_{3h}, p_h)$ and $(\tilde{\B}_h,B_{3h},\tilde{\E}_h,E_{3h},\tilde{\j}_h,j_{3h})$) still works well for the 2.5D case. However, we observe poor performance of the monolithic multigrid method applied to this block for an island coalescence and $\RH>0.01$. Robust solvers for this inner problem require further investigation and, as mentioned before, we apply a direct solver to this block in the 2.5D numerical results in the next section.

\section{Numerical results}\label{sec:numericalresultsHall}
As before, the numerical results were implemented in Firedrake.
Moreover, we replaced the Laplace term $-\Delta \u$ in our implementation by $-2\nabla \cdot \varepsilon(\u)$, where $\varepsilon(\u) := 1/2(\nabla \u + \nabla \u^\top)$ denotes the symmetric gradient. This allows us to also consider alternative boundary conditions 
\begin{equation}
	\u = \mathbf{0} \text{ on } \Gamma_D , \qquad \frac{2}{\Re}  \varepsilon( \u )\cdot \n = p \n \text{ on } \Gamma_N
\end{equation}
with $\Gamma_D\cup \Gamma_N = \partial \Omega$.
Note that both formulations are equivalent for the boundary conditions $\u=\mathbf{0}$ on $\partial \Omega$ which we consider in this work \cite[Chap.~15]{Quarteroni2017}.

\subsection{Verification and convergence order}\label{sec:verificationandconvergenceorder}
In the first example, we consider the method of manufactured solutions for a smooth given solution to verify the implementation of our solver and report convergence rates. We employ the Picard iteration for the stationary problem from Algorithm \ref{alg:picard-s}.
The right-hand sides and boundary conditions are calculated corresponding to the analytical solution
\begin{gather}
	\begin{split}
		\u(x,y,z) = \begin{pmatrix}
			\cos(y) \\ \sin(z) \\ \exp(x)
		\end{pmatrix}, \quad
		p(x,y,z) = y \sin(x) \exp(z), \quad 
		\B(x,y,z)  = \begin{pmatrix}
			\sin(z) \\ \sin(x) \\ \cos(y)
		\end{pmatrix},\\
		\E(x,y,z) = \begin{pmatrix}
			x \sin(x) \\ \exp(y) \\ z^3
		\end{pmatrix}, \quad
		\j(x,y,z) = \begin{pmatrix}
			\cos(y z) \\ \exp(xz) \\ \sinh(x)
		\end{pmatrix}.  \qquad \qquad \qquad \quad
	\end{split}
\end{gather}
We used second order $\mathbb{BDM}$-elements for $\u_h$, second order $\mathbb{NED}1$-elements for $\E_h$ and $\j_h$, second order $\mathbb{RT}$-elements for $\B_h$ and first order $\mathbb{DG}$-elements for $p_h$ on $\Omega=(0,1)^3$.
Based on the standard error estimates for these spaces, one would expect third order convergence in the $L^2$-norm for $\u_h$ and  second order convergence for $p_h$, $\B_h$, $\E_h$ and $\j_h$. This is numerically verified by Table \ref{tab:convOrder}.
\begin{table}[!htb]
	\begin{centering}
		\resizebox{\textwidth}{!}{
		\begin{tabular}{c |c c|c c|c c|c c|c c} 
			\toprule
			h & $\|\u-\u_h\|_0$ & rate & $\|p-p_h\|_0$ & rate & $\|\B-\B_h\|_0$ & rate & $\|\E-\E_h\|_0$ & rate & $\|\j-\j_h\|_0$ & rate  \\
			\midrule
			1/4 & 3.08E-04 & - & 3.52E-02 & - & 2.44E-03 & - & 9.57E-03 & - & 6.77E-03 & - \\
			1/8 & 4.50E-05 & 2.78 & 6.58E-03 & 2.42 & 6.04E-04 & 2.02 & 2.50E-03 & 1.93 & 1.79E-03 & 1.92 \\
			1/16 & 5.99E-06 & 2.91 & 1.36E-03 & 2.27 & 1.50E-04 & 2.01 & 6.32E-04 & 1.99 & 4.53E-04 & 1.98 \\
			1/32 &7.72E-07 & 2.96 & 2.99E-04 & 2.19 & 3.74E-05 & 2.00 & 1.58E-04 & 2.00 & 1.14E-04 & 1.99 \\
			\bottomrule
		\end{tabular}}
		\caption{$L^2$-error and convergence order.}
		\label{tab:convOrder}
	\end{centering}
\end{table}

\subsection{Lid-driven cavity problem}
As in Section \ref{sec:numericalresults}, we consider a lid-driven cavity problem for a background magnetic field $\B_0=(0,1,0)^\top$ which determines the boundary conditions $\B\cdot \n = \B_0\cdot \n$ on $\partial \Omega$ and set $\f = \mathbf{0}$ for $\Omega=(-0.5, 0.5)^3$. The boundary condition $\u=(1,0,0)^\top$ is imposed at the boundary $y=0.5$ and homogeneous boundary conditions elsewhere. 

Since we consider non-homogeneous boundary conditions in this problem the boundary conditions for $\E$ and $\j$ have to be chosen in a compatible way, which we derive in the following. 
From \eqref{eq:HallMHDj} we can deduce the necessary condition that 
\begin{equation}\label{eq:bcscompatibility}
	\Reminv \j \times \n = \E \times \n + (\u \times \B) \times \n - \RH (\j \times \B) \times \n
\end{equation}
has to hold on $\partial \Omega$. 

On a face that does not correspond to $y=0.5$, we have $\u = (0,0,0)^\top$. Then it is clear that \eqref{eq:bcscompatibility} is fulfilled if we choose $\E\times\n=\j\times\n=\mathbf{0}$ on these faces.

On the face $y=0.5$, we have that $\n = (0,1,0)^\top$ and hence \eqref{eq:bcscompatibility} simplifies to
\begin{equation}
	\begin{cases}
		&	\Reminv  j_3 = -E_3 -1 + \RH j_1,\\
		&	\Reminv  j_2 = E_2,\\
		&	\Reminv  j_1 = E_1 + \RH j_3.
	\end{cases}
\end{equation}
If we choose $\E\times\n=\mathbf{0}$ it follows that 
\begin{equation}
	\j \times \n =\frac{1}{\Reminv + \Rem \RH^2} 
	\begin{pmatrix}
		\Rem \RH \\ 0 \\ 1
	\end{pmatrix} 
	\times \n.
\end{equation}

In Table \ref{tab:ldcstatRERHall}, we present iteration numbers for the Picard and Newton linearisations for the stationary version of the lid-driven cavity problem. Note that we applied our scalable solver with a monolithic multigrid method for the Schur complement approximation here. The direct solver that we mentioned earlier is only used for 2.5D results. We have used the same elements for $\u_h$, $\B_h$, $\E_h$ and $\j_h$ and $p_h$ as in the previous example. Moreover, we have used a coarse mesh of $6\times6\times6$ cells and 3 levels of refinement for the multigrid method resulting in an $48\times48\times48$ mesh with 29.2 million DoFs. One can observe good robustness in the reported ranges of $\RH$ for both linearisations. The Newton linearisation shows slightly better non-linear convergence, while the linear iterations are slightly smaller in most cases for the Picard iteration.

Table \ref{tab:ldctimeRERHall} shows the corresponding results for the time-dependent version of the lid-driven cavity problem. Here, we have chosen a time step of $\Delta t=0.01$ and iterated until the final time of $T=0.1$. We iterated some of the cases until the final time of $T=1.0$ to confirm that the reported iteration numbers remain representative for longer final times. We have chosen the L-stable BDF2 method for the time-discretisation where the first time step was computed by Crank-Nicolson. We observe good robustness in both the nonlinear and linear iteration numbers for this problem.

\begin{table}[htbp!]
	\centering
	\begin{tabular}{r|ccc|ccc}
		\toprule
		& \multicolumn{3}{c|}{Picard} & \multicolumn{3}{c}{Newton} \\
		\midrule
		$\RH\backslash\Re$  &1 &     100 &    1,000 &1 & 100&     1,000  \\
		\midrule
		0.0 & ( 4) 4.8 & ( 4) 5.5 & ( 4)10.0 & ( 3) 6.0 & ( 4) 4.3 & ( 4) 8.8 \\
		0.1 & ( 4) 5.0 & ( 4) 4.8 & ( 4)10.0 & ( 3) 6.0 & ( 4) 4.3 & ( 4) 9.3 \\
		1.0 & ( 4) 5.3 & ( 4) 4.5 & ( 5)10.2 & ( 3) 5.0 & ( 4) 4.3& ( 4) 12.0 \\

		\bottomrule	
	\end{tabular}
	\caption{Iteration counts for the stationary lid-driven cavity problem. The entries of the table correspond to: (Number of nonlinear iterations) Average number of linear iterations per nonlinear step.\label{tab:ldcstatRERHall}\newline}
	
	\begin{tabular}{r|ccc|ccc}
		\toprule
		& \multicolumn{3}{c|}{Picard} & \multicolumn{3}{c}{Newton} \\
		\midrule
		$\RH\backslash\Re$  &1 &     1,000 &    10,000 & 1 & 1,000 & 10,000 \\
		\midrule
		0.0 & (3.0) 5.6 & (3.1) 2.2 & (3.2) 2.0 & (2.1) 7.5 & (3.1)  2.2 & (3.2)  2.0 \\ 
		0.1 & (3.0) 5.6 & (3.1) 2.2 & (3.2) 2.0 & (2.1) 7.5 & (3.1)  2.2 & (3.2) 2.0 \\ 
		1.0 & (3.0) 5.8 & (3.1) 2.2 & (3.2) 2.0 & (2.2) 7.3 & (3.1)  2.2 & (3.2) 2.0\\ 
		
		\bottomrule	
	\end{tabular}
	
	\caption{Iteration counts for the time-dependent lid-driven cavity problem.\label{tab:ldctimeRERHall}}
	
\end{table}

\newcolumntype{C}{ >{\centering\arraybackslash} m{3.0cm} }
\newcolumntype{D}{ >{\centering\arraybackslash} m{2cm} }

\begin{figure}[htbp!]
	\centering
	\begin{tabular}{|D |C |C |C |C |}
		\hline
		$\Rem = 10$ &	\includegraphics[width=2.5cm]{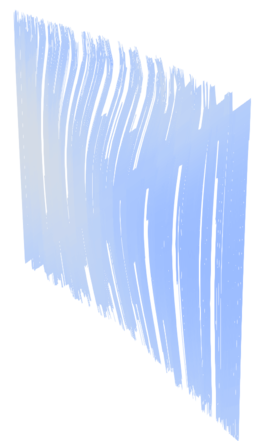} &
		\includegraphics[width=2.5cm]{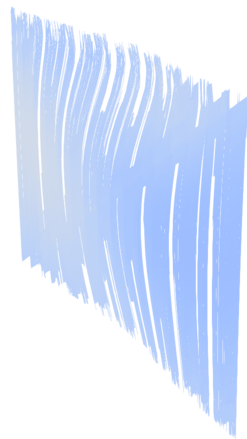} &
		\includegraphics[width=2.5cm]{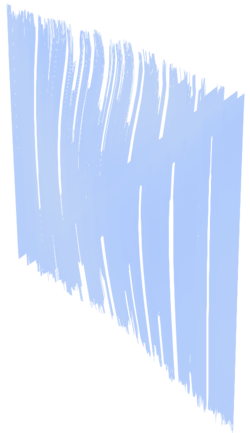} &
		\includegraphics[width=2.5cm]{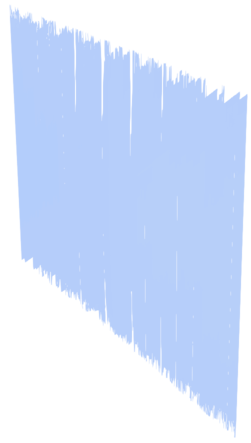} \\
		\hline
		$\Rem = 50$ &	\includegraphics[width=2.5cm]{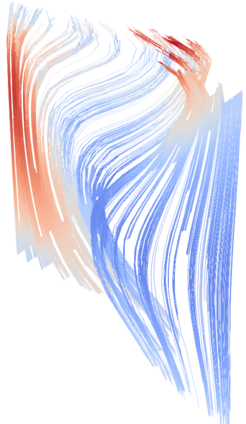} &
		\includegraphics[width=2.5cm]{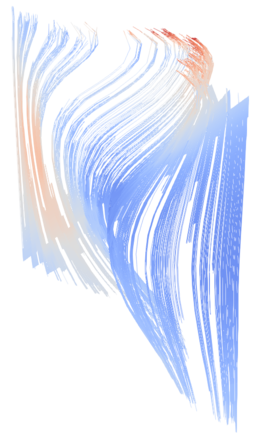} &
		\includegraphics[width=2.5cm]{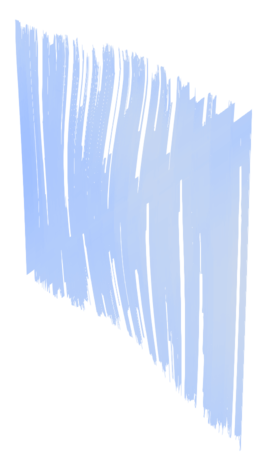} &
		\includegraphics[width=2.5cm]{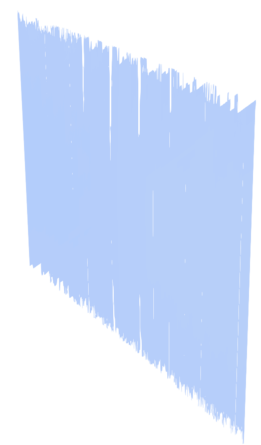} \\
		\hline
		$\Rem = 100$ &	\includegraphics[width=2.5cm]{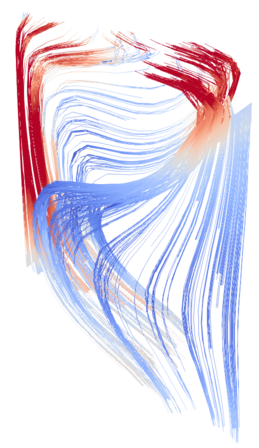} &
		\includegraphics[width=2.5cm]{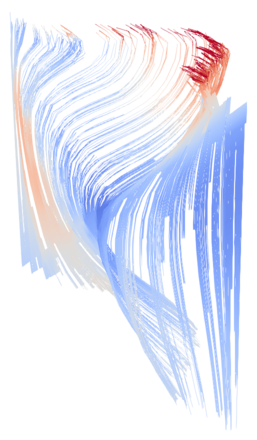} &
		\includegraphics[width=2.5cm]{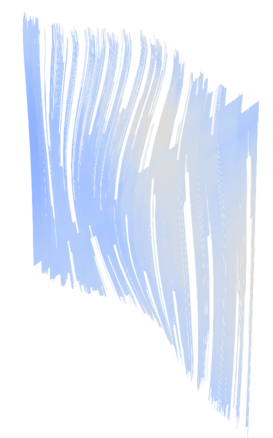} &
		\includegraphics[width=2.5cm]{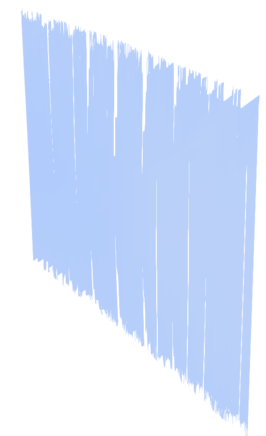} \\
		\hline
		&  $\RH = 0.0$ & $\RH = 0.01$ & $\RH = 0.1$ & $ \RH = 1.0$\\
		\hline
	\end{tabular}
	\caption{Streamlines of the magnetic field for the stationary lid-driven cavity problem for different values of $\Rem$ and $\RH$.\label{fig:StreamlinesldcHall}}
\end{figure}

Figure \ref{fig:Streamlinesldc} shows plots of the magnetic field for different values of $\Rem$ and $\RH$. For $\RH=0$ one can nicely observe the physical phenomenon that for the standard MHD equations the magnetic fields lines tend to be advected by the fluid flow the higher $\Rem$ is chosen. For increasing $\RH$ one can see that this effect is damped until for $\RH=1$, where the influence of the fluid flow is negligible and the magnetic field is close to the background magnetic field in the direction of $(0,1,0)^\top$. 

\subsection{Test of conservative scheme for $\u \times\n = \mathbf{0}$ }
In this section, we want to numerically verify our results from Section \ref{sec:schemeutimesn} for the boundary conditions $\u \times \n = \mathbf{0}$. Here, we used $\Omega=[0,1]^3$ and a mesh of $12\times12\times12$ cells. We chose the interpolant of the following functions as the initial conditions 
\begin{align}
		\u^0(x,y,z) &= \begin{pmatrix}
			-\sin(\pi(x-0.5))\cos(\pi(y-0.5))z(z-1) \\ 
			\cos(\pi(x-0.5))\sin(\pi(y-0.5))z(z-1) \\ 
			0
		\end{pmatrix},\\
		\B^0(x,y,z)  &= \begin{pmatrix}
			-\sin(\pi x)\cos(\pi y) \\ 
			\cos(\pi x) \sin(\pi y) \\ 
			0
		\end{pmatrix},
\end{align}
which satisfy the boundary conditions $\u^0\times \n = \mathbf{0}$,  $\B^0\times\n=\mathbf{0}$ and the constraints $\nabla \cdot \u^0=\nabla \cdot \B^0=0$. 
Remember that the interpolant of divergence-free functions is still divergence-free for $\mathbb{RT}$ and $\mathbb{BDM}$ elements, see \eqref{eq:InterpolationPreserveRT}. We enforce this property in our implementation by using a sufficiently high quadrature degree in the evaluation of the degrees of freedom for the $\mathbb{RT}$ and $\mathbb{BDM}$ elements; see earlier in this thesis in Section \ref{sec:interp-boundary-data}. 
Here, we discretise $\u$ with $\mathbb{NED}1$-elements and $p$ with $\mathbb{CG}_1$-elements. 

For the computation of the magnetic helicity we determine a discrete vector-potential such that $\nabla \times \mathbf{A}_h = \B_h$ by the system
\begin{equation}
	\left(\nabla \times \mathbf{A}_h, \nabla \times \k_h\right) = \left(\B_h, \nabla \times \k_h \right) \quad \forall\ \k_h \in \Hhc.
\end{equation}
We solve this singular system with GMRES preconditioned by ILU, which is known to be convergent if the problem is consistent \cite{Ipsen1998}. 

Although, the scheme \eqref{alg:helicityutimesn} contains multiple auxiliary variables, it can be solved efficiently with a fixed point iteration \cite[Section 4]{hu2021helicity}. For the  time step from $t_k$ to $t_{k+1}$ we compute iterative solutions $\left( \u^{(k+1,j)}_{h}, P^{(k+\frac{1}{2},j)}_h,  \B^{(k+1,j)}_{h}, \E^{(k+\frac{1}{2},j)}_{h}, \j^{(k+\frac{1}{2},j)}_{h}, \mathbf{H}^{(k+\frac{1}{2},j)}_{h},\bm \omega^{(k+\frac{1}{2},j)}_{h}\right)$ until the stopping criterion
\begin{equation}
	\frac{\|\u^{(k+1,j+1)}_{h} - \u^{(k+1,j)}_{h}\|}{\|\u^{(k+1,j)}_{h}\|} + \frac{\|\B^{(k+1,j+1)}_{h} - \B^{(k+1,j)}_{h}\|}{\|\B^{(k+1,j)}_{h}\|} < \text{TOL}
\end{equation}
is satisfied for a given tolerance $\text{TOL}$. We initialise the iteration with the values from time step $k$ and first determine the updates  $\left (\E^{(k+\frac{1}{2},j+1)}_{h}, \j^{(k+\frac{1}{2},j+1)}_{h},\mathbf{H}^{(k+\frac{1}{2},j+1)}_{h},\bm \omega^{(k+\frac{1}{2},j+1)}_{h}\right)$ by solving \eqref{alg:helicity-cross-j2} - \eqref{alg:helicity-cross-w2}
with right-hand sides of the level $j$. Then, we update the velocity and pressure by
\begin{subequations}
	\begin{alignat}{2}
		\frac{1}{\Delta t} (\u^{(k+1,j+1)}_{h}, \v_h) + (\nabla P^{(k+\frac{1}{2},j+1)}_h, \v_h) &= (\F_h, \v_h)  & \quad \forall\  \v_h \in \Hhd,\\
		(\nabla Q_h, \u^{(k+1,j+1)}_h) &= 0 &\quad \forall \ Q_h \in H^1_0(\Omega), 
	\end{alignat}
\end{subequations}
with
\begin{equation}
	\F_h = \frac{1}{\Delta t} \u^{k}_h + S \j^{(k+\frac{1}{2},j+1)}_h \times \mathbf{H}^{(k+\frac{1}{2},j+1)}_h + \frac{1}{2}\left(\u^{(k+1,j)}_h + \u^{k}_h\right)\times \bm \omega^{(k+\frac{1}{2},j+1)}_h.
\end{equation}
The magnetic field is updated by solving
\begin{equation}
	\frac{1}{\Delta t} (\B^{(k+1,j+1)}_{h}, \C_h) = \frac{1}{\Delta t} (\B^{k}_{h}, \C_h)- (\nabla \times \E^{(k+\frac{1}{2},j+1)}_h, \C_h)  \quad \forall\  \C_h \in \Hhd.
\end{equation}

Figure \ref{fig:uHcurl} shows plots of the different conserved quantities for $\Re=\Rem=\infty$ and $\RH=0.5$. One can clearly see that the energy and hybrid helicity remain constant over time, while the cross and fluid helicity are not conserved. These are the observations we expected from the theory in Section \ref{sec:schemeutimesn}. Moreover, $\operatorname{div}(\B_h)$ and the magnetic helicity also show good preservation with small oscillations on the machine precision level.

In Figure \ref{fig:differentRe}, we show plots of the energy and hybrid helicity for $\RH=0.1$ and multiple finite values of $\Re$ and $\Rem$. This test confirms that both quantities are indeed only conserved in the ideal limit of $\Re=\Rem=\infty$.

Finally, Figure \ref{fig:differentRH} compares the cross and hybrid helicity for different values of $\RH$ in the ideal limit of $\Re=\Rem=\infty$. One can observe that the cross helicity is indeed only conserved for $\RH=0$, which corresponds to the standard MHD equations. On the other hand, the hybrid helicity is conserved for all tested values of $\RH$. Note that the hybrid helicity corresponds for $\RH=0$ to the magnetic helicity.

\begin{figure}[htbp!]
	\centering
	\includegraphics[width=12cm]{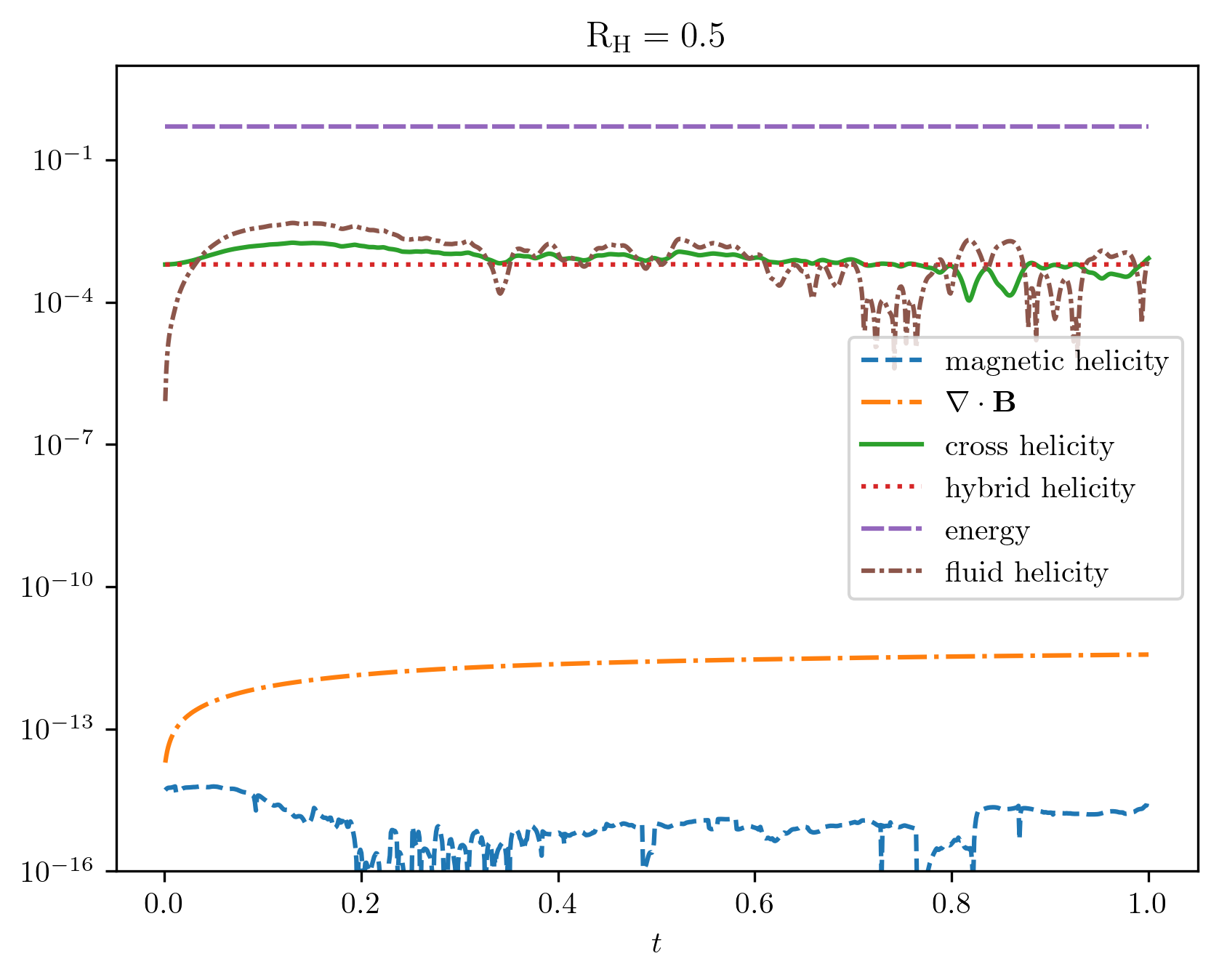} 
	\caption{Plot of conserved quantities for $\u \times \n = \mathbf{0}$ in the ideal limit.}
	\label{fig:uHcurl}
\end{figure}

\begin{figure}[htbp!]
	\centering
	\begin{tabular}{cc}
		\includegraphics[width=7.5cm]{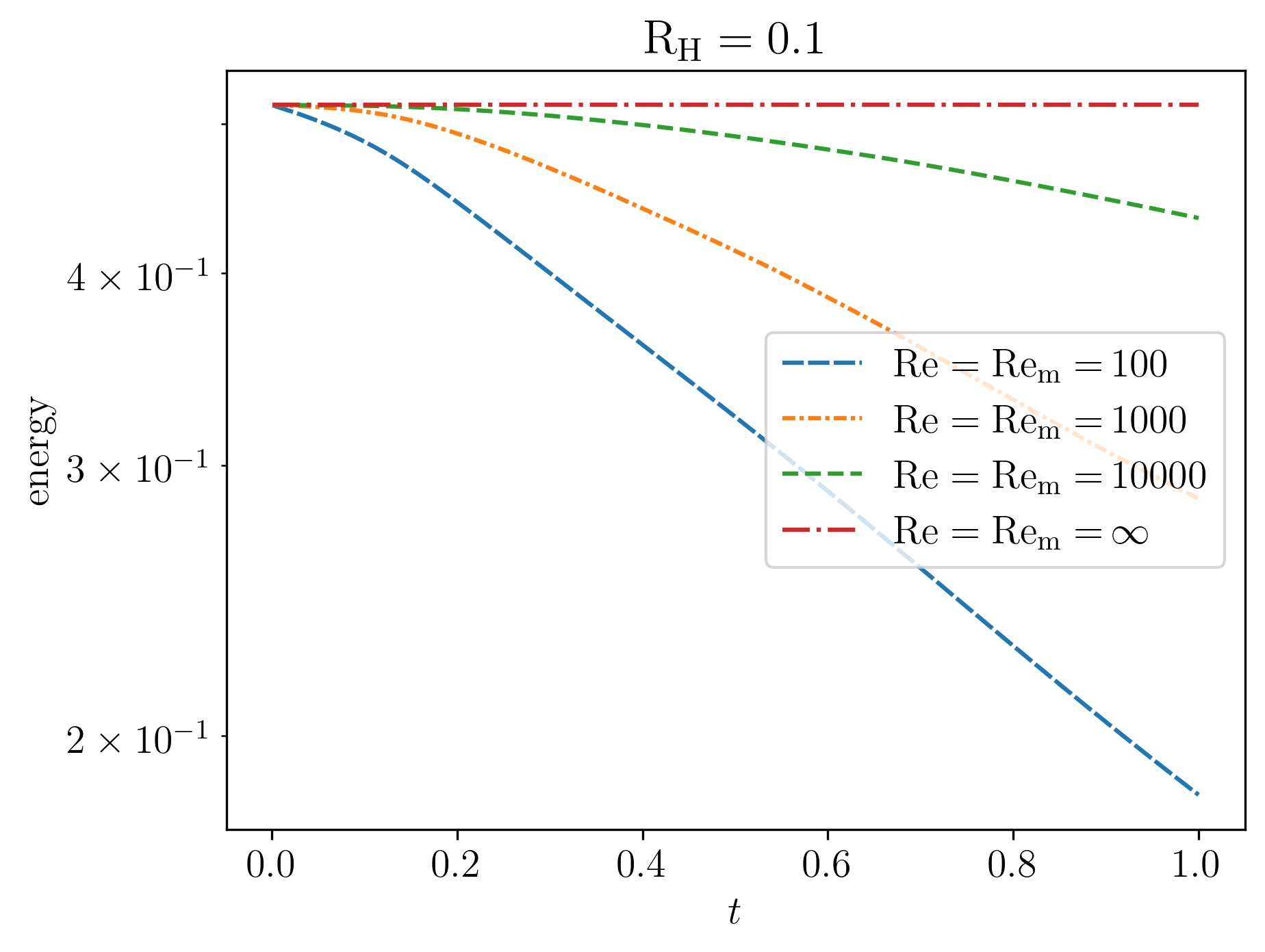} &
		\includegraphics[width=7.5cm]{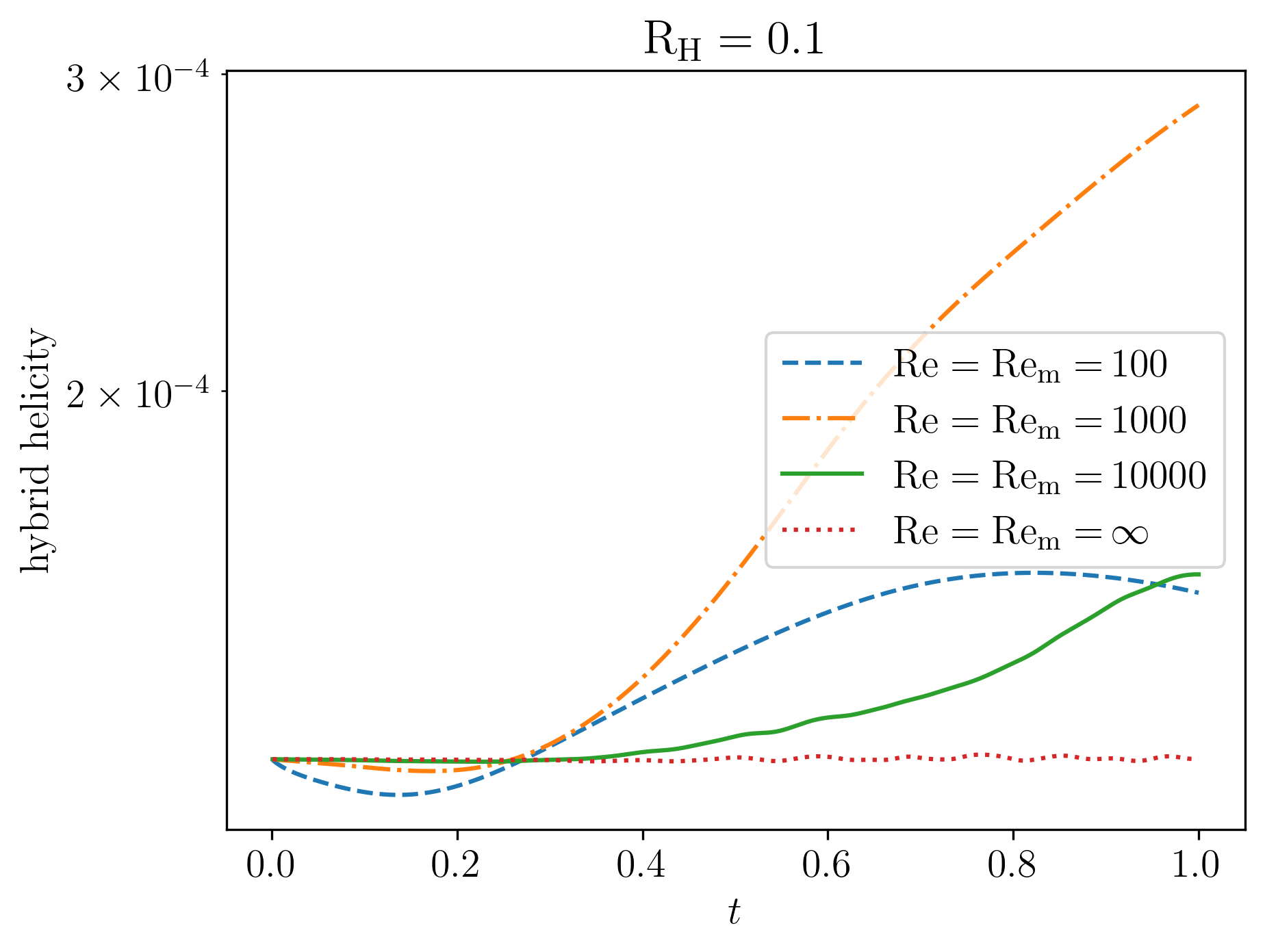}
	\end{tabular}
	\caption{Plots of the energy (left) and hybrid helicity (right) for different values of $\Re$ and $\Rem$ for $\u\times \n = \mathbf{0}$.}
	\label{fig:differentRe}
	
\end{figure}

\begin{figure}[htbp!]
	\centering
	\begin{tabular}{cc}
		\includegraphics[width=7.5cm]{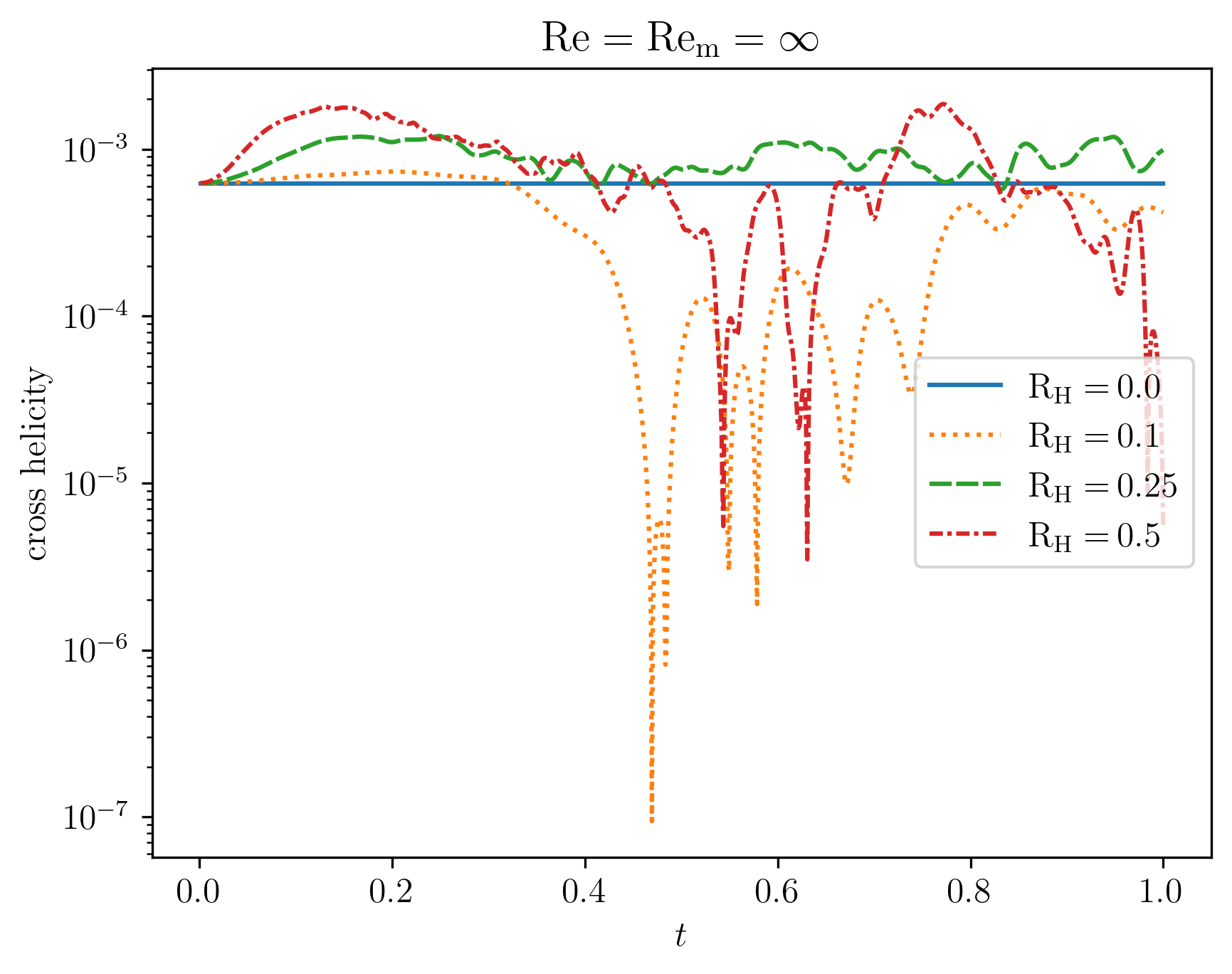} &
		\includegraphics[width=7.5cm]{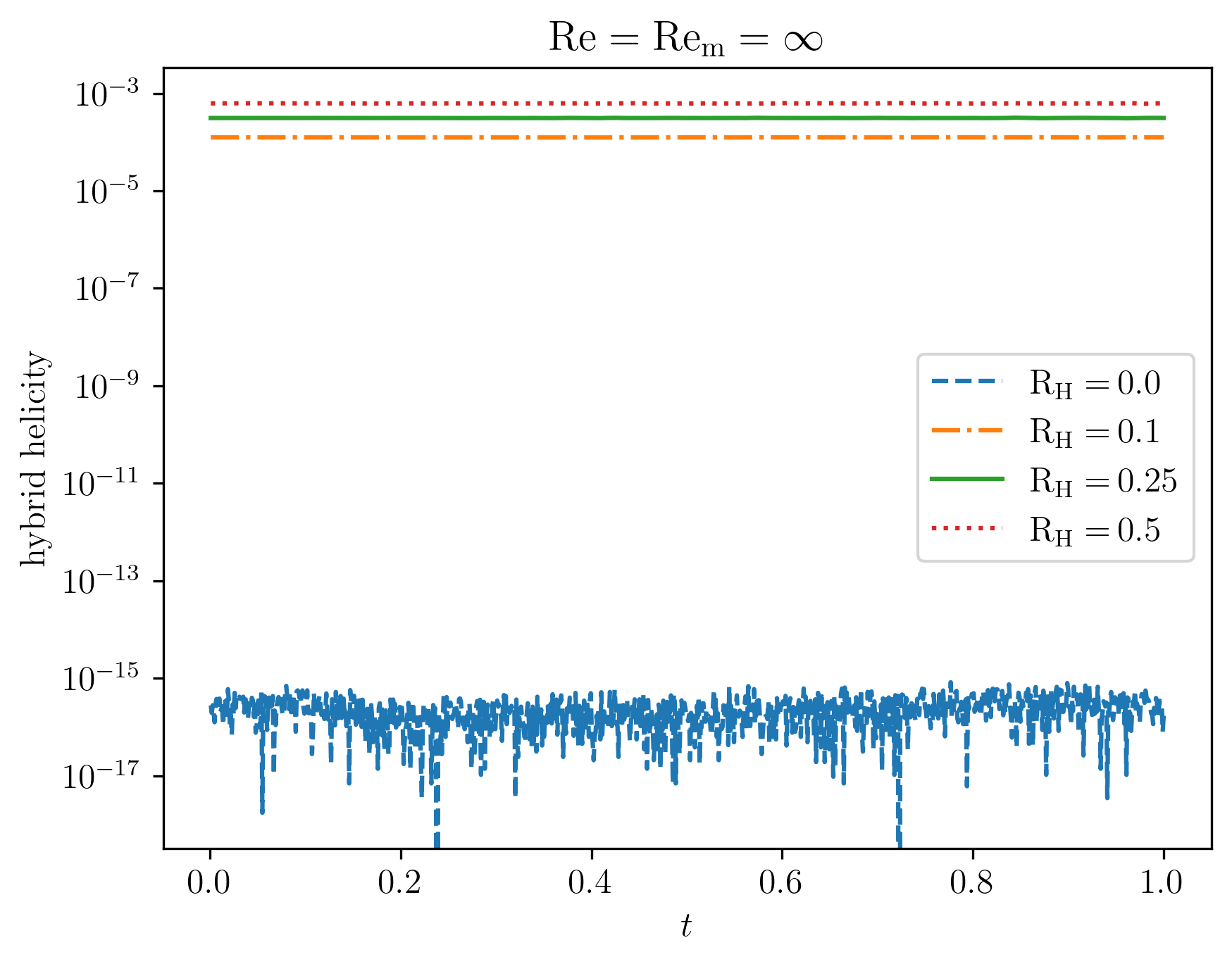}
	\end{tabular}
	\caption{Plots of the cross helicity (left) and hybrid helicity (right) for different values of $\RH$ for $\u\times \n = \mathbf{0}$ in the ideal limit of $\Re=\Rem=\infty$.}
	\label{fig:differentRH}
\end{figure}

\subsection{Test of conservative scheme for $\u \cdot \n = 0$ }
In this test, we  verify our results from Section \ref{sec:schemeucdotn} for the boundary conditions $\u \cdot \n = 0$. Here, we use the same initial conditions for $\B^0$ as before and 
\begin{gather}
	\begin{split}
		\u^0 = \nabla \times \v_{\text{pot}} \quad \text{with} \quad 
		\v_{\text{pot}}(x,y,z) = \frac{1}{\pi}\begin{pmatrix}
			\sin(\pi y)\sin(\pi z) \\ 
			\sin(\pi x)\sin(\pi z) \\ 
			\sin(\pi x)\sin(\pi y)
		\end{pmatrix},
	\end{split}
\end{gather}
which satisfy the boundary condition $\u^0\cdot \n = 0$ and $\nabla \cdot \u^0=0$. We discretise $\u$ with $\mathbb{RT}_1$-elements and $p$ with $\mathbb{DG}_0$-elements. We solve the system with a similar fixed point iteration to the one we described in the last subsection. The iteration coincides with that used in \cite[Section 6]{gawlik2020}.

In contrast to the case $\u \times \n = \mathbf{0}$, we now enforce $\nabla \cdot \u_h = 0$ precisely over time. All conserved properties are plotted in Figure \ref{fig:uHdiv}. Remember that the hybrid helicity is not conserved for this scheme and therefore not displayed here. Moreover, corresponding plots to Figure \ref{fig:differentRe} and  \ref{fig:differentRH} show similar results and are therefore omitted here.

\begin{figure}[htbp!]
	\centering
	\includegraphics[width=12cm]{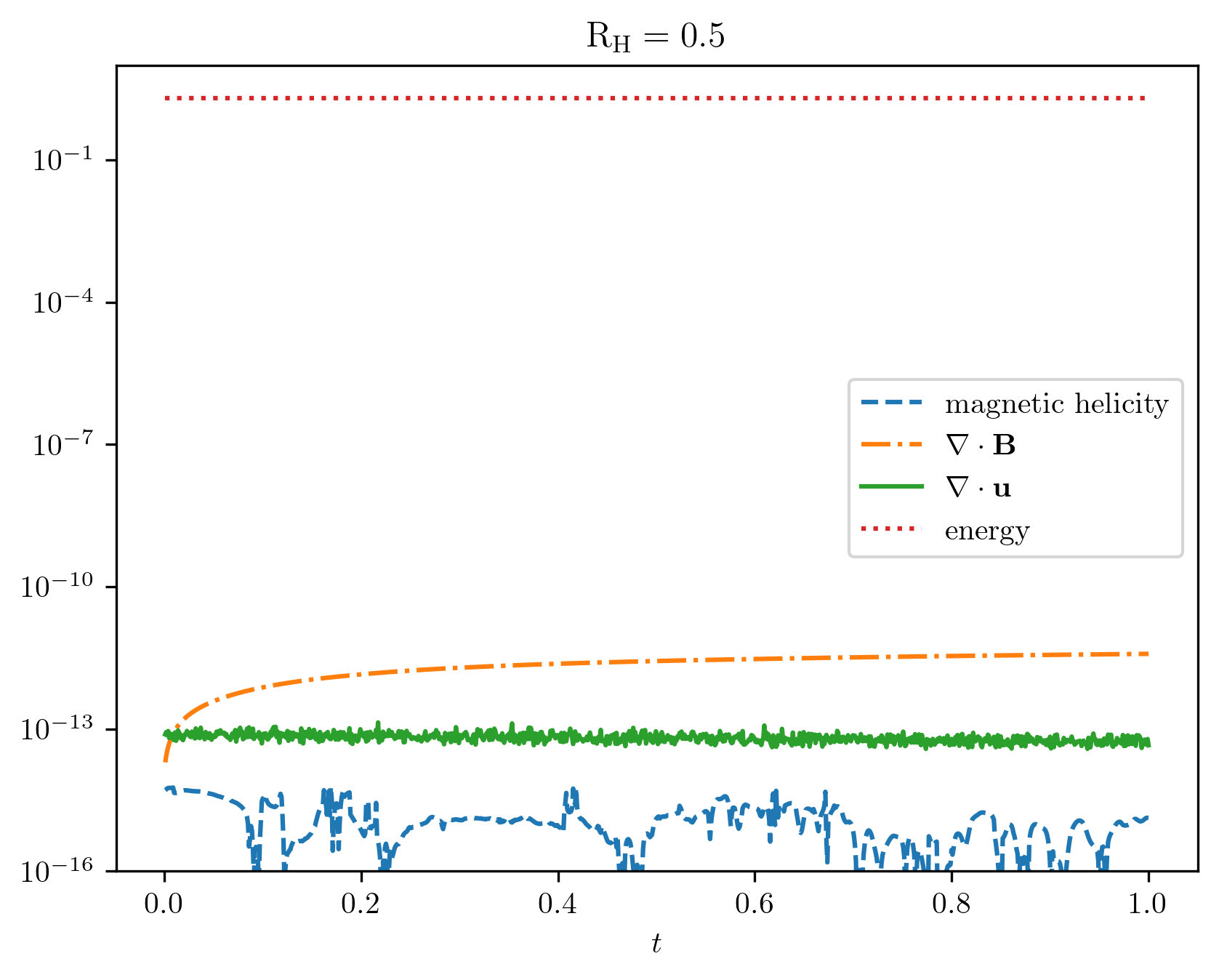} 
	\caption{Plot of conservative quantities for $\u \cdot \n = 0$ in the ideal limit.}
	\label{fig:uHdiv}
\end{figure}


\subsection{Island coalescence problem}
Finally, we consider a 2.5-dimensional island coalescence problem to model a magnetic reconnection process. We use the same setup as in Section \ref{sec:islandcoal}. For the additional variables, we set the equilibrium solution
\begin{align}
	u_{3, eq} = B_{3, eq}= 0, \qquad \tilde{\j}_{eq} = \mathbf{0}, \qquad j_{3,eq} = \scurl \tilde{\B}_{eq}. 
\end{align}
The components $\tilde{\E}_{eq}$ and $E_{3,eq}$ of the electric field are computed by the equations \eqref{eq:islandcoal-j} and \eqref{eq:islandcoal-j3}. Since we use a direct solver for the solution of the Schur complement, we only considered a base mesh $20\times20$ cells and three levels of refinement here resulting in an $160\times 160$ mesh. We iterated until the final time $T=12.0$ with a fixed step size of $\Delta t = 0.025$. We considered a length scale of $L=1$, a reference value for the magnetic field of $\overline{B}=1$, a reference density of $\rho_0=1$ and a Alfv\'en velocity of $v_A=1$, i.e., the Lundquist number is given here as $S_L = 1/\eta$ and coincides with $\Rem$.

Figure \ref{fig:reconHall} shows the reconnection rate for different choices of $\RH$ at $\Rem=\Re=100,500,1{,}000,1{,}500$. All graphs have in common that the reconnection process happens faster for higher Hall parameters. This is consistent with the results of other numerical experiments \cite[Section 4.3]{Morales2005}\cite{Huba2003}}. We also observe that additional peaks occur for high Hall parameters and Reynolds numbers.

 For $\Rem=\Re=100$ and $\Rem=\Re=500$ one can observe that the height of the peaks increases with growing Hall parameters. At $\Rem = \Re = 1{,}000$ this trend is broken and for $\Rem=\Re=1{,}500$ the heights of the peaks starts to decrease for higher Hall parameters. This observation matches the results shown in Figure 2 in \cite{Chacon2006Hall} qualitatively well, even though a slightly different problem setup is considered. In this figure the resistivity is plotted against the peak reconnection rate. Since the authors consider the case $\eta = \nu$ with other reference values and length scales fixed, varying the resistivity corresponds in our case to varying the values of $\Re=\Rem$. Figure 2  in \cite{Chacon2006Hall} shows that for decreasing $\eta$ (i.e. increasing $\Rem$) the peak reconnection rate increases until a certain value of $\eta$ is reached and then starts to decrease, which matches our findings. Furthermore, Figure 2 demonstrates that the decrease in the peak reconnection rate is smaller the higher the Hall parameter is chosen. While this trend is more obvious in Figure 2 due to the much larger considered range of approximately $10^{-5} \leq \eta \leq 10^{-3}$ (i.e. $10^3 \leq \Rem \leq 10^5$) the same trend is indicated in our results.

\begin{figure}[htbp!]
	\centering
	\begin{tabular}{cc}
		\includegraphics[width=7.0cm]{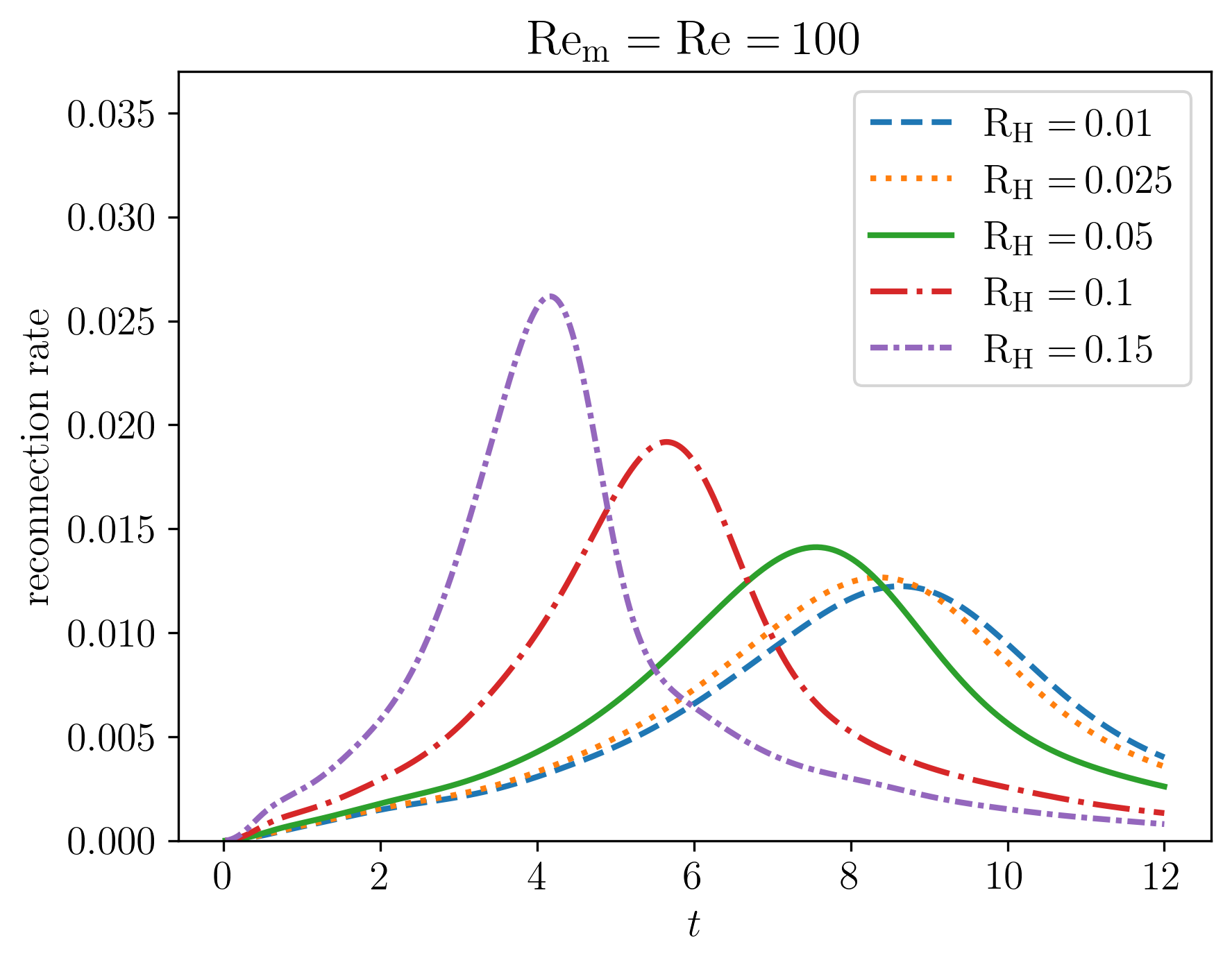} &
		\includegraphics[width=7.0cm]{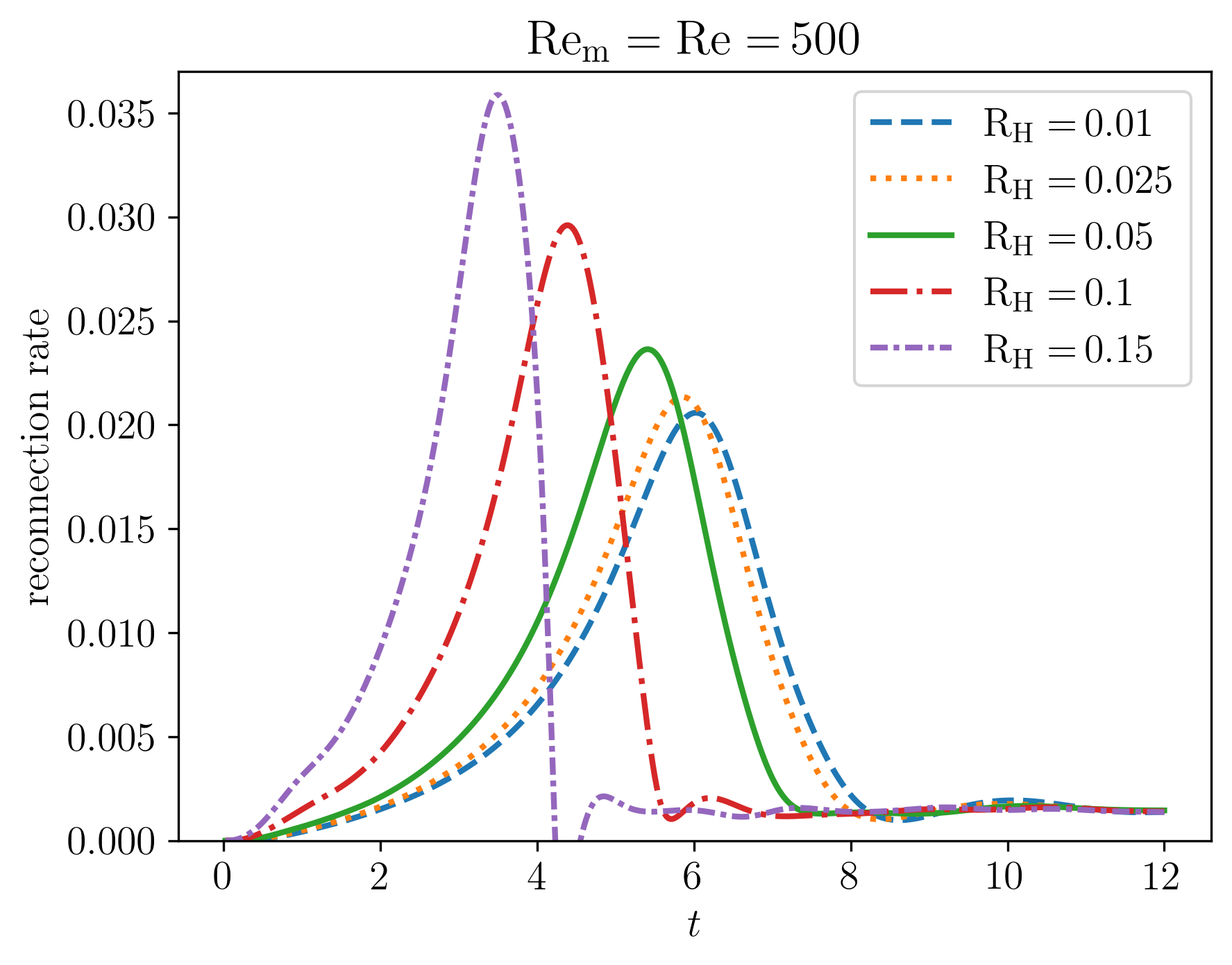} \\
		\includegraphics[width=7.0cm]{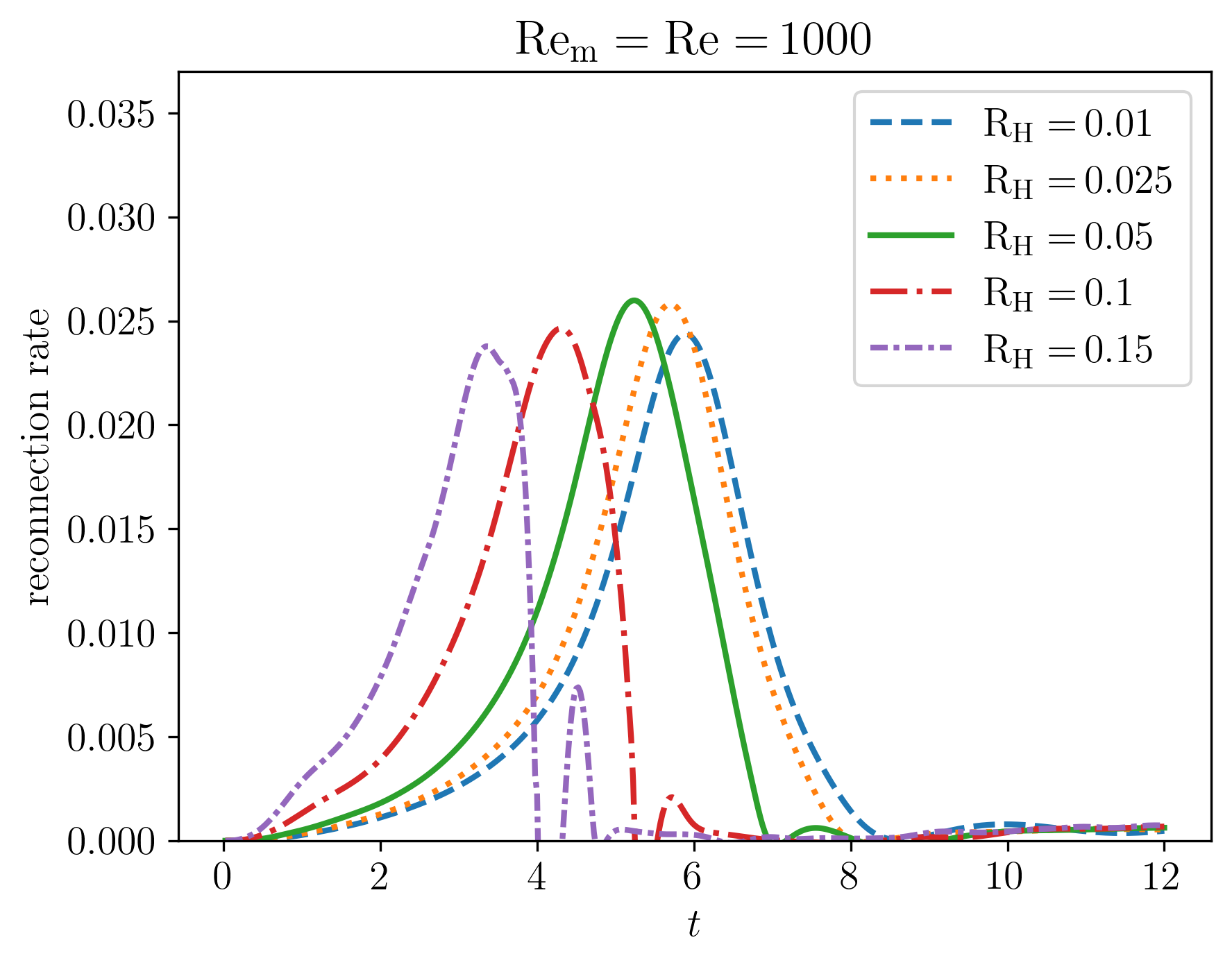} &
		\includegraphics[width=7.0cm]{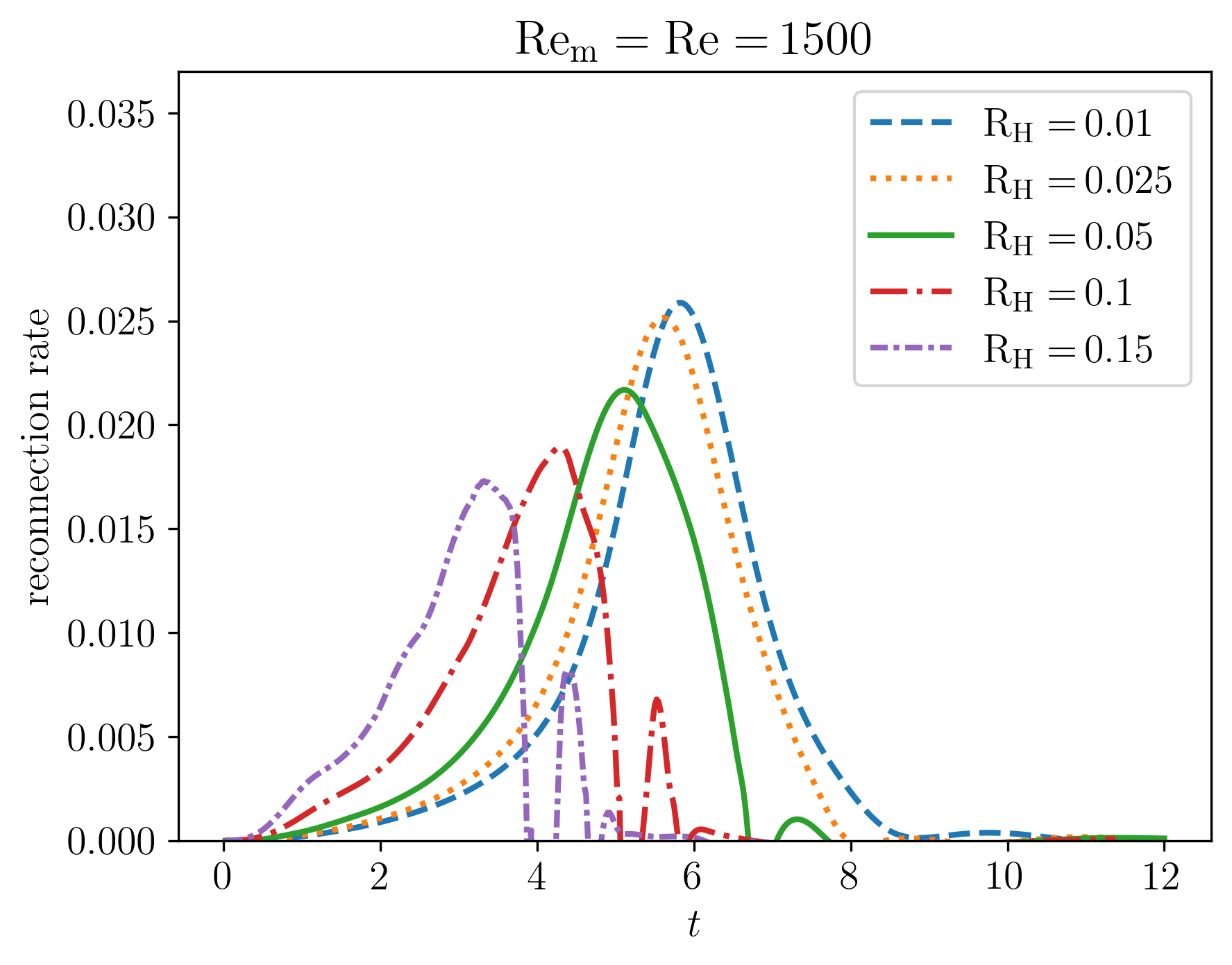} \\
	\end{tabular}
	\caption{Reconnection rates for an island coalescence problem for different choices of $\RH$.}
	\label{fig:reconHall}
\end{figure}

\renewcommand{\div}{\nabla \cdot}
\chapter{Bifurcation analysis and robust solvers for anisothermal MHD models}\label{chap:4}

In the final chapter of this thesis, we wish to consider discretisations and preconditioners for anisothermal MHD models. In particular, we investigate the Boussinesq approximation \cite{boussinesq1903theorie,Oberbeck1879} used in the modelling of MHD convection. This approximation assumes that the flow is buoyancy-driven and that density differences only appear in the buoyancy term, while other parameters depend neither on the density or temperature. Moreover, we perform a bifurcation analysis for a magnetic Rayleigh--B\'enard problem and investigate the influence of the coupling number $S$ on the bifurcation diagrams.

\section{Formulation and discretisation}\label{sec:AnIsoFormulationAndDiscretisation}
The dimensional formulation of the anisothermal MHD equations with Boussinesq approximation is given by 
\begin{subequations}
	\label{eq:MHDBouDim}
	\begin{align}
		 \partial_t \u - 2 \nu \div \eps \u +  \u \cdot \nabla \u + \nabla p \qquad \qquad \qquad & \nonumber \\ 
		 + \frac{1}{\rho_0 \mu_0 \eta} \B \times  (\E + \u \times \B) &=
		 - \beta (\theta-\theta_0) g\mathbf{e}_3, \label{eq:MHDBouDim1}\\
		\div \u&=0, \label{eq:MHDBouDim2}\\
		\E + \u \times \B - \eta \vcurl \B &= \mathbf{0}, \label{eq:MHDBouDim3}\\
		\partial_t \B + \vcurl \E &= \mathbf{0}, \label{eq:MHDBouDim4}\\
		\partial_t \theta -\alpha \Delta \theta + \u \cdot \nabla \theta &= 0, \label{eq:MHDBouDim5}\\
		\div \B &= 0, \label{eq:MHDBouDim6}
	\end{align}
\end{subequations}
subject to the boundary conditions
\begin{equation}
	\u=\mathbf{0}, \quad  \E \times \n = \mathbf{0}, \quad \B \cdot \n=0 \quad \text{ on } \del \Omega,
\end{equation}
and
\begin{equation}
	\theta = \theta_b \text{ on } \del \Omega_D, 	\quad \quad \nabla \theta\cdot \n = 0 \text{ on } \del \Omega \backslash \del \Omega_D,
\end{equation}
for a given temperature distribution $\theta_b$ and a non-empty subset $\Omega_D$ of $\partial \Omega$.
In the above model, $\theta$ denotes the temperature, $\theta_0$ a reference temperature, $\rho_0$ a reference density, $\beta$ the thermal expansion, $g$ the magnitude of acceleration due to gravity, $\mathbf{e}_3$ a unit vector in $z$-direction, the buoyancy direction, and $\alpha$ the thermal diffusivity. 

In this chapter, we are not going to focus on regularity and well-posedness results for the continuous version of these equations. Results in this direction can be found in \cite{Bian2020,Ghani2021, pan2020global}.


For the derivation of non-dimensional version, we introduce the new unknowns
\begin{align}
	\x^\star &= \frac{\x}{L},\\
	t^\star &= \frac{\ou}{L}t,\\ 
	\u^\star(\x^\star, t^\star)&=\frac{\u(\x,t)}{\ou}, \quad \ou = \frac{\alpha}{L},\\
	p^\star(\x^\star, t^\star)&= \frac{p(\x,t)L^2}{\rho_0 \alpha^2}, \\
	\B^\star(\x^\star, t^\star)&= \frac{\B(\x,t)}{\overline{B}},\\
	\E^\star(\x^\star, t^\star) &= \frac{\E(\x,t)}{\ou\overline{B}}, \\
    \theta^\star(\x^\star, t^\star) &= \frac{\theta(\x,t) - \theta_0}{\overline{\theta}}, \quad \overline{\theta} = \theta_1 - \theta_0,
\end{align}
with characteristic values for the magnetic field $\overline{B}$, the length scale $L$ and a second reference temperature $\theta_1 > \theta_0$ (e.g., $\theta_1$ and $\theta_0$ can denote the temperature of the hot and cold plate in a B\'enard type problem). Note that we do not include viscous dissipation in the temperature equation \eqref{eq:MHDBouDim5}.

After dropping the stars, this leads to the non-dimensional system
\begin{subequations}
	\label{eq:MHDBou}
	\begin{align}
		\partial_t \u - 2 \Pr \div \eps \u +  \u \cdot \nabla \u + \nabla p + S\ \B \times  (\E + \u \times \B) &=
		\Ra\, \Pr\, \theta \mathbf{e}_3, \label{eq:MHDBou1}\\
		\div \u&=0, \label{eq:MHDBou2}\\
		\E + \u \times \B - \frac{\Pr}{\Pm} \vcurl \B &= \mathbf{0}, \label{eq:MHDBou3}\\
		\partial_t \B + \vcurl \E &= \mathbf{0}, \label{eq:MHDBou4}\\
		\partial_t \theta - \Delta \theta + \u \cdot \nabla \theta &= 0, \label{eq:MHDBou5}\\
		\div \B &= 0, \label{eq:MHDBou6}
	\end{align}
\end{subequations}
with the Prandtl number $\Pr$, magnetic Prandtl number $\Pm$, Rayleigh number $\Ra$ and coupling number $S$ given by
\begin{equation}
	\Pr = \frac{\nu}{\alpha}, \qquad \Pm=\frac{\nu}{\eta}, \qquad \Ra = \frac{\beta g \overline{\theta} L^3}{\nu \alpha}\quad  \text{ and } \quad S = \frac{\overline{B}^2 L^2}{\rho_0 \mu_0 \eta \alpha }.
\end{equation}
For a further description of these unknowns and typical parameter values in applications we refer to Section \ref{sec:IntroAnisothermalMHDmodels}. Note that the coupling number $S$ can also be expressed in terms of the Chandrasekhar number $Q$ as
\begin{equation}
	S = \Pr\, Q  \quad \text{ with } \quad Q = \frac{\overline{B}^2 L^2}{\mu_0 \rho_0 \nu \eta}.
\end{equation}

For the stationary formulation of \eqref{eq:MHDBou}, we combine as in the previous chapters the equations \eqref{eq:MHDBou4} and \eqref{eq:MHDBou5} to the augmented Lagrangian formulation
\begin{equation}
	-\frac{\Pr}{\Pm} \nabla \div \B + \vcurl \E = \mathbf{0}.
\end{equation}
For a weak formulation, we look for the new unknown $\theta$ in \mbox{$H^1_{0, \del \Omega_D} \coloneqq \{\tau \in H^1(\Omega) \, | \, \tau = 0 \text{ on } \partial \Omega_D \}$}. Hence, the weak formulation for the homogeneous, stationary problem is given by: find \mbox{$(\u, p, \theta, \E, \B) \in X \coloneqq \mathbf{H}^1_0 \times L^2_0 \times H^1_{0, \partial \Omega_D} \times  \mathbf{H}_0(\curl) \times \mathbf{H}_0(\mathrm{div})$} such that for all \mbox{$(\v, q, \tau, \F, \C) \in X$} there holds
\begin{subequations}
\label{eq:BoussinesqStationary}
\begin{align}
	2 \Pr (\div \eps \u , \div \eps \v) + ((\u \cdot \nabla) \u, \v) -(p, \div \v) + S (\B \times \E, \v) \nonumber\\ + S (\B \times (\u\times \B), \v) - \Ra\, \Pr\, (\theta \mathbf{e}_3, \v) &= 0,\\
	- (\div \u, q) &= 0, \\
	(\E, \F) + (\u \times \B, \F) - \frac{\Pr}{\Pm} (\B, \vcurl \F) &= 0,\\
	\frac{\Pr}{\Pm} (\div \B, \div \C)  + (\vcurl \E, \C) &= 0,\\
	(\nabla \theta, \nabla \tau) + (\u \cdot \nabla \theta, \tau) &= 0.
\end{align}
\end{subequations}
For a finite element approximation, we approximate the temperature $\theta$ with $\mathbb{CG}_k$ elements. The finite element spaces for  $\B$ and $\E$ are chosen as in Chapter \ref{chap:2} and Chapter \ref{chap:3}, namely $\mathbb{RT}_k$ for $\B$ and $\mathbb{NED}1_k$ for $\E$ in 3D and $\mathbb{CG}_k$ for $E$ in 2D. We mention in each of the following sections which discretisation we choose for $\u$ and $p$.

In the previous chapters, the main purpose of the Picard type iteration was to prove well-posedness results and derive more accurate Schur complement approximations for the development of our preconditioners. The well-posedness proof for the Picard iteration of the standard MHD system from \cite{Hu2020} is straightforward to extend to the temperature-dependent case. We only investigate the full Newton linearisation here, since it outperformed in nearly all cases in the previous chapters the Picard iteration in terms of iteration numbers.

\section{Bifurcation analysis for a 2D magnetic Rayleigh-B\'enard problem}\label{sec:BifurcationAnalysis}

The author wants to thank Nicolas Boull\'e for many helpful discussions which helped to improve the content of this section.

In this section, we want to perform a bifurcation analysis for a two-dimensional magnetic Rayleigh-B\'enard problem. Our goal is to compute a bifurcation diagram for the bifurcation parameter $\Ra$ in the range between 0 and 100,000 at a high coupling number of $S=1{,}000$. Since these diagrams are quite challenging to compute directly for high coupling numbers,  we start by investigating the bifurcation diagram over $\Ra$ at a low coupling number of $S=1$ in Section \ref{sec:bif1}. We then proceed in Section \ref{sec:bif2} to choose $S$ as bifurcation parameter ranging from 1 to 1{,}000 with fixed $\Ra=100{,}000$. Finally, we use the obtained results at $\Ra=100{,}000$ and $S=1{,}000$ as initial guesses to compute the desired bifurcation diagram over $0\leq \Ra \leq 100{,}000$ at $S=1{,}000$ in Section \ref{sec:bif3}.

Another goal is to study the effect of the magnetic field on the arising bifurcations in comparison to the standard Rayleigh-B\'enard problem for the three unknowns $(\u, p, \theta)$, i.e., $\B=E=0$. Therefore, we  compare our results to the ones presented in \cite{Boulle2022}. The outline of this section is heavily influenced by this manuscript and we use a similar problem setup adapted for the magnetic and electric fields and similar numerical techniques to compute our numerical results.
For more information about the magnetic Rayleigh-B\'enard problem and its bifurcation analysis we refer to \cite{Yang2021,HAN2018370,NAFFOUTI2014714,Akhmedagaev2020,burr_muller_2002,Nandukumar_2015}. 

In the following, we consider the unit square domain $\Omega=(0,1)^2$ with coordinates $(x_1, x_3)$, no-slip boundary conditions for $\u$ and a background magnetic field that points in the direction $(0,1)^\top$. Further, we choose the horizontal walls to be thermally conducting and the vertical walls to be insulating. In summary, this leads to the boundary conditions
\begin{align}
	\begin{split}
	\u = \mathbf{0} \text{ on } \partial \Omega, \quad \theta = 
	\begin{cases}
		1, & \text{ on } \{x_3=0\}, \\
		0, & \text{ on } \{x_3 = 1\},
	\end{cases},
    \quad
  	\nabla \theta \cdot \n = 0 \text{ on }  \{x_1=0,1\}, \\
    \B\cdot \mathbf{n} = (0,1)^\top \mathbf{n} \text{ on } \partial \Omega \quad \text{ and } \quad
  	E  = 0 \text{ on } \partial \Omega. \qquad \qquad \quad 
  	\end{split}
\end{align}
Recall from Section \ref{sec:MHDModel} that the electric field is a scalar field in two dimensions. 

The trivial steady state solution for these boundary conditions, also called the conduction state, is given by
\begin{align}\label{eq:trivialsol}	
	\begin{split}
	\u_0 = \mathbf{0}, \qquad p_0 = \Ra\, \Pr\left(x_3 - \frac{1}{2} x_3^2 - \frac{1}{3}\right), \qquad  \theta_0 = 1-x_3, \\
	 \B_0 = (0,1)^\top \quad  \text{ and } \quad  E_0 = 0.  \qquad \qquad \qquad
	\end{split}
\end{align}

The problem has two symmetries that determine the behaviour of the arising bifurcations:
\begin{equation}\label{eq:symmetry1}
	[u_1, u_2, \theta, B_1, B_2](x_1, x_3) \to [-u_1, u_2, \theta, B_1, -B_2] (1-x_1, x_3)
\end{equation}
and 
\begin{equation}\label{eq:symmetry2}
	[u_1, u_2, \theta, B_1, B_2](x_1, x_3) \to [u_1, -u_2, 1-\theta, -B_1, B_2] (x_1, 1-x_3)
\end{equation}
which can be easily verified by a direct computation.
In our bifurcation diagrams we always just display one of these four corresponding solutions. The evolution of the solutions will be represented in terms of $\|\u\|^2$, $\|\theta\|^2$ and $\|\B\|^2$.

Since we want to compare our numerical results to \cite{Boulle2022}, we also use Taylor-Hood elements of degree 2, i.e., $[\mathbb{CG}_2]^2 \times \mathbb{CG}_1$ to discretise $(\u, p)$ in this section.  Furthermore, we use a triangular mesh with $50\times50$ square cells where each square cell is split into four triangles by the two diagonals of each square. We use this symmetric mesh to preserve the symmetries of the problem. Note that \cite{Boulle2022} uses a quadrilateral mesh. 

 Since the deflation algorithm, which we introduce in the next paragraph, might require the solution of hundreds of thousands of nonlinear iterations in total to compute a full bifurcation diagram, we choose this rather coarse mesh here to decrease the computation time. Moreover, we apply a direct solver to solve the arising linear systems. In the next section, we introduce a scalable preconditioner that allows to solve these equations efficiently and robustly on much finer grids. If one is interested in more accurate solutions for certain parameters, on can then use the solution of the $50\times50$ grid as an initial guess and recompute the solution on finer meshes in a nested iteration. 

We compute our bifurcation diagrams with a technique called \emph{deflated continuation}. Deflation \cite{Deflation2015} is a method to compute multiple solutions of a nonlinear equations. To introduce this method, we rewrite our system \eqref{eq:MHDBou} as $F(\Phi, \lambda)=0$ where $\Phi = (\u, p, \theta, \B, E)$ and $\lambda \in \{\Ra, S\}$ denotes the bifurcation parameter. Assuming that Newton's method has found a solution $\Phi_1$ with $F(\Phi_1, \lambda)=0$, the deflation algorithm continues by trying to find a root of the deflated residual
\begin{equation}
	F_1(\Phi, \lambda) \coloneqq \mathcal{M}(\Phi, \Phi_1) F(\Phi, \lambda).
\end{equation}
The operator $\mathcal{M}$ should be constructed in a way that eliminates solutions which are close to the already found solution $\Phi_1$ and is close to 1 away from $\Phi_1$. In this work, we choose $\mathcal{M}$ as 
\begin{equation}
\mathcal{M}(\Phi, \Phi_1) \coloneqq \left( \frac{1}{\|\u - \u_1 \|^2 + \|\nabla (\u - \u_1)\|^2 + \|\theta - \theta_1\|^2 + \|\B - \B_1\|^2} + 1  \right).
\end{equation}
Newton's method can then be applied to $F_1$ from the same initial guess. The process can be repeated to discover multiple solutions. After the deflation method has found multiple solutions for a fixed $\lambda$, deflated continuation proceeds by using the computed solutions for $\lambda$ as initial guesses for a parameter continuation from $\lambda$ to $\lambda \pm \Delta \lambda$ (the sign expresses whether we do forward or backward continuation), i.e., we apply a simple 0-th order continuation here with fixed step-size $\Delta \lambda$. The algorithm then iteratively continues by applying a deflated continuation step to each found solution at $\lambda \pm \Delta \lambda$ until the final value of $\lambda$ is reached.
 
The performance of the deflation algorithm and which solutions are found heavily depends on the available initial guesses for the initial deflation step. We use the same approach described in \cite{Boulle2022} to provide initial guesses by computing unstable eigenmodes linearised around the trivial solution \eqref{eq:trivialsol}. The initial guesses are then given by the sum of the trivial solution and the normalised eigenmodes. Note that this approach is used in \cite{Boulle2022} to provide initial guesses for backward continuation starting from $\Ra=100{,}000$. In general, backward continuation allows to compute more complex bifurcation diagrams and especially allows the computation of disconnected branches which might not be found with forward continuation.

To compute the unstable eigenmodes, we consider the perturbation ansatz
\begin{equation}
	\begin{pmatrix}
		\u\\ \theta   \\ \B
	\end{pmatrix}
= 
	\begin{pmatrix}
		\u_0 \\ \theta_0 \\ \B_0
	\end{pmatrix}
+
\begin{pmatrix}
	\tilde{\u}  \\ \tilde{\theta} \\ \tilde{\B}
\end{pmatrix}
e^{\lambda t}
\end{equation}
for small perturbations $\tilde{\u}, \tilde{\theta},\tilde{\B} \ll 1$.
Linearising system \eqref{eq:MHDBou} around the trivial solution $(\u_0, p_0, \theta_0, \B_0, E_0)$ and inserting the perturbation ansatz leads to the generalised eigenvalue problem
\begin{equation}\label{eq:EVproblem}
		\resizebox{\textwidth}{!}{%
$\begin{bmatrix}
		\FF & -\nabla & \Ra\, \Pr\, \mathbf{e}_3 & \GG& - S\,\B_0 \times\\
		\nabla \cdot & 0 & 0 & 0 & 0\\
		- \nabla \theta_0 \cdot & 0 & \Delta - \u_0 \cdot \nabla & 0 & 0\\
		0 & 0 & 0 & \frac{\Pr}{\Pm}\nabla \nabla \cdot & - \vcurl \\
		- \times \B_0 & 0 & 0 &  -\frac{\Pr}{\Pm}\scurl \, - \u_0 \cdot & -I 
	\end{bmatrix}
\begin{bmatrix}
		\tilde{\u}  \\ p \\  \tilde{\theta} \\ \tilde{\B}\\ E
\end{bmatrix}	
= \lambda
	\begin{bmatrix}
		I & 0 & 0 & 0 & 0\\
		0 & 0 & 0 & 0 & 0\\
		0 & 0 & I & 0 & 0 \\
		0 & 0 & 0 & I & 0\\
		0 & 0 & 0 & 0 & 0\\
	\end{bmatrix}
\begin{bmatrix}
	\tilde{\u}  \\ p \\ \tilde{\theta} \\ \tilde{\B} \\ E
\end{bmatrix}$}
\end{equation}
with 
\begin{align}
\FF \tilde{\u} &= 2 \Pr \nabla \cdot \varepsilon(\tilde{\u}) - \u_0 \cdot \nabla \tilde{\u} - \tilde{\u} \cdot \nabla \u_0 - S\, \B_0 \times (\tilde{\u} \times \B_0), \\
	\GG\, \tilde{\B} &= - S\, \tilde{\B} \times E_0 - S\, \tilde{\B}\times (\u_0 \times \B_0) - \S\, \B_0\times(\u_0 \times \tilde{\B}).
\end{align}
Note that we changed our block matrix notation here slightly by including the operators directly in the matrix rather than assigning a name to each operator as previously done, e.g., in Table \ref{tab:Operators}.

We solve this eigenvalue problem with a Krylov--Schur solver \cite{Stewart2002} that is implemented in the library SLEPc \cite{SLEPc}. The real and imaginary parts of the computed eigenvalues determine the stability of the corresponding eigenmode. If the real part of all eigenvalues is negative, the solution is stable. If at least one eigenvalue has a positive real part the solution is unstable, with the type of instability depending on whether the associated imaginary part is zero or non-zero.

As mentioned before, this technique has been used in \cite{Boulle2022} to provide initial guesses at $\Ra=100{,}000$ to compute a bifurcation diagram with complex solution patterns and disconnected branches. The same approach still works for the magnetic Rayleigh-B\'enard problem at a small coupling number of $S=1$, but fails to provide initial guesses for higher coupling numbers at $S=1{,}000$ from which Newton's method is able to converge. Therefore, we proceed with the approach outlined at the beginning of this section in which we use deflated continuation for $0\leq\Ra\leq 100{,}000$ at $S=1$ and for $1\leq S\leq 1{,}000$ at $\Ra=100{,}000$ to obtain initial guesses at $\Ra=100{,}000$ and $S=1{,}000$. As we will see in Section \ref{sec:bif3} this allows us to compute all primary bifurcations as well a disconnected branch that we were not able to find with forward continuation.

\subsection{Bifurcation analysis for  $0\leq \Ra \leq 100{,}000$  with  \mbox{$S=1$}}\label{sec:bif1}
We start by analysing the stability of the conducting state \eqref{eq:trivialsol} in the range of $0\leq \Ra \leq 100{,}000$ at $S=1$ by solving the aforementioned eigenvalue problem \eqref{eq:EVproblem}. We observe 11 supercritical bifurcations emanating from the conducting state in this range. The growth rates of the first 10 unstable eigenfunctions are displayed in Figure \ref{fig:eigsplot_S1_PM1}. This plot looks nearly identical to \cite[Fig.\ 1]{Boulle2022} for the standard Rayleigh--B\'enard problem which indicates that the effect of the magnetic field at $S=1$ is almost negligible for the bifurcation patterns. The 11th supercritical bifurcation starts at $\Ra_c^{(11)}=99{,}528$. For the standard Rayleigh--B\'enard problem the 11th supercritical bifurcation starts slightly above 100,000 and is hence not included in \cite{Boulle2022}.

Similarly, the plots of the eigenfunctions in Figure \ref{fig:eigenmodes_S1_PM1} look similar to \cite[Fig.\ 2]{Boulle2022}. Figure \ref{fig:eigenmodes_S1_PM1} also includes the critical Rayleigh numbers $\Ra_c$ which differ from the ones of the non-magnetic problem by around one percent. The critical Rayleigh numbers indicate when a steady states become unstable, i.e., when the real part of the eigenvalues crosses the zero-line. We compute the critical Rayleigh numbers accordingly to \cite[Sec.~B]{Boulle2022} by solving a generalised eigenvalue problem, where we interpret $\Ra_c$ as the eigenvalue in 

\begin{equation}\label{eq:EVproblemRac}
	\resizebox{\textwidth}{!}{%
		$\begin{bmatrix}
			\FF & -\nabla & 0 & \GG& - S\,\B_0 \times\\
			\nabla \cdot & 0 & 0 & 0 & 0\\
			- \nabla \theta_0 \cdot & 0 & \Delta - \u_0 \cdot \nabla & 0 & 0\\
			0 & 0 & 0 & \frac{\Pr}{\Pm}\nabla \nabla \cdot & - \vcurl\\
			- \times \B_0 & 0 & 0 &  -\frac{\Pr}{\Pm}\scurl \,   - \u_0 \cdot & -I 
		\end{bmatrix}
		\begin{bmatrix}
			\tilde{\u}  \\ p \\  \tilde{\theta} \\ \tilde{\B}\\ E
		\end{bmatrix}	
		= \Ra_c
		\begin{bmatrix}
			0 & 0 & -\Pr\, \mathbf{e}_3 & 0 & 0\\
			0 & 0 & 0 & 0 & 0\\
			0 & 0 & 0 & 0 & 0 \\
			0 & 0 & 0 & 0 & 0\\
			0 & 0 & 0 & 0 & 0\\
		\end{bmatrix}
		\begin{bmatrix}
			\tilde{\u}  \\ p \\ \tilde{\theta} \\ \tilde{\B} \\ E
		\end{bmatrix}$}.
\end{equation}

\begin{figure}[htbp!]
	\centering
	\begin{subfigure}{0.32\textwidth}
		\includegraphics[width=4.8cm]{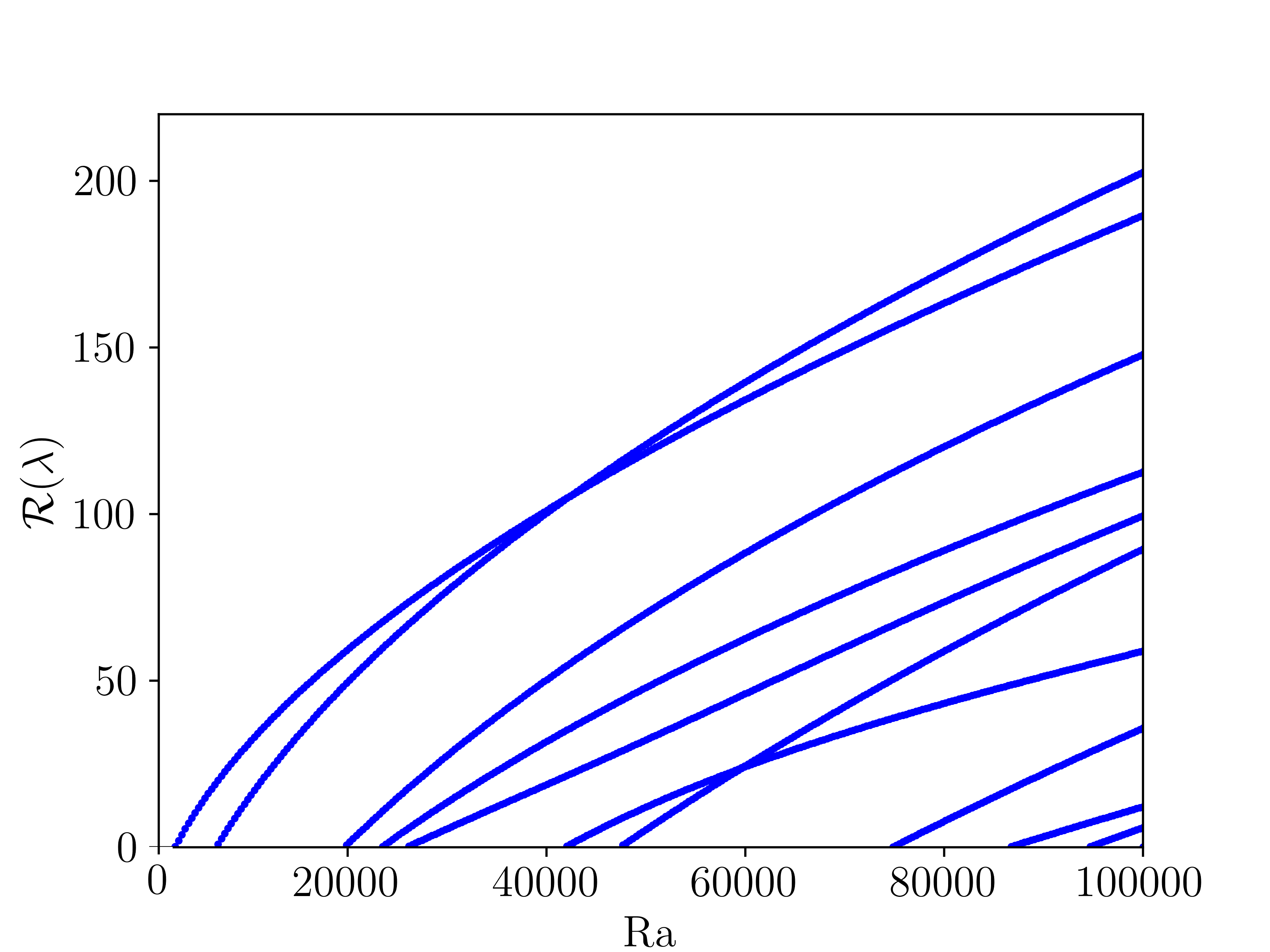}
		\subcaption{Over $\Ra$  at  \mbox{$S=1$}.}
		\label{fig:eigsplot_S1_PM1}
	\end{subfigure}
    \hspace{-0.3cm}
	\begin{subfigure}{0.32\textwidth}
		\includegraphics[width=4.8cm]{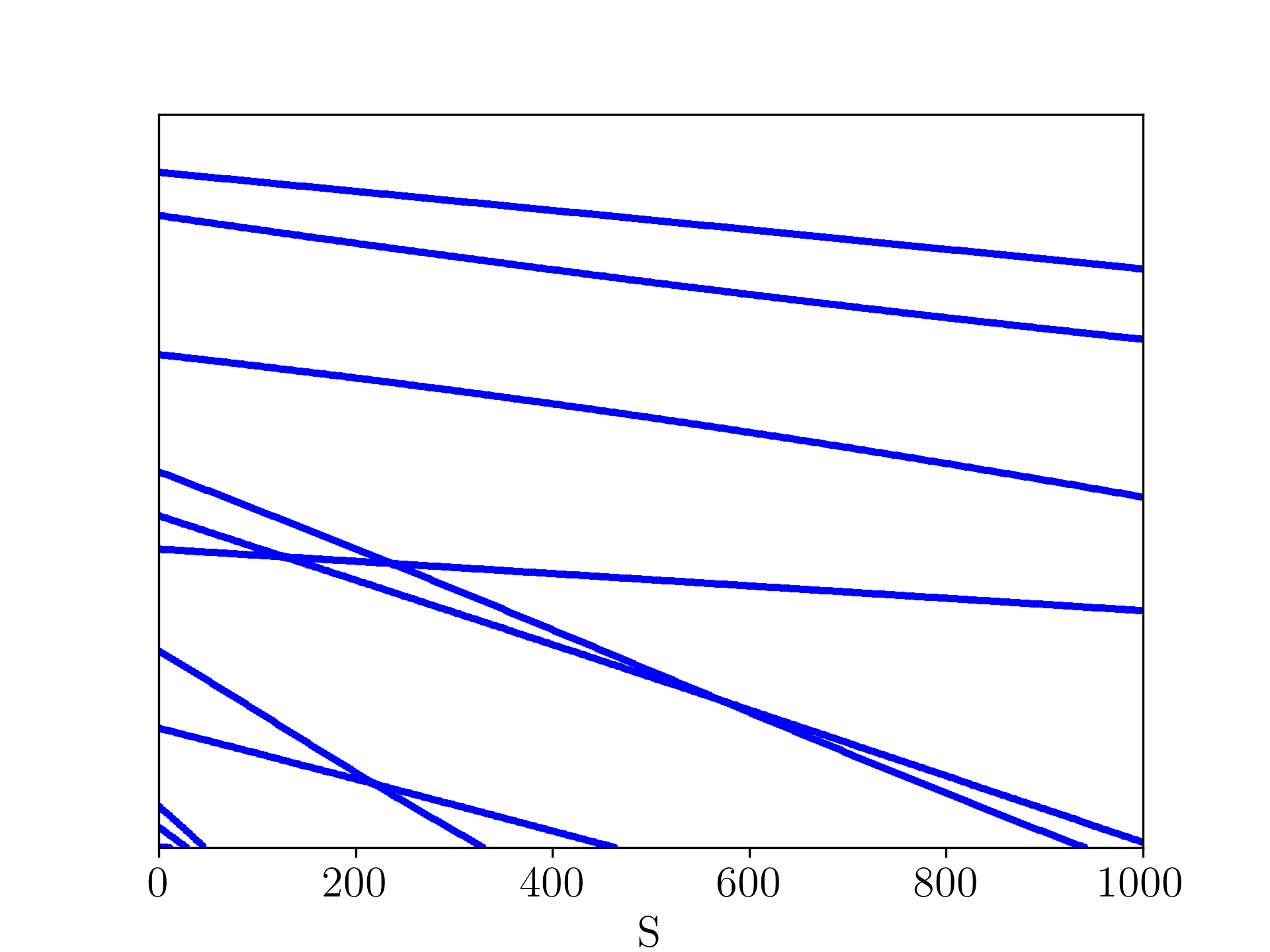}
		\subcaption{Over $S$  at \mbox{$\Ra=100{,}000$}.}
		\label{fig:eigsplot_S1000_PM1}
	\end{subfigure}
    \hspace{-0.3cm}
	\begin{subfigure}{0.32\textwidth}
		\includegraphics[width=4.8cm]{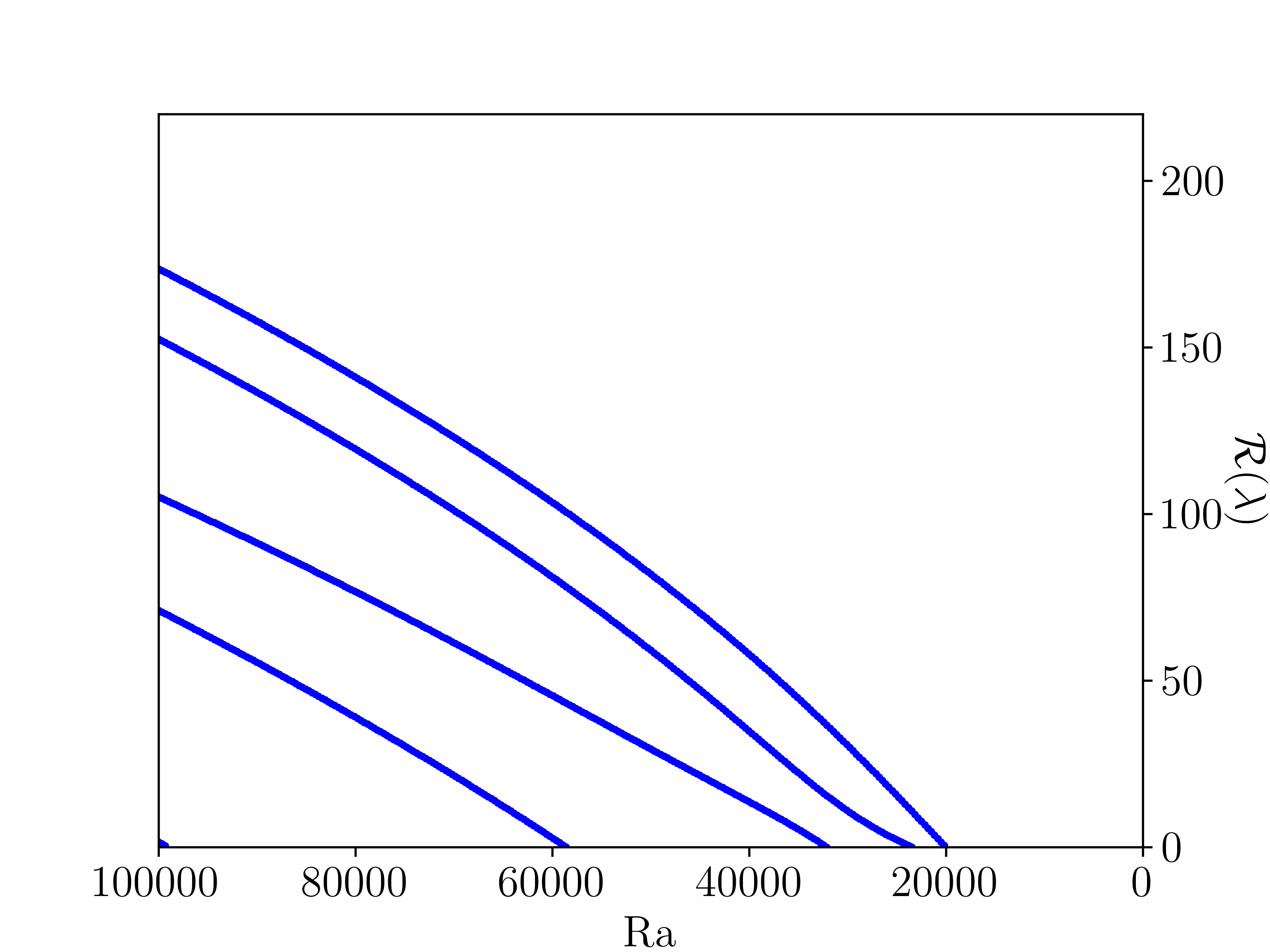}
		\subcaption{Over $\Ra$ at \mbox{$S=1{,}000$}.}
		\label{fig:eigsplot_S1000_PM1_3}
	\end{subfigure}
	\caption{Growth rates of eigenmodes emanating from the conducting state.
		\label{fig:eigsplots}}
\end{figure}

\begingroup
\renewcommand{\arraystretch}{2.0}
\begin{figure}[htbp!]
	\centering
	\newcommand{\mywidth}{2.5cm}
	\begin{tabular}{ccccc}
		$\Ra_c^{(1)} = 2{,}609$ & $\Ra_c^{(2)} = 6{,}756$ & $\Ra_c^{(3)} = 19{,}647$ & $\Ra_c^{(4)} = 23{,}408$ & $\Ra_c^{(5)} = 25{,}903$ \\
		\includegraphics[width=\mywidth]{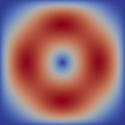} &
		\includegraphics[width=\mywidth]{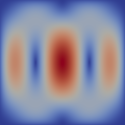} &
		\includegraphics[width=\mywidth]{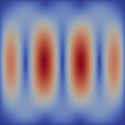} &
		\includegraphics[width=\mywidth]{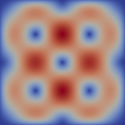} &
		\includegraphics[width=\mywidth]{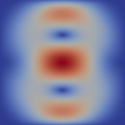} \\
		\includegraphics[width=\mywidth]{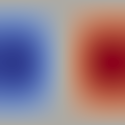} &
		\includegraphics[width=\mywidth]{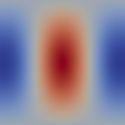} &
		\includegraphics[width=\mywidth]{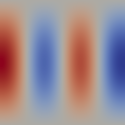} &
		\includegraphics[width=\mywidth]{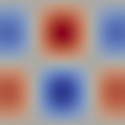} &
		\includegraphics[width=\mywidth]{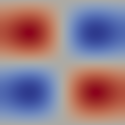} \\
		\vspace{0.5cm}
		\includegraphics[width=\mywidth]{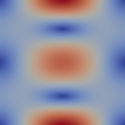} &
		\includegraphics[width=\mywidth]{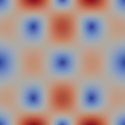} &
		\includegraphics[width=\mywidth]{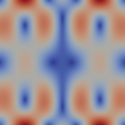} &
		\includegraphics[width=\mywidth]{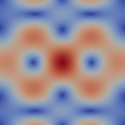} &
		\includegraphics[width=\mywidth]{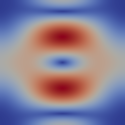} \\		
		$\Ra_c^{(6)} = 41{,}799$ & $\Ra_c^{(7)} = 47{,}364$ & $\Ra_c^{(8)} = 74{,}761$ & $\Ra_c^{(9)} = 86{,}462$ & $\Ra_c^{(10)} = 94{,}524$ \\
		\includegraphics[width=\mywidth]{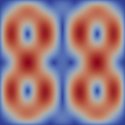} &
		\includegraphics[width=\mywidth]{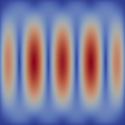} &
		\includegraphics[width=\mywidth]{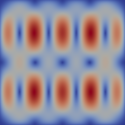} &
		\includegraphics[width=\mywidth]{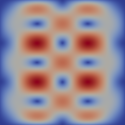} &
		\includegraphics[width=\mywidth]{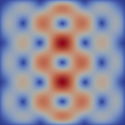} \\
		\includegraphics[width=\mywidth]{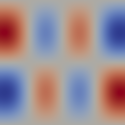} &
		\includegraphics[width=\mywidth]{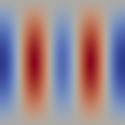} &
		\includegraphics[width=\mywidth]{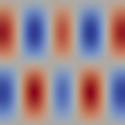} &
		\includegraphics[width=\mywidth]{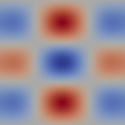} &
		\includegraphics[width=\mywidth]{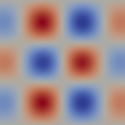} \\
		\includegraphics[width=\mywidth]{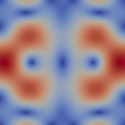} &
		\includegraphics[width=\mywidth]{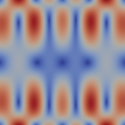} &
		\includegraphics[width=\mywidth]{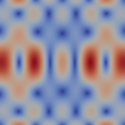} &
		\includegraphics[width=\mywidth]{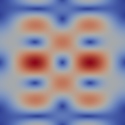} &
		\includegraphics[width=\mywidth]{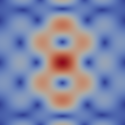} \\
	\end{tabular}
	\caption{First 10 eigenfunctions of the primary bifurcations that emanate from the conducting state \eqref{eq:trivialsol} for $0\leq \Ra \leq 100{,}000$ with \mbox{$S=1$}. The top row shows the magnitude of the velocity, the middle row the temperature and the bottom row the magnitude of the magnetic field for the different critical Raleigh numbers. \label{fig:eigenmodes_S1_PM1}}
\end{figure}
\endgroup

\FloatBarrier
Since the effect of the magnetic field at $S=1$ is negligible, we do not report a full bifurcation diagram for this case. Instead, we track the evolution of four branches in Figure \ref{fig:bifurcation_Ra_S1_diagram} which we use to generate initial guesses at $Ra=100{,}000$ and $S=1{,}000$. Branches 1, 2 and 3 correspond to the primary bifurcations that arise from the third, fourth and seventh eigenmodes at $\Ra_c^{(3)} = 19{,}647$, $\Ra_c^{(4)} = 23{,}408$ and $\Ra_c^{(7)} = 47{,}364$. The tracking of branch 6 allows us to compute a disconnected branch in the final diagram for $0\leq \Ra \leq 100{,}000$ with $S=1{,}000$. We plot branch 6 with a dashed line to highlight the fact that it will become a disconnected branch in the final diagram. We have chosen a step size of $\Delta \Ra = 500$ for this diagram.

\begin{figure}[htbp!]
	\centering
	\begin{tabular}{cc}
		\includegraphics[width=7.0cm]{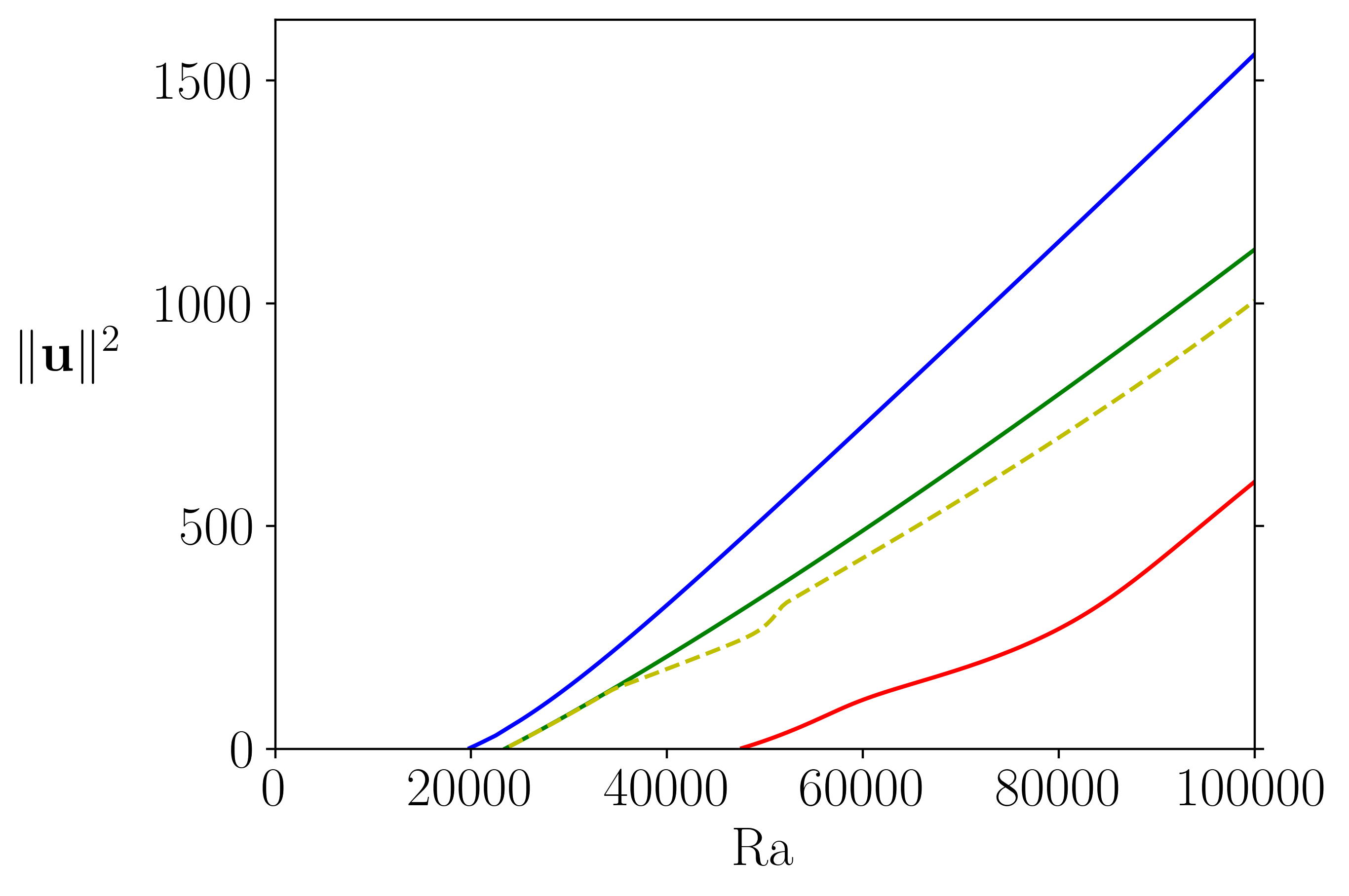} &
		\includegraphics[width=7.0cm]{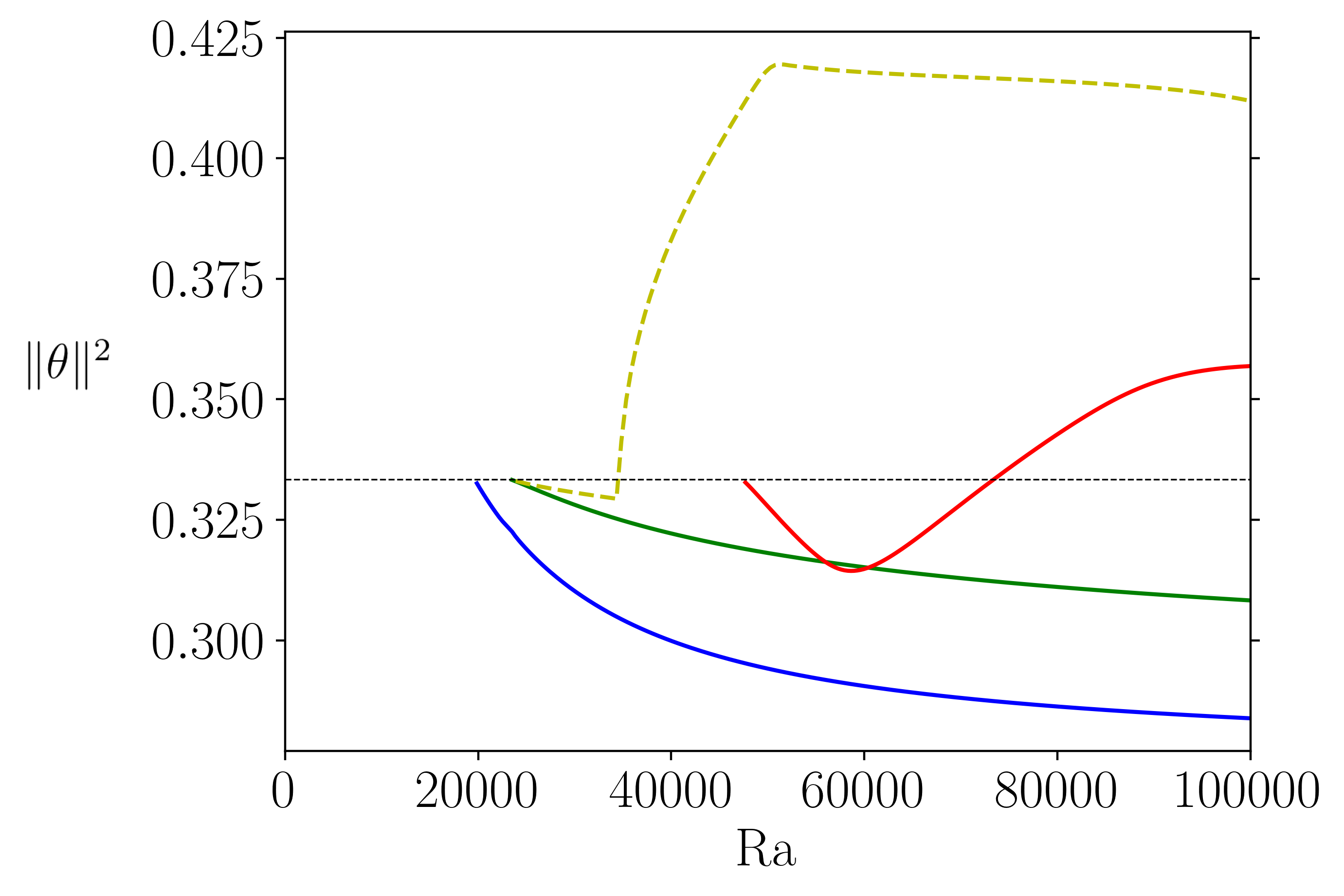} \\
	\end{tabular}
    \includegraphics[width=7.8cm]{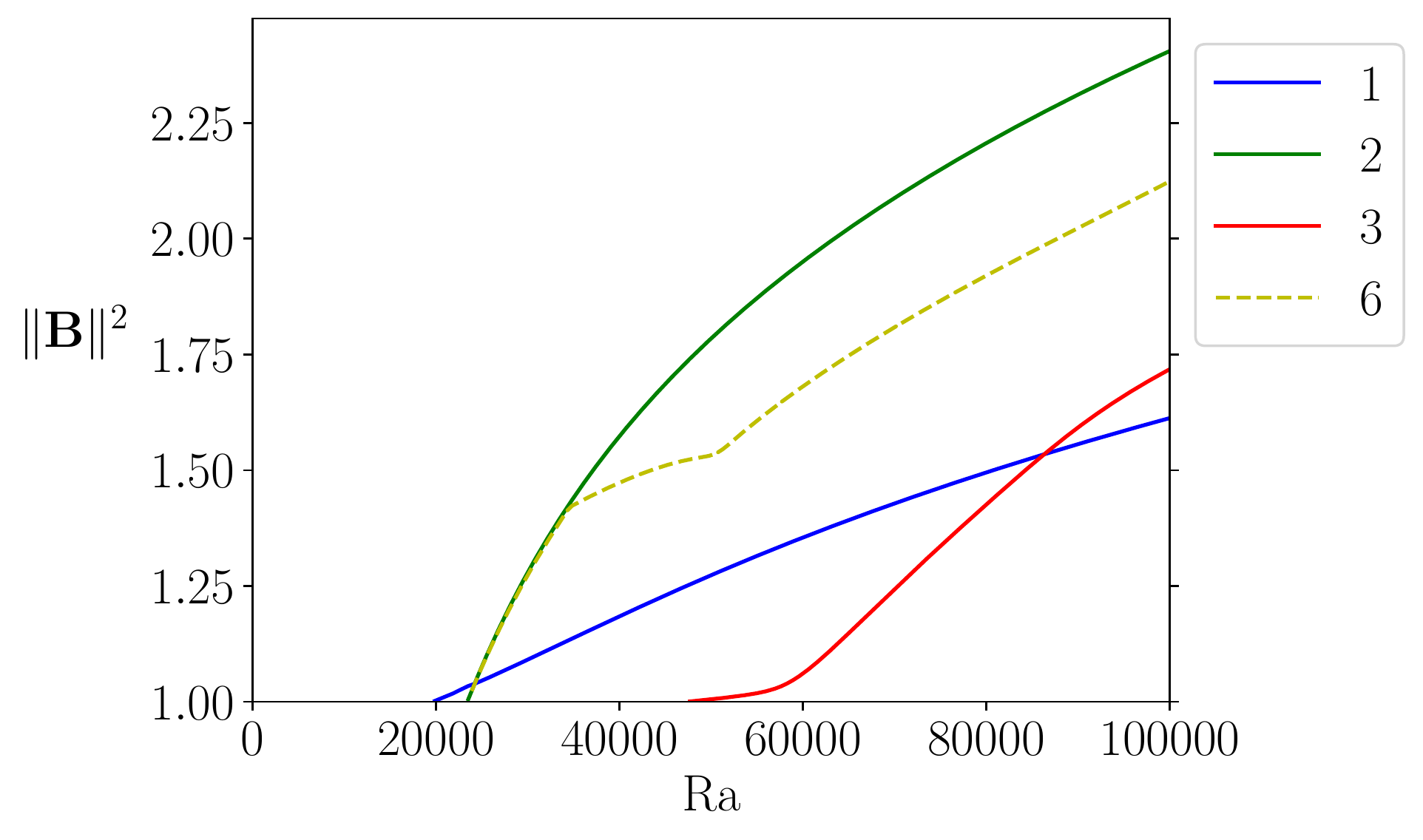}
	\caption{Bifurcation diagram over $0 \leq \Ra \leq 100{,}000$ with $S=1$.}
	\label{fig:bifurcation_Ra_S1_diagram}
\end{figure}

We also include more detailed graphs of the evolution of these four branches in Figure \ref{fig:bifurcation_Ra_S1_diagram_branch_1_2_3_6}. This figure includes plots of the magnitude of the velocity, temperature and magnetic field and the real part of the eigenvalues with largest real part. The number of displayed eigenvalues is manually chosen for each branch, but we typically display the first 4 or 6 eigenvalues. We highlight eigenvalues with $\mathcal{R}(\lambda)=0$ in green if the corresponding imaginary part is zero and in red if it is non-zero to indicate steady bifurcations and Hopf bifurcations. Note that our chosen step size $\Delta \Ra$ is not always sufficient to find eigenvalues that fulfil $\mathcal{R}(\lambda)=0$ precisely. In any case, we highlight the eigenvalue that has the smallest absolute magnitude in the real part. This provides a sufficient indication if the corresponding imaginary parts is zero or non-zero. 

Since the presented branches in Figure  \ref{fig:bifurcation_Ra_S1_diagram_branch_1_2_3_6} do not show noticeable differences to the non-magnetic case from \cite{Boulle2022}, we do not further analyse each figure in detail here, but we will do so in the next two sections.

\begin{figure}[htbp!]
	\begin{subfigure}{\textwidth}
	\centering
	\vspace{-0.5cm}
	\includegraphics[width=16.0cm]{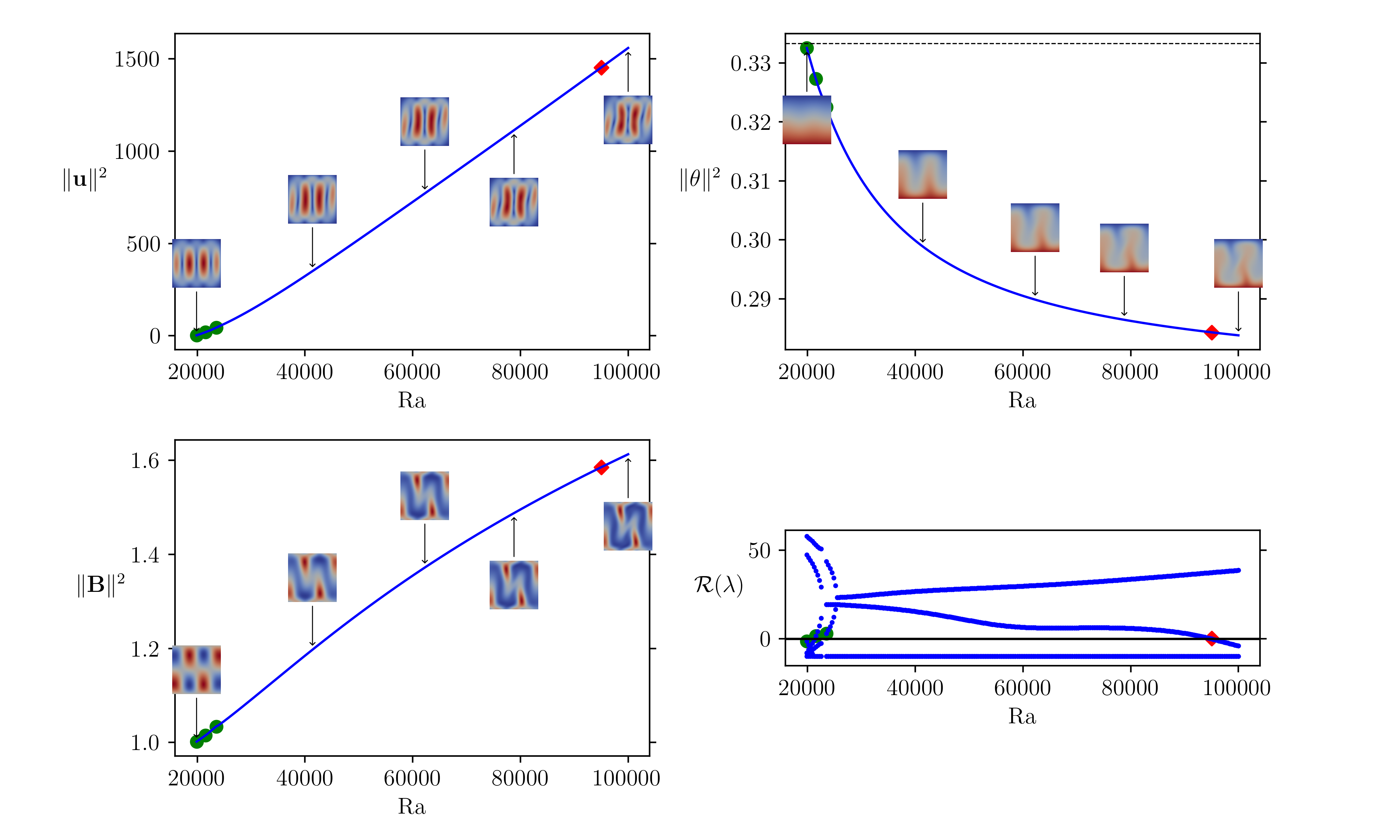}
	\vspace{-0.5cm}
	\subcaption{Evolution of branch 1 over $\Ra$ for $S=1$. Eigenvalues with $\mathcal{R}(\lambda)=0$ are highlighted in green if $\mathcal{I}(\lambda) = 0$ and in red if $\mathcal{I}(\lambda) \neq 0$.}
	\label{fig:bifurcation_Ra_S1_diagram_branch_1}
    \end{subfigure}

    \vspace{0.7cm}
    \begin{subfigure}{\textwidth}
	\centering
	\includegraphics[width=16.0cm]{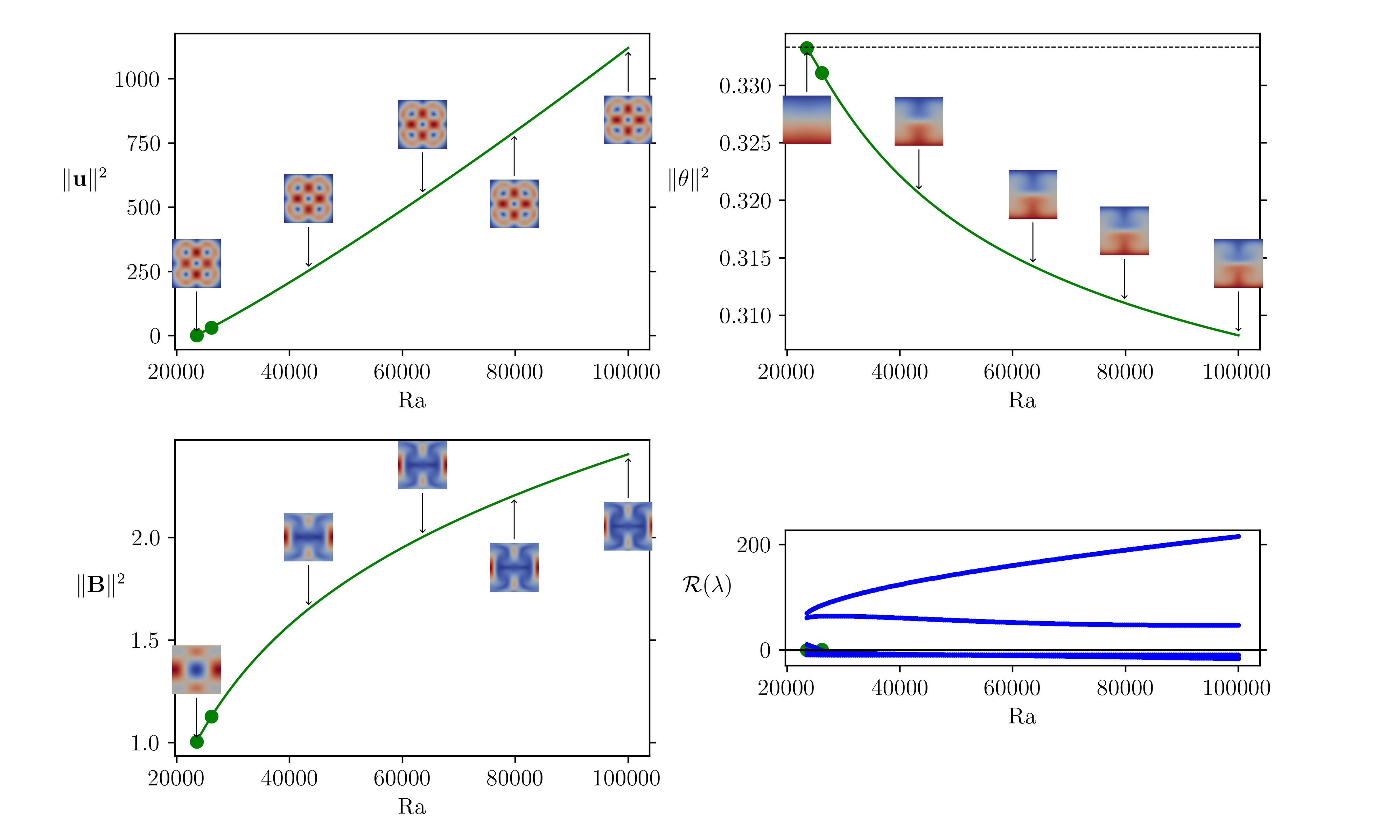}
	\vspace{-0.5cm}
	\subcaption{Evolution of branch 2 over $\Ra$ for $S=1$.}
	\label{fig:bifurcation_Ra_S1_diagram_branch_2}
    \end{subfigure}
\end{figure}
\begin{figure}
    \ContinuedFloat
    \begin{subfigure}{\textwidth}
	\centering
	\vspace{-0.5cm}
	\includegraphics[width=16.0cm]{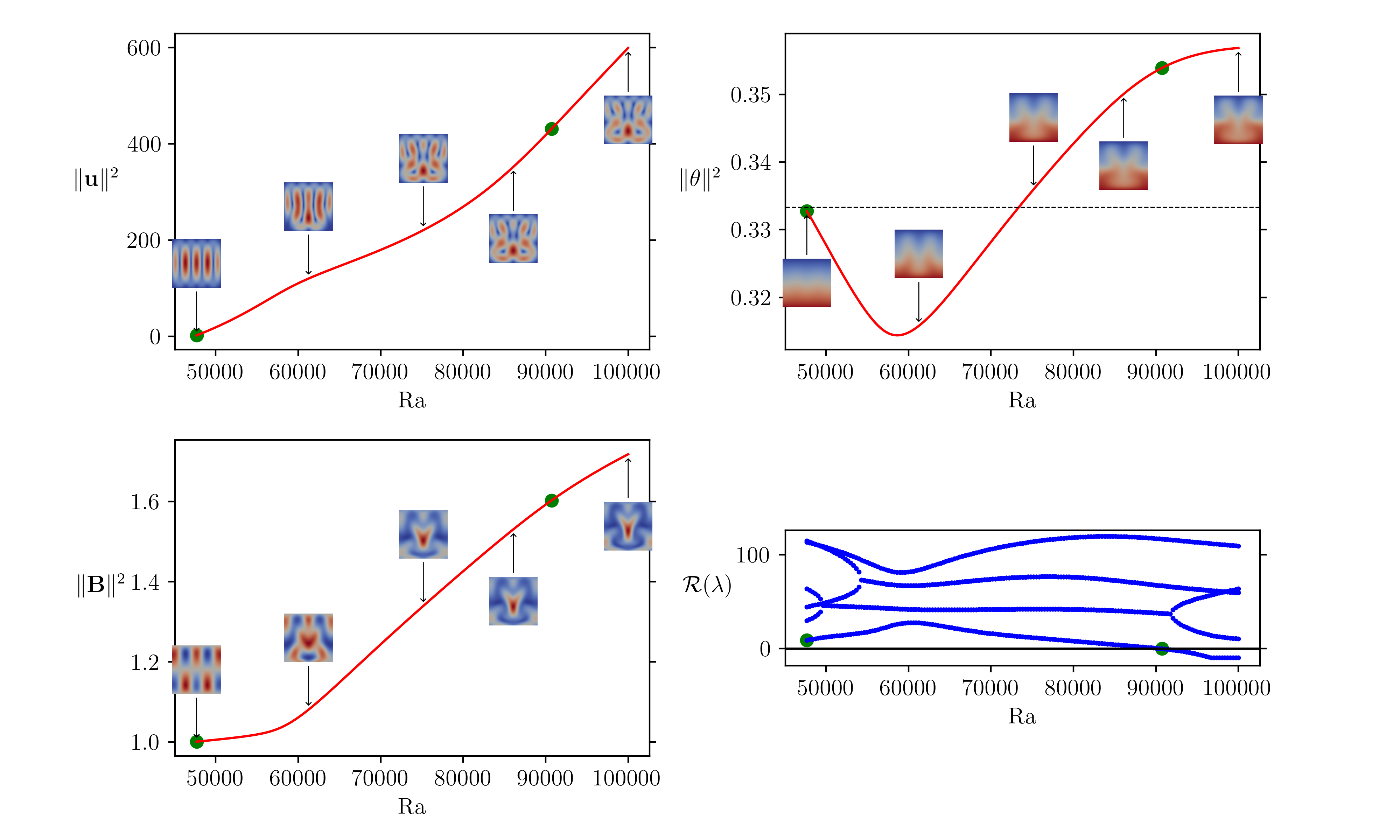}
	\vspace{-0.5cm}
	\subcaption{Evolution of branch 3 over $\Ra$ for $S=1$.}
	\label{fig:bifurcation_Ra_S1_diagram_branch_3}
    \end{subfigure}

   \vspace{0.7cm}
   \begin{subfigure}{\textwidth}
	\centering
	\includegraphics[width=16.0cm]{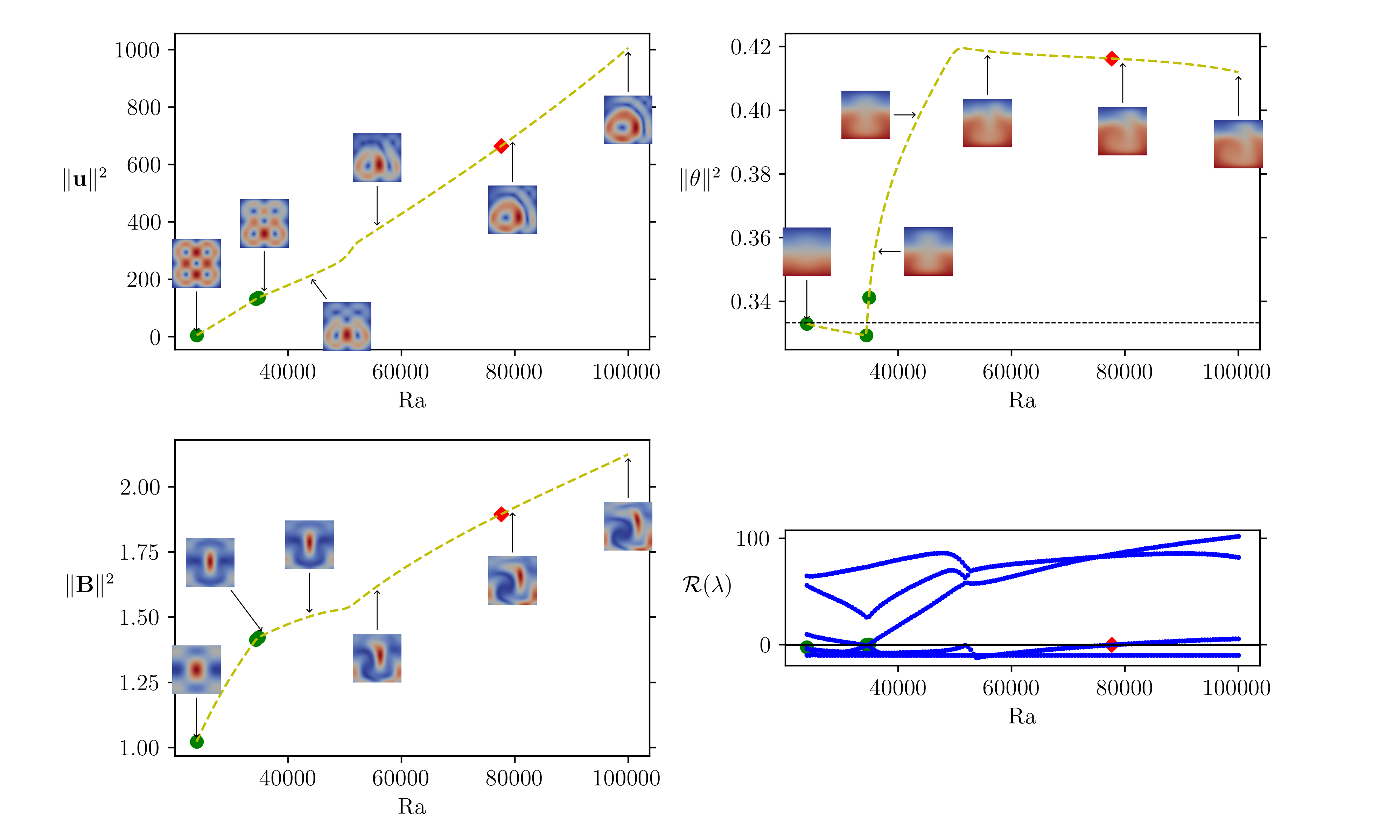}
	\vspace{-0.5cm}
	\subcaption{Evolution of branch 6 over $\Ra$ for $S=1$.}
	\label{fig:bifurcation_Ra_S1_diagram_branch_6}
	\end{subfigure}
\vspace{0.5cm}
    \caption{Evolution of branch 1, 2, 3 and 6 over $\Ra$ for $S=1$.}
	\label{fig:bifurcation_Ra_S1_diagram_branch_1_2_3_6}
\end{figure}

\FloatBarrier
\subsection{Bifurcation analysis for  $1\leq S \leq 1{,}000$  with  \mbox{$\Ra=100{,}000$}}\label{sec:bif2}
We continue by choosing the coupling number $S$ as the bifurcation parameter for fixed $\Ra=100{,}000$ and use the four computed solutions at $\Ra=100{,}000$ from the previous Section \ref{sec:bif1} as initial guesses in our deflated continuation algorithm over $1\leq S \leq 1{,}000$. Note that we perform forward deflated continuation this time and we choose a step size of $\Delta S = 10/3$. 
 
First, we report the growth rate of the most unstable eigenmodes in Figure~\ref{fig:eigsplot_S1000_PM1}. Note that Figure \ref{fig:eigsplot_S1000_PM1} is a continuation of Figure \ref{fig:eigsplot_S1_PM1} at $\Ra=100{,}000$ in the direction of $S$ and hence starts with 11 unstable eigenmodes at \mbox{$S=1$}. Also note that the smallest growth rate at $S=1$ is barely visible due to the small magnitude of the real part of its eigenvalue. Contrary to the Rayleigh number, the growth rates decrease with increasing coupling number, leaving only 5 unstable eigenmodes at $S=1{,}000$. 

Analogously to  \eqref{eq:EVproblemRac}, we can compute the critical coupling numbers $S_c$ by solving the generalised eigenvalue problem

\begin{equation}\label{eq:EVproblemS}
	\resizebox{\textwidth}{!}{%
		$\begin{bmatrix}
			\FF & -\nabla & \Ra\, \Pr\, \mathbf{e}_3 & \GG& 0\\
			\nabla \cdot & 0 & 0 & 0 & 0\\
			- \nabla \theta_0 \cdot & 0 & \Delta - \u_0 \cdot \nabla & 0 & 0\\
			0 & 0 & 0 & \frac{\Pr}{\Pm}\nabla \nabla \cdot & - \vcurl \\
			- \times \B_0 & 0 & 0 &  -\frac{\Pr}{\Pm}\scurl \,  - \u_0 \cdot & -I 
		\end{bmatrix}
		\begin{bmatrix}
			\tilde{\u}  \\ p \\  \tilde{\theta} \\ \tilde{\B}\\ E
		\end{bmatrix}	
		= S_c
		\begin{bmatrix}
			\B_0 \times (\cdot \times \B_0) & 0 & 0 & 0 & \B_0 \times \\
			0 & 0 & 0 & 0 & 0\\
			0 & 0 & 0 & 0 & 0 \\
			0 & 0 & 0 & 0 & 0\\
			0 & 0 & 0 & 0 & 0\\
		\end{bmatrix}
		\begin{bmatrix}
			\tilde{\u}  \\ p \\ \tilde{\theta} \\ \tilde{\B} \\ E
		\end{bmatrix}$}.
\end{equation}
Note that the other terms in the linearisation of the Lorentz force which would appear on the right-hand side vanish, since both $\u_0=\mathbf{0}$ and $E_0=0$.
This results in the critical coupling numbers $S_c^{(1)} = 12$, $S_c^{(2)} = 28$, $S_c^{(3)} = 47$, $S_c^{(4)} = 329$, $S_c^{(5)} = 463$ and $S_c^{(6)} = 940$. Plots of the eigenfunctions can be found in Figure \ref{fig:eigenmodes_S1_PM1_2}.

\begingroup
\renewcommand{\arraystretch}{2.0}
\begin{figure}[htbp!]
	\centering
	\newcommand{\mywidth}{2.0cm}
	\begin{tabular}{cccccc}
		$S_c^{(1)} = 12$ & $S_c^{(2)} = 28$ & $S_c^{(3)} = 47$ & $S_c^{(4)} = 329$ & $S_c^{(5)} = 463$ & $S_c^{(6)} = 940$ \\
		\includegraphics[width=\mywidth]{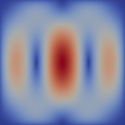} &
		\includegraphics[width=\mywidth]{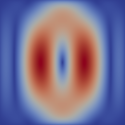} &
		\includegraphics[width=\mywidth]{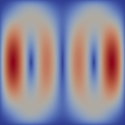} &
		\includegraphics[width=\mywidth]{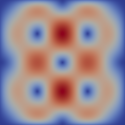} &
		\includegraphics[width=\mywidth]{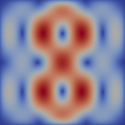} &
		\includegraphics[width=\mywidth]{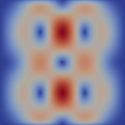} \\
		\includegraphics[width=\mywidth]{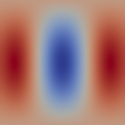} &
		\includegraphics[width=\mywidth]{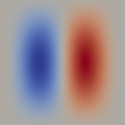} &
		\includegraphics[width=\mywidth]{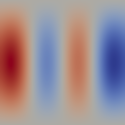} &
		\includegraphics[width=\mywidth]{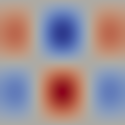} &
		\includegraphics[width=\mywidth]{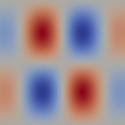} &
		\includegraphics[width=\mywidth]{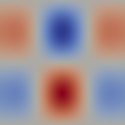} \\
		\includegraphics[width=\mywidth]{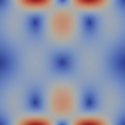} &
		\includegraphics[width=\mywidth]{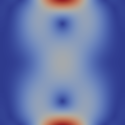} &
		\includegraphics[width=\mywidth]{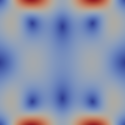} &	
		\includegraphics[width=\mywidth]{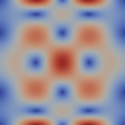} &
		\includegraphics[width=\mywidth]{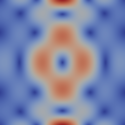} &
		\includegraphics[width=\mywidth]{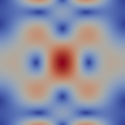} \\
	\end{tabular}
	\vspace{0.2cm}
	\caption{Eigenfunctions of the primary bifurcations that emanate from the conducting state \eqref{eq:trivialsol}  for $1\leq S \leq 1{,}000$  with  \mbox{$\Ra=100{,}000$}. The top row shows the velocity, the middle row the temperature and the bottom row the magnetic field for the different critical coupling numbers. \label{fig:eigenmodes_S1_PM1_2}}
\end{figure}
\endgroup

In Figure \ref{fig:bifurcation_S_Ra100000_diagram} we present the continuation of the 4 branches that we started with in the last section. As we can see this will result in 5 initial guesses at $\Ra = 100{,}000$ and $S=1{,}000$ since branch 3 gives rise to two bifurcations. Remember that the main goal of Section \ref{sec:bif1} and Section \ref{sec:bif2} was to deliver these initial guesses to compute the bifurcation diagram over Ra for $S=1{,}000$. This time we analyse each evolution in more detail. 

\begin{figure}[htbp!]
	\centering
	\begin{tabular}{cc}
		\includegraphics[width=7.0cm]{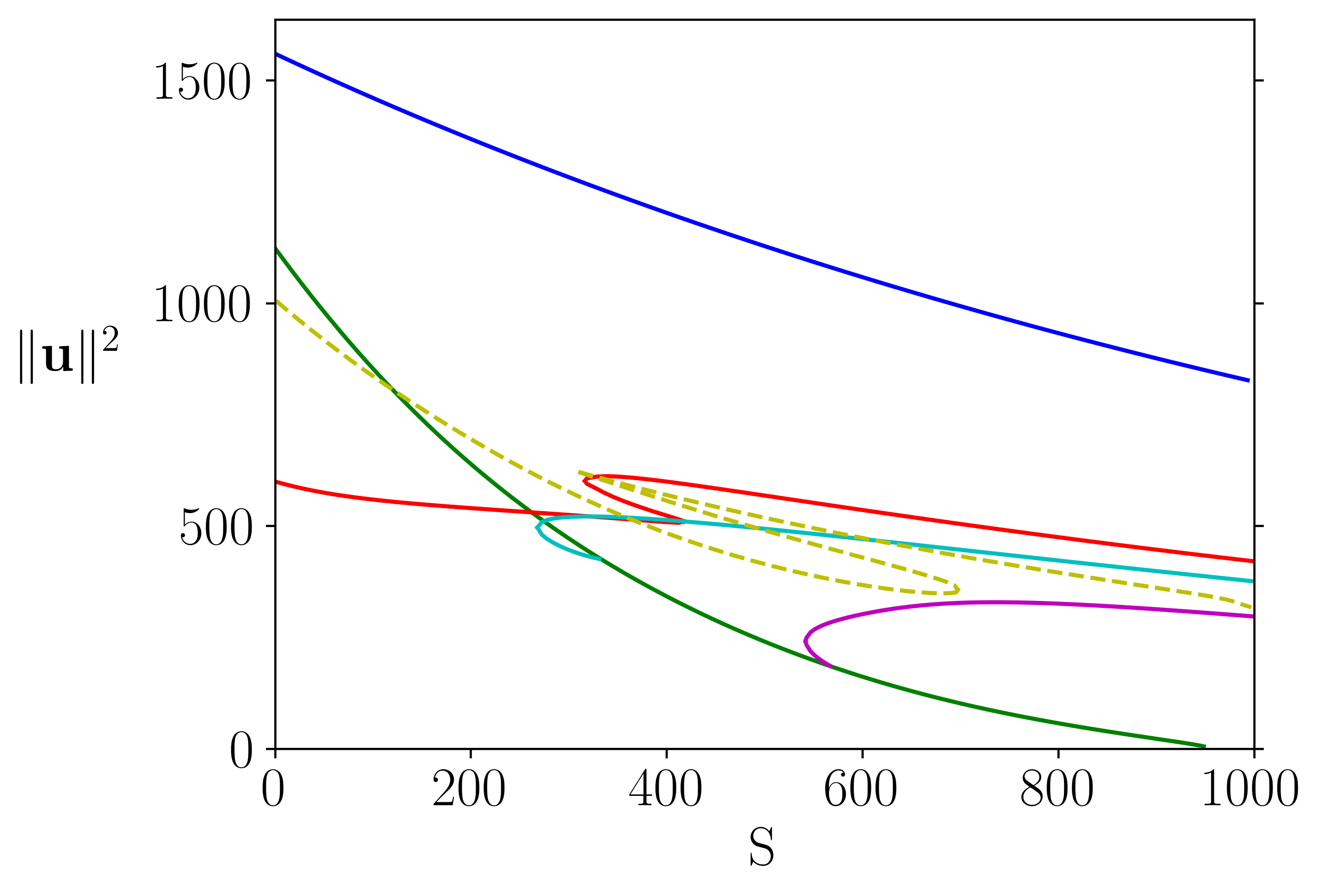} &
		\includegraphics[width=7.0cm]{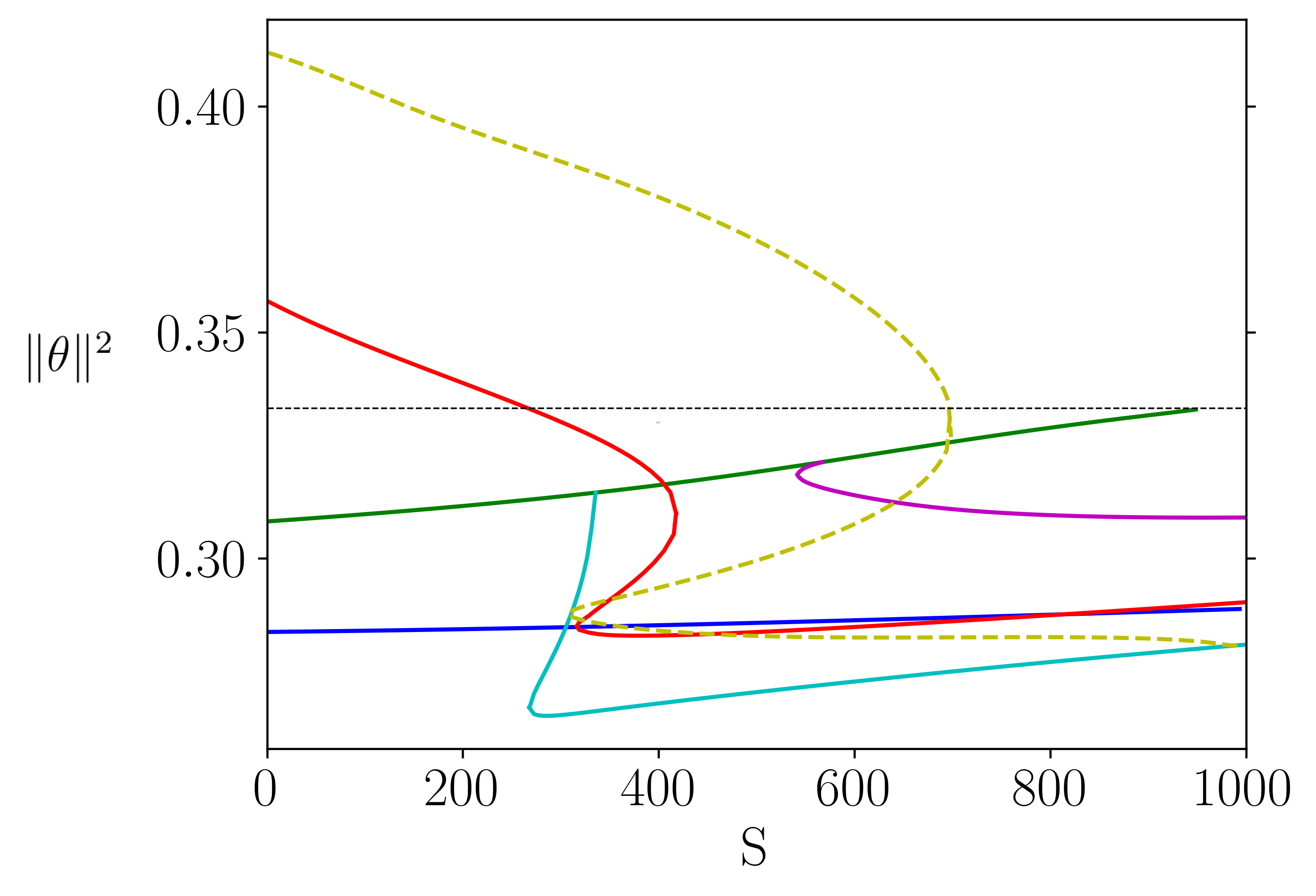} \\
	\end{tabular}
	\includegraphics[width=7.8cm]{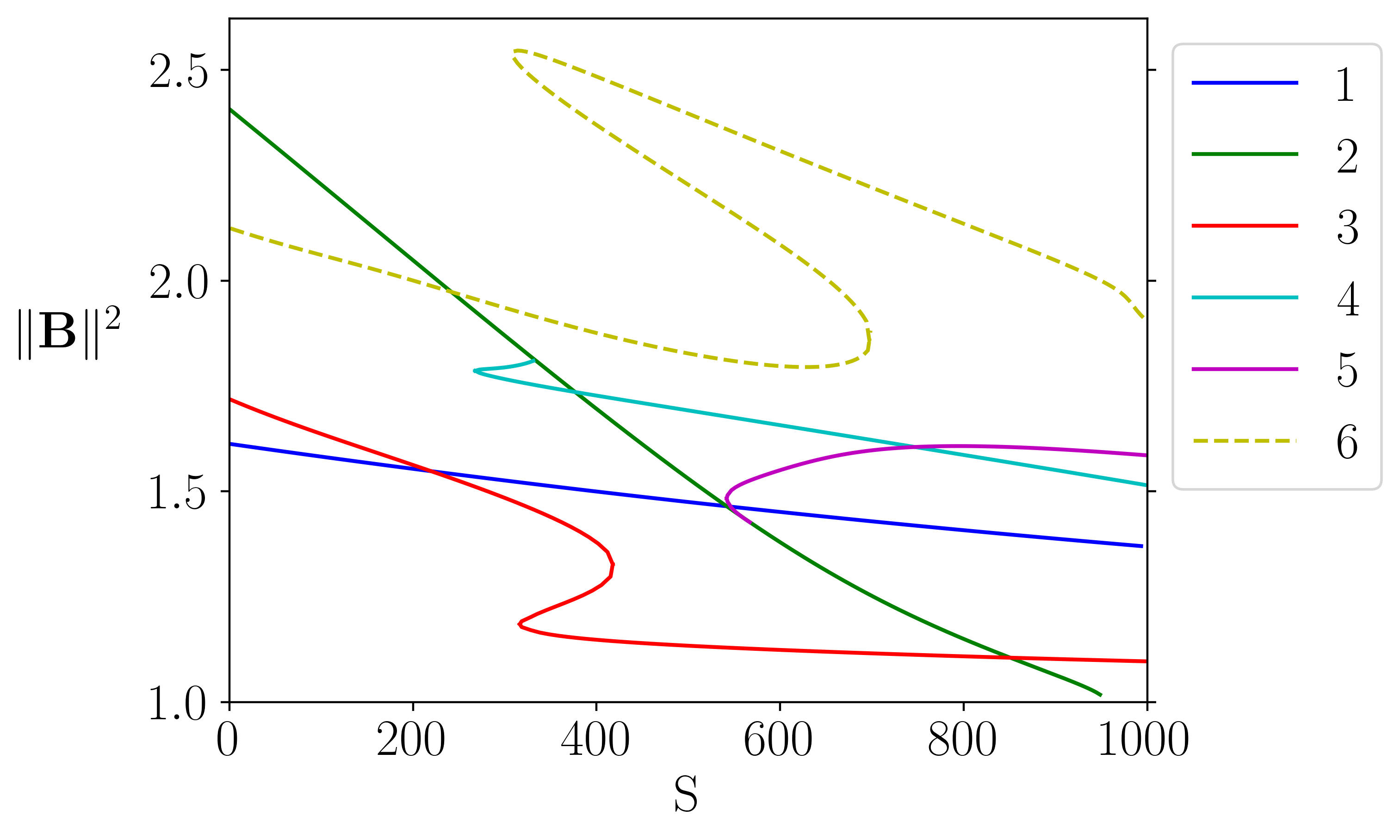}
	\caption{Bifurcation diagram over $1 \leq S \leq 1{,}000$ with $\Ra=100{,}000$.}
	\label{fig:bifurcation_S_Ra100000_diagram}
\end{figure}

Figure \ref{fig:bifurcation_S_Ra100000_diagram_branch_1} shows the evolution of branch 1. The shape of $\u$, $\theta$ and $\B$ does not change notably with increasing $S$. However, it is very interesting to notice that the unstable solution at $S=1$ starts to become more stable with increasing $S$ until it turns stable at $S \approx 700$. The corresponding imaginary part is positive, which indicates that a Hopf bifurcation is emerging at this point. 


Branch 2 transitions at $S=940$ into the conducting state, as can be seen in Figure \ref{fig:bifurcation_S_Ra100000_diagram_branch_2}. Since the sixth critical coupling number $S^{(6)}_c$ is also located at $S = 940$ this indicates that branch 2 is a primary bifurcation in the diagram over S. Hence, the fourth primary bifurcation  in the diagram over Ra investigated in the previous section and the sixth primary bifurcation in the diagram over S correspond to the same bifurcation. 

Moreover, branch 2 has two eigenvalues with $\mathcal{R}(\lambda)=0$ and $\mathcal{I}(\lambda)=0$ at $S \approx 270$ and $S \approx 540$ highlighted with a green dots in Figure \ref{fig:bifurcation_S_Ra100000_diagram_branch_2}. The two bifurcations that emerge at these points result in branch 4 and branch 5 which are further illustrated in Figure \ref{fig:bifurcation_S_Ra100000_diagram_branch_4} and \ref{fig:bifurcation_S_Ra100000_diagram_branch_5}.

Branch 3 is shown in Figure \ref{fig:bifurcation_S_Ra100000_diagram_branch_3} and has two turning points in the interval at around $S \approx 420$  and $S\approx 315$ which are also indicated by eigenvalues with vanishing real part in the eigenvalue plots. Note that the S-shaped form creates multi-valued regions in the graph. Therefore, we include three eigenvalue plots for each region.

The eigenvalue plots of branch 4 and branch 5 in Figure \ref{fig:bifurcation_S_Ra100000_diagram_branch_4} and Figure  \ref{fig:bifurcation_S_Ra100000_diagram_branch_5}  have one turning point. In both graphs the zero eigenvalue at the start of the bifurcation and at the turning point can be seen. Finally, branch 6 in Figure \ref{fig:bifurcation_S_Ra100000_diagram_branch_6} has again a more complex S-shaped form. 

\begin{figure}[htbp!]
	\begin{subfigure}{\textwidth}
		\centering
		\includegraphics[width=16.0cm]{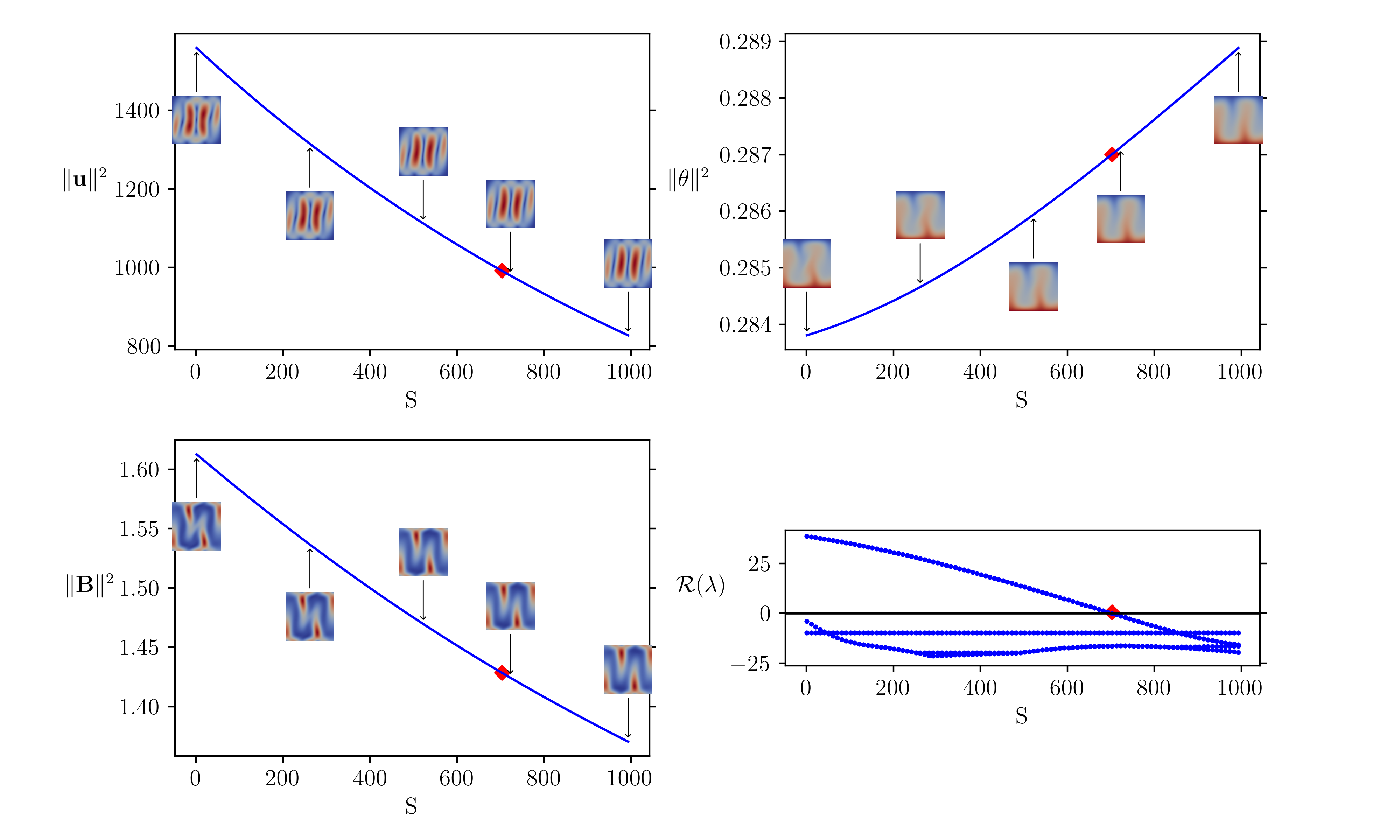}
		\vspace{-0.5cm}
		\subcaption{Evolution of branch 1 over $S$ for $\Ra=100{,}000$.}
		\label{fig:bifurcation_S_Ra100000_diagram_branch_1}
	\end{subfigure}
	
	\vspace{1.0cm}
	\begin{subfigure}{\textwidth}
		\centering
		\includegraphics[width=16.0cm]{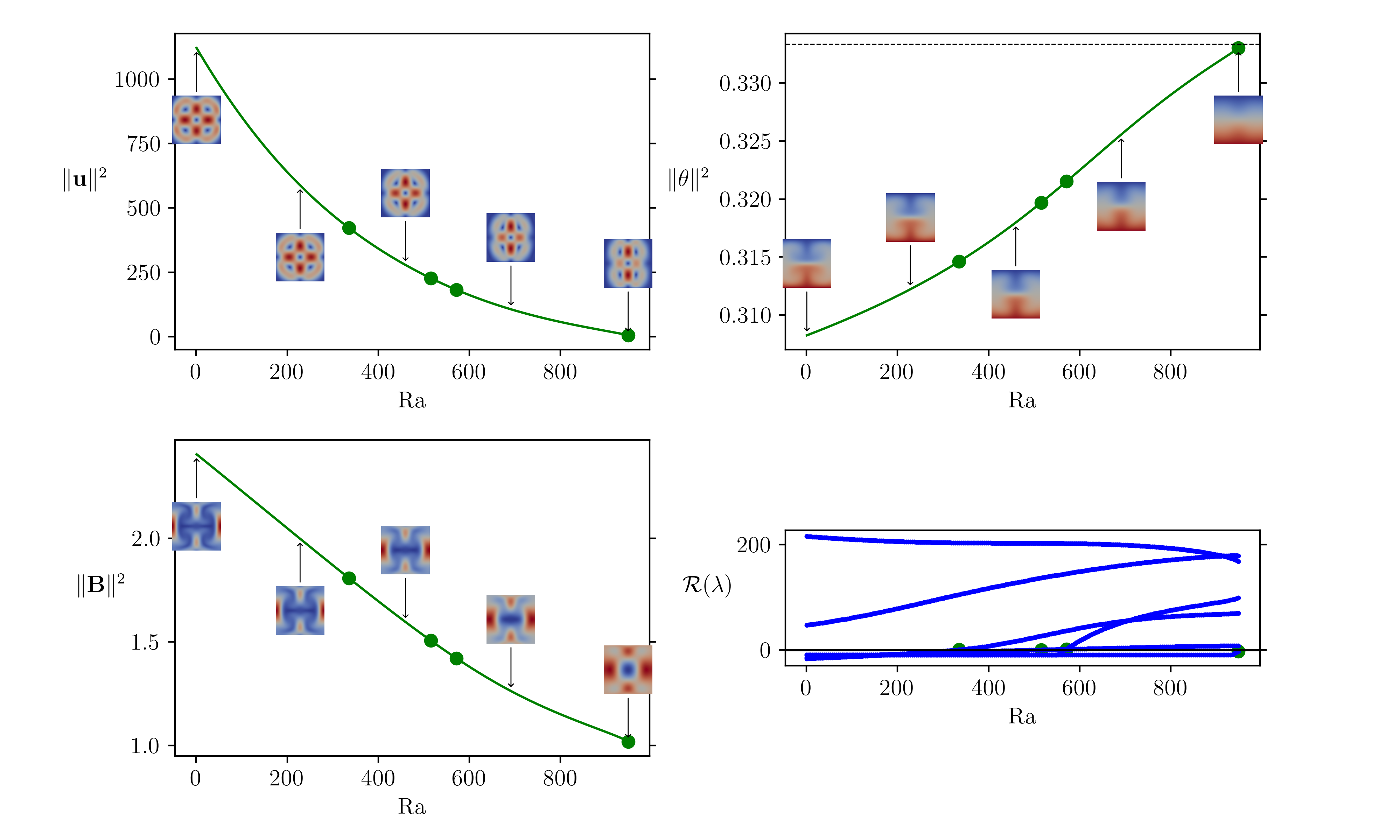}
		\vspace{-0.5cm}
		\subcaption{Evolution of branch 2 over $S$ for $\Ra=100{,}000$.}
		\label{fig:bifurcation_S_Ra100000_diagram_branch_2}
	\end{subfigure}
\end{figure}
\begin{figure}
	\ContinuedFloat
	\begin{subfigure}{\textwidth}
		\centering
		\vspace{-0.3cm}
		\includegraphics[width=16.0cm]{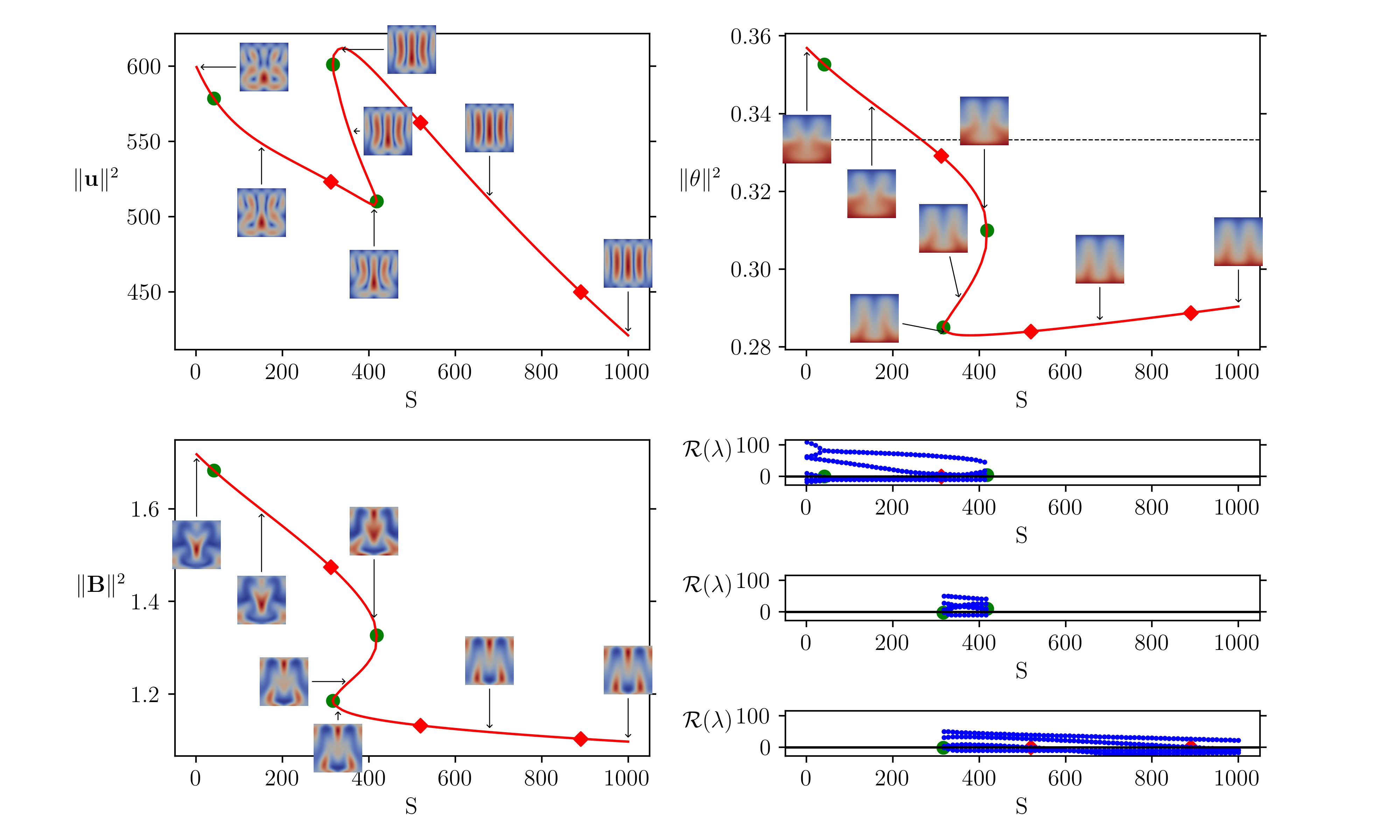}
		\vspace{-0.5cm}
		\subcaption{Evolution of branch 3 over $S$ for $\Ra=100{,}000$.}
		\label{fig:bifurcation_S_Ra100000_diagram_branch_3}
	\end{subfigure}

	\vspace{0.5cm}
	\begin{subfigure}{\textwidth}
		\centering
		\includegraphics[width=16.0cm]{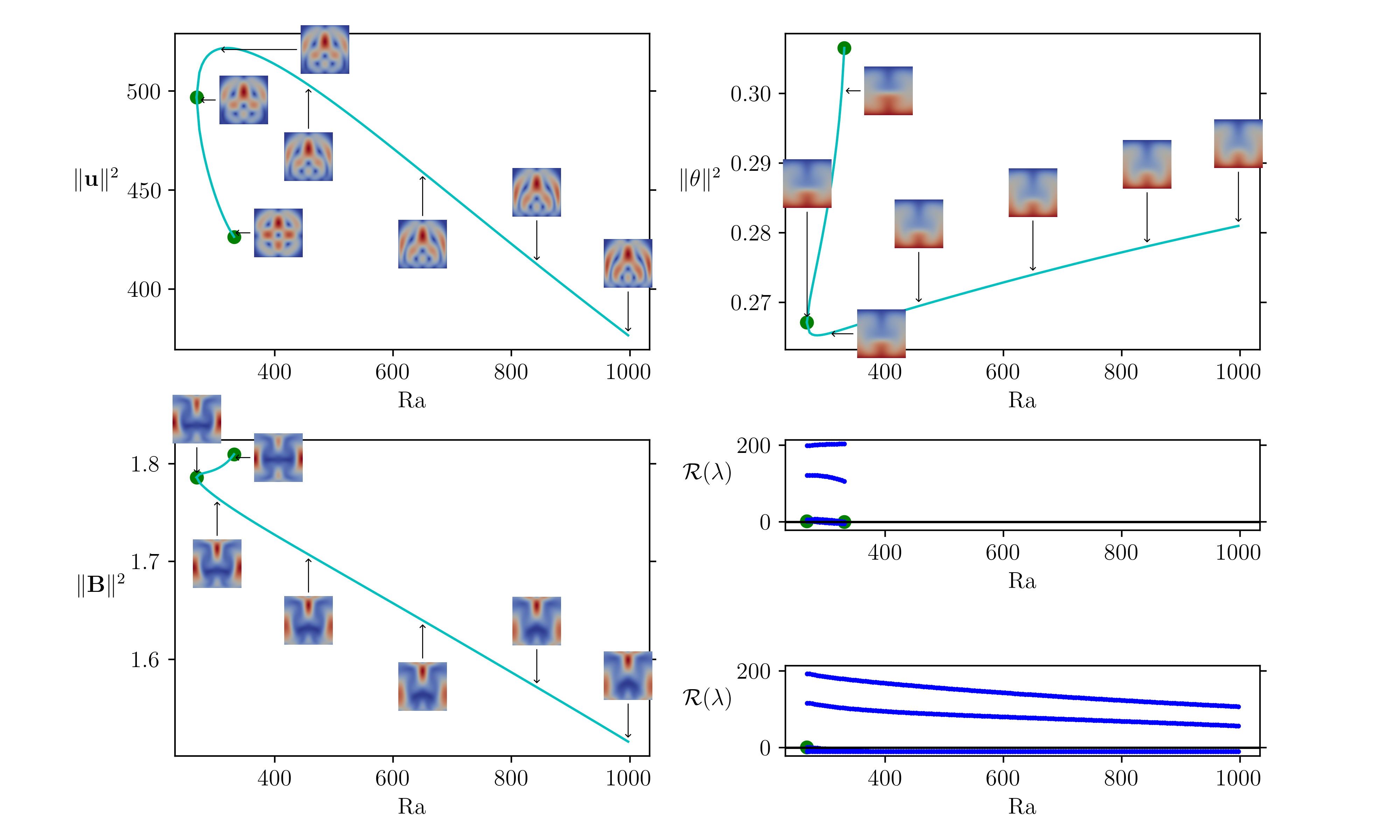}
		\vspace{-0.2cm}
		\subcaption{Evolution of branch 4 over $S$ for $\Ra=100{,}000$.}
		\label{fig:bifurcation_S_Ra100000_diagram_branch_4}
	\end{subfigure}
\end{figure}
\begin{figure}
\ContinuedFloat
\begin{subfigure}{\textwidth}
	\centering
	\vspace{-0.5cm}
	\includegraphics[width=16.0cm]{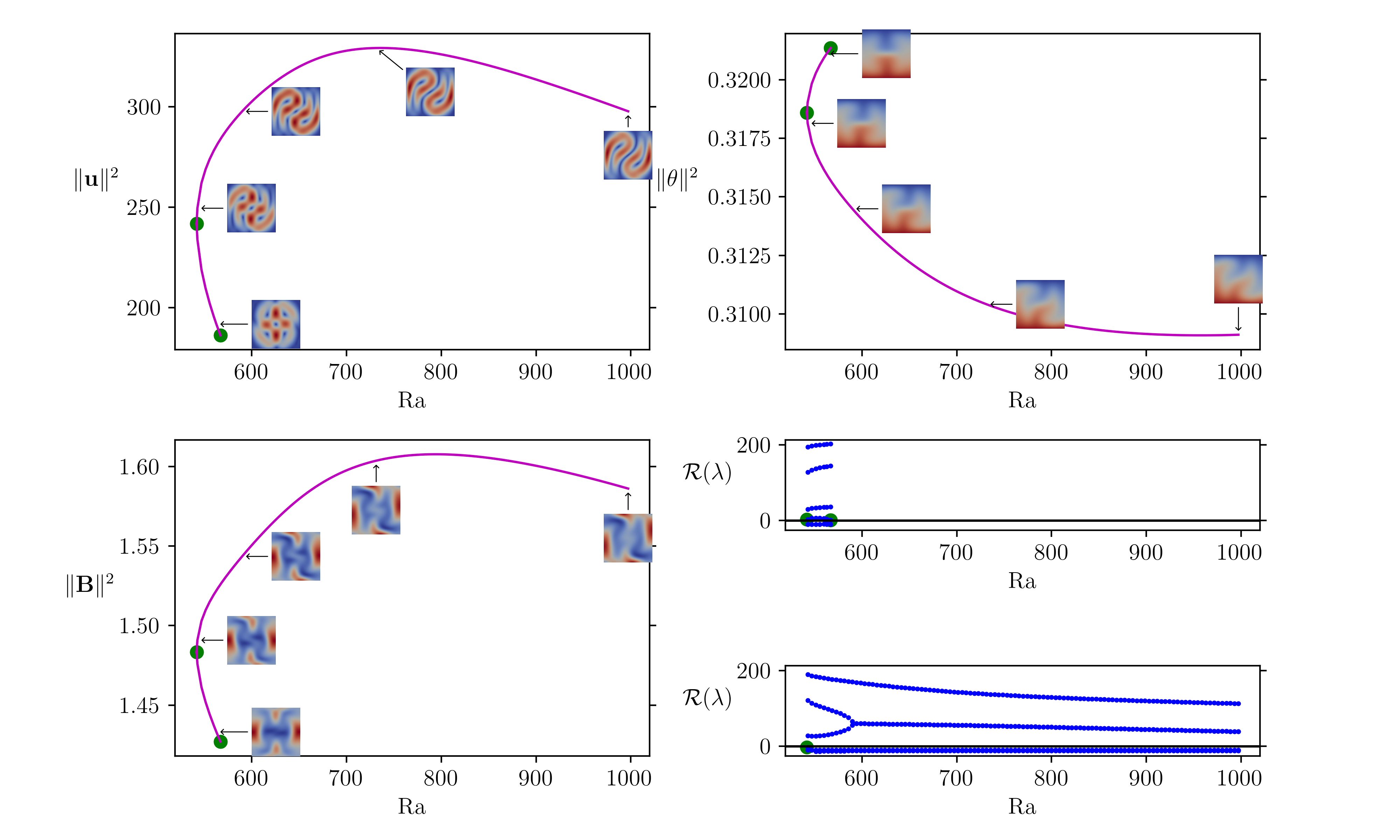}
	\vspace{-0.5cm}
	\subcaption{Evolution of branch 5 over $S$ for $\Ra=100{,}000$.}
	\label{fig:bifurcation_S_Ra100000_diagram_branch_5}
\end{subfigure}

\vspace{0.5cm}
\begin{subfigure}{\textwidth}
	\centering
	\includegraphics[width=16.0cm]{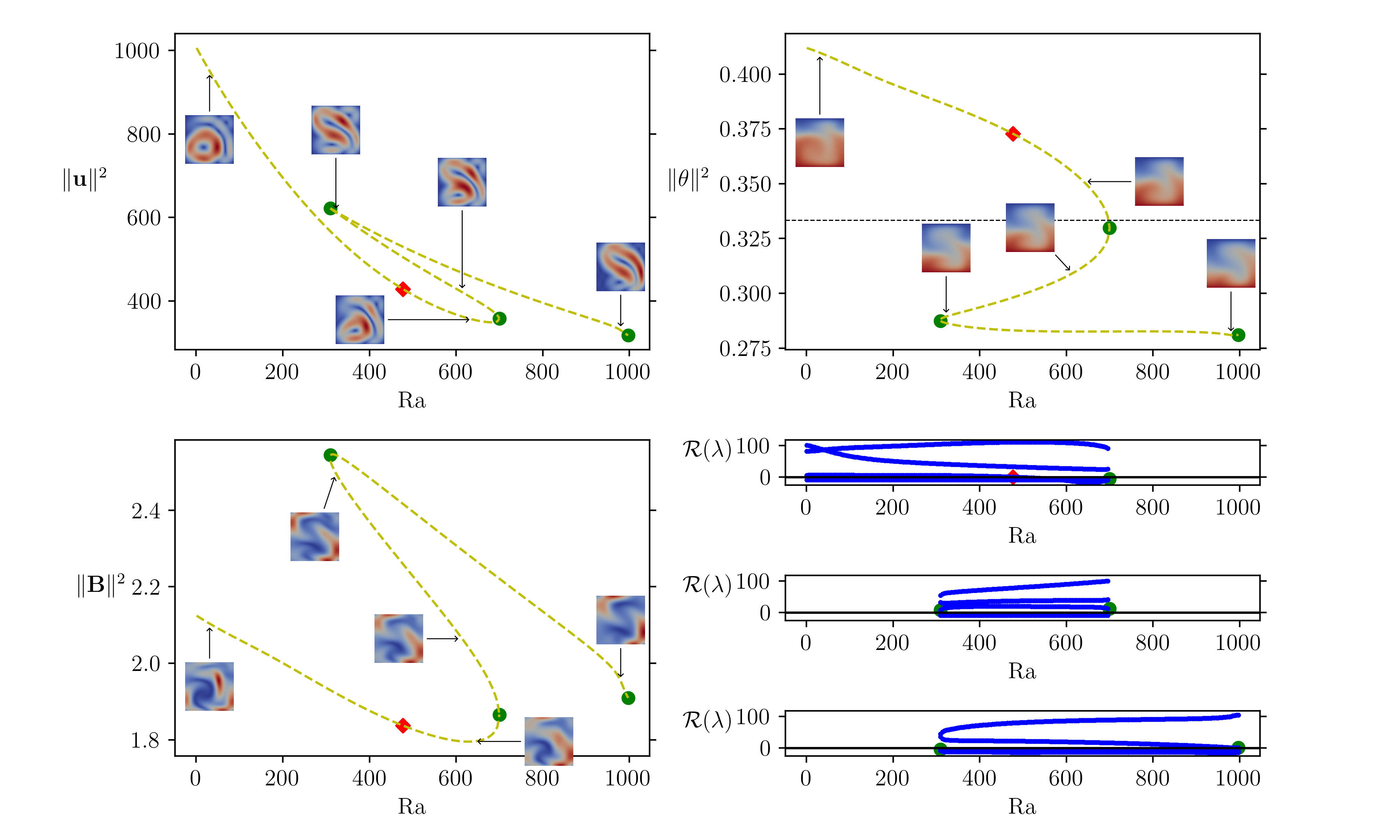}
	\vspace{-0.5cm}
	\subcaption{Evolution of branch 6 over $S$ for $\Ra=100{,}000$.}
	\label{fig:bifurcation_S_Ra100000_diagram_branch_6}
\end{subfigure}
	\caption{Evolution of branches 1, 2, 3, 4, 5 and 6 over $S$ for $\Ra = 100{,}000$.}
	\label{fig:bifurcation_S_Ra100000_diagram_branch_1_2_3_4_5_6}
\end{figure}

\FloatBarrier
\subsection{Bifurcation analysis for  $0\leq \Ra \leq 100{,}000$  with  \mbox{$S=1{,}000$}}\label{sec:bif3}

Finally, we investigate the bifurcation diagram for $0\leq \Ra \leq 100{,}000$  with  \mbox{$S=1{,}000$}. The growth rates of the five eigenmodes which emanate from the conducting state are shown in Figure \ref{fig:eigsplot_S1000_PM1_3}. These five eigenmodes are the extension of the five eigenmodes at $S=1{,}000$ from Figure \ref{fig:eigsplot_S1000_PM1}. In the following, we focus on the first four eigenmodes (since the fifth one only exists in the short interval $[99{,}041; 100{,}000]$) and display them in Figure \ref{fig:eigenmodes_S1000_PM1_3}.

 It is worth mentioning that the velocity and temperate eigenmodes both show a pattern that is oriented in the direction $(0,1)^\top$ of the trivial magnetic field $\B_0$. This indicates that for larger coupling number at $S=1{,}000$ the instabilities which are aligned with the magnetic field occur for smaller Rayleigh numbers.  In general, the order of eigenmodes changes with growing $S$. For example the 7th eigenmode in Figure \ref{fig:bifurcation_Ra_S1_diagram} at $S=1$ corresponds now to the fourth eigenmode at $S=1{,}000$.
 
 For a comparison, we have also included critical Rayleigh numbers and eigenmodes plots for a trivial magnetic field in the perpendicular direction of $(1,0)^\top$ in Figure \ref{fig:eigenmodes_S1_PM1_4}. Here, we just have three primary bifurcations left in the interval $0 \leq \Ra \leq 100{,}000$, but the third velocity eigenmode is now oriented in the direction of the magnetic field while the eigenmodes oriented in $(0,1)^\top$ occur here for $\Ra>100{,}000$. This seems to underline our observation that for high coupling numbers instabilities whose velocity field is aligned with the magnetic field occur for smaller Rayleigh numbers.

\begingroup
\renewcommand{\arraystretch}{2.0}
\begin{figure}[htbp!]
	\centering
	\newcommand{\mywidth}{2.0cm}
	\begin{tabular}{cccc}
		$\Ra_c^{(1)} = 20{,}029$ & $\Ra_c^{(2)} = 23{,}404$ & $\Ra_c^{(3)} = 32{,}095$ & $\Ra_c^{(4)} = 58{,}571$ \\ 
		\includegraphics[width=\mywidth]{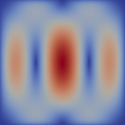} &
		\includegraphics[width=\mywidth]{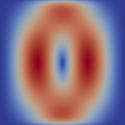} &
		\includegraphics[width=\mywidth]{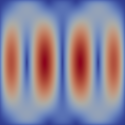} &
		\includegraphics[width=\mywidth]{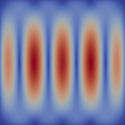} \\
		\includegraphics[width=\mywidth]{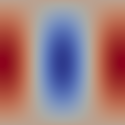} &
		\includegraphics[width=\mywidth]{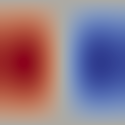} &
		\includegraphics[width=\mywidth]{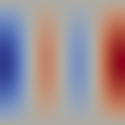} &
		\includegraphics[width=\mywidth]{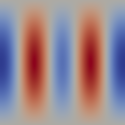} \\
		\vspace{0.5cm}
		\includegraphics[width=\mywidth]{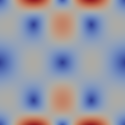} &
		\includegraphics[width=\mywidth]{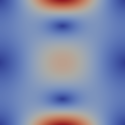} &
		\includegraphics[width=\mywidth]{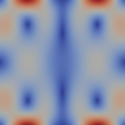} &	
		\includegraphics[width=\mywidth]{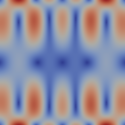} \\
	\end{tabular}
	\caption{First 4 eigenfunctions of the primary bifurcations that emanate from the conducting state \eqref{eq:trivialsol}  for  $0\leq \Ra \leq 100{,}000$  with  \mbox{$S=1{,}000$}. Note that all the velocity and temperature eigenmodes show a pattern in pointing in the direction of $\B_0 = (0,1)^\top$.  \label{fig:eigenmodes_S1000_PM1_3}}
\end{figure}
\endgroup

\begingroup
\renewcommand{\arraystretch}{2.0}
\begin{figure}[htbp!]
	\centering
	\newcommand{\mywidth}{2.0cm}
	\begin{tabular}{ccc}
		$\Ra_c^{(1)} = 24{,}483$ & $\Ra_c^{(2)} = 60{,}583$ & $\Ra_c^{(3)} = 81{,}818$ \\ 
		\includegraphics[width=\mywidth]{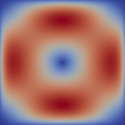} &
		\includegraphics[width=\mywidth]{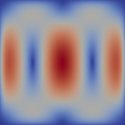} &
		\includegraphics[width=\mywidth]{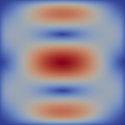} \\
		\includegraphics[width=\mywidth]{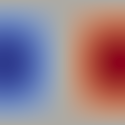} &
		\includegraphics[width=\mywidth]{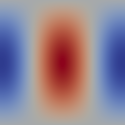} &
		\includegraphics[width=\mywidth]{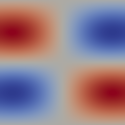} \\
		\vspace{0.5cm}
		\includegraphics[width=\mywidth]{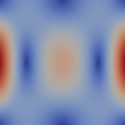} &
		\includegraphics[width=\mywidth]{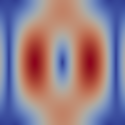} &
		\includegraphics[width=\mywidth]{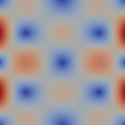} \\
	\end{tabular}
	\caption{First 4 eigenfunctions of the primary bifurcations that emanate from the conducting state \ref{eq:trivialsol} for  $0\leq \Ra \leq 100{,}000$  with \mbox{$S=1{,}000$} for the alternative magnetic field $\B_0=(1,0)^\top$. Note that here the third eigenmode is oriented in the direction of $\B_0 = (1,0)^\top$. \label{fig:eigenmodes_S1_PM1_4}}
\end{figure}
\endgroup

The bifurcation diagrams for the previously tracked branches 1--6 can be found in Figure \ref{fig:bifurcation_Ra_S1000_diagram}. Further, we include two secondary bifurcations named branch 7 and  branch 8. Note that we used backward continuation here starting at $\Ra = 100{,}000$ from the initial guesses obtained in the previous section and used a step size of $\Delta \Ra = 1000/3$. We proceed by analysing each branch in detail.

\begin{figure}[htbp!]
	\centering
	\begin{tabular}{cc}
		\includegraphics[width=7.0cm]{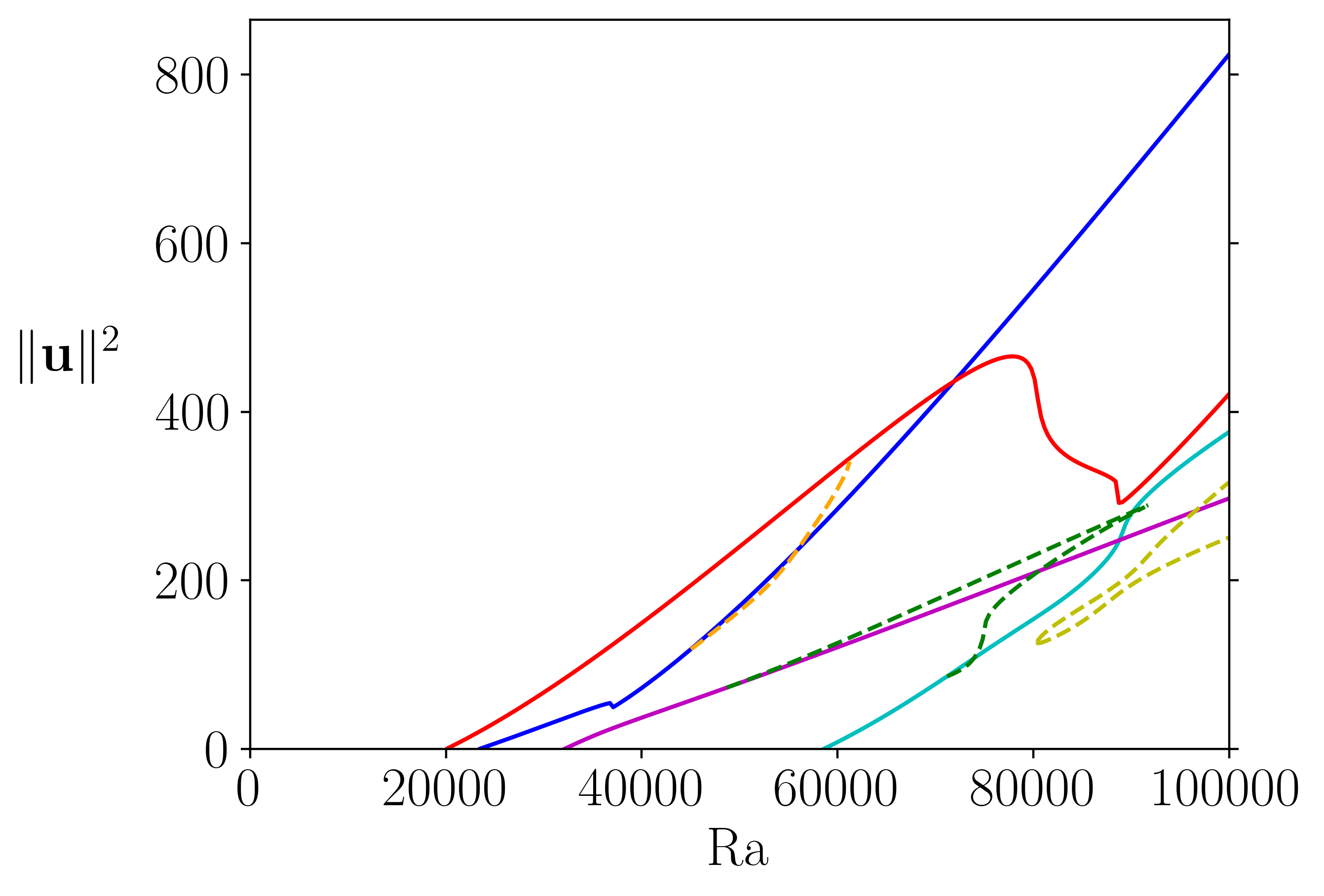} &
		\includegraphics[width=7.0cm]{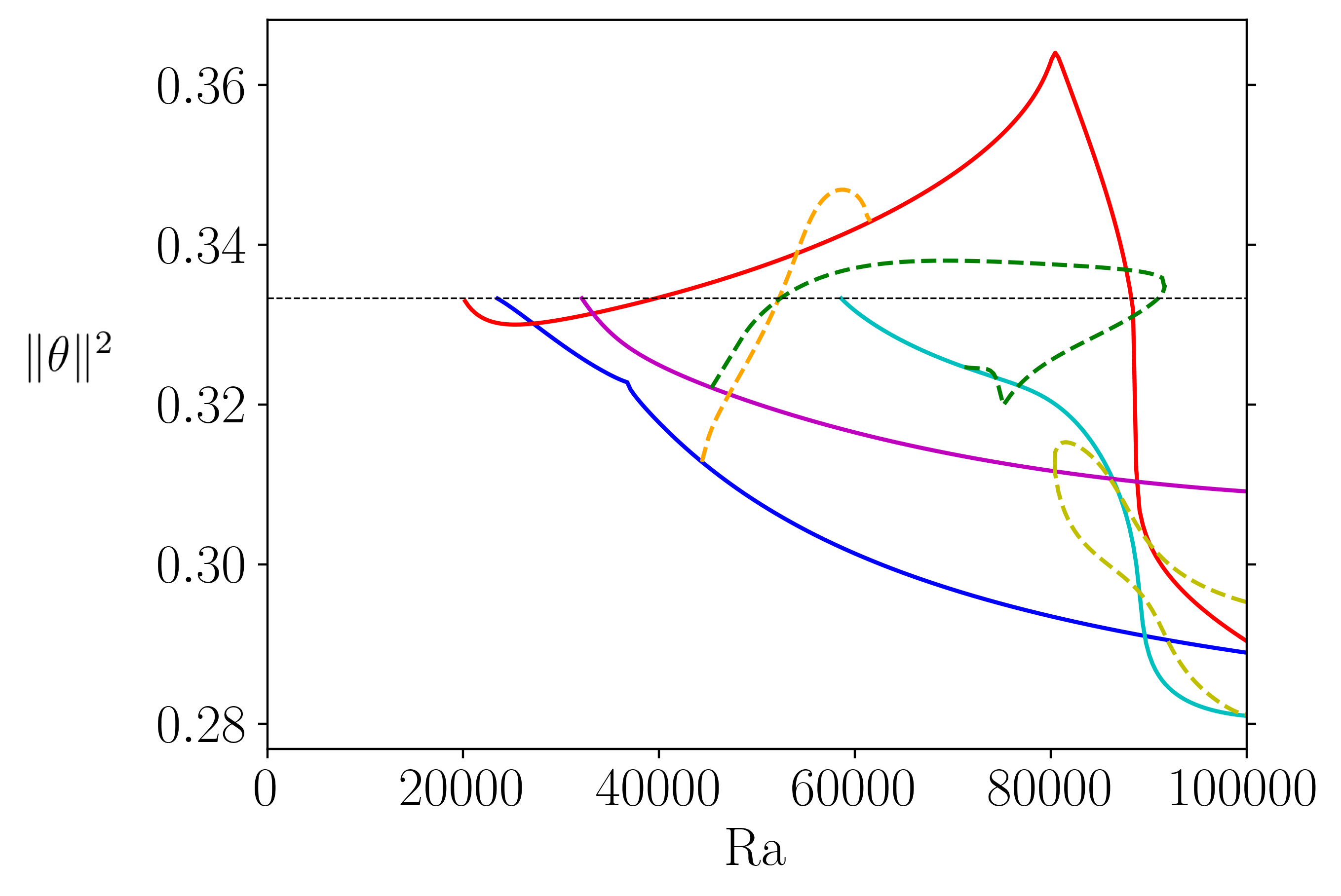} \\
	\end{tabular}
	\includegraphics[width=7.8cm]{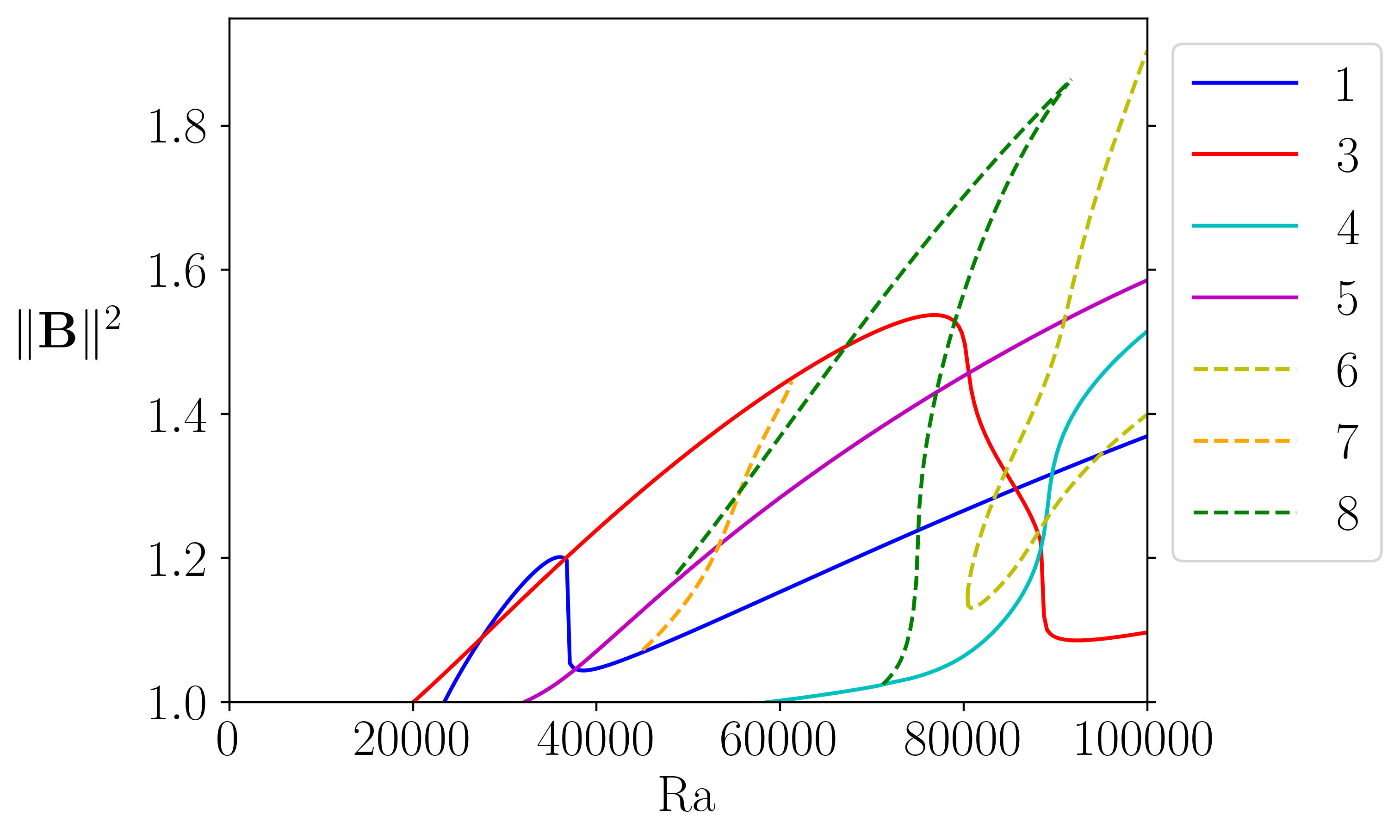}
	\caption{Bifurcation diagram over $0 \leq \Ra \leq 100{,}000$ with $S=1{,}000$.}
	\label{fig:bifurcation_Ra_S1000_diagram}
\end{figure}

Branch 1,  illustrated in Figure \ref{fig:bifurcation_Ra_S1000_diagram_branch_1}, evolves as expected until $\Ra \approx 36{,}800$ where the shape of $\u$, $\theta$ and $\B$ changes rapidly. For $\Ra < 36{,}800$ the standard continuation algorithm fails to converge and we were only able to continue the branch by deflation. A more detailed diagram of branch 1 in this range of $\Ra$ can be found in Figure \ref{fig:bifurcation_Ra_S1000_diagram_branch_1_close_up}. Here, we can see the S-shaped form of branch 1 which shows that the second primary bifurcation that originates at $\Ra_c^{(2)}=23404$ transitions into branch 1. The author would have expected that branch 1 for $\Ra > 36{,}800$ is connected to the third eigenmode, which would corresponds to the expected evolution of the blue line in the $\u$-, $\theta$- and $\B$-diagram. However, the deflated continuation algorithm was not able to continue branch 1 in this direction. Neither a much smaller step-size $\Delta \Ra$ nor a simulation on more refined grid of $100\times100$ cells was able to produce the expected result. We therefore conjecture that the bifurcation diagram is more complicated than our expectations. 

The deflated continuation algorithm also found a rather surprising evolution of branch 3 shown in Figure \ref{fig:bifurcation_Ra_S1000_diagram_branch_3}. Again we tried smaller step sizes and a more refined grid, but this line was the only one discovered. One could have expected that the fourth primary bifurcation emerging at $\Ra_c^{(4)} = 58{,}571$ and branch 3 belong to the same branch, since both branches show similar patterns in the solutions, e.g., the five vertical stripes in the velocity field.
To further investigate the evolution, we have provided a more detailed diagram for branch 3 in Figure \ref{fig:bifurcation_Ra_S1000_diagram_branch_3_close_up} in the range between $84{,}000$ and $94{,}000$. This diagram indeed seems to underline that we  follow branch 3 here rather than starting to follow secondary bifurcations. There exist two turning points at $\Ra \approx 88{,}743 $ and $\Ra \approx 88{,}904$ showing an S--shaped form for of this branch. Moreover, there emerges a secondary bifurcation at $\Ra \approx 84{,}677$. Another secondary bifurcation connects the turning point at $\Ra \approx 88{,}904$ with the bifurcation emerging at $\Ra \approx 92{,}165$. We have coloured these secondary branches in grey to indicate that we have found these branches but will not provide a detailed diagram for them in a separate figure. Finally, we have found a Hopf bifurcation emerging at $\Ra \approx 88{,}693$. Based on the eigenvalue plots we believe that we have found all bifurcations emerging in that interval and that the bifurcation diagram in this interval is complete. 

Branch 1 and branch 3 have a zero eigenvalue at $\Ra \approx 44{,}000$ and  $\Ra \approx 62{,}000$ as the stability plots in Figure \ref{fig:bifurcation_Ra_S1000_diagram_branch_1} and Figure \ref{fig:bifurcation_Ra_S1000_diagram_branch_3} show. The corresponding secondary bifurcation is included in Figure \ref{fig:bifurcation_Ra_S1000_diagram_branch_7} and directly connects branch 1 and branch 3. It is interesting to note that branch 7 starts from branch 1 with a symmetric solution, then breaks its symmetry and retains it again when merging into branch 3.  

The primary bifurcation emerging at $\Ra_c^{(3)}=32{,}095$ and $\Ra_c^{(4)}=58{,}571$ are illustrated in Figure \ref{fig:bifurcation_Ra_S1000_diagram_branch_4} and Figure \ref{fig:bifurcation_Ra_S1000_diagram_branch_5}. They have zero eigenvalues at $\Ra\approx 47{,}000$ and $\Ra \approx 72{,}000$.
We found another secondary bifurcation that connects these two branches and included it in Figure \ref{fig:bifurcation_Ra_S1000_diagram_branch_8}. The corresponding zero eigenvalues of branch 4 and branch 5 are highlighted at $\Ra \approx 72{,}000$ and $\Ra \approx 47{,}000$ in Figure \ref{fig:bifurcation_Ra_S1000_diagram_branch_4} and  Figure \ref{fig:bifurcation_Ra_S1000_diagram_branch_5}.

The final branch is the disconnected branch 6 in Figure \ref{fig:bifurcation_Ra_S1000_diagram_branch_6}. We want to emphasise again that we were only able to find this solution by tracking the branch from $S=1$ with deflated continuation over both $\Ra$ and $S$ in Section \ref{sec:bif1} and \ref{sec:bif2}.

In conclusion, we have observed rather surprising evolutions of some branches and have seen how increasing the coupling number can stabilise unstable branches. Moreover, we have observed that the order of unstable eigenmodes changes with increasing S and instabilities whose velocity is aligned with the magnetic field arise for smaller Rayleigh numbers at high S. Finally, we have outlined how disconnected branches can be found at high coupling numbers.

\begin{figure}[htbp!]
	\begin{subfigure}{\textwidth}
		\centering
		\includegraphics[width=16.0cm]{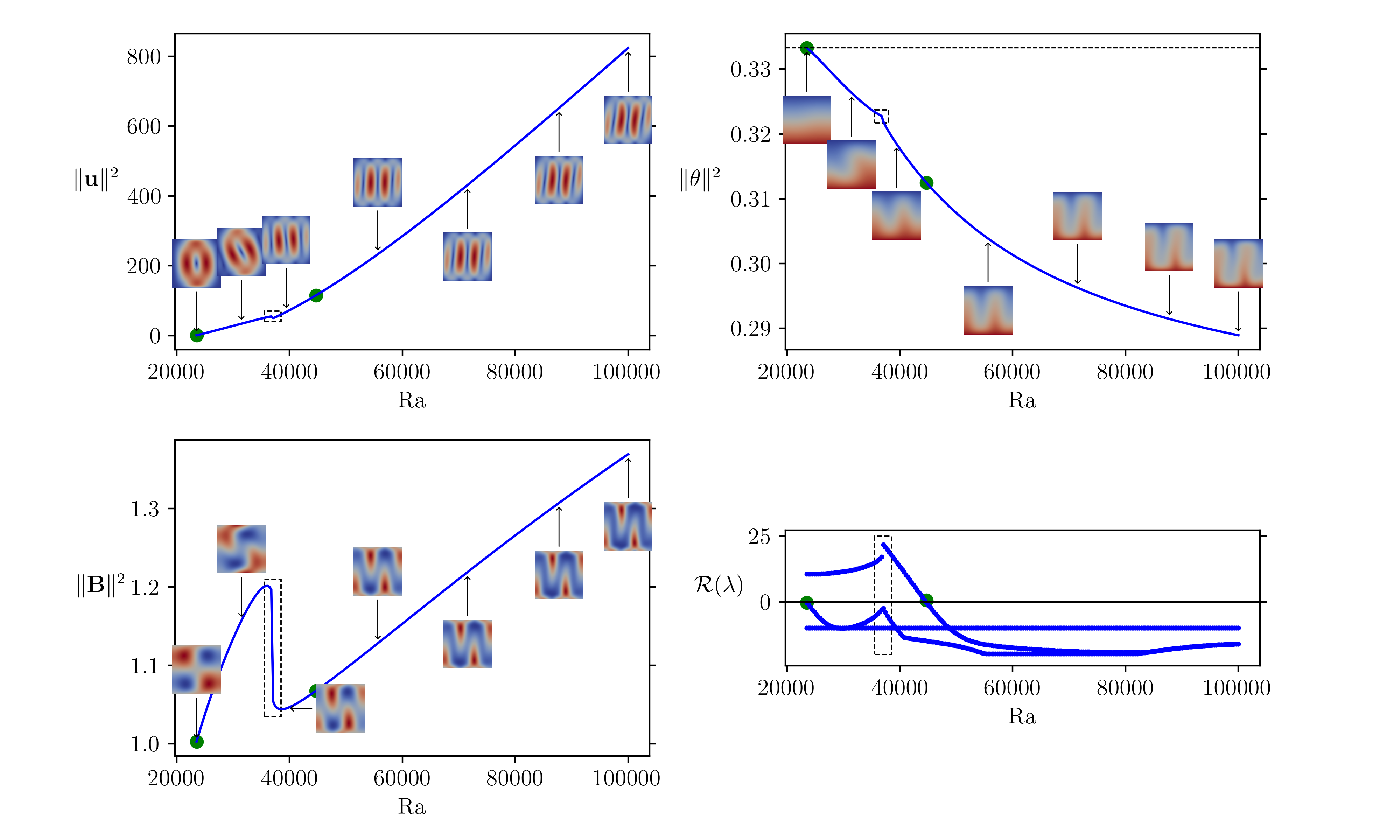}
		\vspace{-0.5cm}
	    \subcaption{Evolution of branch 1 over $\Ra$ for $S=1{,}000$. A more detailed plot of the area highlighted by the dashed rectangle can be found in the next Figure \ref{fig:bifurcation_Ra_S1000_diagram_branch_1_close_up}.}
        \label{fig:bifurcation_Ra_S1000_diagram_branch_1}
	\end{subfigure}
	
	\vspace{1.0cm}
	\begin{subfigure}{\textwidth}
		\centering
		\includegraphics[width=16.0cm]{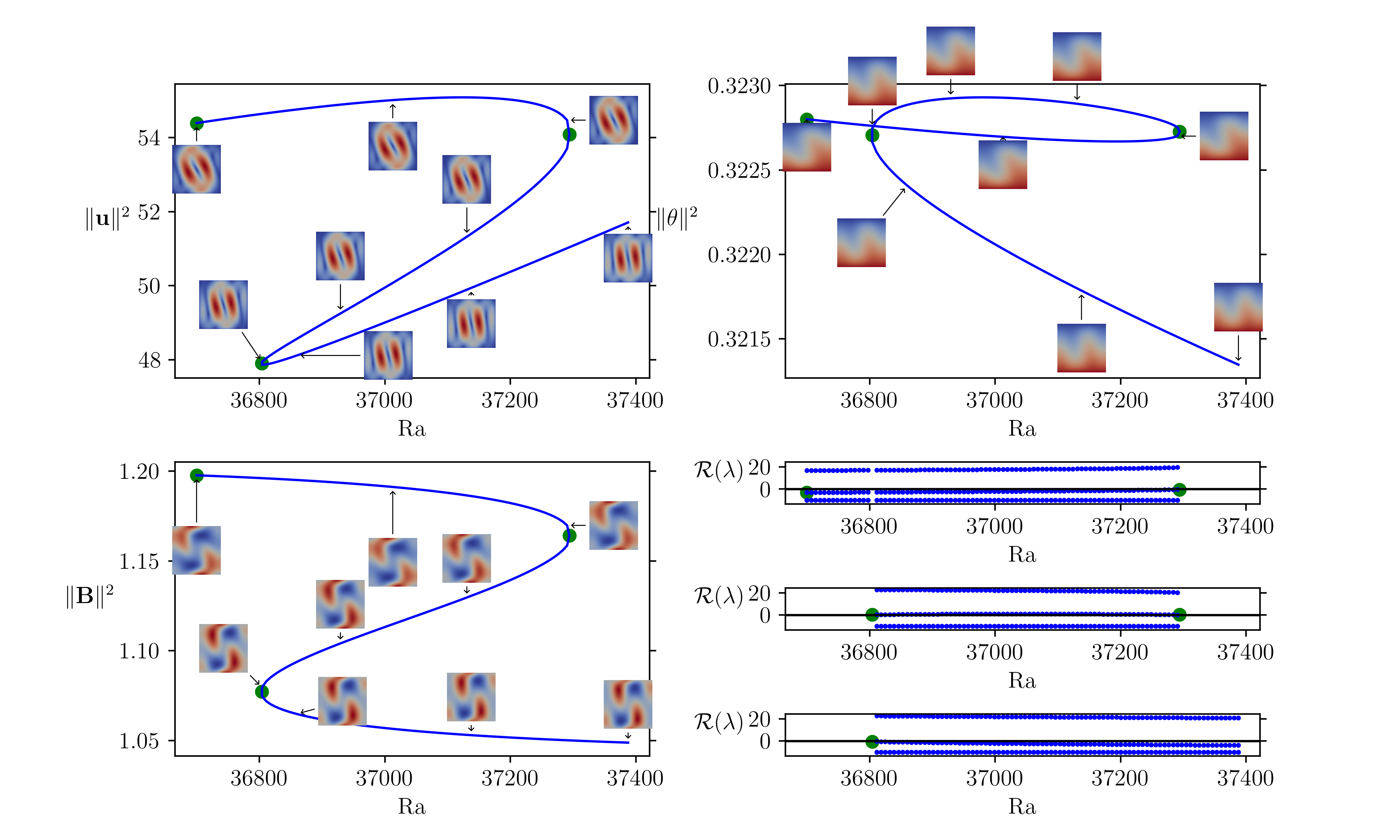}
		\vspace{-0.5cm}
	    \subcaption{More detailed evolution of branch 1 over $\Ra$ for $S=1{,}000$.}
       \label{fig:bifurcation_Ra_S1000_diagram_branch_1_close_up}
	\end{subfigure}
\end{figure}
\begin{figure}
	\ContinuedFloat
	\begin{subfigure}{\textwidth}
		\centering
		\vspace{-0.3cm}
		\includegraphics[width=16.0cm]{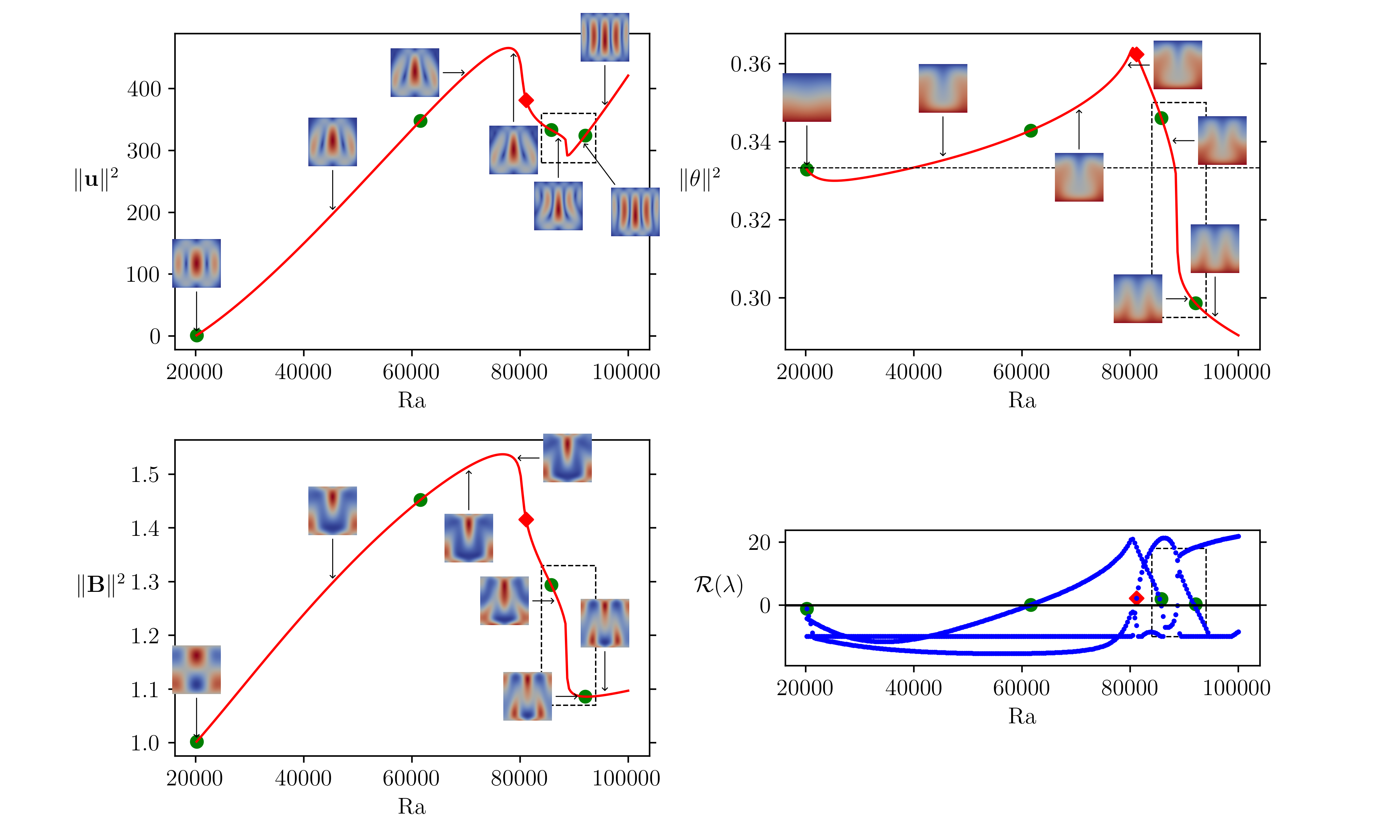}
		\vspace{-0.5cm}
	    \subcaption{Evolution of branch 3 over $\Ra$ for $S=1{,}000$. A more detailed plot of the area highlighted by the dashed rectangle can be found in the next Figure \ref{fig:bifurcation_Ra_S1000_diagram_branch_3_close_up}.}
        \label{fig:bifurcation_Ra_S1000_diagram_branch_3}
	\end{subfigure}
	
	\vspace{1.0cm}
	\begin{subfigure}{\textwidth}
		\centering		
		\includegraphics[width=16.0cm]{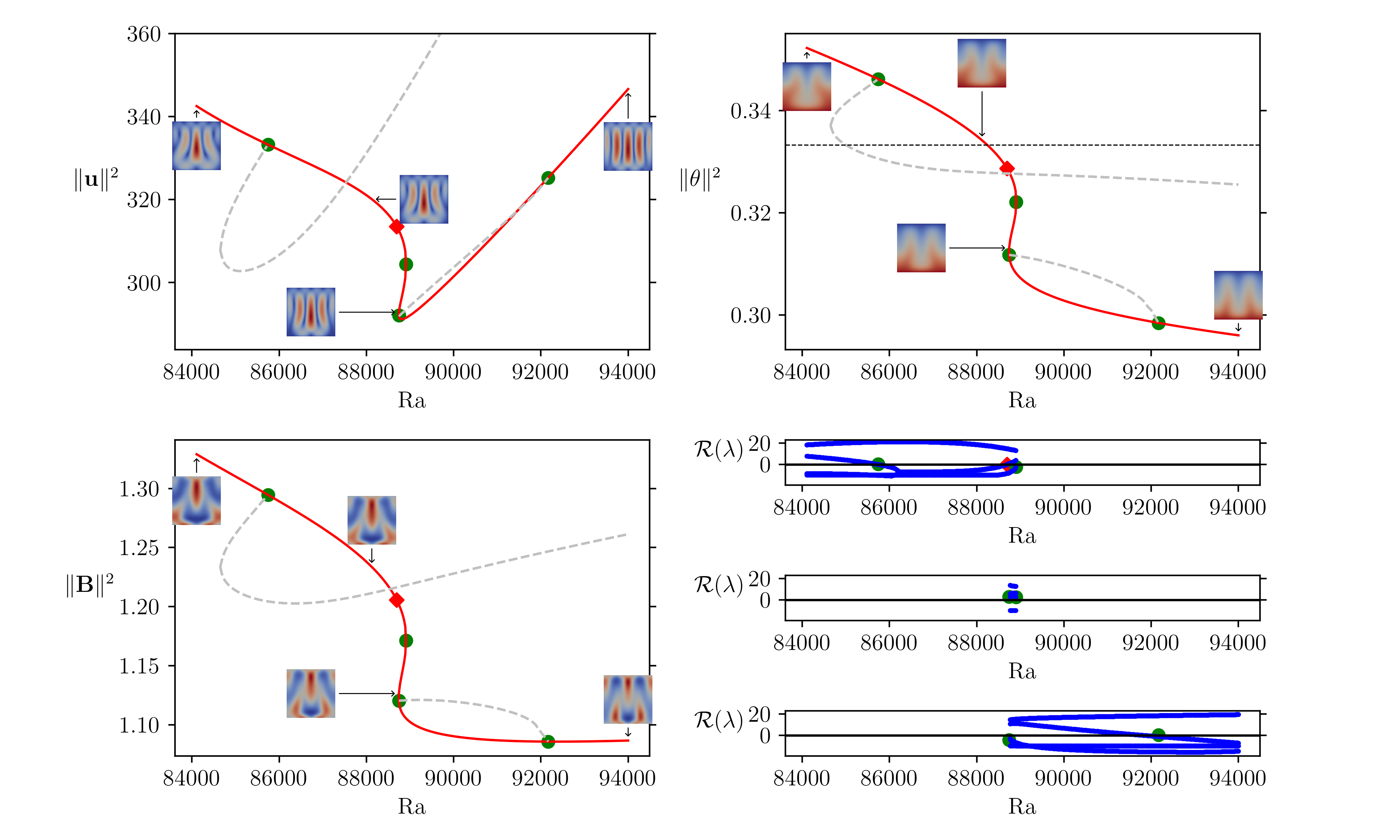}
		\vspace{-0.5cm}
	    \subcaption{More detailed evolution of branch 3 over $\Ra$ for $S=1{,}000$.}
		\label{fig:bifurcation_Ra_S1000_diagram_branch_3_close_up}
	\end{subfigure}
\end{figure}
\begin{figure}
	\ContinuedFloat
	\begin{subfigure}{\textwidth}
		\centering
		\vspace{-0.3cm}
		\includegraphics[width=16.0cm]{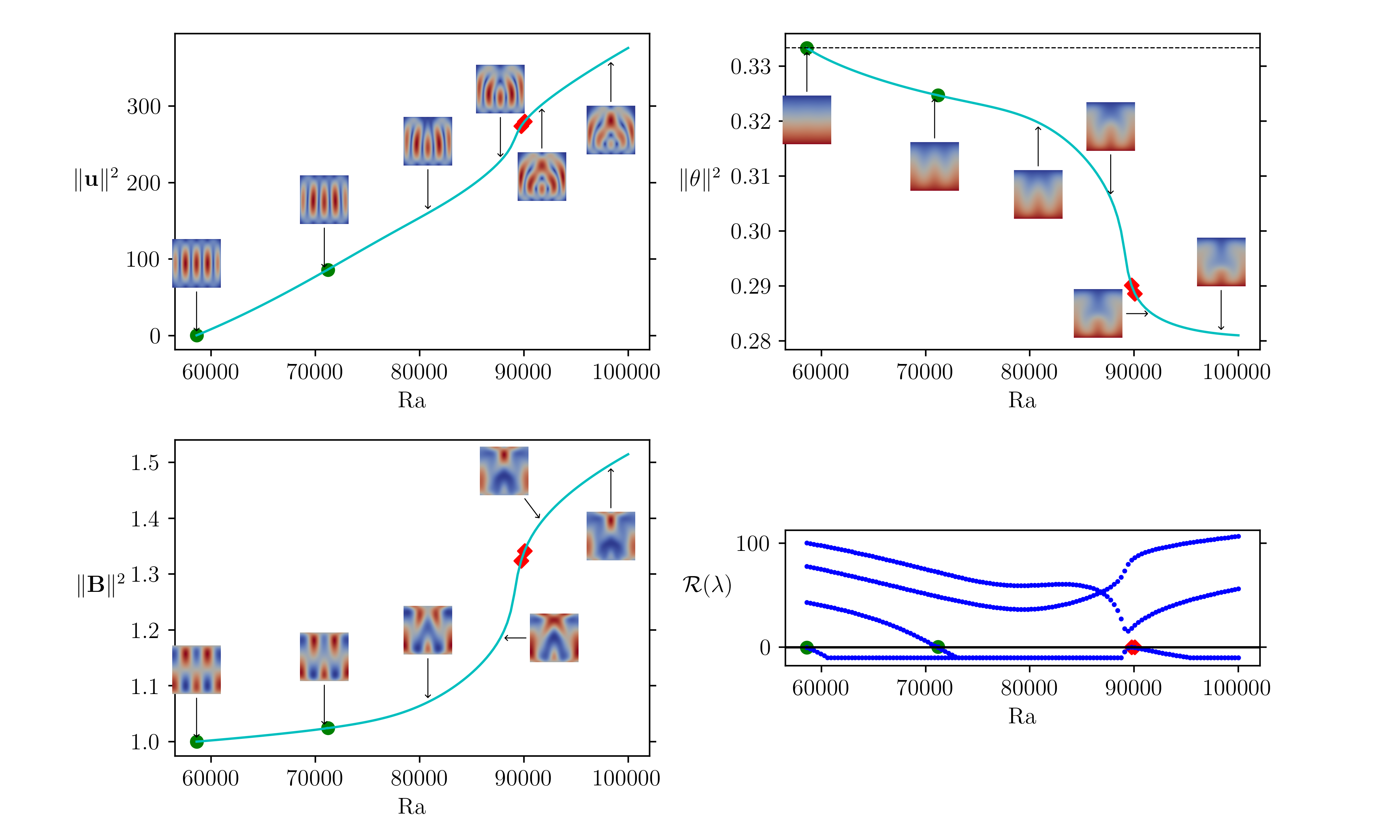}
		\vspace{-0.5cm}
	    \subcaption{Evolution of branch 4 over $\Ra$ for $S=1{,}000$.}
		\label{fig:bifurcation_Ra_S1000_diagram_branch_4}
	\end{subfigure}
	
	\vspace{1.0cm}
	\begin{subfigure}{\textwidth}
		\centering		
		\includegraphics[width=16.0cm]{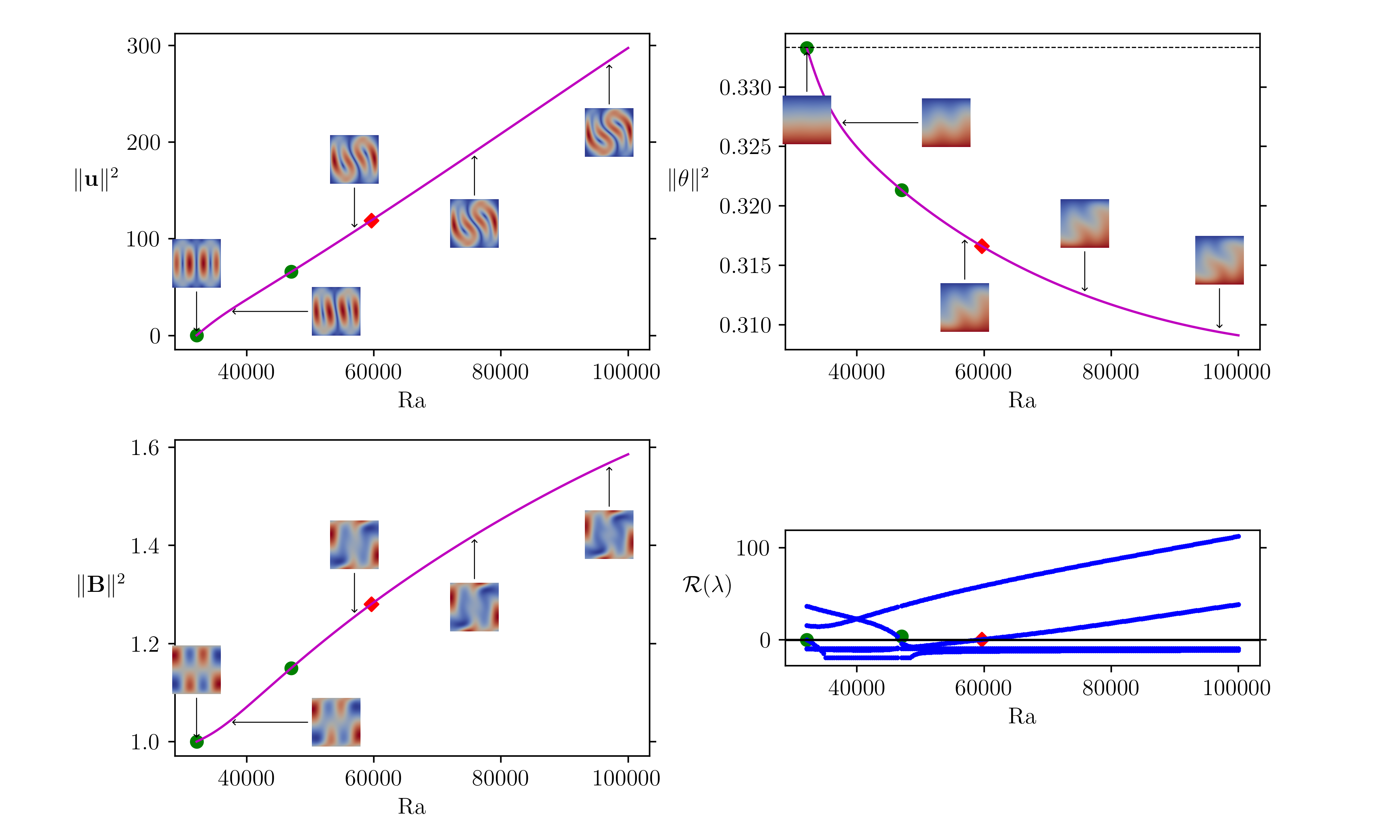}
		\vspace{-0.5cm}
	    \subcaption{Evolution of branch 5 over $\Ra$ for $S=1{,}000$.}
        \label{fig:bifurcation_Ra_S1000_diagram_branch_5}
	\end{subfigure}
\end{figure}
\begin{figure}
	\ContinuedFloat
	\begin{subfigure}{\textwidth}
		\centering
		\vspace{-0.5cm}
		\includegraphics[width=16.0cm]{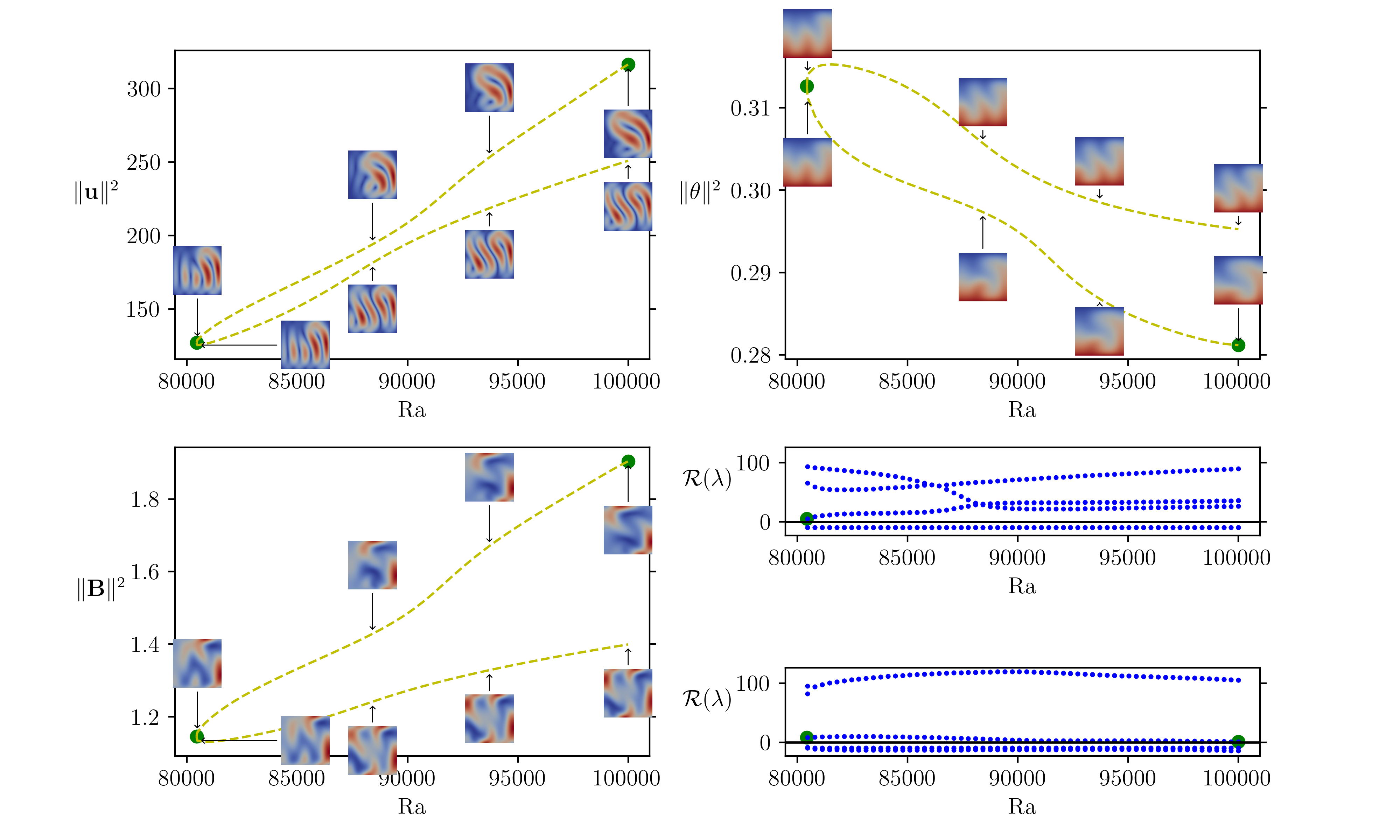}
		\vspace{-0.5cm}
	    \subcaption{Evolution of branch 6 over $\Ra$ for $S=1{,}000$.}
        \label{fig:bifurcation_Ra_S1000_diagram_branch_6}
	\end{subfigure}
	
	\vspace{0.7cm}
	\begin{subfigure}{\textwidth}
		\centering
		\includegraphics[width=16.0cm]{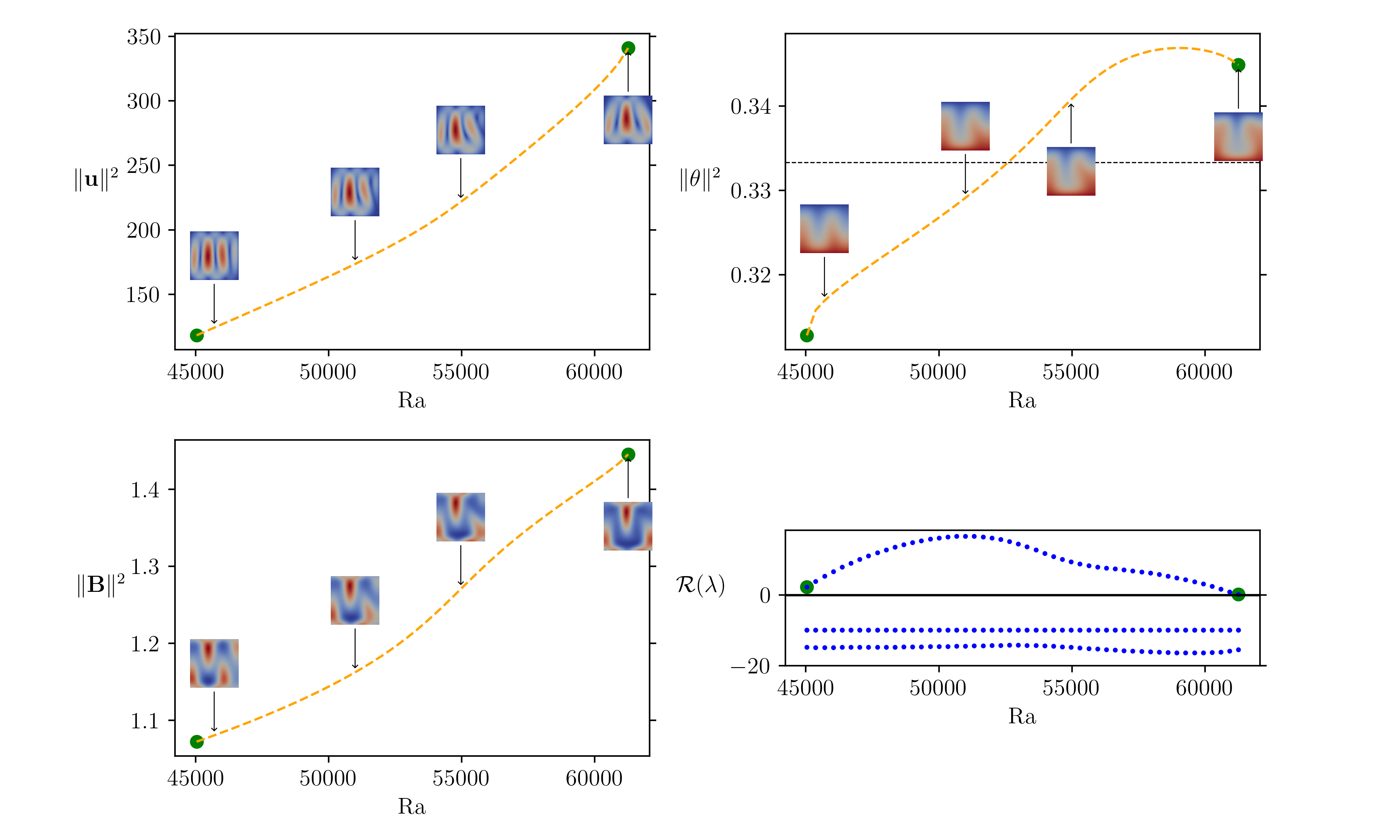}
		\vspace{-0.5cm}
	    \subcaption{Evolution of branch 7 over $\Ra$ for $S=1{,}000$.}
        \label{fig:bifurcation_Ra_S1000_diagram_branch_7}
	\end{subfigure}
\end{figure}

\begin{figure}
	\ContinuedFloat
	\begin{subfigure}{\textwidth}
		\centering
		\vspace{-0.5cm}
		\includegraphics[width=16.0cm]{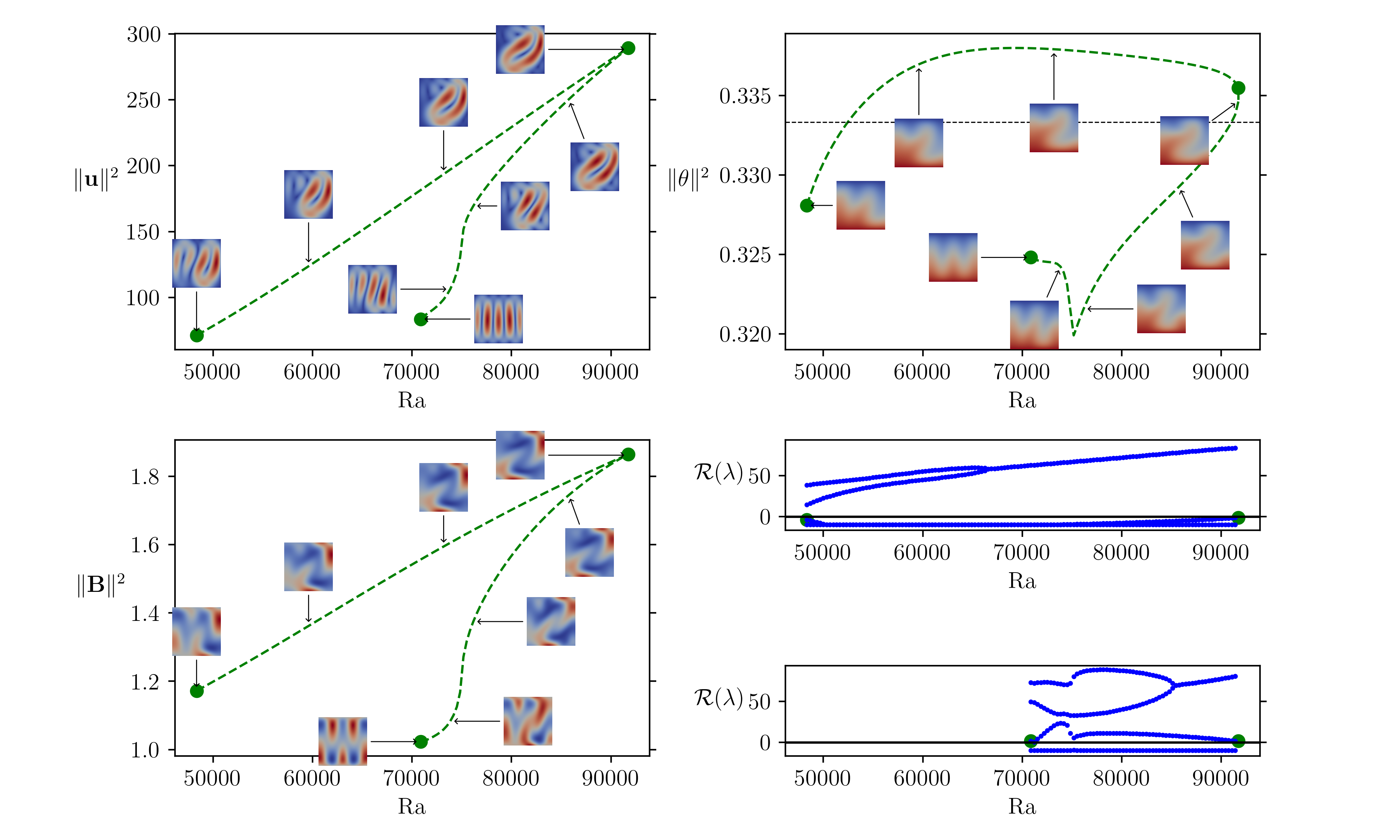}
		\vspace{-0.5cm}
		\subcaption{Evolution of branch 8 over $\Ra$ for $S=1{,}000$.}
		\label{fig:bifurcation_Ra_S1000_diagram_branch_8}
	\end{subfigure}
    \caption{Evolution of branch 1, 3, 4, 5, 6, 7, 8 over $\Ra$ for $S=1{,}000$.}
    \label{fig:bifurcation_S_Ra100000_diagram_branch_1_2_3_4_5_6_7_8}
\end{figure}

\FloatBarrier
\section{An augmented Lagrangian block preconditioner for anisothermal MHD models }\label{sec:BlockPreconAniso}
In this section, we extend the block preconditioner that we developed in Chapter~\ref{chap:2} to the temperature--dependent case. The goal is to built a preconditioner that is robust in the linear iteration counts with respect to the parameters $\Ra$, $\Pr$, $\Pm$ and $S$. This preconditioner can, e.g., be used to compute more refined solutions for the bifurcation problem studied in the previous section. Recall that due to the many nonlinear iterations required for the deflated continuation algorithm we chose a rather coarse mesh of $50 \times 50$ cells with a direct linear solver in the previous section. If one is interested in particular solutions with a finer resolution, the solutions of the coarse mesh can then be used as initial guesses and this preconditioner can be applied. Furthermore, our solver could be used to precondition the Krylov-Schur solver that we used in the previous section to solve the arising eigenvalue problems.

While the literature on preconditioners for anisothermal MHD models does not seem to be too rich, the preconditioning of standard Rayleigh-B\'enard problems or thermal convection-driven flows is more well-known, cf. \cite{Farrell2021Aniso,Howle2012} or \cite{Elman2014} for an overview.  

In each nonlinear step, we have to solve a system of the form
\begin{equation}
	\begin{bmatrix}
		\FF & \nabla & -\Ra\, \Pr\, \mathbf{e}_3 & \GG&  S\,\B_0 \times\\
		-\nabla \cdot & 0 & 0 & 0 & 0\\
		 \nabla \theta^n \cdot & 0 & -\Delta + \u^n \cdot \nabla & 0 & 0\\
		0 & 0 & 0 & -\frac{\Pr}{\Pm}\nabla \nabla \cdot &  \vcurl  \\
		 \times \B_0 & 0 & 0 &  \frac{\Pr}{\Pm}\scurl \,   +\u^n \cdot & I 
	\end{bmatrix}
	\begin{bmatrix}
		x_{\u}  \\ x_{p} \\  x_{\theta} \\ x_{\B}\\ x_{\E} 
	\end{bmatrix}	
= 
\begin{bmatrix}
	R_{\u}  \\ R_{p} \\  R_{\theta} \\ R_{\B}\\ R_{\E} 
\end{bmatrix}	
\end{equation}
with 
\begin{align}
	\FF \u &= -2 \Pr \nabla \cdot \varepsilon(\u) + \u^n \cdot \nabla \u + \u \cdot \nabla \u^n + S\, \B^n \times (\u \times \B^n), \\
	\GG \, \B &=  S\, \B \times \E^n + S\, \B\times (\u^n \times \B^n) + \S\, \B^n\times(\u^n \times \B).
\end{align}
Since the equations for $\B$ and $\E$ do not include the temperature, we group as before $(\u,p,\theta)$ and $(\B, \E)$ together and use 
\begin{equation}
	\tilde{\mathcal{S}}^{(\u, p, \theta)} = 
	\begin{bmatrix}
		-\frac{\Pr}{\Pm}\nabla \nabla \cdot &  \vcurl \\
		 \frac{\Pr}{\Pm}\scurl \,   + \u^n \cdot & I 
	\end{bmatrix}
\end{equation}
as an approximation for the outer Schur complement. Equally, we use the same multigrid method that was described in Section \ref{sec:solverformagneticblock} as a preconditioner for the Schur complement. 

For the top left block, we use a similar idea that was introduced in \cite{Farrell2021Aniso}. The first step is to perform a further Schur complement approximation grouping together $(\u,\theta)$ and $p$. For convenience, we reorder the top left block correspondingly to 
\begin{equation}\label{eq:3by3form}
	\begin{bmatrix}
		\FF & -\Ra\, \Pr\, \mathbf{e}_3 & \nabla \\
		\nabla \theta^n \cdot & -\Delta + \u^n\cdot \nabla & 0 \\
		-\nabla \cdot & 0 & 0
	\end{bmatrix},
\end{equation}
cf. \cite[page 452]{Elman2014}.
As explained in Section \ref{sec:solverforschurcomp}, we again control the Schur complement of this system by adding the augmented Lagrangian term $-\gamma \nabla \div \u$ to the top-left block $\FF$ for a large value of $\gamma$. Identity \eqref{eq:SchurCompNS} resp.\ \cite[Theorem 3.2]{Bacuta2006} can still be applied to our system by interpreting \eqref{eq:3by3form} as a $2\times2$-block system
\begin{equation}
	\begin{bmatrix}
		\AA & \BB^\top \\
		\BB & 0
	\end{bmatrix}
\quad \text{ with }
\AA = 
\begin{bmatrix}
		\FF & -\Ra\, \Pr\,  \mathbf{e}_3 \\
		\nabla \theta^n \cdot & -\Delta + \u^n\cdot \nabla 
\end{bmatrix}
\quad \text{ and }
\BB =
\begin{bmatrix}
	-\nabla \cdot & 0
\end{bmatrix}.
\end{equation}
Hence, a pressure mass matrix scaled by $-1/\gamma$ remains a good approximation for the inner Schur complement provided $\gamma$ is chosen large enough, e.g., $\gamma = 10^4$.

The second step is to apply the parameter-robust multigrid method, which we also introduced in Section \ref{sec:solverforschurcomp}, monolithically to the $(\u, \theta)$-block $\AA$. We found that this approach results in the most robust iteration counts with respect to $\Ra$ and $\Pr$ and report numerical results for this approach in the next section. Alternatively, one could try to eliminate $\u$ or $\theta$ with a further Schur complement approximation. However, both Schur complements are not straight-forward to compute.

\section{Numerical results}
In this section, we present numerical results for the block preconditioner that we introduced in Section \ref{sec:BlockPreconAniso}. 
The results were again produced on ARCHER2, the UK national supercomputer. 
\subsection{Stationary magnetic double glazing problem}
We start by investigating a magnetic double glazing problem on the domain \mbox{$\Omega = (-1/2, 1/2)^d$} for dimension $d=2,3$. The problem setup for the hydrodynamic part is taken from \cite[Section 5.3]{Farrell2021Aniso}. The boundary conditions are chosen to be 
\begin{equation}
	\u = \mathbf{0} \text{ on } \partial \Omega,  \quad \nabla \theta \cdot \n = 0  \, \text{  on } \partial \Omega \backslash(\Gamma_H \cup \Gamma_C), \quad 
	\theta = 
	\begin{cases}
		1, & \text{ on } \Gamma_H, \\
		0, & \text{ on } \Gamma_C,
	\end{cases}
\end{equation}
where the hot and cold boundaries are defined as $\Gamma_H = \{x_1=-1/2\}$ and \mbox{$\Gamma_C = \{x_1 = 1/2\}$}. The system is completed with the magnetic boundary conditions given by 
\begin{equation}
	\B \cdot \n = (0,1)^\top \n , \qquad E = 0 \text{ on } \partial \Omega,
\end{equation}
in two dimensions and 
\begin{equation}
	\B \cdot \n = (0,0,1)^\top \n , \qquad \E\times \n = \mathbf{0} \text{ on } \partial \Omega,
\end{equation}
in three dimensions. The magnetic field can, e.g., be used to control the heat transfer similarly to what we have seen for the magnetic Rayleigh-B\'enard problem where instabilities in the velocity were mainly aligned with the magnetic field.

Note that this setup is different to the magnetic Rayleigh-B\'enard problem studied in the previous Section \ref{sec:BifurcationAnalysis}, since the direction of gravity is now perpendicular to the primary heat flow from the hot to the cold plate. In particular, this problem does \emph{not} have the symmetries \eqref{eq:symmetry1} and \eqref{eq:symmetry2} and
\begin{equation}\label{eq:trivialsol2}	
	\u_0 = \mathbf{0}, \quad \theta_0 = 1-x_1 \quad \text{ and } \quad \B_0 = (0,1)^\top, \text{ resp.}, \, \B_0 = (0,0,1)^\top 
\end{equation}
is no longer a trivial solution of this system. This is due to the fact that there does not exist a solution for $p$ to the equation 
\begin{equation}
	\nabla p = \begin{pmatrix}
		0 \\ 1-x_1
	\end{pmatrix}
\end{equation}
since the necessary condition $\scurl\,  (0, 1-x_1)^\top = 0$ is not fulfilled. 

We investigate the performance of our preconditioner from Section \ref{sec:BlockPreconAniso} for different values of $\Ra$, $\Pr$, $\Pm$ and $S$. Note that the coefficient of the magnetic subsystem is given by $\Pr/\Pm$. Hence reducing $\Pr$ increases the coupling in the hydrodynamic and magnetic block at the same time. This corresponds to increasing the Reynolds numbers $\Re$ and $\Rem$ at the same time in the numerical results of the previous chapters. Therefore, we start by choosing $\Pr = \Pm$ in the first experiment to mimic the previous experiments where just the fluid Reynolds number $\Re$ is varied.

Since we investigate a stationary problem, we perform parameter continuation as previously explained in Section \ref{sec:algorithmdetails}. Here, we choose the steps 1, 10, 100, $1{,}000$, $10{,}000$, $30{,}000$, $100{,}000$ for $\Ra$, 1.0, 0.1, 0.03, 0.01, 0.003, 0.001 for $\Pr$ and 1, 100, 1{,}000, 10{,}000 for $S$. For this problem, we chose a base mesh of $16 \times 16$ cells with 6 levels of refinement resulting in 81.8 million degrees of freedom in two dimensions. In three dimensions, the base mesh had $6\times 6 \times 6$ cells with 3 levels of refinement and 25.8 million DoFs. We used the same $\mathrm{H}(\mathrm{div}) \times L^2$-conforming discretisation for $(\u, p)$ that we introduced in Section \ref{sec:Discretisationchap2}.

The numerical results for the two-dimensional case for $\Pr = \Pm$ are shown in Table \ref{tab:hc_pr_eq_pm2d}. We observe very well controlled linear iteration counts for Rayleigh numbers ranging from 1 to 100,000 and Prandtl numbers ranging from 1.0 to 0.001. The nonlinear iterations mainly grow with increasing Rayleigh number while they remain nearly constant with respect to $\Pr$. The missing entry for the hardest case of $\Pr=0.001$ and $\Ra=100$,000 is due to the failure of nonlinear convergence. This entry might be computable with a smaller step size in the continuation algorithm.

\begin{table}[htbp!]
	\centering
	\begin{tabular}{r|cccc}
		\toprule
		$\Pr\backslash\Ra$ & 1 &     100 &   10,000 &  100,000 \\
		\midrule
		1.0   &  ( 2) 3.0 &  ( 4) 3.2 &  ( 5) 4.8 &  ( 5) 5.0 \\
		0.1   &  ( 2) 2.5 &  ( 5) 2.6 &  ( 5) 4.6 &  ( 6) 4.7 \\
		0.01  &  ( 2) 2.5 &  ( 3) 3.7 &  ( 6) 4.2 &  ( 6) 5.3 \\
		0.001 &  ( 2) 2.5 &  ( 3) 2.7 &  ( 7) 4.0 &  NF \\
		\bottomrule
	\end{tabular}
	\caption{(Nonlinear iterations) Average outer Krylov iterations per nonlinear step for the two-dimensional magnetic double glazing problem with $\Pr = \Pm$. NF indicates that this entry was not computable due to the failure of nonlinear convergence. \label{tab:hc_pr_eq_pm2d}}
\end{table}

 Table \ref{tab:hc_pr_neq_pm2d} shows the iteration numbers for the case of stronger magnetic coupling where $\Pm=1$ and is not decreased with $\Pr$. The linear iteration counts remain again fairly constant in the reported range for $\Ra$ between 1 and 100,000 and $\Pr$ between 1.0 and 0.01. The nonlinear convergence was more difficult to achieve for smaller $\Pr$, which is the reason why this table only report iteration numbers down to $\Pr=0.01$. 
 In summary, in both cases the linear solver performs very well in the reported parameter ranges. The performance of the solver is mainly limited by the convergence of the nonlinear scheme, which could be improved by choosing smaller step sizes in the continuation algorithm.

\begin{table}[htbp!]
	\centering
	\begin{tabular}{r|cccc}
		\toprule
		$\Pr\backslash\Ra$ & 1 &     100 &   10,000 &  100,000 \\
		\midrule
		1.0  &  ( 2) 3.0 &  ( 4) 3.2 &  ( 5) 4.8 &  ( 5) 5.0 \\
		0.1  &  ( 2) 2.5 &  ( 6) 2.5 &  ( 6) 4.0 &  ( 6) 4.7 \\
		0.03 &  ( 3) 2.0 &  ( 6) 2.3 &  ( 6) 4.0 &  ( 6) 5.0 \\
		0.01 &  ( 3) 2.0 &  ( 8) 2.5 &  ( 6) 5.7  &   NF  \\
		\bottomrule
	\end{tabular}
	\caption{Iteration counts for the two-dimensional magnetic double glazing problem for \Pm=1.\label{tab:hc_pr_neq_pm2d}}
\end{table}

Streamline plots for different values of $\Ra$ with $\Pm=\Pr=0.01$ can be found in Figure \ref{fig:2d_hc}. For $\Ra=1$, the plot of $\theta$ and $\B$ is mainly determined by the boundary conditions we applied, i.e., we observe a linearly decreasing temperature from the hot to the cold plate and straight magnetic field lines pointing in the direction $(0, 1)^\top$. The velocity streamlines show a circular pattern
and with increasing Rayleigh number, high fluid velocities start to concentrate in a ring and complex eddies emerge in the corners of the domain. Moreover, the temperature profile and magnetic field lines start to rotate around the centre of the domain with increasing Rayleigh number.
\begin{figure}[htbp!]
	\centering
	\begin{tabular}{ccc}
		\includegraphics[width=4.0cm]{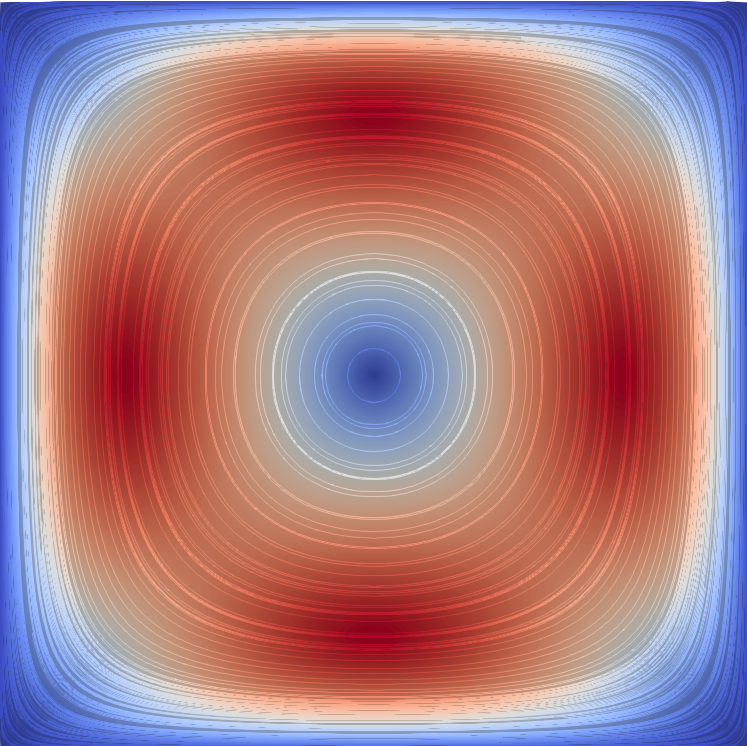} &
		\includegraphics[width=4.0cm]{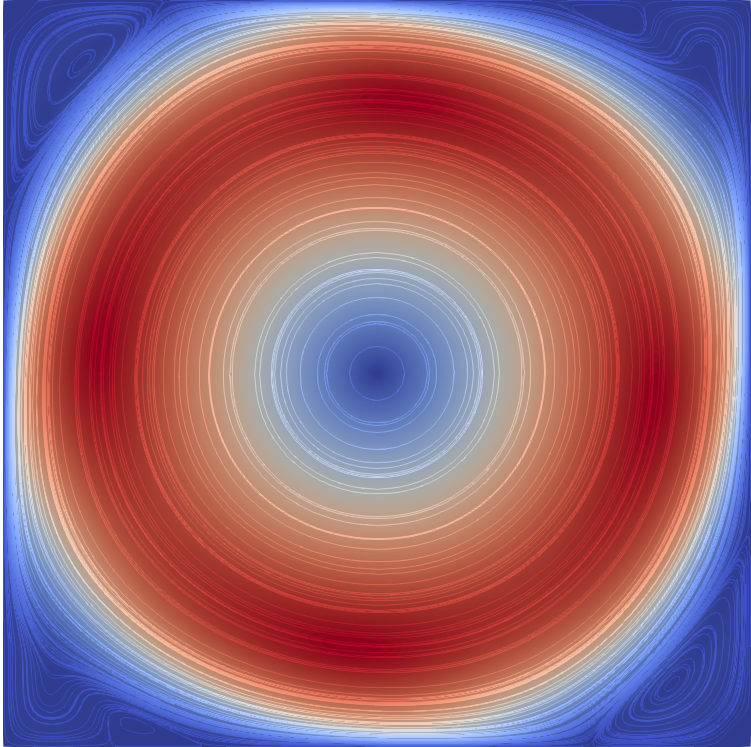} &
		\includegraphics[width=4.0cm]{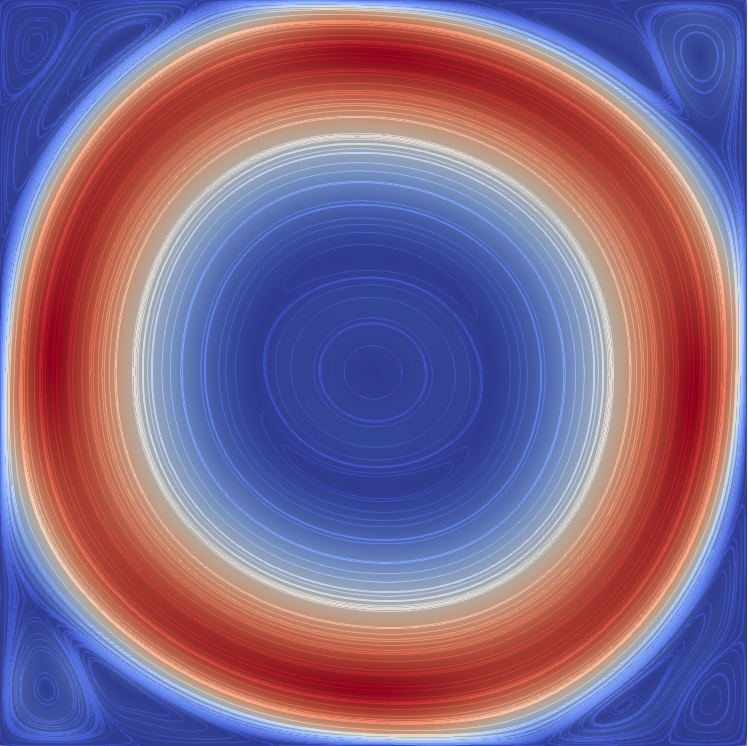} \\
		\includegraphics[width=4.0cm]{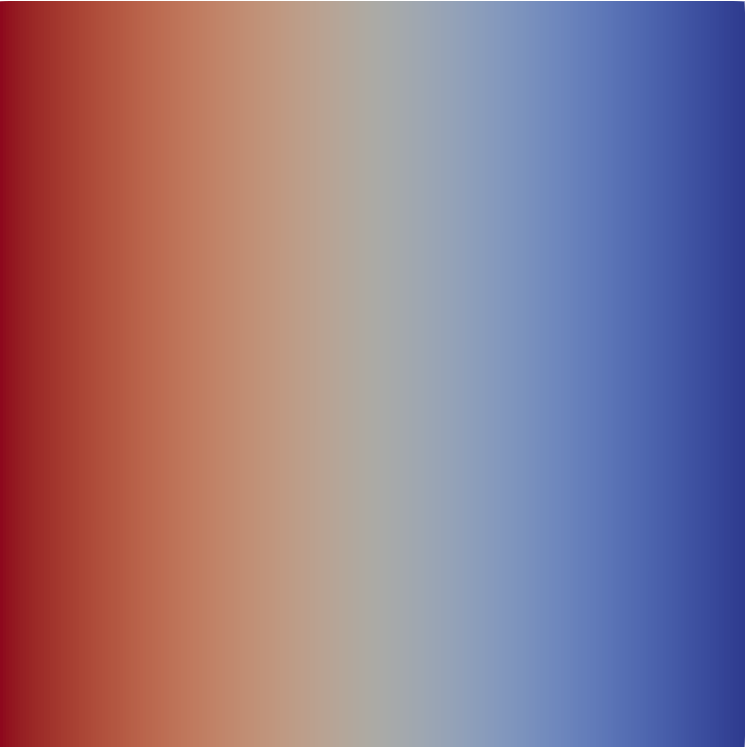} &
		\includegraphics[width=4.0cm]{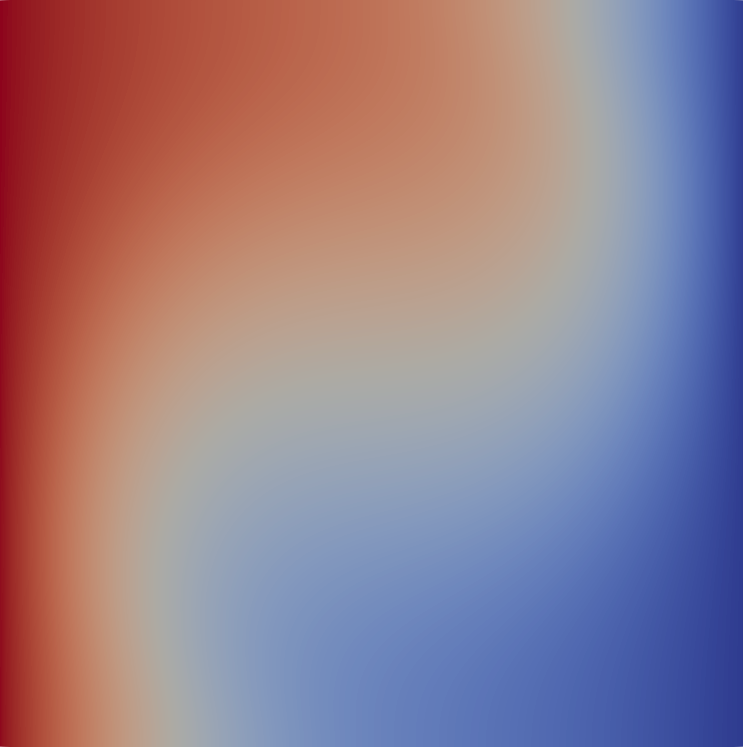} &
		\includegraphics[width=4.0cm]{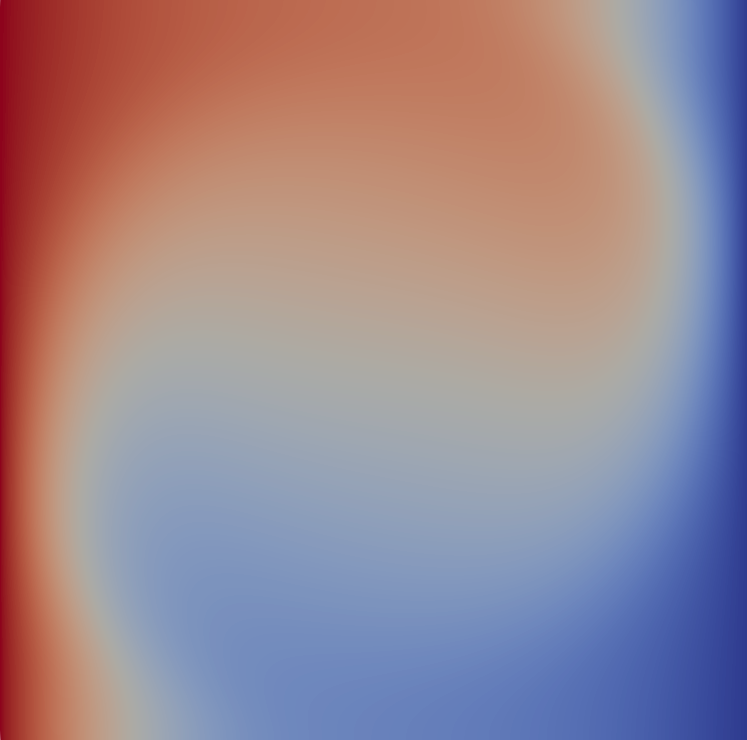} \\
		\includegraphics[width=4.0cm]{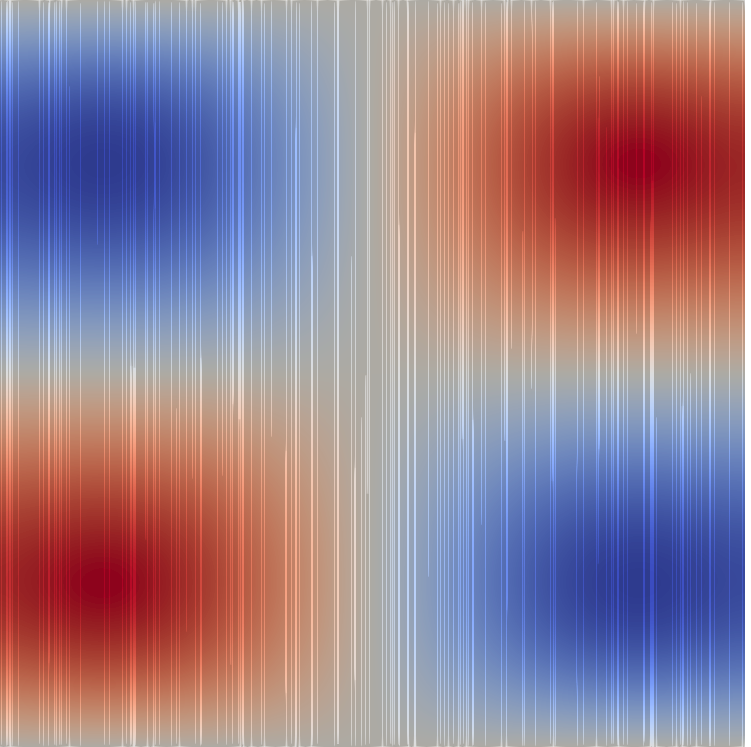} &
		\includegraphics[width=4.0cm]{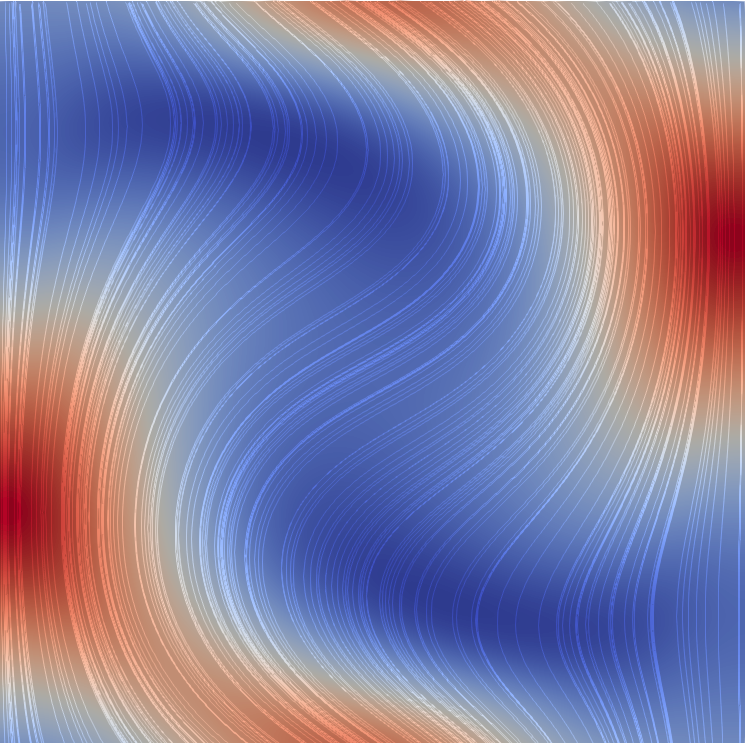} &
		\includegraphics[width=4.0cm]{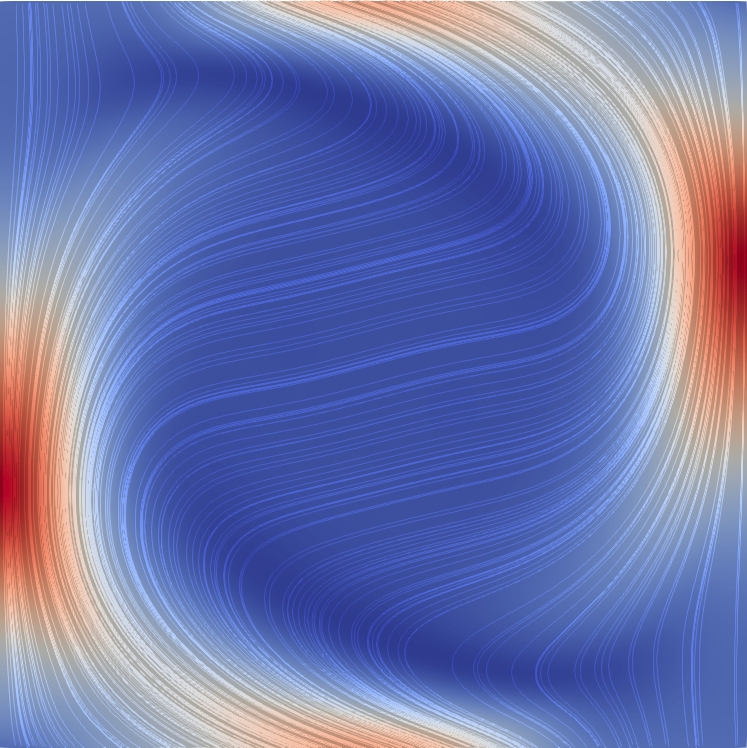} \\
		$\Ra=1$ & $\Ra=10,$000 & $\Ra=100,$000\\
	\end{tabular}
	\caption{Streamlines for the two-dimensional stationary double glazing problem for different values of $\Ra$ with $\Pm=\Pr=0.01$. The top row shows the magnitude of the velocity $\u$, the middle row the temperature $\theta$ and the bottom row the magnitude of the magnetic field $\B$. \label{fig:2d_hc}}
\end{figure}

Moreover, we report the iteration numbers for varying coupling numbers $S$. 
We want to mention that the results presented in Table \ref{tab:hc_S_Ra_2d} were computed after our license for the supercomputer ARCHER2 expired. Therefore, the iteration numbers were obtained on our local workstation and we had to decrease the number of refinements from 6 to 4 levels for a $16 \times 16$ base mesh. Hence, e.g., the iteration numbers in the first row of Table \ref{tab:hc_pr_eq_pm2d} and Table \ref{tab:hc_S_Ra_2d} do not coincide. 
For the coarser mesh, one can see that the number of linear iterations at $S=1$ is higher, while fewer nonlinear iterations are used in comparison to Table \ref{tab:hc_pr_eq_pm2d}. In any case, the iteration numbers remain fairly well controlled when $\Ra$ is increased and completely robust if $S$ is increased. 

\begin{table}[htbp!]
	\centering
	\begin{tabular}{r|cccc}
		\toprule
		S/Ra &       1 &     100 &   10{,}000 &  100{,}000 \\
		\midrule
		1     &  ( 2) 3.0 &  ( 2) 5.5 &  ( 4) 6.2 &  ( 3) 9.3 \\
		100   &  ( 2) 3.5 &  ( 2) 5.5 &  ( 4) 6.5 &  ( 3) 9.7 \\
		1{,}000  &  ( 2) 3.0 &  ( 2) 6.0 &  ( 4) 5.8 &  ( 4) 6.8 \\
		10{,}000 &  ( 2) 3.0 &  ( 2) 6.0 &  ( 3) 6.0 &  ( 3) 7.3 \\
		\bottomrule
	\end{tabular}
	\caption{Iteration counts for the two-dimensional magnetic double glazing problem for varying $S$. Note that this table was computed for 4 levels of refinement of the $16 \times 16$ base mesh instead of the 6 levels that were used in all other two-dimensional experiments in this section.\label{tab:hc_S_Ra_2d}}
\end{table}

Now, we report iteration numbers for the three dimensional version for the cases $\Pm=\Pr$ in Table \ref{tab:hc_pr_eq_pm3d} and $\Pm=1$ in Table \ref{tab:hc_pr_neq_pm3d}. Similar to the three dimensional results for the standard MHD equation from Chapter \ref{chap:2}, we observe that it is harder to achieve parameter robustness in three dimensions. This might be for the same reasons outlined at the beginning of Section \ref{sec:3dresults}. For $\Pm=\Pr$, we see a moderate increase in the iteration counts if $\Ra$ is increased from 1 to 100,000. As before, the two missing entries in Table \ref{tab:hc_pr_eq_pm3d} are due to failure of the nonlinear iteration.

In the case of stronger magnetic coupling where $\Pm$ remains 1 we see more growth in the iteration numbers for high values of $\Ra$ and low values of $\Pr$. In this case, the missing entries are in fact due to the failure of the linear solver. This observation is analogous to results reported in Section \ref{sec:stationaryliddrivencavityproblemin3d} where we observed that the monolithic multigrid solver struggles to deal with high magnetic Reynolds numbers in three dimensions. 
\begin{table}[htbp!]
	\centering
	\begin{tabular}{r|ccccc}
		\toprule
		$\Pr\backslash\Ra$ & 1 &     100 & 1,000&   10,000 &  100,000 \\
		\midrule
		1.0   &  ( 2) 4.0 &  ( 3) 5.0 &  ( 3) 6.7 &  ( 4) 6.8 &  ( 4) 7.2 \\
		0.1   &  ( 2) 3.5 &  ( 3) 6.0 &  ( 4) 6.2 &  ( 5) 7.6 &  ( 5)13.2 \\
		0.01  &  ( 2) 5.5 &  ( 3)12.7 &  ( 4)15.2 & NF & NF \\
		\bottomrule
	\end{tabular}
	\caption{Iteration counts for the three-dimensional magnetic double glazing problem with $\Pr = \Pm$.\label{tab:hc_pr_eq_pm3d}}
\end{table}
\vspace{-0.6cm}
\begin{table}[htbp!]
	\centering
	\begin{tabular}{r|ccccc}
		\toprule
		$\Pr\backslash\Ra$ & 1 &   100 & 1,000 &  10,000 &  100,000 \\
		\midrule
		1.0  &  ( 2) 4.0 &  ( 3) 5.0 &  ( 3) 6.7 &  ( 4) 6.8 &  ( 4) 7.2 \\
		0.1  &  ( 3) 3.0 &  ( 3) 6.0 &  ( 5)10.6 &  ( 5)17.8 &  ( 5)22.8 \\
		0.03 &  ( 3) 8.0 &  ( 4)12.5 &  LF  &  LF  &   LF  \\
		\bottomrule
	\end{tabular}
	\caption{Iteration counts for the three-dimensional magnetic double glazing problem for \Pm=1. LF indicates that this entry was not computable due to the failure of linear convergence. \label{tab:hc_pr_neq_pm3d}}
\end{table}

In Figure \ref{fig:3d_hc} we present streamline and contour plots of the three double glazing problem at $\Pm=\Pr=1$ where we compare the case of $\Ra=1$ and $\Ra=100{,}000$. For $\Ra=1$, the velocity has a similar circular pattern as in two dimensions with high velocities concentrating at the centre of the domain. For higher $\Ra$, high velocities start to concentrate near the left and right boundaries. The magnetic field lines point in the direction of $(0,0,1)^\top$ for $\Ra=1$ and start to twist around the centre for higher $\Ra$. Finally, the contour plots for the temperature show a linear decrease from the left to the right boundary with constant values along the y-axis for $\Ra=1$. For $\Ra=100{,}000$ they remain mostly constant along the y-axis with a more complex structure emerging in the other directions.

\begin{figure}[htbp!]
	\centering
	\vspace{-0.5cm}
	\hspace{-0.2cm}\includegraphics[width=11cm]{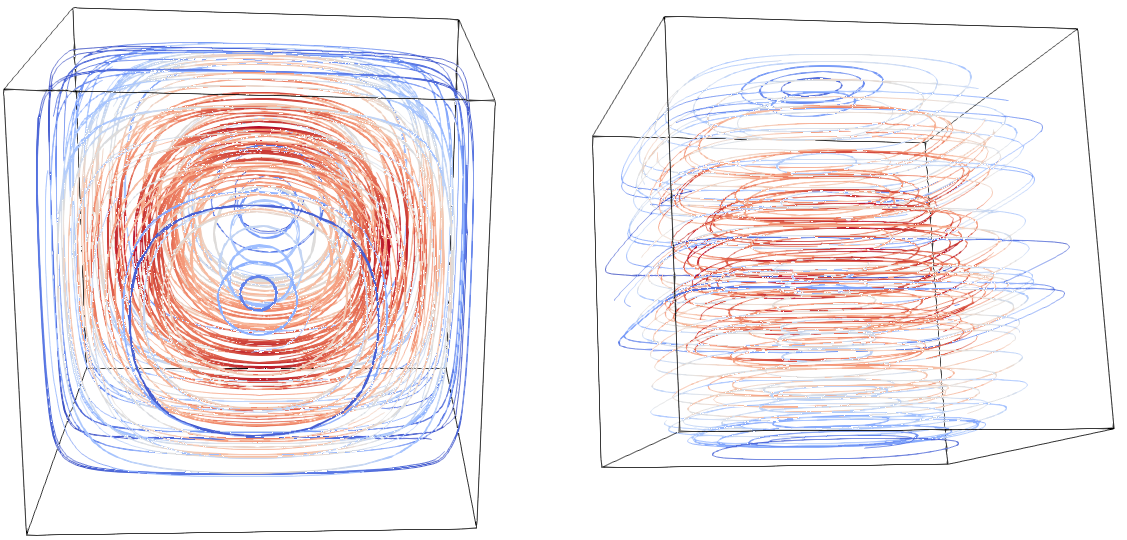} \\
	\vspace{-0.05cm}
	{\small (a) Streamlines for velocity with view from front (left) and top (right) at $\Ra = 1$.}\\
	\vspace{0.2cm}
	\hspace{-0.2cm}\includegraphics[width=11cm]{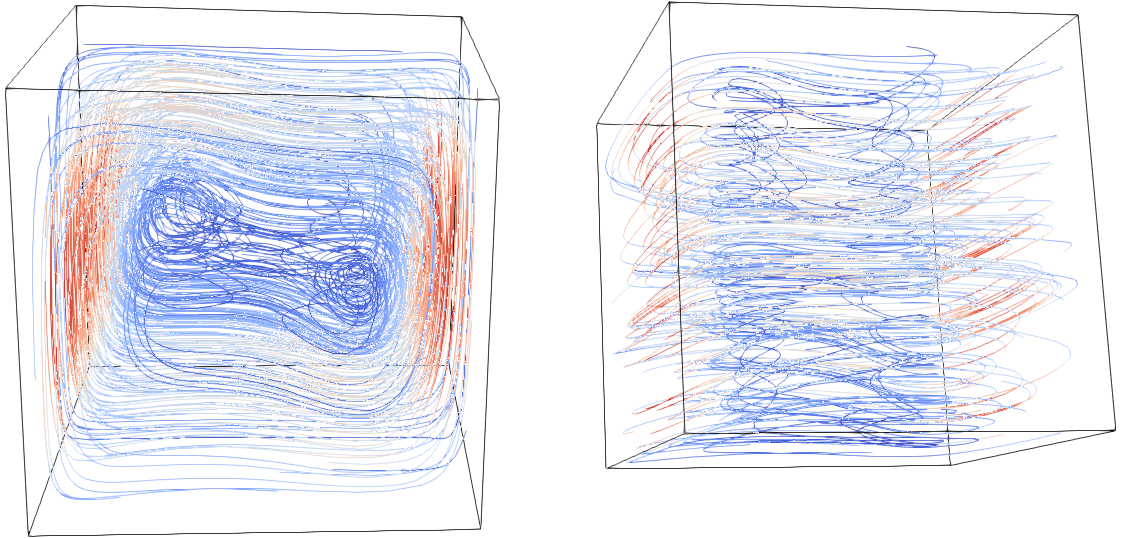}\\ 
	{\small (b) Streamlines for velocity with view from front (left) and top (right) at $\Ra = 10^5$.}\\
	\vspace{0.2cm}
	\hspace{-0.2cm}\includegraphics[width=11cm]{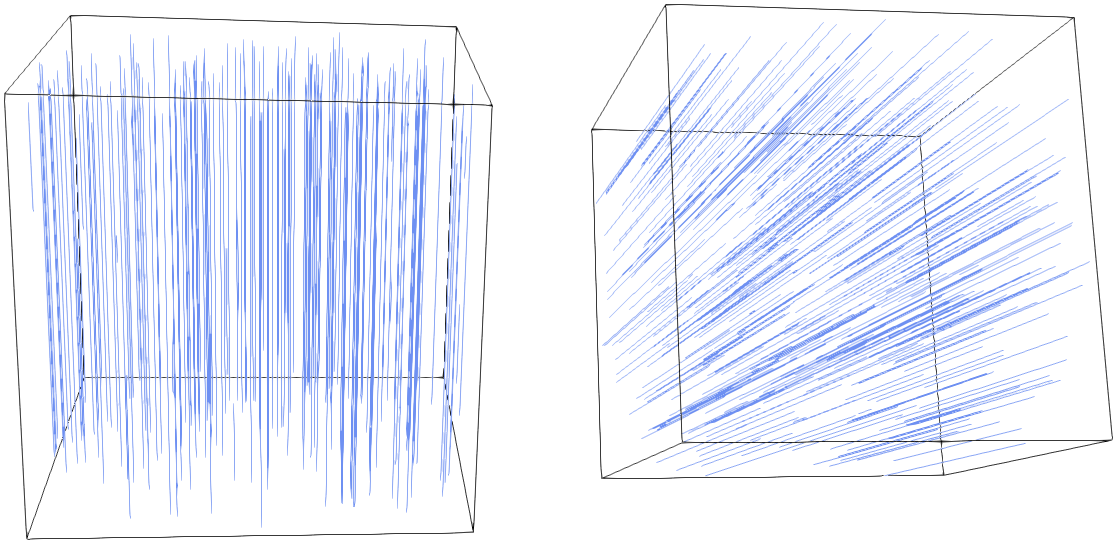} \\
	{\small (c) Streamlines for magnetic field with view from front (left) and top (right) at $\Ra = 1$.}\\
	\vspace{0.3cm}
	\hspace{-0.2cm}\includegraphics[width=11cm]{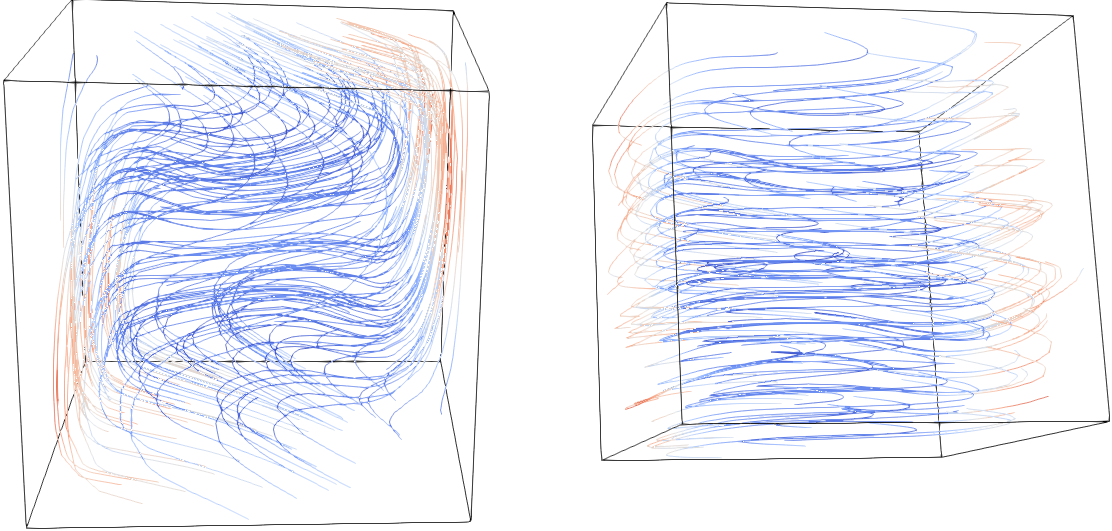}\\ 
	{\small (d) Streamlines for magnetic field with view from front (left) and top (right) at $\Ra = 10^5$.}\\
\end{figure}
\begin{figure}[htbp!]
	\centering
	\includegraphics[width=5.5cm]{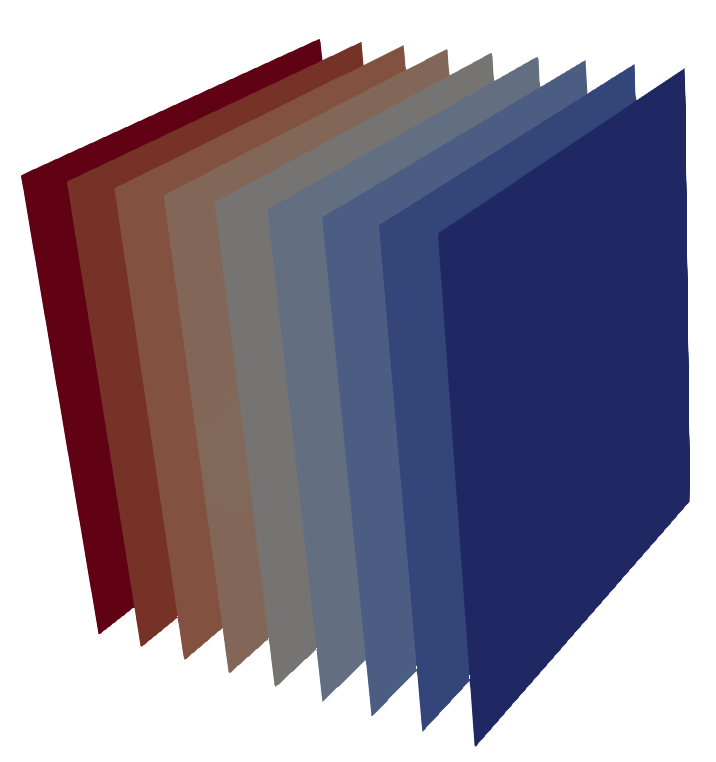} \qquad
	\includegraphics[width=5.5cm]{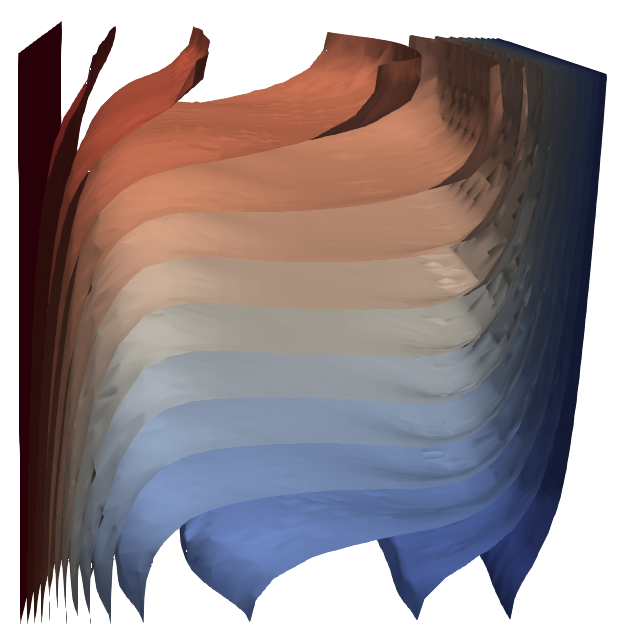}\\ 
	{\small (e) Temperature contours at $\Ra = 1$. \qquad (f) Temperature contours at $\Ra = 10^5$.}
	\caption{Streamlines for the three-dimensional stationary magnetic double glazing problem for different values of $\Ra$ with $\Pm=\Pr=1$.\label{fig:3d_hc}}
\end{figure}

\subsection{Magnetic channel cooling problem}
Finally, we consider a cooling problem where the magnetic field is applied perpendicular to the temperature gradient. We choose the channel to be $\Omega = (0,10) \times (-1,1)$ with the boundary conditions
\begin{equation}
	\u = 
		\begin{cases}
		1, & \text{ on } \{x_1 = 0\}, \\
		0, & \text{ else on } \partial \Omega,
	\end{cases},  
   \quad \nabla \u \cdot \n = \mathbf{0}  \, \text{  on } \{x_1=10\},
    \quad \B \cdot \n = (0,1)^\top \n, \quad E = 0,
\end{equation}
\begin{equation}
    \quad \nabla \theta \cdot \n = 0  \, \text{  on } \{x_1=10\}, \quad 
    	\theta = 
    \begin{cases}
    	1, & \text{ on } \{x_1 < 1\} \cap \partial \Omega, \\
    	-x_1+2, &\text{ on }  \{1<x_1<2\}\cap \partial \Omega,\\
    	0, & \text{ on } \{x_1 > 2\}\cap\{x_2=-1,1\}\cap \partial \Omega.
    \end{cases}
\end{equation}
For this problem, we use Scott--Vogelius elements of order 2 on barycentrically refined grids. We use a base mesh of $90\times 16$ cells and 4 levels of refinement with 62.9 million degrees of freedom. We start by reporting results for the stationary problem in Table \ref{tab:cc_pr_eq_pm2d_stat}. The linear iteration counts are again very well controlled for $\Ra$ up to 100,000. As before, the the nonlinear solver failed to converge for the cases of $\Pr=0.1$ and $\Pr=0.01$ at $\Ra=100$,000. 

\begin{table}[htbp!]
	\centering
	\begin{tabular}{r|cccc}
		\toprule
		$\Pr\backslash\Ra$ & 1  & 100 &  10,000 &  100,000 \\
		\midrule
		1.0  &  ( 3) 3.3 &  ( 3) 3.7 &  ( 6) 6.3 &  ( 7) 8.1 \\
		0.1  &  ( 3) 3.0 &  ( 3) 3.0 &  ( 6) 4.7 &  NF  \\
		0.01 &  ( 4) 2.8 &  ( 2) 2.5 &  ( 8) 5.9 &  NF \\
		\bottomrule
	\end{tabular}
	\caption{Iteration counts for the two-dimensional stationary magnetic cooling channel problem for $\Pm=\Pr$.\label{tab:cc_pr_eq_pm2d_stat}}
\end{table}

The plots in Figure \ref{fig:2d_cooling} demonstrate the effect of different values of Ra on the flow of the fluid. For $\Ra=1$, we mainly see a straight flow of the fluid in $y$-direction. Similarly, the magnetic field lines mainly point in the direction of $(0,1)^\top$ and the temperature profile is mainly determined by the applied boundary conditions with a linear decrease between $x_1=1$ and $x_1=2$. For increasing Rayleigh number one can see a rotational pattern occurring in the left part of the channel at $\Ra=1{,}000$ for the velocity. Further increasing $\Ra$ to 10,000 stretches this pattern in the $x$-direction and high velocities mainly occur near the inlet. The temperature starts to decrease in the interior more quickly and high temperatures are smeared out in the initial upper half. 
The magnetic field lines adopt an S-shaped form in the right part of the channel while at the inlet the lines have a similar pattern to the velocity field. 

Note that we consider $\Pr \leq 1$ in this problem for which the thermal boundary layers are thicker than the velocity boundary layers. According to Figure \ref{fig:2d_cooling} we seem to fully resolve the velocity boundary layers with our mesh size.
\begin{figure}
	\begin{subfigure}{\textwidth}
		\centering
		\begin{tabular}{c}
			\includegraphics[width=12cm]{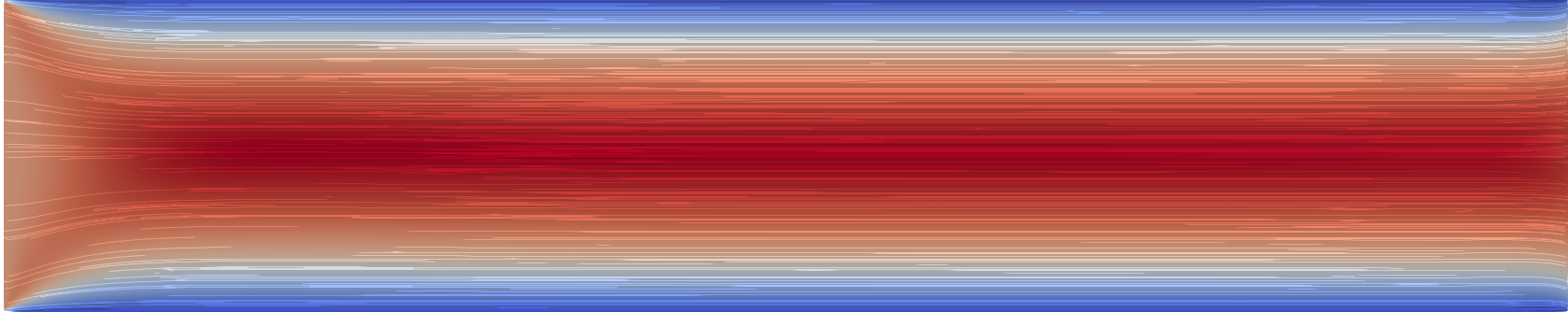} \\
			\includegraphics[width=12cm]{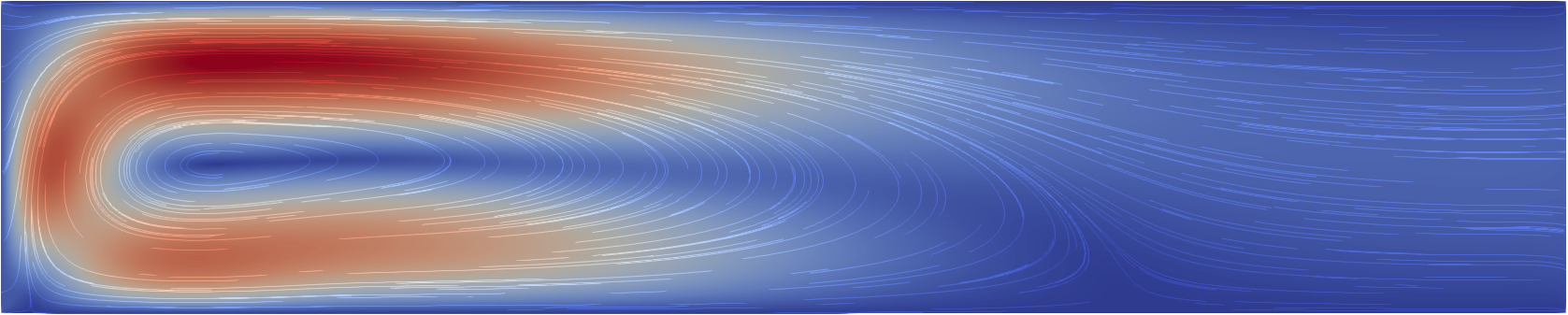} \\
			\includegraphics[width=12cm]{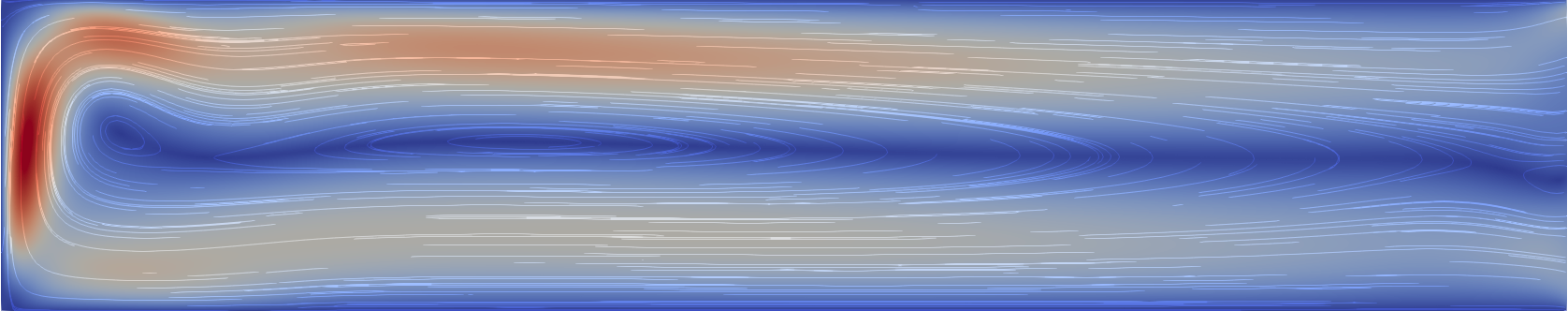} \\
		\end{tabular}
		\subcaption{ Streamlines for velocity at $\Ra = 1$, $\Ra=1,$000 and $\Ra=10,$000.}
	\end{subfigure}
	
	\vspace{1.0cm}
	\begin{subfigure}{\textwidth}
		\centering
		\begin{tabular}{c}
			\includegraphics[width=12cm]{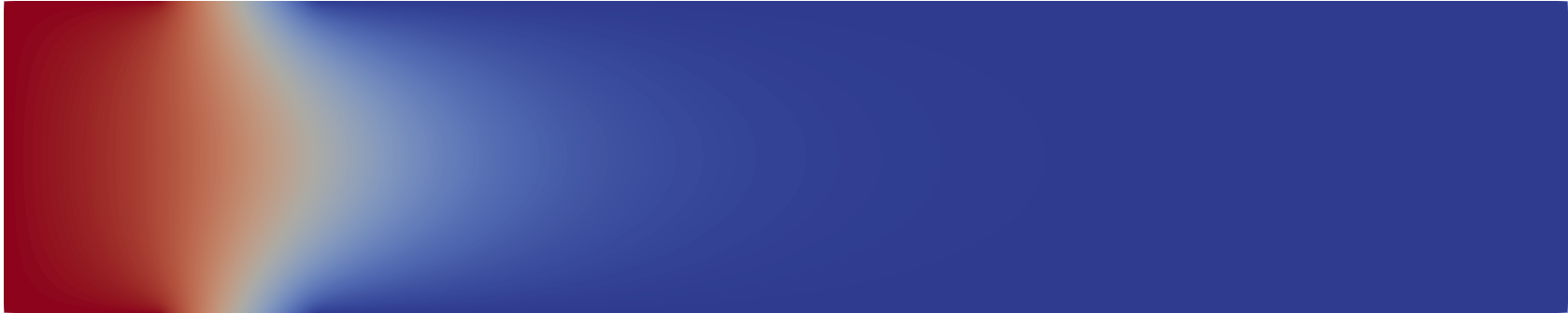} \\
			\includegraphics[width=12cm]{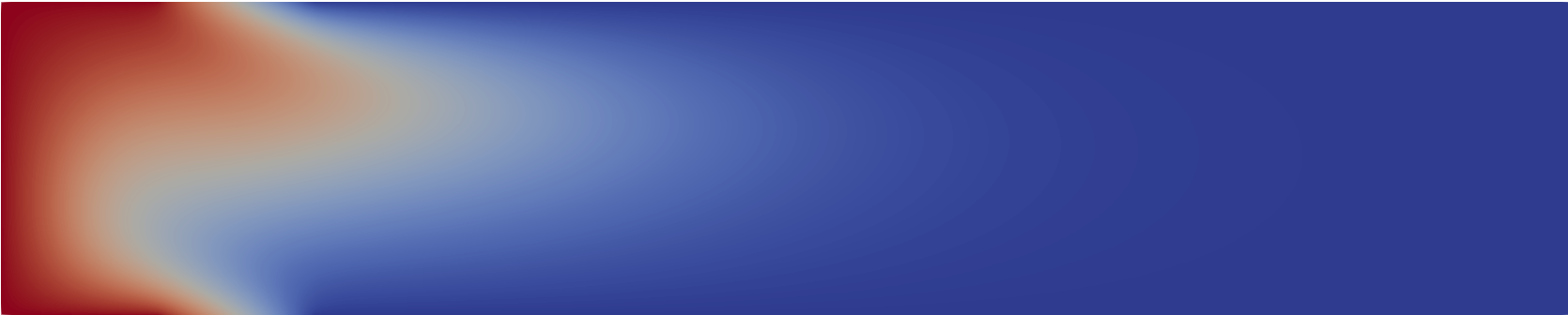} \\
			\includegraphics[width=12cm]{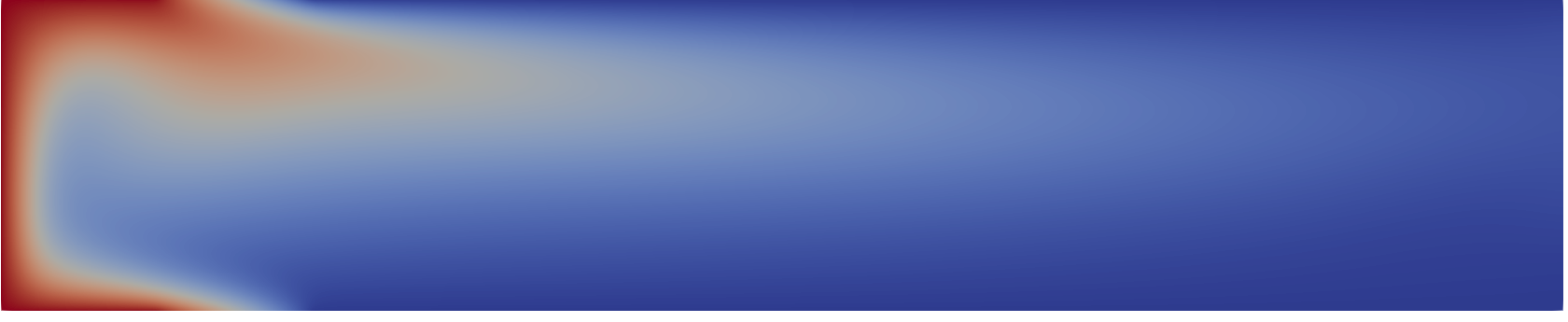} \\
		\end{tabular}
		\subcaption{Temperature at $\Ra = 1$, $\Ra=1,$000 and $\Ra=10,$000.}
	\end{subfigure}
\end{figure}
\begin{figure}[ht!]
	\ContinuedFloat
	\begin{subfigure}{\textwidth}
		\centering
		\begin{tabular}{c}
			\includegraphics[width=12cm]{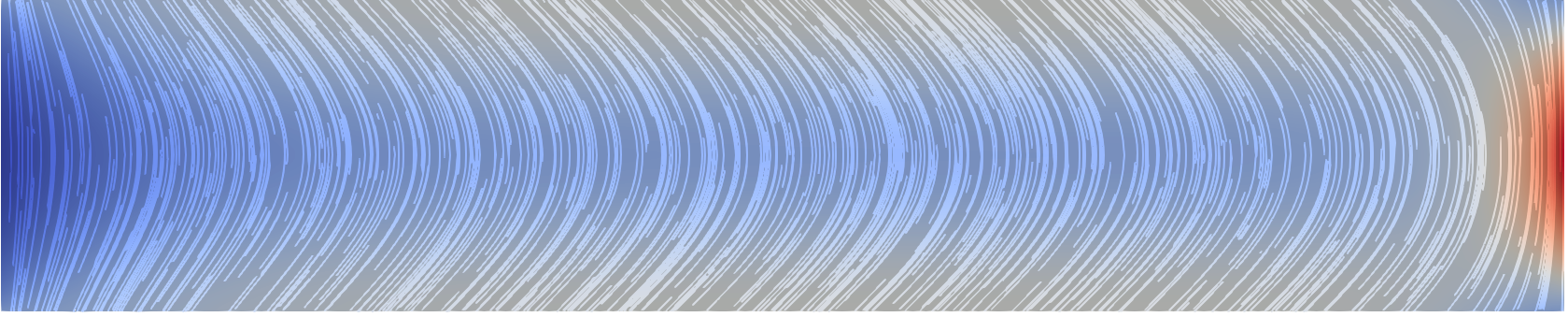} \\
			\includegraphics[width=12cm]{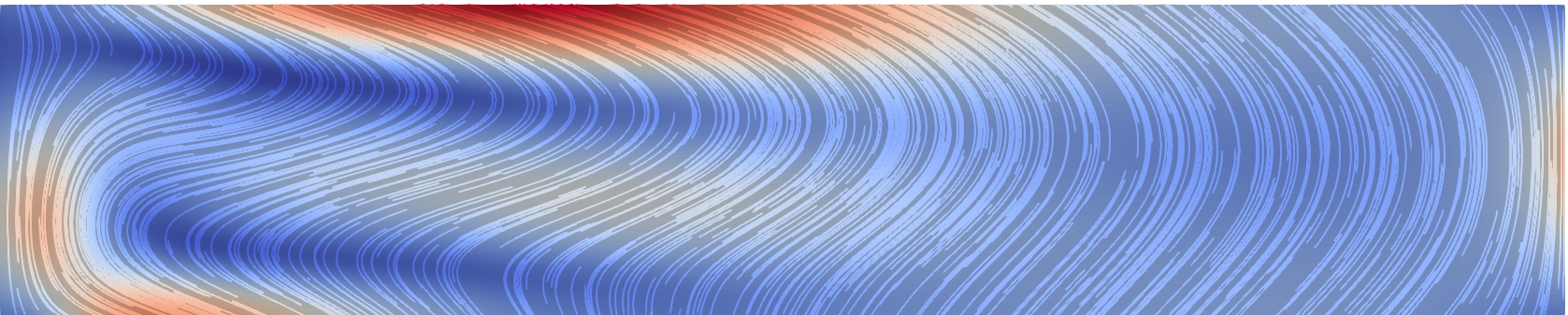} \\
			\includegraphics[width=12cm]{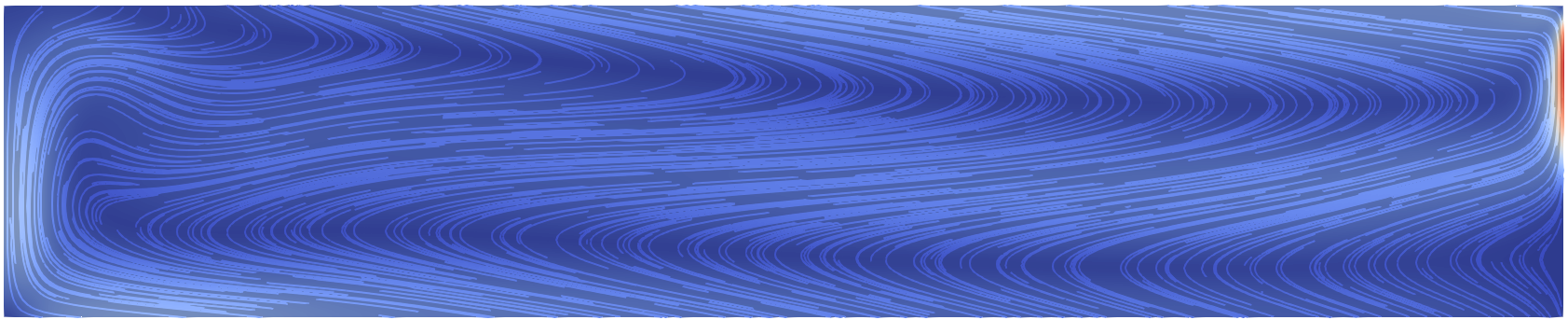} \\
		\end{tabular}
		\subcaption{Streamlines for magnetic field at $\Ra = 1$, $\Ra=1,$000 and $\Ra=10,$000.}
	\end{subfigure}
	\caption{Streamline plots for the stationary magnetic cooling channel problem.\label{fig:2d_cooling}}
\end{figure}

We also investigated the time-dependent version of this problem. We chose the L-stable BDF2 time-stepping method, where the first time step was computed with Crank-Nicolson. We iterated until a final time of $T=1$ with a step size of $\Delta t= 0.01$. Overall, we see well-controlled iteration numbers. The iteration numbers seem to decrease the smaller $\Pr$ is. Since the transient problem does not rely on parameter continuation, we were able to to compute results for a larger range of parameters without observing nonlinear convergence issues. Only in a few cases the first time step struggled to converge in the nonlinear iteration. In this case, we ran the first ten time steps with a decreased step size of $\Delta t= 0.001$ to overcome this problem.

\begin{table}[htbp!]
	\centering
	\begin{tabular}{r|ccc}
		\toprule
		$\Pr\backslash\Ra$ & 1  &   10,000 &  100,000 \\
		\midrule
		1.0 & (1.0) 1.6 & (2.1) 3.1 & (2.8) 5.5 \\
		0.1 & (1.0) 1.1 &  (1.9) 1.9 & (3.1) 4.3 \\
		0.01 & (1.0) 1.1 & (1.9) 1.5 & (2.8) 4.3 \\
		0.001 & (1.0) 1.0 & (1.1) 1.3 & (2.1) 2.5 \\
		\bottomrule
	\end{tabular}
	\caption{Iteration counts for the two-dimensional transient magnetic cooling channel problem for $\Pm=\Pr$.\label{tab:cc_pr_neq_pm2d_time}}
\end{table}

\chapter{Conclusion and Outlook}\label{sec:conclusionandoutlook}

In this thesis we investigated parameter-robust preconditioners and structure-preserving discretisations for several MHD models. 

In Chapter \ref{chap:2}, we have presented scalable block preconditioners for an augmented Lagrangian formulation of the incompressible resistive MHD equations that exhibit parameter-robust iteration counts in most cases. We described how to control the outer Schur complement of a Picard-type and full Newton linearisation and introduced a special monolithic multigrid method to solve the electromagnetic block. This method shows very good robustness with respect to $\Re$ and $\S$ in both the stationary and transient settings. The linear solver is also fully $\Rem$-robust in two dimensions; in three dimensions, it is able to efficiently compute results for higher parameters than was previously possible. Furthermore, our solvers allow the use of fully implicit methods for time-dependent problems. 

We aim to include stabilisation techniques for high magnetic Reynolds numbers in future work and further investigate how to develop a robust multigrid method for the problem including the term $\vcurl(\u^n \times \B)$. This would enable a more robust solver for the most difficult case of stationary problems in three dimensions at high magnetic Reynolds numbers.

In Chapter \ref{chap:3}, we have introduced a structure-preserving finite element discretisation for the incompressible Hall MHD equations that enforces $\nabla \cdot \B = 0$ precisely and proved the well-posedness and convergence of a Picard-type linearisation. Furthermore, we presented formulations that preserve the energy, magnetic and hybrid helicity precisely on the discrete level in the ideal limit for two types of boundary conditions. Finally, we investigated a block preconditioning strategy that works well as long as $\RH$ and $S$ or $\Rem$ are not chosen too high at the same time.

In future work, we want to improve the robustness of our solver with respect to the Hall parameter, especially in the 2.5-dimensional case where we currently use a direct solver to solve the electromagnetic block. This would also enable us to consider the island coalescence problem on much finer grids.  Furthermore, we are curious to investigate further if there exists a scheme that also preserves the hybrid helicity at the same time as the other quantities for the case $\u\cdot \n = 0$.

In Chapter \ref{chap:4}, we investigated anisothermal MHD models. In the first part, we performed a bifurcation analysis for a two-dimensional magnetic Rayleigh-B\'enard problem. We used deflated continuation to compute a bifurcation diagram over the parameter $0 \leq \Ra \leq 100{,}000$ at a high coupling number of $S=1{,}000$. In order to provide useful initial guesses in this regime to find complex solution patterns and disconnected branches, we started with a bifurcation analysis over $0 \leq \Ra \leq 100{,}000$  for $S=1$. We  then proceeded to use the coupling number $S$ as a bifurcation parameter at $\Ra = 100{,}000$ to construct initial guesses at $\Ra = 100{,}000$ and $S=1{,}000$. We demonstrated how an increasing coupling number can stabilise an unstable branch in one case. Moreover, we were able to find a disconnected branch with our approach that was not discovered by starting directly at $S=1{,}000$.

In future work, we want to further investigate the two branches in the diagram for $0 \leq \Ra \leq 100{,}000$ at $S=1{,}000$ that showed a somewhat surprising evolution to us. Furthermore, it would also be interesting to study the dependence on the magnetic Prandtl number $\Pm$ on the bifurcation analysis. 

In the second part of Chapter 4, we extended the augmented Lagrangian block preconditioner from Chapter \ref{chap:2} to the anisothermal version and verified numerically the good robustness of our scheme with respect to the unknowns $\Ra, \Pr, \Pm$ and $S$ in the two-dimensional stationary and transient settings. Similar to the three-dimensional results for the standard MHD equations, a further investigation is needed to make this solver fully robust with respect to high $\Pm$ and small $\Pr$ in the stationary three-dimensional setting.
\newpage

\singlespacing
\bibliography{_literature}

\end{document}